\documentclass[english,reqno,11pt,empty]{amsart}
\usepackage{amsmath, amsthm, amsfonts}
\usepackage{hyperref}
\hypersetup{
	colorlinks=true,
		citecolor=blue!60!black,
		linkcolor=red!60!black,
		urlcolor=green!40!black,
		filecolor=yellow!50!black,
	breaklinks=true,
	pdfpagemode=UseNone,
	bookmarksopen=false,
}
\usepackage{tcolorbox}
\usepackage{amsmath,amsthm,amssymb}
\usepackage{amscd,indentfirst,epsfig}
\usepackage{latexsym}
\usepackage{times}
\usepackage{enumerate}
\usepackage{mathrsfs}
\usepackage{stmaryrd}
\usepackage{amsopn}
\usepackage{amsmath}
\usepackage{amssymb,dsfont,mathtools}
\usepackage{amsfonts,bm}
\usepackage{amsbsy,amsmath}
\usepackage{amscd}
\usepackage{xcolor}
\usepackage{mathtools}

\usepackage[mono=false]{libertine}
\usepackage[T1]{fontenc}
\usepackage{amsthm}
\usepackage[normalem]{ulem} 
\usepackage[cal=euler, scr=boondoxo]{}
\usepackage{microtype}

\usepackage{numprint}

 \usepackage[mediumspace,mediumqspace,Grey,squaren]{SIunits}

 \usepackage{array}

\usepackage{latexsym}
\usepackage{enumerate}
\usepackage{mathrsfs}
\usepackage{stmaryrd}
\usepackage{amsopn}
\usepackage{amsmath,amsthm,amssymb}
\usepackage{mathtools}
\usepackage{amsfonts}
\usepackage{soul}
\usepackage{amsbsy}
\usepackage{amscd,indentfirst,epsfig}
\usepackage{amsfonts,latexsym,verbatim,amsbsy}

\usepackage{pgf,tikz}
\usepackage{amsmath,amsthm,amssymb}
\usepackage{amscd,indentfirst,epsfig}
\usepackage{latexsym}
\usepackage{times}
\usepackage{enumerate}
\usepackage{mathrsfs}
\usepackage{stmaryrd}
\usepackage{amsopn}
\usepackage{amsmath}
\usepackage{amssymb}
\usepackage{amsfonts,bm,dsfont}

\usepackage{dsfont,mathtools}
\usepackage{enumerate}
\usepackage{mathrsfs,esint}
\usepackage{stmaryrd}

\usepackage{amsopn,bm}
\usepackage{amsmath}
\usepackage{amssymb}
\usepackage{amsfonts}
\usepackage{amsbsy}
\usepackage{amscd,indentfirst,epsfig}
\usepackage{amsfonts,amsmath,latexsym,amssymb,verbatim,amsbsy,textcomp}
\usepackage{amsthm}
\usepackage{dsfont}
\usepackage{colordvi}
\usepackage{pstricks}
\usepackage{csquotes}

\usepackage{natbib}
\def \leq {\leqslant}                                                                
\def \geq {\geqslant}                                                                
\def\ind#1{\lower5pt\hbox{$\scriptstyle #1$}}                                        
\def \d {\mathrm{d}}                                                                 
\def \ds {\displaystyle}                                                             
\def \R{\mathbb R}                                                                   
\def\S{\mathbb S}                                                                    
\def\Z{\mathbb Z}                                                                    
\def\N{\mathbb N}                                                                    
\def \C{\mathbb C}																	 
\def\Cs{\mathcal C}																	 
\def\A{\mathcal A}																	 
\def\D{\mathscr D}																	 
\def\P{\mathcal P}																	 
\def\M{\mathcal M}																	 
\def \F {\mathcal{F}}																 
\def \T {\mathbb{T}}                                                                 
\numberwithin{equation}{section}                                                     
\newcommand{\vertiii}[1]{{\left\vert\kern-0.25ex\left\vert\kern-0.25ex\left\vert #1  
    \right\vert\kern-0.25ex\right\vert\kern-0.25ex\right\vert}}                      
\newcommand{\verti}[1]{{\left\vert\kern-0.25ex\left\vert\kern-0.25ex\left\vert #1    
    \right\vert\kern-0.25ex\right\vert\kern-0.25ex\right\vert}}						 
\def\e{\varepsilon}      															    
\def \m {\bm{\varpi}}															        
\def \ho {h^{0}} 																	    
\def \hu {h^{1}}																		
\def \huu {\Psi}																		
\def \Mo {\Delta_{0}}																	
\def \Rs {\mathcal{R}}																	
\def \cS {\mathcal{S}}																	
\def \sE {\bm{\Xi}}																		
\def \ra {\Big\rangle}																	
\def \la {\Big\langle}																	
\def \en {\vartheta}																	
\def \vE {\theta}																		
\def \lae {\lambda_{\e}}													            
\def \re {\alpha}																		

\def\E{\mathcal{E}} 																	
\def \B{\mathcal{B}}                                         							
\def \Y {\mathbb{Y}}                                         							
\def\W {\mathbb{W}} 
\def \H {\mathcal{H}}                                        							

\def\vet{v_{\ast}}                                           							
\def \vb {v_\ast}                                           							
\def \Q{\mathcal{Q}}                                        							
\def \LL {\mathscr{L}_{\re}} 							 							    
\def \LLe {\mathscr{L}_{\re(\e)}}						 							    
\def \G {\mathcal{G}}							 										

\newtheorem{theo}{Theorem}[section]                          							
\newtheorem{prop}[theo]{Proposition}                         							
\newtheorem{cor}[theo]{Corollary}                            							
\newtheorem{lem}[theo]{Lemma}                                							
\newtheorem{hyp}[theo]{Assumption}                    									
\newtheorem{defi}[theo]{Definition}                          							
\newtheorem{nb}[theo]{Remark}                                							
 
\begin{document}

\title[Fluid dynamic limit of Boltzmann equation for granular hard--spheres]{Fluid dynamic limit of Boltzmann equation for granular hard--spheres in a nearly elastic regime}

\author{Ricardo {\sc Alonso}}
\address{$^1$Texas A\&M University at Qatar, Science Department, Education City, Doha, Qatar.}
\address{$^2$Departamento de Matem\'atica, PUC-Rio, Rio de Janeiro, Brasil.} 
\email{ricardo.alonso@qatar.tamu.edu}

\author{Bertrand {\sc Lods}}

\address{Universit\`{a} degli Studi di Torino \& Collegio Carlo Alberto, Department of Economics and Statistics, Corso Unione Sovietica, 218/bis, 10134 Torino, Italy.}\email{bertrand.lods@unito.it}

\author{Isabelle {\sc Tristani}}
\address{D\'epartement de Math\'ematiques et Applications, \'Ecole Normale Sup\'erieure, CNRS, PSL University, 75005 Paris, France.}\email{isabelle.tristani@ens.fr}

\maketitle
\begin{abstract}
 In this paper, we provide the first rigorous derivation of hydrodynamic equations from the Boltzmann equation for inelastic hard spheres with small inelasticity. The hydrodynamic system that we obtain is an incompressible Navier-Stokes-Fourier system with self-consistent forcing terms and, to our knowledge, it is thus the first hydrodynamic system that properly describes rapid granular flows {consistent with the kinetic formulation}. To this end, we write our Boltzmann equation in a non dimensional form using the dimensionless Knudsen number which is intended to be sent to~$0$. There are several difficulties in such derivation, the first one coming from the fact that the original Boltzmann equation is free-cooling and, thus, requires a self-similar change of variables to introduce an homogeneous steady state. {Such a homogeneous state} is not explicit and is heavy-tailed, which is a major obstacle to {adapting} energy estimates and spectral analysis. Additionally, a central challenge is to understand the relation between the restitution coefficient, which quantifies the energy loss at the microscopic level, and the Knudsen number. This is achieved by identifying the correct \emph{nearly elastic regime} to capture nontrivial hydrodynamic behavior. We are, then, able to prove exponential stability uniformly with respect to the Knudsen number for solutions of the rescaled Boltzmann equation in a close to equilibrium regime. Finally, we prove that solutions to the Boltzmann equation converge {in a specific weak sense} towards a hydrodynamic limit which depends on time and space variables only through macroscopic quantities. {Such macroscopic quantities are} solutions to a suitable modification of the incompressible Navier-Stokes-Fourier system {which appears to be new in this context.}
\end{abstract}
\vspace{.5cm}

\vspace{0.4cm}

\bigskip
\bigskip
{
  \hypersetup{linkcolor=blue}
 \tableofcontents
}


\section{Introduction}

The derivation of hydrodynamic models from suitable nonlinear (and possibly non conservative) kinetic equations is a challenging problem which has attracted a lot of attention in the recent years. Besides the well-documented literature dealing with the Boltzmann equation (see Section \ref{sec:liter} hereafter), a large variety of new kinetic models and limiting processes have been considered, spanning from high friction regimes for kinetic models of swarwing (see e.g. \cite{karper,figalli} for the Cucker-Smale model) to the reaction-diffusion limit for Fitzhugh-Nagumo kinetic equations \cite{crevat}. For  fluid-kinetic systems, the literature is even more important, we mention simply here the works \cite{goudona,goudonb} dealing with light or fine particles regimes for the Vlasov-Navier-Stokes system and refer to \cite{daniel} for the more recent advances on the subject. We also mention the challenging study of gas of charged particles submitted to electro-magnetic forces (Vlasov-Maxwell-Boltzmann system) for which several incompressible fluid limits have been derived recently in the monograph \cite{arsenio}.

We consider in the present paper the paradigmatic example of non conservative kinetic equations given by the  Boltzmann equation for inelastic hard spheres. In a regime of small inelasticity, we derive in a suitable hydrodynamic limit an incompressible Navier-Stokes-Fourier system with self-consistent forcing terms. This provides, to the best of our knowledge, the first rigorous derivation of hydrodynamic system  from kinetic  granular flows in physical dimension $d \geq 2$. 

\subsection{Multiscale descriptions of granular gases}\label{sec:multiscale}
Granular materials are ubiquitous in nature and understanding the behaviour of granular matter is a relevant challenge from both the physics and mathematics viewpoints. Various descriptions of granular matter have been proposed in the literature, see {\cite{garzo}}. An especially relevant one consists in viewing granular systems as clusters of a large number of discrete macroscopic particles (with size exceeding \unit{1}{\micro\meter}, significantly larger than the one of a typical particle described in classical kinetic theory) suffering dissipative interactions. One speaks then of \emph{rapid granular flows} or \emph{gaseous granular matter}. If the number of particles is large enough, it is then common to adopt a kinetic modelling based upon suitable modification of the Boltzmann equation. As usual in kinetic theory, it is then particularly relevant to deduce from this kinetic description the fluid behaviour of the system. This means, roughly speaking, that we look at the granular gas at a scale larger than the mesoscopic one and aim to capture the hydrodynamical features of it through the evolution of macroscopic quantities like \emph{density, bulk velocity} and \emph{temperature} of the gas which satisfy suitable hydrodynamics equations. \smallskip

\textit{\textbf{One of the main objects of the present work is to make a \emph{first rigorous link} between these two co-existing descriptions by deriving a suitable modification of incompressible Navier-Stokes equation from the Boltzmann equation for inelastic hard-spheres as the Knudsen number goes to zero.
}}
\smallskip


Recall that the Knudsen number $\e$ is proportional to the mean free path between collisions and in order to derive hydrodynamic equations from the Boltzmann equation, the usual strategy consists, roughly speaking, in performing a perturbation analysis in the limit $\e \to0$ (meaning that the mean free path is negligible when compared to the typical physical scale length). We point out that these questions are perfectly understood in the elastic case (\textit{molecular gases}) for which rigorous results on the hydrodynamic limits of the Boltzmann equation have been obtained, we refer to the next Section \ref{sec:liter} for more details and to \cite{SR} for an up-to-date review.

\medskip

The picture in the context of granular gases is quite different.  In fact, a satisfying hydrodynamic equation that properly describes rapid granular flows is still a controversial issue among the physics community.  The continuous loss of kinetic energy makes granular gases an \emph{open system} as far as thermodynamics is concerned. Moreover, no non-trivial steady states exist in granular gases without an external energy supply which makes granular gases a prototype of \emph{non-equilibrium} systems. This is  an important obstacle in the derivation of hydrodynamical equations from the kinetic description since it is expected that \emph{equilibrium states} play the role of the typical \emph{hydrodynamic solution} where time-space dependence of the single-particle distribution function $F(t,x,v)$ occurs only through suitable hydrodynamic fields like density $\varrho(t,x)$, bulk velocity $u(t,x)$, and temperature $\vE(t,x)$.  An additional difficulty is related to the size of particles and scale separation. Recall that granular gases involve \emph{macroscopic particles} whose size is much larger than the one described by the usual Boltzmann equation with elastic interactions referred to as \emph{molecular gases}.  As the hydrodynamic description occurs on large time scales (compared to the mean free time) and on large spatial scales (compared to the mean free path) the mesoscopic -- continuum scale separation is problematic to justify in full generality for granular gases.  We refer to {\cite[Section 3.1, p. 102]{garzo}} for more details on this point and observe here that the main concern is related to the time scale induced by the evolution of the temperature (see \eqref{eq:TeHaff} herafter). In particular, as observed in {\cite{garzo}}, this problem can only be answered with a fine spectral analysis of the linearized Boltzmann equation that ensures that the $d + 2$ hydrodynamic modes associated to density, velocity and temperature decay more slowly than the remaining kinetic excitations at large times. This is the only way that the hydrodynamic excitations emerge as the dominant dynamics.   All these physically grounded obstacles make the derivation of hydrodynamic equations from the Boltzmann equation associated to granular gases a reputedly challenging open problem. Quoting~{\cite{brey}}: 
\begin{displayquote}
\textit{``the context of the hydrodynamic equations remains uncertain. What are the relevant space and time scales? How much inelasticity can be described in this way?''}
\end{displayquote}
\textit{\textbf{The present paper is, to the best of our knowledge, the first rigorous answer to these relevant problems}, at least in dimension $d \geq 2$.} We already mentioned that the key point in our analysis is to identify the correct regime which allows to answer these questions: the \emph{nearly elastic} one.  In this regime the energy dissipation rate in the systems happens in a controlled fashion since the inelasticity parameter is compensated accordingly to the number of collisions per time unit.  This process mimics viscoelasticity as particle collisions become more elastic as the collision dissipation mechanism increases in the limit $\e\to0$ (see Assumption \ref{hyp:re} below).  In this way, we are able to consider a re-scaling of the kinetic equation in which a peculiar intermediate asymptotic emerges and prevents the total cooling of the granular gas.

\smallskip

Other regimes can be considered depending on the rate at which kinetic energy is dissipated; for example, an interesting regime is the \emph{mono-kinetic} one which considers the extreme case of infinite energy dissipation rate.  In this way, the limit is formally described by enforcing a Dirac mass solution in the kinetic equation yielding the \emph{pressureless Euler system} (corresponding to sticky particles).  Such a regime has been rigorously addressed in the one-dimensional framework in the interesting contribution {\cite{jabin}}. It is an open question to extend such analysis to higher dimensions {since the approach of {\cite{jabin}} uses the so-called Bony functional which is a tool specifically tailored for 1D kinetic equations.}

\subsection{The Boltzmann equation for granular gases}
We consider here the (freely cooling) Boltzmann equation which provides a statistical description of  identical smooth hard spheres suffering binary and \emph{inelastic collisions}:
\begin{equation}\label{Bol-}
\partial_{t}F(t,x,v) + v\cdot \nabla_{x} F(t,x,v)=\Q_{\re}(F,F)\end{equation}
supplemented with initial condition $F(0,x,v)=F_{\mathrm{in}}(x,v)$, where $F(t,x,v)$ is the density of granular particles having position $x \in \T^{d}$ and velocity $v \in \R^{d}$ at time $t\geq0$  and $d \geq 2$. We consider here for simplicity the case of \emph{flat torus} 
\begin{equation}\label{torus}
\T_{\ell}^{d}=\R^{d}\slash (2\pi\,\ell\,\Z)^{d}\end{equation}
for some typical length-scale $\ell >0$. This corresponds to periodic boundary conditions:
$$F(t,x+2\pi\,\ell\bm{e}_{i},v)=F(t,x,v) \qquad i=1,\ldots,d$$
where $\bm{e}_{i}$ is the $i$-th vector of the canonical basis of $\R^{d}.$ The collision operator $\Q_{\re}$ is defined in weak form as
\begin{equation}\label{co:weak}
\begin{split}
 \int_{\R^{d}} \Q_{\re} (g,f)(v)\, \psi(v)\d v  =  \frac{1}{2}  \int_{\R^{2d}} f(v)\,g(v_{\ast})\,|v-v_{\ast}|
\mathcal{A}_{\re}[\psi](v,v_{\ast})\d v_{\ast}\d v,
\end{split}
\end{equation}
where
\begin{equation}    \label{coll:psi} \mathcal{A}_{\re}[\psi](v,v_{\ast}) =
    \int_{\S^{d-1}}(\psi(v')+\psi(v_{\ast}')-\psi(v)-\psi(v_{\ast}))b(\sigma \cdot \widehat{u})\d{\sigma},
\end{equation}
and the post-collisional velocities $(v',v_{\ast}')$ are given
by
\begin{equation}\begin{split}
\label{co:transf}
  v'=v+\frac{1+\re}{4}\,(|u|\sigma-u),&
\qquad  v_{\ast}'=v_{\ast}-\frac{1+\re}{4}\,(|u|\sigma-u),\\
\text{where} \qquad u=v-v_{\ast},& \qquad \widehat{u}=\frac{u}{|u|}.
\end{split}
\end{equation}
Here, $\d\sigma$ denotes the Lebesgue measure on $\S^{d-1}$ and the angular part $b=b(\cos\theta)$ 
of the collision kernel appearing in \eqref{coll:psi} is a non-negative measurable mapping integrable over~$\S^{d-1}$.  There is no loss of generality assuming
$$\int_{\S^{d-1}}b(\sigma \cdot \widehat{u})\d{\sigma}=1, \qquad \forall \, \widehat{u} \in \S^{d-1}.$$
An additional technical assumption on the angular kernel $b(\cdot)$ will be needed in the sequel, namely, in the rest of the paper, we suppose that there exists $r>2$ such that 
\begin{equation}\label{eq:conditionb}
\int_{-1}^{1}b(s)\left[\left(1-s\right)^{\frac{d-3r}{2r}}\left(1+s\right)^{\frac{d-3}{2}}+\left(1+s\right)^{\frac{d-3r}{2r}}\left(1-s\right)^{\frac{d-3}{2}}\right]\d s < \infty.\end{equation} We just mention here that we need this integral to be finite to get bounds on the bilinear operator~$\Q_{\re}$ on $L^{r}_{v}$ for $r \geq 2$ (see Theorem~\ref{theo:Ricardo}). 
As a consequence, if $b \in L^{\infty}(\S^{d-1})$ and $d \geq 3$, condition~\eqref{eq:conditionb} holds true for any $r \in [2,3)$. In particular, our assumption includes the case of hard spheres in dimension $d=3$.

The fundamental distinction between the classical elastic Boltzmann equation and that associated to granular gases lies in the role of the parameter $\re \in (0,1)$, the \emph{coefficient of restitution}.  This coefficient is given by the ratio between the magnitude of the normal component (along the line of separation between the centers of the two spheres at contact) of the relative velocity after and before the collision (see Appendix \ref{Sec21} for the detailed microscopic velocities). The case $\re = 1$ corresponds to perfectly elastic collisions  where kinetic energy  is conserved. However, when $\re  < 1$, part of the kinetic energy of the relative motion is lost since
\begin{equation} \label{eq:lossenergy}
|v'|^{2}+|\vb'|^{2}-|v|^{2}-|\vb|^{2}=-\frac{1-\re^{2}}{4}|u|^{2}\,\left(1-\sigma \cdot \widehat{u}\right) \leq 0.
\end{equation}
It is assumed in this work that $\alpha$ is independent of the relative velocity $u$ (refer to {\cite{A}},~{\cite{ALCMP}}, and {\cite{ALT}} for the viscoelastic restitution coefficient case).  Notice that the microscopic description \eqref{co:transf} preserves the momentum
$$v'+\vb'=v+\vb$$
and, taking $\psi=1$ and then $\psi(v)=v$ in \eqref{co:weak} yields to the following conservation of macroscopic density and bulk velocity
$$\dfrac{\d}{\d t}\bm{R}(t):=\dfrac{\d}{\d t}\int_{\R^{d}\times \T^{d}_{\ell}}F(t,x,v)\d v\d x=0, \qquad \frac{\d}{\d t}\bm{U}(t):=\frac{\d}{\d t}\int_{\R^{d}\times\T^{d}_{\ell}}v F(t,x,v)\d v\d x=0\,.$$
Consequently, there is no loss of generality in assuming that
$$\bm{R}(t)=\bm{R}(0)=1, \qquad \bm{U}(t)=\bm{U}(0)=0 \qquad \forall t \geq0.$$
As mentioned, the main contrast between elastic and inelastic gases is that in the latter the \emph{granular temperature}
$$\bm{T}(t):=\frac{1}{|\T^{d}_{\ell}|}\int_{\R^{d}\times \T^{d}_{\ell}}|v|^{2}F(t,x,v)\d v\d x$$
is constantly decreasing
$$\dfrac{\d}{\d t}\bm{T}(t)=-(1-\re^{2})\mathcal{D}_{\re}(F(t),F(t)) \leq 0\,,\qquad \forall t \geq 0.$$
Here $\mathcal{D}_{\re}(g,g)$ denotes the normalised energy dissipation associated to $\Q_{\re}$, see {\cite{MiMo2}}, given by
\begin{equation}\label{eq:Dre}
\mathcal{D}_{\re}(g,g):=\frac{\gamma_{b}}{4}\int_{\T^{d}_{\ell}}\frac{\d x}{|\T^{d}_{\ell}|}\int_{\R^d \times \R^d}g(x,v)g(x,\vb)|v-\vb|^{3}\d v\d \vb
\end{equation}
with 
$$\gamma_{b}:=\int_{\S^{d-1}}\frac{1-\sigma \cdot \widehat{u}}{2}\,b(\sigma\cdot \widehat{u})\d \sigma=|\S^{d-2}|\int_{0}^{\pi}b(\cos\theta)\,\left(\sin\theta\right)^{d-2}\,\sin^{2}\left(\frac{\theta}{2}\right)\d\theta.$$
In fact, it is possible to show that 
$$\lim_{t\to\infty}\bm{T}(t)=0$$
which expresses the \emph{total cooling of granular gases}.  Determining the exact dissipation rate of the granular temperature is an important question known as \emph{Haff's law}, see {\cite{haff}}. 
 
\subsection{Navier-Stokes scaling}  To capture some hydrodynamic behaviour of the gas,  we need to write the above equation in \emph{nondimensional form} introducing the dimensionless Knudsen number
$$\e:=\dfrac{\text{ mean free path }}{\text{ spatial length-scale }}$$
which is assumed to be small. We introduce then a rescaling of time and space to capture the hydrodynamic limit and introduce the particle density
\begin{equation}\label{eq:Scaling}
F_{\e}(t,x,v)=F\left(\frac{t}{\e^{2}},\frac{x}{\e},v\right), \qquad t \geq 0.
\end{equation}
In this case, we choose for simplicity $\ell=\e$ in \eqref{torus} which ensures now that $F_{\e}$ is defined on $\R^{+}\times \T^{d}\times \R^{d}$ with $\T^{d}
=\T_{1}^{d}$. {From now on, we assume for simplicity that the torus $\T^{d}$ is equipped with the normalized Lebesgue measure, i.e. $|\T^{d}|= 1.$} 
 It is well-know that, in the classical elastic case, this scaling leads to the incompressible Navier-Stokes, however, other scalings are possible that yield different hydrodynamic models. Under such a scaling, the typical number of collisions per particle per time unit is~$\e^{-2}$, more specifically, $F_{\e}$ satisfies the rescaled Boltzmann equation
\begin{subequations}
\begin{equation}\label{Bol-e}
\e^{2}\partial_{t}F_{\e}(t,x,v) + \e\,v\cdot \nabla_{x} F_{\e}(t,x,v)=\Q_{\re}(F_{\e},F_{\e}), \qquad (x,v) \in \T^{d}\times\R^{d}\,,
\end{equation}
supplemented with initial condition
\begin{equation}\label{eq:init}
F_{\e}(0,x,v)=F_{\mathrm{in}}^{\e}(x,v)=F_{\mathrm{in}}(\tfrac x\e, v)\,.
\end{equation}
\end{subequations}
Conservation of mass and density is preserved under this scaling, consequently, we assume that 
$$\bm{R}_{\e}(t)=\int_{\R^{d}\times \T^{d}}F_{\e}(t,x,v)\d v\d x=1, \quad \bm{U}_{\e}(t)=\int_{\R^{d}\times \T^{d}}F_{\e}(t,x,v)v\d v\d x=0, \quad \forall\, t \geq 0,$$
whereas the cooling of the granular gas is now given by the equation
\begin{equation}\label{eq:TeHaff}
\frac{\d}{\d t}\bm{T}_{\e}(t)= -\frac{1-\re^{2}}{\e^{2}}\mathcal{D}_{\re}(F_{\e}(t),F_{\e}(t)),\end{equation}
where  $\bm{T}_{\e}(t)= \displaystyle \int_{\R^{d}\times\T^{d}}|v|^{2}F_{\e}(t,x,v)\d v\d x.$

\begin{nb}\label{nb:init}
{From now on we will always assume that 
$${\int_{\T^d \times \R^{d}}F_{\mathrm{in}}^{\e}(x,v)\left(\begin{array}{c}1 \\v \\|v|^{2}\end{array}\right)\d v \d x}=\left(\begin{array}{c}1 \\ 0 \\ E_{\mathrm{in}}\end{array}\right)\,$$
with $E_{\mathrm{in}} >0$ fixed and independent of $\e.$ It is important to emphasize that, in the sequel, all the threshold values on $\e$ and the various constants involved are actually depending \emph{only} on this initial choice.}\end{nb}
 
\subsection{Self-similar variable and homogeneous cooling state} Various forcing terms have been added to \eqref{Bol-e} depending on the underlying physics.  Forcing terms prevent the total cooling of the gas (heated bath, thermal bath, see {\cite{vill}} for details) since they act as an energy supply source to the system and induce the existence of a non-trivial steady state. These are, however, systems different from the free-cooling Boltzmann equation \eqref{Bol-e} that we aim to investigate here. 

To understand better this free-cooling scenario, it is still possible to introduce an intermediate asymptotics and a steady state to work with. This is done by performing a self-similar change of variables
\begin{subequations}\label{eq:ScalING}
\begin{equation}
\label{eq:Fescal}
F_{\e}(t,x,v)={V}_{\e}(t)^{d}f_{\e}\big(\tau_{\e}(t),x,V_{\e}(t)v\big)\,,
\end{equation}
with 
\begin{equation}\label{eq:tauVe}
\tau_{\e}(t):=\frac{1}{c_{\e}}\log(1+ c_{\e}\,t)\,, \quad V_{\e}(t)=(1+c_{\e}\,t)\,, \quad t \geq0, \;\, c_{\e} > 0\,.
\end{equation}
With the special choice 
\begin{equation}\label{eq:ce}
c_{\e}=\frac{1-\re}{\e^2}\,,
\end{equation}
\end{subequations}
we can prove that $f_{\e}$ satisfies 
\begin{equation}\label{BE0}
\e^{2}\partial_{t} f_{\e}(t,x,v)+\e v \cdot \nabla_{x} f_{\e}(t,x,v) + \kappa_{\re}\,\nabla_{v}\cdot (v f_{\e}(t,x,v))=  \Q_{\re}(f_{\e},f_{\e})\,,
\end{equation}
with initial condition 
$$f_{\e}(0,x,v)=F^{\e}_{\mathrm{in}}(x,v).$$ 
Here 
$$\kappa_{\re}=1-\re > 0, \qquad \forall \re \in (0,1).$$
The underlying drift term $\kappa_{\re}\nabla_{v}\cdot (v f(t,x,v))$ acts as an energy supply which prevents the total cooling down of the gas.  Indeed, it has been shown in a series of papers ({\cite{MiMo1,MiMo2,MiMo3}}) that there exists a \emph{spatially homogeneous} steady state  $G_{\re}$ to \eqref{BE0} which is unique for $\re \in (\re_{0},1)$ for an explicit threshold value $\re_{0} \in (0,1)$. More specifically, for $\re \in (\re_{0},1)$, there exists a unique solution $G_{\re}$ to the spatially homogeneous steady equation
$$ {\kappa_{\re}} \nabla_{v}\cdot (v G_{\re}(v))=  \Q_{\re}(G_{\re},G_{\re})\,,$$  
with 
$$\int_{\R^{d}}G_{\re}(v)\d v=1, \qquad \int_{\R^{d}}G_{\re}(v)\,v\d v=0.$$
Moreover,  
\begin{equation}\label{eq:GMli}
\lim_{\re\to1^{-}}\|G_{\re}-\M\|_{L^{1}(\langle v \rangle^2)}=0\,,
\end{equation}
where $\M$ is the Maxwellian distribution
\begin{equation}\label{eq:max}
\M(v)=G_{1}(v)=(2\pi\en_{1})^{-\frac{d}{2}}\exp\left(-\frac{|v|^{2}}{2\en_{1}}\right), \qquad v \in \R^{d}\,,
\end{equation}
for some explicit temperature $\en_{1} >0$.  The Maxwellian distribution $\M(v)$ is a steady solution for $\re=1$ and its prescribed temperature $\en_{1}$ (which ensures \eqref{eq:GMli} to hold) will play a role in the rest of the analysis. We refer to Appendix \ref{appen:homog} for more details and explanation of the role of $\en_{1}$. 

Notice also that the equation in self-similar variables~\eqref{BE0} preserves mass and {\em vanishing} momentum. Indeed, a simple computation based on~\eqref{co:weak} gives that 
$$
{\d \over \d t} \int_{\R^{d}\times \T^{d}}f_{\e}(t,x,v) v \d v\d x = \frac{\kappa_{\re}}{\e^{2}} \int_{\R^{d}\times \T^{d}}f_{\e}(t,x,v)v\d v\d x\,.
$$
Consequently, the assumption made in Remark~\ref{nb:init} and the fact that $G_\re$ has mass $1$ and vanishing momentum imply that for any $t \geq 0$, we have
$$
\int_{\R^{d}\times \T^{d}}f_{\e}(t,x,v) \left(\begin{array}{c}1 \\v \end{array}\right)\d v \d x=\left(\begin{array}{c}1 \\ 0\end{array}\right) .
$$ \medskip

Three main questions are addressed in this work regarding the solution to \eqref{BE0}:
\begin{enumerate}[{\bf (Q1)}]
\item First, we aim to prove the existence and uniqueness of solutions to \eqref{BE0} in a close to equilibrium setting, i.e. solutions which are defined \emph{globally} in time and such that
\begin{equation}
\label{eq:unifDelta}
\sup_{t\geq 0}\|f_{\e}(t)-G_{\re}\| \leq \delta\end{equation}
for some positive and explicit $\delta >0$ in a suitable norm $\|\cdot\|$ of a functional space to be identified.  The close-to-equilibrium setting is quite relevant for very small Knudsen numbers given the large number of collisions per unit time which keeps the system thermodynamically relaxed. 
\item More importantly (though closely related), the scope here is to provide estimates on the constructed solutions $f_{\e}$ which are \emph{uniform with respect to $\e$}. This means that, in the previous point, $\delta>0$ is independent of $\e$.  In fact, we are able to prove \emph{exponential time decay} for the difference $\|f_{\e}(t) - G_{\re}\|$.
\item Finally, we aim to prove that, as $\e\to 0$, the solution $f_{\e}(t)$ converges towards some hydrodynamic solution which depends on $(t,x)$ only through macroscopic quantities $(\varrho(t,x),u(t,x),\vE(t,x))$ which are solutions to a suitable modification of the incompressible Navier-Stokes system.\
\end{enumerate}

\noindent
The central underlying assumption in the previous program is the following relation between the restitution coefficient and the Knudsen number.
\begin{hyp}\label{hyp:re}
The restitution coefficient $\re(\cdot)$ is a continuously decreasing function of the Knudsen number $\e$ satisfying the optimal scaling behaviour
\begin{equation}\label{eq:scaling}
\re(\e)=1-\lambda_{0}\e^{2}+\mathrm{o}(\e^{2})
\end{equation}
with $\lambda_{0} \geq 0$.
\end{hyp}
Indeed, a careful spectral analysis of the linearized collision operator around $G_{\re}$ shows that unless one assumes $1-\re$ comparable to $\e^{2}$ the eigenfunction associated to the energy dissipation would explode and prevent \eqref{eq:unifDelta} to hold true.  In fact, we require $\lambda_0$ to be relatively small with respect to the eigenvalues associated to other kinetic excitations. 
As mentioned before, in this regime the energy dissipation rate is controlled along time by mimicking a viscoelastic property in the granular gas which is at contrast to other regimes such as the mono-kinetic limit.  In viscoelastic models, nearly elastic regimes emerges naturally on large-time scale, see {\cite{BCG, ALCMP,ALT}} for details.

\smallskip

Because $\e \to 0$, Assumption \ref{hyp:re} means that the limit produces a model of the cumulative effect of \textbf{\textit{nearly elastic}} collisions in the \textbf{\textit{hydrodynamic regime}}.  Two situations are of interest in our analysis
\begin{enumerate}
\item[\underline{Case 1:}] If $\lambda_{0}=0$ the cumulative effect of the inelasticity is too weak in the hydrodynamic scale and the expected model is the classical Navier-Stokes equations. In this case, to ensure that $\overline{\lambda}_{\e} >0$, we need the additional assumption that $\alpha$ satisfies 
$$\re(\e) = 1 - \e^2 \eta(\e)$$
for some function $\eta(\cdot)$ which is positive on some interval $(0, \bar\e)$. This technical assumption will be made in all the sequel.

\smallskip

\item[\underline{Case 2:}] If $0 < \lambda_{0}< \infty$, the cumulative effect is visible  in the hydrodynamic scale and we expect a model different  from the Navier-Stokes equation accounting for that.  As we mentioned, we require $\lambda_{0}$ to be relatively small compared to some explicit quantities completely determined by the mass and energy of the initial datum, say, $0 < \lambda_{0} \ll 1$ with some explicit upper bounds on $\lambda_{0}$. 
\end{enumerate}
We wish to emphasize here that, without Assumption \ref{hyp:re}, it appears hopeless to resort to any kind of linearized technique, which is somehow at the basis of the Navier-Stokes scaling. Indeed, even in the spatially homogeneous case, the asymptotic behaviour of the Boltzmann equation is not clearly understood far from the elastic case (see the discussion in the introduction of~\cite{MiMo3}). We strongly believe that we captured with Assumption \ref{hyp:re} the correct regime that brings together the delicate balance between inelasticity and Knudsen number  adapted to the hydrodynamic asymptotics for the constant restitution coefficient case. We also remark that it is very likely that the more adapted model of viscoelastic hard spheres  will display naturally such balance and enjoy the nearly inelastic regime in the long-time dynamic (see \cite{ALCMP,ALT} for more details).

\subsection{Main results} The main results are both concerned with the solutions to \eqref{BE0}. The first one is the following Cauchy theorem regarding the existence and uniqueness of close-to-equilibrium solutions to \eqref{BE0}. A precise statement is given in Theorem \ref{theo:main-cauc1} in Section \ref{sec:Cauchy}. 

\begin{theo}\label{theo:main-cauc}
Under Assumption \ref{hyp:re}, one can construct two suitable Banach spaces $X_{1} \subset X$  such that, for~$\e,\lambda_0$ and $\eta_0$ sufficiently small with respect to the initial mass and energy,  if
$$\|F_{\mathrm{in}}^{\e}-G_{\re(\e)}\|_{X} \leq \e\,\eta_0$$
then the inelastic Boltzmann equation \eqref{BE0} has a unique solution 
$$f_{\e} \in\mathcal{C}\big([0,\infty); X\big) \cap L^1\big([0,\infty); X_{1}\big)$$ satisfying
\begin{equation*}
\left\|f_{\e}(t)-G_{\re(\e)}\right\|_{X}\leq C\e\eta_0\,\exp\left(-\overline{\lambda}_{\e}\,t\right), \qquad \forall t >0\end{equation*}
for some positive constant $C >0$ independent of $\e$ and where $\lim_{\e\to 0}\bar{\lambda}_{\e}=\lambda_{0}$.
\end{theo}    

Theorem \ref{theo:main-cauc} completely answers queries {\bf (Q1)} and {\bf (Q2)} where the functional spaces $X_{1} \subset X$ are chosen to be $ {L^{1}_{v}L^2_x}$-based Sobolev spaces
$$X= {\W^{k,1}_{v}\W^{m,2}_{x}}(\langle v\rangle^{q}), \qquad X_{1}={\W^{k,1}_{v}\W^{m,2}_{x}}(\langle v\rangle^{q+1})$$
for suitable choice of $m,k,q$. Exact notations for the functional spaces are introduced in Section \ref{sec:nota}. The close-to-equilibrium solutions we construct are shown to decay with a rate that can be made uniform with respect to the Knudsen number $\e$.  Recall here that, since Assumption \ref{hyp:re} is met, the homogeneous cooling state depends on $\e$ and $G_{\re(\e)} \to \M$ as $\e \to 0$. We also point out that $\overline{\lambda}_{\e} \simeq \frac{1-\re(\e)}{\e^{2}}$ is the energy eigenvalue of the linearized operator (see Theorem \ref{cor:mu} hereafter).

\smallskip

The estimates on the solution $f_{\e}$ provided by Theorem \ref{theo:main-cauc} are enough to answer {\bf (Q3)} in the following  (we refer to Section \ref{sec:hydr} for a more accurate statement provided by Theorem \ref{theo:CVNS}).

\begin{theo}\label{theo:CVNS-int}
Under the assumptions of Theorem \ref{theo:main-cauc}, set
$$f_{\e}(t,x,v)=G_{\re(\e)} + \e\,h_{\e}(t,x,v)\,,$$
with $h_{\e}(0,x,v)=h_{\mathrm{in}}^{\e}(x,v)=\e^{-1}\left(F^{\e}_{\mathrm{in}}-G_{\re(\e)}\right)$. For a suitable class of "well-prepared" initial datum $h^{\e}_{\mathrm{in}}$ (see Theorem \ref{theo:CVNS} for a precise definition) and any  $T >0$, the family $\left\{ h_{\e} \right\}_{\e} $ converges in some weak sense to a limit $\bm{h}=\bm{h}(t,x,v)$ which is such that 
\begin{equation}\label{eq:hlimint}
\bm{h}(t,x,v)=\left(\varrho(t,x)+u(t,x)\cdot v + \frac{1}{2}\vE(t,x)(|v|^{2}-d\en_{1})\right)\M(v)\,,
\end{equation}
where $(\varrho,u,\vE)=(\varrho(t,x),u(t,x),\vE(t,x))$ are suitable solutions to the following  \emph{incompressible Navier-Stokes-Fourier system with forcing}
\begin{equation}\label{eq:NSFint}
\begin{cases}
\partial_{t}u-{\frac{{\bm \nu}}{\en_1}}\,\Delta_{x}u + {\en_{1}}\,u\cdot \nabla_{x}\,u+\nabla_{x}p=\lambda_{0}u\,,\\[6pt]
\partial_{t}\,\vE-\frac{\gamma}{\en_{1}^{2}}\,\Delta_{x}\vE{+}\en_{1}\,u\cdot \nabla_{x}\vE=\dfrac{\lambda_{0}\,\bar{c}}{2(d+2)}\sqrt{\en_{1}}\,\vE\,,\\[8pt]
\mathrm{div}_{x}u=0, \qquad \varrho + \en_{1}\,\vE = 0\,,
\end{cases}
\end{equation}
subject to initial conditions $(\varrho_{\mathrm{in}},u_{\mathrm{in}},\vE_{\mathrm{in}})$ (entirely determined by the limiting behaviour of $h_{\rm in}^{\e}$ as $\e \to 0$). The viscosity $\bm{\nu} >0$ and heat conductivity $\gamma >0$ are explicit and $\lambda_{0} >0$ is the parameter appearing in \eqref{eq:scaling}. The parameter $\bar{c} >0$ is depending on the collision kernel $b(\cdot)$.
\end{theo}
{The precise notion of weak convergence in the above Theorem \ref{theo:CVNS-int} is very peculiar and strongly related to the \emph{a priori} estimates used for the proof of Theorem \ref{theo:main-cauc}. The mode of convergence is detailed in Theorem \ref{theo:strong-conv}, see also Section \ref{sec:hydro} for more details}.\medskip

It is classical for incompressible Navier-Stokes equations, see \cite[Section 1.8, Chapter~I]{majda}, that the pressure term $p$ acts as a Lagrange multiplier due to the constraint $\mathrm{div}_{x} u=0$ and it is recovered (up to a constant) from the knowledge of $(\varrho,{u},\vE)$. 

\medskip

We point out that the above incompressible Navier-Stokes-Fourier system \eqref{eq:NSFint} with the self-consistent forcing terms on the right-hand-side is a new system of hydrodynamic equations that, to our knowledge, has never been rigorously derived earlier to describe granular flows. We also notice that the last two identities in \eqref{eq:NSFint} give respectively the incompressibility condition and a strong Boussinesq relation (see the discussion in Section~\ref{sec:hydr}). It is important to point out that in the case $\lambda_{0}=0$, one recovers the classical incompressible Navier-Stokes-Fourier system derived from elastic Boltzmann equation, see~{\cite{SR}}.  This proves continuity with respect to the restitution coefficient $\re$. 

Moreover, in both cases $\lambda_{0}=0$ or $\lambda_{0} >0$, the limiting system \eqref{eq:NSFint} is \emph{conservative} (for all quantities $\varrho(t,x),u(t,x),\theta(t,x)$) which illustrates the perfect balance of the self-similar scaling in the hydrodynamic limit.

\smallskip

We finally mention that  Theorem \ref{theo:CVNS-int} together with the relations \eqref{eq:ScalING} provide also a quite precise description of the hydrodynamic behaviour of the original problem \eqref{Bol-e} in physical variables. In this framework, the aforementioned \underline{Case 2} for which $\lambda_{0} >0$ enjoys some special features for which uniform-in-time error estimates can be obtained. Turning back to the original problem \eqref{Bol-e} not only gives a precise answer to \emph{Haff's law} (with an explicit cooling rate of the granular temperature $\bm{T}_{\e}(t)$) but also describes the cooling rate of the \emph{local temperature} $\ds\int_{\R^{d}}F_{\e}(t,x,v)|v|^{2}\d v$. We refer to Section~\ref{sec:orig} and Appendix \ref{Appendix-PP} for a more detailed discussion. 
\newpage

Let us summarize here the main original features of this paper:
\begin{itemize}
\item We identify the correct regime of weak inelasticity (Assumption \ref{hyp:re}) which, with a novel use of self-similarity techniques, allows to balance uniformly, in terms of the Knudsen number, the in-and-out fluxes of energy and allows to exploits fully the non Gaussian steady state in the spatially inhomogeneous setting.

\item We craft a fine linear analysis that leads to uniform spectral estimates to the spatially inhomogeneous inelastic Boltzmann linearized operator in terms of $\e$ which renders in particular the appearance of a new negative eigenvalue $-\bar{\lambda}_\e$ driving the evolution of the energy and performing a crucial role in the nonlinear analysis.

\item We introduce a sophisticated argumentation (including some non standard Gronwall Lemma) exploiting fully the interplay between linear and nonlinear estimates. This approach  leads to uniform estimates for the nonlinear spatially inhomogeneous inelastic Boltzmann model in terms of the Knudsen number as well as some long-time decay of the solutions to \eqref{BE}.

\item We bring a precise quantification of the macroscopic observables in the hydrodynamic limit yielding first to a modified Navier-Stokes-Fourier system (completely new in this context) and also to a rigorous derivation of the (local) Haff law.
\end{itemize}

The reader will experience a self-contained and detailed presentation including the material corresponding to the full derivation of the modified Navier-Stokes-Fourier system and the relevant estimates for the Boltzmann collision operator.

\subsection{Hydrodynamic limits in the elastic case}\label{sec:liter}
The derivation of hydrodynamic limits from the elastic Boltzmann equation is an important problem which received a lot of attention and its origin can be traced back at least to D. Hilbert exposition of its 6th problem at the 1900 International Congress of Mathematicians. We refer the reader to {\cite{SR,golse}} for an up-to-date description of the mathematically relevant results in the field. Roughly speaking three main approaches are adopted for the rigorous derivation of hydrodynamic limits.

\begin{enumerate}[A)]
\item Many of the early mathematical justifications of hydrodynamic limits of the Boltzmann equation are based on (truncated) asymptotic expansions of the solution around some hydrodynamic solution
\begin{equation}\label{eq:C-E}
F_{\e}(t,x,v)=F_{0}(t,x,v)\left(1+\sum_{n}\e^{n}F_{n}(t,x,v)\right)\end{equation}
where, typically
\begin{equation}\label{eq:F0C-E}
F_{0}(t,x,v)=\frac{\varrho(t,x)}{(2\pi\vE(t,x))^{\frac{d}{2}}}\exp\left(-\frac{|v-u(t,x)|^{2}}{2\vE(t,x)}\right)\end{equation}
is a \emph{local Maxwellian} associated to the macroscopic fields which is required to satisfy the limiting fluid dynamic equation. This approach (or a variant of it based upon Chapman-Enskog expansion) leads to the first rigorous justification of the compressible Euler limit up to the first singular time for the solution of the Euler system in {\cite{caflisch}} (see also {\cite{lacho}} for more general initial data and a study of initial layers). In the same way, a justification of the incompressible Navier-Stokes limit has been obtained in {\cite{demasi}}. This approach deals mainly with strong solutions for both the kinetic and fluid equations.

\smallskip

\item Another important line of research concerns weak solutions and a whole program on this topic has been introduced in {\cite{BaGoLe1,BaGoLe2}}. The goal is to prove the convergence of the renormalized solutions to the Boltzmann equation (as obtained in {\cite{diperna}}) towards weak solutions to the compressible Euler system or to the incompressible Navier-Stokes equations. This program has been continued exhaustively and the convergence have been obtained in several important results (see {\cite{golseSR,golseSR1,jiang-masm,lever,lions-masm1,lions-masm2}} to mention just a few).  We remark that, in the notion of renormalized solutions for the classical Boltzmann equation, a crucial role is played by the entropy dissipation ($H$-theorem) which asserts that the entropy of solutions to the Boltzmann equation is non increasing
$$\dfrac{\d}{\d t}\int_{\R^{d}\times \T^{d}}F_{\e}\log F_{\e}(t,x,v)\d v\d x \leq 0.$$
This \emph{a priori} estimate is fully exploited in the construction of renormalized solutions to the classical Boltzmann equation and is also fundamental in some justification arguments for the Euler limit, see {\cite{sr}}. 

\smallskip

\item A third line of research deals with strong solutions close to equilibrium and exploits a careful spectral analysis of the linearized Boltzmann equation. Strong solutions to the Boltzmann equation close to equilibrium have been obtained in a weighted $L^{2}$-framework in the work {\cite{ukai}} and the \emph{local-in-time} convergence of these solutions towards solution to the compressible Euler equations have been derived in {\cite{nishida}}. For the limiting incompressible Navier-Stokes solution, a similar result have been carried out in {\cite{bu}} for smooth  global solutions in $\R^{3}$ with a small initial velocity field. The smallness assumption has been recently removed in {\cite{isabelles}} allowing to treat also non global in time solutions to the Navier-Stokes equation.  These results as well as {\cite{briant}} exploit a very careful description of the spectrum of the linearized Boltzmann equation derived in {\cite{ellis}}.  We notice that they are framed in the space $L^{2}(\M^{-1})$ where the linearized Boltzmann operator is self-adjoint and coercive. The fact that the analysis of \cite{ellis} has been extended recently in \cite{gervais} to larger functional spaces of the type $L^{2}_{v}(\langle \cdot\rangle^{q})$ opens the gate to some refinements of several of the aforementioned results. We also mention here the work {\cite{zhao}} which deals with an energy method in $L^{2}(\M^{-1})$ spaces (see also {\cite{guo,guo2}} and~\cite{rachid}) in order to prove the \emph{strong convergence} of the solutions to the Boltzmann equation towards the incompressible Navier-Stokes equation without resorting to the work of {\cite{ellis}}. 
\end{enumerate}

\noindent
We mention finally that the work {\cite{bmam}} was the main inspiration to answer questions {\bf (Q1)-(Q2)}. Indeed, in~{\cite{bmam}}, the first estimates on the elastic Boltzmann equation in Sobolev spaces with polynomial weight (based on $L^{1}$) are obtained \emph{uniformly with respect to the Knudsen number $\e$}.  To answer question {\bf (Q3)}, we will resort to ideas introduced in the theory of renormalized solutions \cite{BaGoLe1,BaGoLe2,golseSR} that we adapt to the notions of solutions we are dealing with here. We notice here already that we cannot resort to the work of \cite{ellis} and need to carefully exploit the properties of the solutions as constructed in Theorem \ref{theo:main-cauc}.

\subsection{The challenge of hydrodynamic limits for granular gases}\label{sec:chall}
There are several reasons which make the derivation of hydrodynamic limits for granular gases a challenging question at the physical level.  In regard of the mathematical aspects of the hydrodynamical limit, several hurdles stand on way when trying to adapt the aforementioned approaches:
\begin{enumerate}[I)]
\item With respect to the strategy given in \textrm{A)}, the main difficulty lies in the identification of the typical \emph{hydrodynamic solution}.  Such solution is such that the time-space dependence of the one-particle distribution function $F(t,x, v)$ occurs only through suitable hydrodynamic fields like density $\varrho(t, x)$, bulk velocity ${u}(t, x),$ and temperature $\vE(t, x)$.  This is the role played by the Maxwellian $F_{0}$ in \eqref{eq:F0C-E} whenever $\re=1$ and one wonders if the homogeneous cooling state $G_{\re}$ plays this role here.  This is indeed the case up to first order capturing the fat tails of inelastic distributions, yet surprisingly, a suitable Maxwellian plays the role of the hydrodynamic solution in the $\e$-order correction.  This Gaussian behaviour emerges in the \emph{hydrodynamic limit} because of the near elastic regime that we treat here.\footnote{See the interesting discussion in {\cite{vill}}, especially the Section 2.8 entitled ``What Is the Trouble with Non-Gaussianity''}
 
\smallskip 
 
\item The direction promoted in \textrm{B)} appears for the moment out of reach in the context of granular gases.  Renormalized solutions in the context of the inelastic Boltzmann equation \eqref{BE} have not been obtained due to the lack of an $H$-Theorem for granular gases.  It is unclear if the classical entropy (or a suitable modification of it) remains bounded in general for granular gases.  

\smallskip

\item Homogeneous cooling states $G_{\re}$ are not explicit, this is a technical difficulty when adapting the approach of {\cite{ellis}} for the spectral analysis of the linearized inelastic Boltzmann equation in the spatial Fourier variable. Partial interesting results have been obtained in {\cite{rey}} (devoted to diffusively heated granular gases) but they do not give a complete asymptotic expansion of eigenvalues and eigenfunctions up to the order leading to the Navier-Stokes asymptotic.  We mention that obtaining an analogue of the work {\cite{ellis}} for granular gases would allow, in particular, to quantify the convergence rate towards the limiting model as in the recent work {\cite{isabelles}}. 

\smallskip

\item A major obstacle to adapt energy estimates and spectral approach lies in the choice of functional spaces. While the linearized Boltzmann operator associated to elastic interactions is self-adjoint and coercive in the weighted $L^{2}$-space $L^{2}_{v}(\M^{-1})$, there is no such ``self-adjoint'' space for the inelastic case. This yields technical difficulties in the study of the spectral analysis of the linearized operator \footnote{Recall that the powerful enlargement techniques for the elastic Boltzmann equation are based on the \emph{knowledge} of the spectral structure in the space $L^{2}(\M^{-1})$ (and Sobolev spaces built on it) which can be extended to the more natural $L^{1}$-setting.}. Moreover, the energy estimates of {\cite{guo,guo2,jiang-masm,zhao}} are essentially based upon the coercivity of the linearized operator. For granular gases, it seems that one needs to face the problem \emph{directly} in a $L^{1}_{v}$-setting.  Points \textrm{III)} and \textrm{IV)} make the approach \textrm{C)} difficult to directly adapt.
\end{enumerate}
\subsection{Notations and definitions}\label{sec:nota}
 {
We first introduce some useful notations for function spaces. For any nonnegative weight function $m\::\:\R^{d}\to \R^{+}$ (notice that all the weights we consider here will depend only on velocity, i.e. $m=m(v)$),
we define $L^q_v L^p_x(m)$, $1 \leq p,q \leq +\infty$, as the Lebesgue space associated to the norm 
$$
\| h \|_{L^q_v L^p_x(m)} = \| \| h(\cdot, v)\|_{L^p_x} \, m(v) \|_{L^q_v}.
$$
We also consider the standard higher-order Sobolev generalizations $\W^{\sigma, q}_v \W^{s,p}_x(m)$ for any $\sigma, s \in \N$ defined by the norm 
$$
\| h \|_{\W^{\sigma, q}_v \W^{s,p}_x(m)} = \sum_{\substack {0\leq s' \leq s, \, 0 \leq \sigma' \leq \sigma, \\ s'+ \sigma' \leq \max(s,\sigma)}} 
\| \| \nabla_x^{s'} \nabla_v^{\sigma'} h(\cdot, v)\|_{L^p_x} \, m(v) \|_{L^q_v}.
$$
This definition reduces to the usual weighted Sobolev space $\W^{s,p}_{v,x}(m)$ when $q=p$ and $\sigma=s$. For $m \equiv 1$, we simply denote the associated spaces by $L^q_vL^{p}_x$ and $\W^{\sigma,q}_v\W^{s,p}_x$.}

\smallskip
\noindent
We consider in the sequel the general weight
$$\m_{s}(v)=(1+ |v|^{2})^{\frac{s}{2}}, \qquad v \in \R^{d},\qquad s \geq 0.$$
On the complex plane, for any $a \in \R$, we set
$$\C_{a}:=\{z \in \C\;;\;\mathrm{Re}z >-a\}, \qquad \C_{a}^{\star}:=\C_{a}\setminus\{0\}$$
and, for any $r >0$, we set
$$\mathbb{D}(r)=\{z \in \C\;;\;|z| \leq r\}.$$
We also introduce the following notion of hypo-dissipativity in a general Banach space.
\begin{defi} Let $(X,\|\cdot\|)$ be a given Banach space. A closed (unbounded) linear operator $A\::\:\D(A) \subset X \to X$ is said to be \emph{hypo-dissipative} on $X$ if there exists  a norm, denoted by~$\vertiii{\cdot}$, equivalent to the $\|\cdot\|$--norm such that $A$ is dissipative on the space $(X,\vertiii{\cdot})$, that is, 
$$\vertiii{(\lambda-A)h} \geq \lambda\,\vertiii{h}, \qquad \forall \lambda >0,\,\;h \in \D(A).$$
\end{defi}
\begin{nb}  This is equivalent to the following (see Proposition 3.23, p. 88 in \cite{engel}):  if $\vertiii{\cdot}_{\star}$ denotes the norm on the dual space $X^{\star}$, for all  $h \in \D(A),$  there exists $\bm{u}_{h} \in X^{\star}$ such that
\begin{equation*}
\left[ \bm{u}_{h},h\right]=\vertiii{h}^{2}=\vertiii{\bm{u}_{h}}_{\star}^{2}\qquad \text{ and } \quad
\mathrm{Re}\left[\bm{u}_{h},A\,h\right] \leq 0,
\end{equation*}
where $\left[\cdot\,,\cdot\right]$ denotes the duality bracket between $(X^{\star},\vertiii{\cdot}_{\star})$ and $(X,\vertiii{\cdot})$.
\end{nb}
For two tensors $A=(A_{i,j}), B=(B_{i,j}) \in \mathscr{M}_{d}(\R)$, we denote by $A:B$ the scalar $(A:B)=\sum_{i,j}A_{i,j}B_{i,j} \in \R$ as the trace of the matrix product $AB$ whereas, for a vector function $w=w(x) \in \R^{d}$, the tensor $(\partial_{x_{i}}w_{j})_{i,j}$ is denoted as $\nabla_{x} w$. We  also write 
$(\mathrm{Div}_{x}A)^{i}=\sum_{j}\partial_{x_{j}}A_{i,j}(x).$\medskip
\subsection{Strategy of the proof}
The strategy used to prove the main results Theorems \ref{theo:main-cauc} and~\ref{theo:CVNS-int} yields to several intermediate results of independent interest.  The approach is perturbative in essence since we are dealing with close-to-equilibrium solutions to \eqref{BE0}.  This means that, in the study of \eqref{BE0}, we introduce the fluctuation $h_{\e}$ around the equilibrium $G_{\re}$ defined through
$$f_{\varepsilon}(t,x,v)=G_{\re}(v)+\varepsilon \,h_{\varepsilon}(t,x,v)\,,$$
and $h_{\e}$ satisfies
\begin{equation}\label{BE}
\begin{cases}
&\partial_{t} h_{\varepsilon}(t,x,v)+\dfrac{1}{\e} v \cdot \nabla_{x} h_{\varepsilon}(t,x,v)-\dfrac{1}{\e^{2}} \LL h_{\varepsilon}(t,x,v) =\dfrac{1}{\e} \Q_{\re}(h_{\varepsilon},h_{\varepsilon})(t,x,v)\,,\\[7pt]
&\;\, h_{\e}(0,x,v)=h_\e^{\mathrm{\mathrm{in}}}(x,v)\,,
\end{cases}
\end{equation}
where $\LL$ is the linearized collision operator (local in the $x$-variable) defined as
$$\LL h(x,v)=\mathbf{L}_{\re}(h)(x,v) - \kappa_{\re}\nabla_{v}\cdot (vh(x,v))\,,$$
with
$$\mathbf{L}_{\re}(h)=2\widetilde{\Q}_{\re}(G_{\re},h)\,,$$
where we set 
\begin{equation} \label{def:Qalphatilde}\widetilde{\Q}_{\re}(f,g)=\frac{1}{2}\left\{\Q_{\re}(f,g)+\Q_{\re}(g,f)\right\}.\end{equation}
We also denote by $\mathscr{L}_{1}$ the linearized operator around $G_{1}=\M$, that is,
$$\mathscr{L}_{1}(h)=\mathbf{L}_{1}(h)=\Q_{1}(\M,h)+\Q_{1}(h,\M).$$
The method of proof requires first a careful spectral analysis of the full linearized operator appearing in \eqref{BE}: 
$$\G_{\re,\e}h:=-\e^{-1} v \cdot \nabla_{x}h +\e^{-2}\LL h.$$
Such a spectral analysis has to be performed in a suitable Sobolev space, our starting point is an enlargement argument developed in {\cite{GMM}} (that has then extended to the case $\e \neq 1$ in {\cite{bmam}}) to study the spectrum of the elastic operator in a large class of Sobolev spaces. 

\smallskip

More precisely, in our approach, we treat $\G_{\re,\e}$ as a \emph{perturbation}\footnote{This perturbation does not fall into the realm of the classical perturbation theory of  unbounded operators as described in {\cite{kato}}. Typically, the domain of $\G_{\re,\e}$ is much smaller than the one of $\G_{1,\e}$ (because of the drift term in velocity) and the relative bound between $\G_{1,\e}$ and $\G_{\re,\e}$ does not converges to zero in the elastic limit $\re\to 1$.} of the elastic linearized operator $\G_{1,\e}$.  The spectrum of $\G_{1,\e}$ in $ {\W^{s,1}_v\W^{\ell,2}_{x}}(\langle v\rangle^{q})$ is well-understood {\cite{bmam}}, so, it is possible to deduce from this characterisation the spectrum of $\G_{\re,\e}$ using ideas from~{\cite{Tr}}.  We \emph{only} study the spectrum of $\G_{\re,\e}$ without requiring knowledge of the decay of the semigroup associated to $\G_{\re,\e}$.  This simplifies the technicalities of the spectral analysis performed in Section \ref{sec:elas} related to Dyson-Phillips iterates which leads to the spectral mapping theorem {\cite{GMM,Tr}}.  Most notably, in this simplified approach one is able to identify the optimal scaling \eqref{eq:scaling} of the restitution coefficient. {It is worth mentioning that capturing the optimal scaling \eqref{eq:scaling} for this linear and spectral approach yields to rather involved and technical analysis with the introduction of  {several kinds of function spaces of type  {$\W^{s,1}_{v}\W^{\ell,2}_{x}(\langle v\rangle^{q})$} and $\W^{s,2}_{v}\W^{\ell,2}_{x}(\langle v\rangle^{q})$} and the properties of the linearized operator on each of those spaces.}

\smallskip

The scaling \eqref{eq:scaling} is precisely the one which allows to preserve exactly $d+2$ eigenvalues in the neighbourhood of zero (recall that $0$ is an eigenvalue of multiplicity $d+2$ in the elastic case).  Recalling that, in any reasonable space, the elastic operator has a spectral gap of size $\mu_{\star} >0$, i.e.
$$\mathfrak{S}(\G_{1,\e})\cap \{z \in \mathbb{C}\;;\,\mathrm{Re}z > -\mu_{\star}\}=\{0\}$$
where $0$ is an eigenvalue of algebraic multiplicity $d+2$ which is associated to the eigenfunctions $\{\M,v_{j}\M,|v|^{2}\M\;,\;j=1,\ldots,d\}$, one can prove the following theorem
\begin{theo}\phantomsection \label{cor:mu} 
Assume that Assumption \ref{hyp:re} is met 
and consider the Banach space
$$\E:=\begin{cases}
\W^{s,2}_{v}\W^{\ell,2}_{x}(\m_{q}), \qquad \ell \in \N,\,s\geq0, \quad  {\ell \geq s}, \qquad  {q >q^{\star}},\\
\text{ or } \\
\W^{s,1}_{v}\W^{\ell,2}_{x}(\m_{q}), \qquad \ell \in \N,\,s\geq0, \quad   {\ell \geq s}, \qquad  {q >2}\,,
\end{cases}$$
{where $q^\star$ is defined in~\eqref{eq:qstar}}. There exists some explicit $\nu_{*} >0$ such that, if $\mu \in (\mu_{\star}-\nu_{*},\mu_{\star})$, there is some explicit $\overline{\e}>0$ depending only on $\mu_{\star}-\mu$ and such that, for all $\e \in (0,\overline{\e})$, the linearized operator 
$$\G_{\re,\e}\::\:\D(\G_{\re,\e}) \subset \E \to \E$$ has the spectral property:
\begin{equation}\label{eq:spectGe}
\mathfrak{S}(\G_{\re,\e}) \cap \{z \in \mathbb{C}\;;\,\mathrm{Re}z \geq -\mu\}=\{\lambda_{1}(\e),\ldots,\lambda_{d+2}(\e)\}\,,
\end{equation}
where $\lambda_{1}(\e),\ldots,\lambda_{d+2}(\e)$ are eigenvalues of $\G_{\e}$ (not necessarily distinct) with 
$$|\lambda_{j}(\e)| \leq \mu_{\star}-\mu \qquad \text{for} \;\, j=1,\ldots,d+2.$$ More precisely, it follows that
\begin{equation*}\begin{split}
\mathfrak{S}(\G_{\re,\e}) \cap \{z \in \mathbb{C}\;;\,\mathrm{Re}z \geq -{\mu}\}
&=\mathfrak{S}(\e^{-2}\mathscr{L}_{\re}) \cap \{z \in \mathbb{C}\;;\,\mathrm{Re}z\geq -{\mu}\}\\
&=\{\lambda_{1}(\e),\ldots,\lambda_{d+2}(\e)\}\,,
\end{split}
\end{equation*} 
with
$$\lambda_{1}(\e)=0, \qquad \lambda_{j}(\e)=\e^{-2}\kappa_{\re(\e)}, \qquad j=2,\ldots,d+1\,,$$
and
\begin{equation}\label{eq:LambdaE}
\lambda_{d+2}(\e)=-\bar{\lambda}_{\e}=-\frac{1-\re(\e)}{\e^{2}}+\mathrm{O}(\e^{2})\,, \qquad \text{ for }\; \e \simeq 0.\end{equation}
\end{theo} 

To prove Theorem \ref{cor:mu}, it is necessary to strengthen several results of {\cite{MiMo3}} and obtain \emph{sharp} convergence rate in the elastic limit for the linearized operator. Typically, one needs to prove that, for suitable topology
\begin{equation}\label{eq:L1-re}
\LL -\mathscr{L}_{1} \simeq (1-\re)\end{equation}
which gives an estimate of the type
$$\left(\G_{\re,\e}-\G_{1,\e}\right)  \simeq \frac{1-\re}{\e^{2}}.$$
{Our aim is of course to capture the properties of  $\G_{\re,\e}$ in space built on $L^{1}_{v}$ but, for technical reasons, we will also need to investigate the above ansatz \eqref{eq:L1-re} in some $L^{2}_{v}$-spaces with polynomial moments.} This is done in Sections \ref{sec:collL} and \ref{sec:elas}. 
\smallskip

After the spectral analysis is performed, in order to prove Theorem \ref{theo:main-cauc} several \emph{a priori estimates} for the solutions to \eqref{BE0} are required. This is done in Section \ref{sec:non}. The crucial point in the analysis lies in the splitting of \eqref{BE0} into a system of two equations mimicking a spectral enlargement method from a PDE perspective (see the Section 2.3 of \cite{MiMoFP} and \cite{bmam} for pioneering ideas on such a splitting).  More precisely, the splitting performed in Sections \ref{sec:non} and \ref{sec:Cauchy} amounts to look for a solution of \eqref{BE} of the form 
$$h_{\e}(t)=\ho_{\e}(t)+\hu_{\e}(t)$$
where $\hu_{\e}(t)$ is solution to the \emph{linearized elastic equation} with a source term involving the reminder $\ho_{\e}(t)$, namely,
\begin{equation}\label{eq:h1-int}
\partial_{t} \hu_{\e}(t)= {\G_{1,\e}\hu_{\e}} + \e^{-1}\Q_{1}(\hu_{\e},\hu_{\e}) + \A_\e \ho_{\e}\end{equation}
having zero initial datum and where $\A_{\e}$ is a regularizing operator (see Section \ref{sec:elas} for a precise definition). In this way we seek $\hu_{\e}(t)$ in the Hilbert space
$$\hu_{\e}(t) \in \W_{v,x}^{m,2}\left(\M^{-1/2}\right)=:\H$$
with  {$m>d$} and prove bounds of the type
$$\sup_{t\geq 0}\Big(\| \hu_{\e}(t) \|^{2}_{\H} + \int^{t}_{0} \| \hu_{\e}(\tau) \|^{2}_{\H_{1}} \d \tau \Big) \leq C\mathcal{K}_{0}$$
where $\H_{1}$ is the domain of $\G_{1,\e}$ in $\H$, $\mathcal{K}_{0}$ depends only on the initial datum $h^{\e}_{\mathrm{in}}$.   With such a splitting, it is possible to fully exploit the elastic problem and treat $h_{\e}$ as a perturbation of this solution.   {It is important to point out already that $\A_{\e}$ is regularizing only in the velocity variable but not in the $x$-variable. Therefore, no gain of integrability can be deduced from the action of $\A_{\e}$. Therefore, since we look for $\hu_{\e} \in \H$ in~\eqref{eq:h1-int}, we need to look for $\ho_{\e}$ \emph{in a space based on $L^{2}_{x}$.} The velocity regularization properties of $\A_{\e}$ allow then to look for 
$$\ho_{\e}(t) \in \W^{k,1}_{v}\W^{m,2}_{x}(\m_{q}), \qquad t \geq 0.$$}
 This is the role of Section \ref{sec:non}. {Of course, to study the equation solved then by~$\ho_{\e}(t)$, a careful study of the linearized operator on spaces $\W^{k,1}_{v}\W^{m,2}_{x}(\m_{q})$ is necessary (see Sections~\ref{sec:collL} and~\ref{sec:elas}) yielding highly technical additional difficulties}.
\smallskip

In Section \ref{sec:Cauchy}, we prove Theorem \ref{theo:main-cauc} introducing a suitable iterative scheme based upon the coupling $\big(\ho_{\e}(t),\,\hu_{\e}(t)\big)$.  We show in practice that the coupled system of kinetic equations satisfied by $\ho$ and $\hu$ is well-posed.  It is fair to say that the bounds for $\ho_{\e}$ and $\hu_{\e}$ given in Sections \ref{sec:non} and~\ref{sec:Cauchy} play the role of suitable energy estimates as the ones established in the purely Hilbert setting {\cite{guo,guo2,zhao}}. In particular, these bounds are sufficient to deduce {a very peculiar type of weak convergence of $h_{\e}(t)$} towards an element in the kernel of the linearized operator $\mathscr{L}_{1}$, in particular, the limit of $h_{\e}$ is necessarily of the form \eqref{eq:hlimint}. {The notion of weak convergence we use here fully exploits the splitting $h_{\e}=\ho_{\e}+\hu_{\e}$ where we prove that $\ho_{\e}$ converges to $0$ \emph{strongly} in $L^{1}((0,T);{L^{1}_v\W^{m,2}_{x}}(\m_{q}))$ whereas $\hu_{\e}$ converges to $\bm{h}$ \emph{weakly} in $L^{2}((0,T)\,;\,{L^2_v\W^{m,2}_{x}}(\M^{-\frac{1}{2}})).$}

\smallskip
Finally, in Section \ref{sec:hydr}, the regularity of $(\varrho,u,\vE)$ obtained via a simple use of Ascoli-Arzela Theorem and the identification of the limiting equations these macroscopic fields satisfy is presented.
With the {notion of weak convergence  at hand} presented above, the approach is simpler but reminiscent of the program established in {\cite{BaGoLe1,BaGoLe2}}. {In particular, we can adapt some of the main ideas of {\cite{golseSR}} regarding the delicate convergence of nonlinear convection terms.}   Detailed computations are included to make the paper as much self-contained as possible {also because, even in the classical ``elastic'' case, it is difficult to find a full proof of the convergence towards hydrodynamic limit for the weak solutions we consider here. For such solutions, details of proof are scattered in the literature and full proof of the convergence of nonlinear terms is sometimes only sketched where most of the full detailed proofs are dealing with the more delicate case of renormalized solutions {\cite{golseSR,golseSR1,lever}}}. In our framework, the terms involving the quadratic operator $\Q_{\re}(h_{\e},h_{\e})$ are treated as source terms which converge in distributions to zero whereas the drift term and the dissipation of energy function $\mathcal{D}_{\re}$ are the objects responsible for the terms in the right-side of the Navier-Stokes system \eqref{eq:NSFint}. We also observe that the derivation of the strong Boussinesq relation is not as straightforward as in the elastic case. Actually, the classical Boussinesq relation
$$\nabla \left( \varrho(t,x)+\en_{1}\vE(t,x)\right)=0$$
is established as in the elastic case.  In the elastic case, this relation implies the strong form of Boussinesq relation mainly because the two functions $\varrho(t,x)$ and $\vE(t,x)$ have zero spatial averages. This cannot be deduced directly in the granular context due to the dissipation of energy.

\subsection{Organization of the paper} The paper is divided into 6 Sections and three Appendices. In the following Section \ref{sec:collL}, we collect several results regarding the collision operator $\LL$ and introduce the splitting of the operator in $\LL=\A_{\re}+\B_{\re}$ as well as the splitting of the full linearized operator $\G_{\re,\e}.$ As mentioned, even if our final goal is to study the collision operator in spaces built on $L^{1}_{v}$-spaces with polynomial weights, we shall also need to resort to estimates of $\LL$ in $L^{2}_{v}$-spaces. Section \ref{sec:elas} is devoted to the spectral analysis of $\G_{\re,\e}$ culminating with the proof of Theorem \ref{cor:mu}. In Section \ref{sec:non}, we derive the fundamental \emph{a priori} estimates on the close-to-equilibrium solutions to \eqref{BE}. It is the most technical part of the work and fully exploits the splitting of the operator $\G_{\re,\e}$ as explained earlier. Section \ref{sec:Cauchy} gives the proof of Theorem \ref{theo:main-cauc} whereas Section \ref{sec:hydr} gives the full proof of the hydrodynamic limit (Theorem \ref{theo:CVNS-int}). In Appendix \ref{appen:homog}, we recall some facts about the granular Boltzmann equation and gives the full proof of a technical result of Section \ref{sec:collL}. Appendix \ref{Appendix-PP} deals with some properties of the solution to \eqref{Bol-e}, i.e. dealing with the original variables, and provides some insights about the local version of Haff's law. In Appendix \ref{sec:hydro1}, we collect some well-known properties useful for the hydrodynamic limit as well as some technical proofs used in Section \ref{sec:hydr}. Finally, Appendix \ref{appen:G1e} gives the proof of two technical results of Section \ref{sec:collL}.

\smallskip
\noindent {\bf Acknowledgements.}    RA gratefully acknowledges the support from O Conselho Nacional de Desenvolvimento Cient\'ifico e Tecnol\'ogico, Bolsa de Produtividade em Pesquisa  - CNPq (303325/2019-4).  BL gratefully acknowledges the financial support from the Italian Ministry of Education, University and Research (MIUR), ``Dipartimenti di Eccellenza'' grant 2018-2022. Part of this research was performed while the second author was visiting the ``D\'epartement de Math\'ematiques et Applications,'' at \'Ecole Normale Sup\'erieure, Paris in February 2019. He wishes to express his gratitude for the financial support and warm hospitality offered by this Institution. IT thanks the ANR EFI:  ANR-17-CE40-0030  and the ANR SALVE: ANR-19-CE40-0004 for their support. We thank Isabelle Gallagher for stimulating discussions and precious advices. 

\section{Summary of useful results about the collision operator}\label{sec:collL}

\subsection{The linearized operators $\mathscr{L}_{\re}$ and $\mathscr{L}_{1}$}
\label{sec:ll}

{In all the sequel, we will use well-known estimates for the bilinear operator $\Q_{\re}(f,g)$ and $\Q_{1}(f,g)$ in several different functional spaces. We refer to {\cite{ACG,AG,MiMo3}} for precise statements.}
A crucial role in our analysis will be played by the fact that, in some suitable sense, $\mathscr{L}_{\re}$ is close to the elastic linearized operator~$\mathscr{L}_{1}$ for $\re \simeq 1.$ 
Let us begin with the following crucial result which also justifies the optimal scaling \eqref{eq:scaling} and optimise the rate of convergence previously derived in \cite[Proposition 3.1 (iii)]{MiMo3} for weights different to the ones considered here. The technical proof is postponed to Appendix \ref{appen:homog}. 


\begin{lem}\label{lem:els}
Let $a=\max\{d-1,2\}$ and $\kappa>\frac{d}{2}$.  
For $k\in\mathbb{N}$ and {$q \geq 0$}, there is a positive constant $\bm{c}_{k,q} >0$ such that for any~$\re \in (0,1]$, 
\begin{equation*}
\begin{aligned}
\| \Q_{1}(g,f) - \Q_{\re}(g,f) \|_{\W^{k,1}_{v}(\m_{q})} 
&\leq \,\bm{c}_{k,q}\,\frac{1-\re}{\re^{a}}\,\Big(\| f \|_{ \W^{k,1}_{v}(\m_{q+2}) }\,\| g \|_{ \W^{k+1,1}_{v}(\m_{q+2}) }\\
&\qquad \qquad \qquad \qquad {+\| g \|_{ \W^{k,1}_{v}(\m_{q+2}) }\,\| f \|_{ \W^{k+1,1}_{v}(\m_{q+2}) }\Big)}
\end{aligned}
\end{equation*}
and  { a constant $\bm{c}_{k,q}(\kappa) >0$ such that, for any $\re\in (0,1]$,}
\begin{equation*}
\begin{aligned}
\| \Q_{1}(g,f) - \Q_{\re}(g,f) \|_{\W^{k,2}_{v}(\m_{q})} 
&\leq \,\bm{c}_{k,q}\,\frac{1-\re}{\re^{a}}\,\Big(\| f \|_{ \W^{k,2}_{v}(\m_{q+\kappa+2}) }\,\| g \|_{ \W^{k+1,2}_{v}(\m_{q+\kappa+2}) }\\
&\qquad \qquad \qquad\qquad  {+\| g \|_{ \W^{k,2}_{v}(\m_{q+\kappa+2}) }\,\| f \|_{ \W^{k+1,2}_{v}(\m_{q+\kappa+2}) }\Big)}\,.
\end{aligned}
\end{equation*}
\end{lem}
By using the previous lemma and estimates on the difference between $G_\re$ and $\M$, one can get an estimate on the difference between $\mathbf{L}_1$ and $\mathbf{L}_{\re}$ with loss of regularity on the argument (see the first part of Proposition~\ref{prop:converLLL0} below). In our analysis, we will also need an estimate on the difference between $\mathbf{L}_1$ and $\mathbf{L}_{\re}$ with no loss of regularity  (i.e. an estimate in the \emph{graph norm}), even if the rate is not anymore optimal. To this end, we here state a lemma which is in the spirit of~\cite[Proposition 3.2]{MiMo3} except from the fact that one of the argument is fixed to be the Maxwellian $\M$. Note that \cite[Proposition 3.2]{MiMo3} gives an estimate on the difference between $\Q_1$ and $\Q_\re$ for general arguments but the  proof heavily relies on the exponential weights considered. It turns out that we can not adapt easily the proof of~\cite[Proposition 3.2]{MiMo3} for polynomial weights. However, by using decay properties of $\M$, we are able, to get an estimate on $\Q_1(\M,\cdot) - \Q_\re(\M,\cdot)$ and its symmetric, which is enough for our purpose. The proof is also postponed to Appendix~\ref{appen:homog}.
\begin{lem} \label{lem:L1Lalpha}
Let {$q \geq 0$} and $k \in \N$. There exist some explicit $\widetilde{\bm{c}}_{q}>0$, $\widetilde{\bm{c}}_{k,q} >0$, $p \in (0,1)$ and~$\re_1 \in (0,1)$ such that 
\begin{align*}
&\| \Q_1(\M,f) - \Q_{\re}(\M,f) \|_{ \W^{k,1}_{v}(\m_{q}) } \\
&\quad + \| \Q_1(f,\M) - \Q_{\re}(f,\M) \|_{ \W^{k,1}_{v}(\m_{q}) }\leq \widetilde{\bm{c}}_{k,q}\,(1-\re)^p\,\| f \|_{ \W^{k,1}_v(\m_{q+1}) }\, , \qquad \re \in [\re_1,1]\,,
\end{align*} 
and
\begin{align*}
&\| \Q_1(\M,f) - \Q_{\re}(\M,f) \|_{ \W^{k,2}_{v}(\m_{q}) } \\
&\quad + \| \Q_1(f,\M) - \Q_{\re}(f,\M) \|_{ \W^{k,2}_{v}(\m_{q}) }\leq \widetilde{\bm{c}}_{k,q}\,(1-\re)^p\,\| f \|_{ \W^{k,2}_v(\m_{q+1}) }\, , \qquad \re \in [\re_1,1]\,.
\end{align*} 
\end{lem}

Let us now investigate the rate of convergence of the equilibrium $G_{\re}$ towards $\M$. An optimal convergence rate in $L^{1}$-spaces is given in  \cite[Step 2, proof of Lemma 4.4]{MiMo3}: there is $C >0$ and~ {$\re_{2} >0$} such that
\begin{equation}\label{appe1}
\|\M - G_{\re}\|_{ L^{1}_{v}(\langle \cdot \rangle m) }\leq C(1-\re)\,,\qquad  {\re\in [\re_2,1]}\,.
\end{equation}
for $ {m(v)=\exp(a\,|v|)}$, $a >0$ small enough. We need to extend this optimal rate of convergence to the Sobolev spaces $\W^{k,j}_v(\m_{q})$ for $j=1,2$ we are considering here.  {The proof of this technical result is postponed to Appendix \ref{appen:homog}.}
\begin{lem}\label{prop:psi}
Let $k\in \N$,  $q \geq 1$ be given. There exist some explicit  $\re_{3} \in (0,1)$ and $C >0$ such that
\begin{equation*}
\|\M - G_{\re} \|_{ \W^{k,1}_{v}(\m_{q}) }+\|\M - G_{\re} \|_{ \W^{k,2}_{v}(\m_{q}) }  \leq C(1-\re)\,,\qquad  {\re \in [\re_3,1]}\,.
\end{equation*}
\end{lem}
{For $j=1,2$, on the underlying space $\W^{k,j}_{v}(\m_{q})$, introduce the operator 
$T_{\re}\::\:\D(T_{\re}) \subset \W^{k,j}_{v}(\m_{q}) \to \W^{k,j}_{v}(\m_{q})$ defined by $\D(T_{\re})=\W^{k+1,j}_{v}(\m_{q+1})$ and 
$$T_{\re}h(v)=-\kappa_{\re}\nabla_v \cdot (v\,  h(v)), \qquad h \in \D(T_{\re}).$$
One sees that $T_{\re}$ is one of the operators responsible for the discrepancy between the domain of $\mathscr{L}_{1}$ and $\LL$. Because of this, we set
$$\P_{\re}\::\:\D(\P_{\re}) \subset \W^{k,j}_{v}(\m_{q}) \to \W^{k,j}_{v}(\m_{q})$$
as $\P_{\re}=\mathbf{L}_{\re}-\mathbf{L}_{1}$ with domain
$$\D(\P_{\re})=\D(\mathscr{L}_{1})=\W^{k,j}_{v}(\m_{q+1}).$$}
One has then the following Proposition in $L^{1}_{v}$ and $L^2_v$-based spaces:
\begin{prop}\phantomsection\label{prop:converLLL0} Consider $k \in \N$ and $q \geq 0$. There exist  some explicit constant $\bm{C}_{k,q} >0$ and $\alpha_\star \in (0,1)$ such that for any $h \in \W^{k+1,1}_{v}(\m_{q+2})$
\begin{equation}\label{eq:llXk0}
\|\P_{\re}h\|_{\W^{k,1}_{v}(\m_{q})}=\|\mathbf{L}_{\re}h-\mathbf{L}_{1}h\|_{\W^{k,1}_{v}(\m_{q})}
\leq \bm{C}_{k,q}(1-\re)\,\|h\|_{\W^{k+1,1}_{v}(\m_{q+2})}\, , \quad  {\alpha \in [\alpha_\star,1]}\,,
\end{equation}
As a consequence, for any $h \in \W^{k+1,1}_{v}(\m_{q+2})$ 
\begin{equation}\label{eq:llXk}
\|\LL h -\mathscr{L}_{1}h\|_{\W^{k,1}_{v}(\m_{q})} \leq \left(\bm{C}_{k,q}(1-\re)+\kappa_{\re}\right)\,\|h\|_{\W^{k+1,1}_{v}(\m_{q+2})}\,, \quad  {\alpha \in [\alpha_\star,1]}\,. 
\end{equation}
In the same way, for any $\kappa > \frac{d}{2}$  and $q \geq 2$, there exist some explicit $\bm{C}_{k,q}(\kappa) >0$ such that for any~$h \in \W^{k,2}_{v}(\m_{q+\kappa+2})$ it holds
\begin{equation}\label{eq:llXk0-L2}
\|\P_{\re}h\|_{\W^{k,2}_{v}(\m_{q})}=\|\mathbf{L}_{\re}h-\mathbf{L}_{1}h\|_{\W^{k,2}_{v}(\m_{q})}
\leq \bm{C}_{k,q}(\kappa)\,(1-\re)\,\|h\|_{\W^{k+1,2}_{v}(\m_{q+\kappa+2})}\, , \quad  {\alpha \in [\alpha_\star,1]}\,,
\end{equation}
As a consequence, for any $h \in \W^{k+1,2}_{v}(\m_{q+\kappa+2})$ 
\begin{equation*}\label{eq:llXk-L2}
\|\LL h -\mathscr{L}_{1}h\|_{\W^{k,2}_{v}(\m_{q})} \leq \left(\bm{C}_{k,q}(\kappa)\,(1-\re)+\kappa_{\re}\right)\,\|h\|_{\W^{k+1,2}_{v}(\m_{q+\kappa+2})}\,, \quad  {\alpha \in [\alpha_\star,1]}\,. 
\end{equation*}
\end{prop}\phantomsection 
\begin{proof}  {Recall (see Theorem \ref{theo:Ricardo}) that, for any $q \geq 0$, there exists some universal positive constant~$C_{q} >0$ such that for any $\re \in (0,1]$,
\begin{equation}\label{eq:qpm}
\|\Q_{\re}^\pm(g,f)\|_{L^{1}_{v}(\m_{q})} \leq C_{q}\|g\|_{L^{1}_{v}(\m_{q+1})}\,\|f\|_{L^{1}_{v}(\m_{q+1})}, \qquad \forall f,g \in L^{1}_{v}(\m_{q+1}),\end{equation}
where we have denoted by $\Q_\alpha^+$ (resp. $\Q_\alpha^-$) the gain (resp. loss) part of the operator $\Q_\alpha$. }
Then, we have
\begin{equation}\label{eq:LLL0}
\begin{split}
\mathbf{L}_{\re} h(v) - &\mathbf{L}_{1}h(v)=\Q_{\re}(h,G_{\re}-\M)(v) + \Q_{\re}(G_{\re}-\M,h)(v) \\
&+\big[\Q_{\re}(h,\M)(v)-\Q_{1}(h,\M)(v)\big] + \big[\Q_{\re}(\M,h)(v)-\Q_{1}(\M,h)(v)\big]\,.
\end{split}
\end{equation}
One thus deduce from \eqref{eq:qpm} and Lemma \ref{lem:els} that
\begin{align*}
\|\P_{\re}h\|_{L^{1}_{v}(\m_{q})} \leq 2\,C_{q}\|h\|_{L^{1}_{v}(\m_{q+1})}\,&\|G_{\re}-\M\|_{L^{1}_{v}(\m_{q+1})} \\
&+ 2\,\bm{c}_{q}\,\frac{1-\re}{ {\re^{a}}}\,\|h\|_{\W^{1,1}_{v}(\m_{q+2})}\,\|\M\|_{\W^{1,1}_{v}(\m_{q+2})}
\end{align*}
where we recall that $a = \max\{d-1,2\}$.
Using now~\eqref{appe1}, one can conclude the proof of~\eqref{eq:llXk0} for $k=0$. 
In order to prove the result for higher-order derivatives,  one argues using the fact that 
\begin{equation}\label{eq:deriv}
\nabla_{v} \Q_{\re}(g,f)=\Q_{\re}(\nabla_v g,f) + \Q_{\re}(g,\nabla_v f).\end{equation}
Then, using \eqref{eq:LLL0}  with the help of the estimate
$$\|T_{\re}h\|_{\W^{k,1}_{v}(\m_{q})} \leq \kappa_{\re}\|h\|_{\W^{k+1,1}_{v}(\m_{q+1})}$$
one deduces \eqref{eq:llXk} from \eqref{eq:llXk0}.  

We prove the result in $L^{2}_{v}$-based space in a similar way. From Theorem~\ref{theo:Ricardo}, there exists $C_{q} >0$ such that 
$$\|\Q_{\re}^{+}(g,f)\|_{L^{2}_{v}(\m_{q})} +\|\Q_{\re}^{+}(f,g)\|_{L^{2}_{v}(\m_{q})}  \leq C_{q}\|f\|_{L^{1}_{v}(\m_{q+1})}\,\|g\|_{L^{2}_{v}(\m_{q+1})}.$$
Indeed, since we have supposed that the condition~\eqref{eq:conditionb} is satisfied for some $r>2$, it in particular holds true for $r=2$, which implies that we can apply Theorem~\ref{theo:Ricardo} for $r=2$. 
On the other hand, it is immediate to check that for $\kappa>\frac{d}{2}$, there exists $C_q(\kappa)>0$ such that
$$
\|\Q_\re^-(g,f)\|_{L^2_v(\m_q)} \leq C_q(\kappa) \|f\|_{L^2_v(\m_{q+1})} \|g\|_{L^2_v(\m_{\kappa+1})}
$$
where we used Cauchy Schwarz inequality. One deduces that
\begin{align*}
&\|\Q_{\re}(h,G_{\re}-\M)\|_{L^{2}_{v}(\m_{q})} + \|\Q_{\re}(G_{\re}-\M,h)\|_{L^{2}_{v}(\m_{q})} \\
&\quad \leq C_q(\kappa) \|h\|_{L^{2}_{v}(\m_{q+1})} \left(\|G_{\re}-\M\|_{L^{1}_{v}(\m_{q+1})}+\|G_{\re}-\M\|_{L^{2}_{v}(\m_{\kappa+1})}\right) \\
&\qquad+ C_q(\kappa) \|h\|_{L^{2}_{v}(\m_{\kappa+1})} \|G_{\re}-\M\|_{L^{2}_{v}(\m_{q+1})} .
\end{align*}
Then from Lemma~\ref{lem:els}, for any $\kappa >\frac{d}{2}$, we have:
\begin{align*}
\|\Q_{\re}(h,\M)-\Q_{1}(h,\M)\|_{L^2_v(\m_q)}&+ \|\Q_{\re}(\M,h)-\Q_{1}(\M,h)\|_{L^2_v(\m_q)} \\
&\quad \leq \bm{c}_{q}(\kappa)\,\frac{1-\re}{\re^a}\|\M\|_{\W^{1,2}_{v}(\m_{q+\kappa+2})}\,\|h\|_{\W^{1,2}_{v}(\m_{q+\kappa+2})}
\end{align*}
where $a={\max(d-1,2)}$. One can then conclude that \eqref{eq:llXk0-L2} holds true for $k=0$ thanks to Lemma~\ref{prop:psi} and the proof for $k >0$ follows from \eqref{eq:deriv}.  Finally, since 
$$\|T_{\re}h\|_{\W^{k,2}_{v}(\m_{q})} \leq \kappa_{\re}\|h\|_{\W^{k+1,2}_{v}(\m_{q+1})}$$
one also deduces the estimate for $\|\LL h-\mathscr{L}_{1}h\|_{\W^{k,2}_{v}(\m_{q})}.$\end{proof}
We will need also to derive an estimate for $\P_{\re}$ in its graph norm, at the price of loosing the sharp convergence rate $(1-\re)$. The estimate is easy to deduce in $L^{1}_{v}$-space but the extension to the~$L^{2}_{v}$ case requires the use of Corollary~\ref{cor:Ricardo}.
\begin{lem}\label{lem:Palphap}
For any $k \in \N$ and $q \geq 0$, there exists $\bm{C}_{k,q} >0$ such that 
\begin{equation} \label{eq:Palpha}
\|\P_{\re}h\|_{\W^{k,1}_{v}(\m_{q})} \leq \bm{C}_{k,q}(1-\re)^{p}\,\|h\|_{\W^{k,1}_{v}(\m_{q+1})}\, , \quad \alpha \in [\alpha_\star,1]\,,
\end{equation}
where $p$ is defined in Lemma~\ref{lem:L1Lalpha} (is independent of both $k$ and $q$).
Moreover, 
there exists $\bar{p} \in (0,1)$ (independent of $k,q$) such that
\begin{equation} \label{eq:PalphaL2}
\|\P_{\re}h\|_{\W^{k,2}_{v}(\m_{q})} \leq \bm{C}_{k,q}\,(1-\re)^{\bar{p}}\,\|h\|_{\W^{k,2}_{v}(\m_{q+1})}\, , \quad \alpha \in [\alpha_\star,1]\,.
\end{equation}
\end{lem}
\begin{proof} The proof of \eqref{eq:Palpha} is a direct consequence of \eqref{eq:qpm} and Lemma~\ref{lem:L1Lalpha} which give
\begin{align*}
\|\P_{\re}h\|_{L^{1}_{v}(\m_{q})} \leq 2\,C_{q}\|h\|_{L^{1}_{v}(\m_{q+1})}\,&\|G_{\re}-\M\|_{L^{1}_{v}(\m_{q+1})} + \,\widetilde{\bm{c}}_{q}\,{(1-\re)^p}\,\|h\|_{L^1_{v}(\m_{q+1})}.
\end{align*}
The proof of \eqref{eq:PalphaL2} is then deduced from \eqref{eq:Palpha} by Riesz-Thorin interpolation. Indeed, \eqref{eq:Palpha} asserts that 
$$\P_{\re} \in \mathscr{B}(L^{1}_{v}(\m_{q+1}),L^{1}_{v}(\m_{q})), \qquad \left\|\P_{\re}\right\|_{\mathscr{B}(L^{1}_{v}(\m_{q+1}),L^{1}_{v}(\m_{q}))} \leq C_{0,q}(1-\re)^{p}.$$
Now, using Corollary \ref{cor:Ricardo}, and recalling that $b$ is such that~\eqref{eq:conditionb} is satisfied for some $r>2$, one proves easily that, for some $r >2$,
$$\P_{\re} \in \mathscr{B}(L^{r}_{v}(\m_{q+1}),L^{r}_{v}(\m_{q}))\,\qquad \sup_{\re\in (0,1)}\left\|\P_{\re}\right\|_{\mathscr{B}(L^{r}_{v}(\m_{q+1}),L^{r}_{v}(\m_{q}))}=C_{r}(q)< \infty.$$
Then, from Riesz-Thorin interpolation Theorem {\cite[Theorem 1.3.4]{grafakos}}
$$\P_{\re} \in \mathscr{B}(L^{2}_{v}(\m_{q+1}),L^{2}_{v}(\m_{q})) \qquad \text{ with } \quad \left\|\P_{\re}\right\|_{\mathscr{B}(L^{2}_{v}(\m_{q+1}),L^{2}_{v}(\m_{q}))}\leq C\,(1-\re)^{p\theta}$$
where $\theta=\frac{r-2}{2(r-1)}$ is such that $\frac{1}{2}=\theta+\frac{1-\theta}{r}$. This proves \eqref{eq:PalphaL2} for $k=0$ setting $\bar{p}=p\theta$ and its extension to $k >0$ follows from \eqref{eq:deriv}.
\end{proof}
The above extends to functional spaces $\W^{k,1}_{v}L^{2}_{x}(\m_{q})$ and $\W^{k,2}_{v}L^{2}_{x}(\m_{q})$ in an easy way:
\begin{cor} \label{cor:Palpha}
Consider  $k \in \N$ and $q \geq 0$. There exists some explicit constant $\widetilde{\bm{C}}_{k,q} >0$ such that
\begin{equation} \label{eq:Palpha2}
\|\P_{\re}h\|_{\W^{k,1}_{v}L^2_x(\m_{q})} \leq \widetilde{\bm{C}}_{k,q}(1-\re)^{\frac{p}{2}}\,\|h\|_{\W^{k,1}_{v}L^2_x(\m_{q+1})}\, , \quad \alpha \in [\alpha_\star,1]\,,
\end{equation} and similarly, 
$$\|\P_{\re}h\|_{\W^{k,2}_{v}L^{2}_{x}(\m_{q})} \leq \widetilde{\bm{C}}_{k,q}(1-\re)^{\bar{p}}\,\|h\|_{\W^{k,2}_{v}L^{2}_{x}(\m_{q+1})}\,, \quad \re \in [\re_{\star},1],$$
where $p>0$ and $\bar{p}$ are defined respectively in Lemmas~\ref{lem:L1Lalpha} and~\ref{lem:Palphap}.
\end{cor}
\begin{proof}
On the one hand, the $\W^{k,1}_{v}L^1_x(\m_{q})$-norm of $\P_\re h$ is estimated using Fubini theorem and~\eqref{eq:Palpha} in~Lemma~\ref{lem:Palphap}:
$$
\|\P_{\re}h\|_{\W^{k,1}_{v}L^1_x(\m_{q})}  \leq {\bm C}_{k,q} (1-\re)^p \, \|h\|_{\W^{k,1}_{v}L^1_x(\m_{q+1})}.
$$
On the other hand, using~\eqref{eq:qpm} and the fact that $\Q_\alpha^\pm$ are local in $x$, one can show that 
$$
\|\P_{\re}h\|_{\W^{k,1}_{v}L^{\infty}_x(\m_{q})} \leq \widetilde{\bm C}_{k,q} \|h\|_{\W^{k,1}_{v}L^{\infty}_x(\m_{q+1})}.
$$
We obtain \eqref{eq:Palpha2} by interpolation. The proof for $L^{2}_{v}L^{2}_{x}$-based spaces is simpler since it is deduced directly from Fubini Theorem.
\end{proof}
 \subsection{Decomposition of $\LL$}\label{sec:hypo} Let us now recall the following decomposition of $\mathscr{L}_{0}$ introduced in {\cite{GMM,Tr}}. For any $\delta\in (0,1)$, we consider the cutoff function $0\leq\Theta_\delta = \Theta_\delta(\xi,\xi_*, \sigma) \in\Cs^{\infty}(\R^d\times \R^d\times \S^{d-1})$, assumed to be bounded by $1$, which equals $1$ on 
$$J_{\delta}:=\left\{(\xi,\xi_{*},\sigma)\in \R^d\times \R^d\times \S^{d-1}\,\Big|\,|\xi|\leq \delta^{-1}\,,\, 2\delta \leq |\xi-\xi_*|\leq \delta^{-1}\,,\,
 |\!\cos\theta| \leq 1-2\delta \right\},$$
and whose support is included in $J_{\delta/2}$ (where $\cos \theta=\langle \frac{\xi-\xi_{*}}{|\xi-\xi_{*}|},\sigma\rangle$). We then set 
\begin{equation*}
\begin{split}
\mathscr{L}_{1}^{S,\delta}h(\xi)& =\int_{\R^d\times \S^{d-1}} 
\big[\M(\xi'_*)h(\xi') +\M(\xi')h(\xi'_*)-\M(\xi)h(\xi_*)\big]\\
&\hspace{6cm}\times|\xi-\xi_*|\,\Theta_\delta (\xi,\xi_{*},\sigma) \d\xi_*\d\sigma\,, \\ 
\mathscr{L}_{1}^{{R,\delta}}h(\xi)&=  \int_{\R^d\times \S^{d-1}}
\big[\M(\xi'_*)h(\xi') +\M(\xi')h(\xi'_*)-\M(\xi)h(\xi_*)\big] \\
&\hspace{6cm}\times|\xi-\xi_*|\,(1-\Theta_\delta(\xi,\xi,\sigma) ) \d\xi_*\d\sigma\,,
\end{split}\end{equation*}
so that $\mathscr{L}_{1}h= \mathscr{L}_{1}^{{S,\delta}}h+ \mathscr{L}_{1}^{{R,\delta}}h-h\Sigma_\M$
where $\Sigma_{\M}$ denotes the mapping 
\begin{equation}\label{eq:SigmaM}
\Sigma_{\M}(\xi)=\int_{\R^{d}}\M(\xi_{*})|\xi-\xi_{*}|\d\xi_{*}, \qquad \xi \in \R^{d}.\end{equation}
 {Recall that there exist $\sigma_0>0$ and $\sigma_1>0$ such that 
\begin{equation} \label{eq:collfreq}
\sigma_0 \, \m_1(\xi) \leq \Sigma_{\M}(\xi) \leq \sigma_1 \, \m_1{(\xi)}, \quad \xi \in \R^d. 
\end{equation}}
Introduce
\begin{equation*} {\mathcal A}^{(\delta)} (h):=  \mathscr{L}_{1}^{S,\delta}(h)\qquad \text{ and } \qquad
{\mathcal B}_{1}^{(\delta)} (h) := \mathscr{L}_{1}^{{R,\delta}}-\Sigma_\M
\end{equation*}
so that $\mathscr{L}_{1}=\mathcal{A}^{(\delta)}+ \mathcal{B}_{1}^{(\delta)}$.  Let us now recall the known hypo-dissipitavity results for the elastic Boltzmann operator in $L^{1}_{v}L^{2}_{x}$ and $L^2_{v,x}$-based Sobolev spaces, see {\cite[Lemmas~4.12,~4.14 \& Lemma~4.16]{GMM}}:
\begin{lem}\phantomsection\label{prop:hypo1}For any $k \in \N$ and $\delta >0,$ there are two positive constants $C_{k,\delta} >0$ and $R_{\delta} >0$ such that
$\mathrm{supp}\left(\mathcal{A}^{(\delta)}f\right)\subset B(0,R_{\delta})$
and
\begin{equation}\label{eq:Adelta}
 {\|\mathcal{A}^{(\delta)}f\|_{\W^{k,2}_{v}(\R^{d})}} \leq C_{k,\delta}\|f\|_{L^{1}_{v}(\m_{1})}, \qquad \forall  f \in L^{1}_v(\m_{1}).\end{equation}
Moreover, the following holds
\begin{enumerate}
\item For any $q > 2$ and any $\delta \in (0,1)$ it holds
{\begin{multline}\label{eq:B1delta}
\int_{\R^d} \|h(\cdot,v)\|_{L^{2}_{x}}^{-1} \left(\int_{\T^d} \left(\B_{1}^{(\delta)}h(x,v)\right) h(x,v) \, \d x\right) \m_q\, \d v \\
\leq \left(\Lambda_{q}^{(1)}(\delta)-1\right)\|h\|_{L^{1}_{v}L^{2}_{x}(\m_{q}\Sigma_\M)}
\end{multline}}
where $\Lambda_{q}^{(1)}\::\:(0,1) \to \R^{+}$ is some explicit function such that $\lim_{\delta\to0}\Lambda_{q}^{(1)}(\delta)=\frac{4}{q+2}.$
\item {For any $q > q^\star$ and any $\delta \in (0,1)$,
\begin{equation}\label{eq:B1deltaL2}
\int_{\R^d\times\T^d} \left(\B_{1}^{(\delta)}h(x,v)\right) h(x,v) \, \d x \,\m_{q}^2(v)\, \d v \\
\leq \left(\sqrt{\sigma_1}\Lambda_{q}^{(2)}(\delta)-\sqrt{\sigma_0}\right)\|h\|_{L^{2}_{v,x}(\m_{q}\sqrt{\Sigma_\M})}
\end{equation}
where $\Lambda_{q}^{(2)}\::\:(0,1) \to \R^{+}$ is explicit and such that $\lim_{\delta\to0}\Lambda_{q}^{(2)}(\delta)={\frac{8}{2q-3}}$.}
\end{enumerate}
\end{lem}\phantomsection
\begin{nb}
Notice that this lemma directly comes from~{\cite{GMM}} but the constants involved in the final estimates are not the same as in~Lemma 4.14 of {\cite{GMM}} where it seems that some multiplicative constants coming from~\eqref{eq:collfreq} have been omitted in some computations of their proof. 
\end{nb}
This leads to the following decomposition of $\LL$:
\begin{equation} \label{eq:splitLalpha}
\mathscr{L}_{\re}=\mathcal{B}_{\re}^{(\delta)} + \mathcal{A}^{(\delta)}\,,\quad \text{where} \quad \mathcal{B}_{\re}^{(\delta)}=\mathcal{B}_{1}^{(\delta)}+\left[\LL-\mathscr{L}_{1}\right]\,.
\end{equation}

\subsection{The complete linearized operator}  The complete linearized operator is given by
$$\mathcal{G}_{\re,\e}h=\e^{-2} \LL(h) - \e^{-1} v \cdot \nabla_{x}h, {\qquad \forall \re \in (0,1]}.$$
With previous decomposition, we have that
$$\mathcal{G}_{\re,\e}=\mathcal{A}_{\e}^{(\delta)}+ \mathcal{B}_{\re,\e}^{(\delta)}$$
where
$$\mathcal{A}_{\e}^{(\delta)}=\e^{-2} \A^{(\delta)}, \qquad \mathcal{B}_{\re,\e}^{(\delta)}=\e^{-2} \mathcal{B}_{\re}^{(\delta)} -\e^{-1} v\cdot \nabla_{x}\,.$$
Notice that
$$\mathcal{B}_{\re,\e}^{(\delta)}-\mathcal{B}_{1,\e}^{(\delta)}=\G_{\re,\e}-\G_{1,\e}=\e^{-2} \mathcal{P}_{\re} + \e^{-2} T_{\re}.$$
We set
\begin{equation}\label{eq:qstar}
q_1:=2, \quad q_2=q^\star:=4\sqrt{\frac{\sigma_{1}}{\sigma_{0}}} + \frac{3}{2},\end{equation}
where we recall that $\sigma_{0}$ and $\sigma_1$ are defined in~\eqref{eq:collfreq}. One has the following properties of $\mathcal{B}_{\re,\e}^{(\delta)}$ in~$L^1_vL^2_x$ and $L^{2}_{v,x}$-based spaces. 
\begin{prop}\phantomsection\label{prop:hypo} 
For $j=1,2$, for any $\ell\geq s \geq 0$ and $q >q_j$ there exist ${\re}^{\dagger}_{j,\ell,s,q} >0$, ${\delta}_{j,\ell,s,q}^{\dagger} >0$ and ${\nu}_{j,\ell,s,q} >0$ such that {for any $\e \in (0,1]$}, 
$$\mathcal{B}_{\re,\e}^{(\delta)} + \e^{-2}{\nu}_{j,\ell,s,q} \;\, \text{ is hypo--dissipative in $\W^{s,j}_v \W^{\ell,2}_{x}(\m_{q})$} 
$$
for any $\re \in ({\re}^{\dagger}_{j,\ell,s,q},1),$ and $\delta \in (0,{\delta}_{j,\ell,s,q}^{\dagger}).$\end{prop}
\begin{proof}
We present here the proof in the space $\W^{s,2}_{v}\W^{\ell,2}_{x}(\m_{q})$, the proof for $\W^{s,1}_{v}\W^{\ell,2}_{x}(\m_{q})$ follows the same lines and is given in Appendix \ref{appen:homog}. Notice that derivatives with respect to the $x$-variable commute with the operator $\mathcal{B}_{\re,\e}^{(\delta)}$ and this allows to prove the result, without loss of generality, in the special case $\ell=s$. We divide the proof in several steps:\\
\noindent$\bullet$ We first consider the case $\ell=0$. We write $\mathcal{B}_{\re,\e}^{(\delta)}(h)=\sum_{i=0}^{3}C_{i}(h)$ with 
$$C_{0}(h)=\e^{-2}\B_{1}^{(\delta)}h, \quad  C_{1}(h)=-\e^{-1} v \cdot \nabla_{x}h,$$ 
$$C_{2}(h)=\e^{-2} \P_{\re}h,\qquad  C_{3}(h)=\e^{-2}  {T}_{\re}h=-\e^{-2} \kappa_{\re}\nabla_{v}\cdot(v\,h(x,v))\,,$$ and correspondingly and with obvious notations,
\begin{equation*}
\int_{\R^{d}\times\T^{d}} \mathcal{B}_{\re,\e}^{(\delta)}(h)(x,v)\,  h(x,v) \m_{q}^{2}(v)\,\d x \d v  
 =:\sum_{i=0}^{3}I_{i}(h).
 \end{equation*}
{First, $I_1(h)=0$ since 
$$
\int_{\T^d} \left(v \cdot \nabla_{x} h(x,v)\right) h(x,v) \, \d x = \frac12 \int_{\T^d} v \cdot \nabla_x h^2(x,v) \, \d x = 0.  
$$
Then, from~\eqref{eq:B1deltaL2}, one has
$$
I_0(h) \leq \e^{-2} \left(\sqrt{\sigma_1}\Lambda_{q}^{(2)}(\delta)-\sqrt{\sigma_0}\right)\|h\|_{L^{2}_{v,x}(\m_{q}\sqrt{\Sigma_\M})}.
$$
Recalling that $\lim_{\delta\to0}\Lambda_{q}^{(2)}(\delta)=8/(2q-3)$, one can choose $\delta >0$ small enough so that, for $q > q^{\star}$, 
$$\sqrt{\sigma_1} \Lambda_{q}^{(2)}(\delta) < \sqrt{\sigma_{0}}.$$ Then, for $q>q^\star$ and $\delta$ small enough, we have
\begin{equation} \label{eq:I0hL2}
I_0(h) \leq \e^{-2} \sqrt{\sigma_0} \left(\sqrt{\sigma_1}\Lambda_{q}^{(2)}(\delta)-\sqrt{\sigma_0}\right)\|h\|_{L^{2}_{v,x}(\m_{q+1/2})}.
\end{equation}
Moreover, it follows from Cauchy-Schwarz inequality and \eqref{eq:PalphaL2} that 
\begin{equation*}\begin{split}
I_2(h) &\leq 
\e^{-2}\|\P_{\re} h\|_{L^2_{v,x}(\m_{q-\frac{1}{2}})}\|h\|_{L^{2}_{v,x}(\m_{q+\frac{1}{2}})}\\
&\leq \e^{-2} {\bm C}_{0,q-\frac{1}{2}} (1-\alpha)^{\bar{p}} \|h\|_{L^{2}_{v,x}(\m_{q+\frac{1}{2}})}^{2} 
\end{split}\end{equation*}
Finally, for $I_3$, one can compute 
\begin{align*}
&\int_{\R^d \times \T^d} \nabla_v\cdot (vh(x,v)) \, h(x,v)   \,\, \m_q^{2}(v) \, \d x\d v
 \\
&\quad= d \|h\|_{L^{2}_{v,x}(\m_{q})}^{2} +  {1 \over 2} \int_{\R^d \times \T^d} v \cdot \nabla_v h^2(x,v)  \m_{q}^{2}(v) \,\d x \d v
 \\
 &\quad = \frac{d}{2} \|h\|_{L^{2}_{v,x}(\m_{q})}^{2} -  {1 \over 2} \int_{\R^d \times \T^d}  h^{2}(x,v) \, v \cdot \nabla_{v}\m_{q}^{2}(v) \,\d x \d v \\
\end{align*}
Since $v \cdot \nabla_{v} \m_{q}^{2}(v)=2q \m_{q}^{2}(v)-2q\m_{q-1}^{2}(v)$  we get
\begin{equation} \label{eq:I3L2}
I_{3}(h) \leq q\kappa_{\re}\e^{-2} \|h\|_{L^{2}_{v,x}(\m_{q})}^{2} \leq q\kappa_{\re}\e^{-2}\|h\|_{L^{2}_{v,x}(\m_{q+\frac{1}{2}})}^{2}\,.
\end{equation}
Gathering the previous estimates, one obtains 
\begin{multline*}
\int_{\R^{d}\times\T^{d}} \mathcal{B}_{\re,\e}^{(\delta)}(h)(x,v)\,  h(x,v) \m_{q}^{2}(v)\,\d x \d v \\
\leq \e^{-2}\,\left({\bm C}_{0,q-\frac{1}{2}} (1-\alpha)^{\bar{p}} + \sqrt{\sigma_0}\left(\sqrt{\sigma_1}\Lambda_{q}^{(2)}(\delta)-\sqrt{\sigma_{0}}\right) + q\kappa_{\re}\right)\|h\|_{L^2_{v,x}(\m_{q+1/2})}.
\end{multline*}
Recalling that $\kappa_{\re}=1-\re$ while $\lim_{\delta\to 0}\sqrt{\sigma_1}\Lambda_{q}^{(2)}(\delta) < \sqrt{\sigma_{0}} $ for $q >q^{\star}$ we can pick $\delta_{2,0,0,q}^{\dagger}$ small enough and then $\re^{\dagger}_{2,0,0,q}\in(0,1)$ close enough to $1$ so that
  \begin{align*}&\nu_{2,0,0,q}:=\\
  & -\inf\left\{{\bm C}_{0,q-\frac{1}{2}} (1-\alpha)^{\bar{p}} + \sqrt{\sigma_0}\left(\sqrt{\sigma_1}\Lambda_{q}^{(2)}(\delta)-\sqrt{\sigma_{0}}\right) + q\kappa_{\re}\,;\,\re \in (\re_{2,0,0,q}^{\dagger},1),\, \delta \in (0,\delta_{2,0,0,q}^{\dagger})  \right\}\end{align*}
is positive and get, for any $\delta \in (0,\delta_{2,0,0,q}^{\dagger})$ and $\re \in (\re^{\dagger}_{2,0,0,q},1)$,
\begin{equation}\label{eq:Bal-L2}\begin{split}
\mathcal{I}:=\int_{\R^{d}\times\T^{d}} \mathcal{B}_{\re,\e}^{(\delta)}(h)(x,v)\,  h(x,v) \m_{q}^{2}(v)\d x \d v &\leq -\e^{-2}\nu_{2,0,0,q}\left\|h\right\|_{L^{2}_{v,x}(\m_{q+\frac{1}{2}})}^{2}\\
 &\leq -\e^{-2}\nu_{2,0,0,q}\|h\|^2_{L^{2}_{v,x}(\m_{q})}\end{split}\end{equation}
which implies  that $\B_{\re,\e}^{(\delta)}+\e^{-2}\nu_{2,0,0,q}$ is dissipative in $L^{2}_{v,x}(\m_{q})$. }
\smallskip  

\noindent
Let now investigate the case $\ell=1$. We consider the norm 
$$\vertiii{h}^{2}= \|h\|_{L^{2}_{v,x}(\m_{q})}^{2}+   \|\nabla_{x} h\|_{L^{2}_{v,x}(\m_{q})}^{2} +\eta\,\|\nabla_{v} h\|_{L^{2}_{v,x}(\m_{q})}^{2}  , $$
for some $\eta>0$, the value of which shall be fixed later on. This norm is equivalent to the classical $\W^{1,2}_{v,x} (\m_{q})$-norm. We shall prove that for some $\nu_{2,1,1,q}>0$, $ \mathcal{B}_{\re,\e}^{(\delta)}+\e^{-2}\nu_{2,1,1,q}$ is dissipative in $ \W^{1,2}_{v,x} (\m_{q})$ for the norm $\vertiii{\cdot}$.  
Notice first that the $x$-derivative commutes with all the above terms $C_{i}(h)$, $i=0,\ldots,3$, i.e.
$$\nabla_{x} \mathcal{B}_{\re,\e}^{(\delta)}h(x,v)= \mathcal{B}_{\re,\e}^{(\delta)}\left(\nabla_{x}h\right)(x,v)$$
so that, according to the previous step
\begin{align}\label{eq:BaldxL2}
\begin{split}
\mathcal{J}_{x}:&=\int_{\R^{d} \times \T^d}  \nabla_x \,\mathcal{B}_{\re,\e}^{(\delta)}( h)(x,v) \cdot \nabla_x h(x,v) \,\, \m_{q}^{2}(v)\,\d x \d v\\
&\leq -\e^{-2} \nu_{2,0,0,q}\|\nabla_{x} h\|_{L^{2}_{v,x}(\m_{q+\frac{1}{2}})}^{2}.
\end{split}
\end{align}
Consider now the quantity
$$\mathcal{J}_{v}:=\int_{\R^{d} \times \T^d}  \nabla_v \, \mathcal{B}_{\re,\e}^{(\delta)}( h)(x,v) \cdot \nabla_v h(x,v) \, \m_{q}^{2}(v)\,\d x \d v.$$
Using the notations above, one notices that $\nabla_{v} C_{1}(h)=-\e^{-1} \nabla_{x} h + C_{1}(\nabla_{v}h)$, so that
\begin{equation}\label{eq:nabBL2}\begin{split}
\nabla_{v} ( \mathcal{B}_{\re,\e}^{(\delta)} h(x,v))&=\e^{-2} \nabla_{v} (\mathcal{B}_{1}^{(\delta)}h) -\e^{-1} \nabla_{x}h + C_{1}(\nabla_{v}h) \\
&\hspace{4cm}+ \e^{-2} \nabla_{v}(\P_{\re}h) + \e^{-2} \nabla_{v} (T_{\re}h)\\
\end{split}
\end{equation} 
Then, it follows from Corollary~\ref{cor:Palpha} that 
\begin{equation}\label{eq:nabL2}
\|\nabla_{v} (\P_{\re} h) \|_{L^{2}_{v,x}(\m_{q-\frac{1}{2}})} \leq \widetilde{\bm{C}}_{1,q-\frac{1}{2}}(1-\re)^{\bar{p}}\left(\|h\|_{L^{2}_{v,x}(\m_{q+\frac{1}{2}})}+\|\nabla_{v} h\|_{L^{2}_{v,x}(\m_{q+\frac{1}{2}})}\right)\,.\end{equation}
Now, 
$$\nabla_{v}\B_{1}^{(\delta)}h=\nabla_{v} [\mathscr{L}_{1}^{{R,\delta}}h -\Sigma_\M  h ]
= \mathscr{L}_{1}^{R,\delta}(\nabla_{v} h) -\Sigma_\M  \nabla_{v} h 
+{\mathcal R} (h) , $$
where 
$${\mathcal R}( h )=\Q_1(h,\nabla_{v} \M)+ \Q_1(\nabla_{v} \M,h) 
- (\nabla_{v}{\mathcal A}^{(\delta)})(h) - {\mathcal A}^{(\delta)}(\nabla_{v} h). $$
From the proof of~{\cite[Lemma~4.14]{GMM}}, we have that
$$\| {\mathcal R}( h )\|_{L^{2}_{v,x}(\m_{q-\frac{1}{2}})}  
\leq  C_\delta \|h\|_{L^{2}_{v,x}(\m_{q+\frac{1}{2}})}. $$
while, according to \eqref{eq:I0hL2}, one has
\begin{multline*}
\e^{-2}\int_{\R^{d}\times \T^{d}}\left[\mathscr{L}_{1}^{R,\delta}(\nabla_{v} h) -\Sigma_\M  \nabla_{v} h \right] \cdot \nabla_{v} h \,\m_{q}^{2}\d x \d v\\
=I_{0}(\nabla_{v} h)
\leq \e^{-2}\sqrt{\sigma_{0}}\left(\sqrt{\sigma_1}\Lambda_{q}^{(2)}(\delta)-\sqrt{\sigma_{0}}\right)\|\nabla_{v} h\|_{L^{2}_{v,x}(\m_{q+\frac{1}{2}})}^{2}.
\end{multline*}
Therefore, the contribution of $\nabla_{v} \B_{1}^{(\delta)}h$ is 
\begin{multline}\label{eq:nabL0L2}
\e^{-2}\int_{\R^{d}\times \T^{d}}\nabla_{v}\B_{1}^{(\delta)}h \cdot \nabla_{v} h \,\m_{q}^{2}\d x \d v \\
\leq \e^{-2}C_\delta	\|h\|_{L^{2}_{v,x}(\m_{q+\frac{1}{2}})}^{2}+\e^{-2}\sqrt{\sigma_{0}}\left(\sqrt{\sigma_1}\Lambda_{q}^{(2)}(\delta)-\sqrt{\sigma_{0}}\right)\|\nabla_{v} h\|_{L^{2}_{v,x}(\m_{q+\frac{1}{2}})}^{2}\end{multline}
where $\sqrt{\sigma_1}\Lambda_{q}^{(2)}(\delta)-\sqrt{\sigma_{0}} <0$ for $\delta$ small enough and $q > q^{\star}.$ Finally, using the short-hand notation 
$$
\nabla_v \cdot (v \, \nabla_v h) =\big(\nabla_v \cdot (v \, \partial_{v_1} h),\cdots,\nabla_v \cdot (v \,\partial_{v_d} h)\big),
$$
we have
$$
\nabla_v h \cdot \nabla_v \left(\nabla_v \cdot (vh)\right) = |\nabla_v h|^2 + \left[\nabla_v \cdot (v\, \nabla_vh)\right] \cdot \nabla_v h. 
$$
Doing similar computations as the ones leading to~\eqref{eq:I3L2}, we obtain:
\begin{equation}\label{eq:divL2}
\begin{split} 
\int_{\R^d \times \T^d}  \nabla_v \, T_\re( h)(x,v) \cdot \nabla_v h(x,v) \m_{q}^{2}(v)\,\d x \d v
\leq (q-1) \kappa_{\re} \|\nabla_v h\|_{L^{2}_{v,x}(\m_{q+\frac{1}{2}})}^{2}.
\end{split}
\end{equation}
Coming back to~\eqref{eq:nabBL2}, Cauchy-Schwarz inequality and estimates \eqref{eq:nabL2},~\eqref{eq:nabL0L2} and~\eqref{eq:divL2} give that  
\begin{multline*}
\mathcal{J}_{v} \leq \e^{-2} (C_\delta +\widetilde{\bm{C}}_{1,q-1/2}(1-\re)^{\bar p} )\| h\|_{L^{2}_{v,x}(\m_{q+\frac{1}{2}})}^{2} + \e^{-1} \|\nabla_x h\|_{L^{2}_{v,x}(\m_q)}\,\|\nabla_{v}h\|_{L^{2}_{v,x}(\m_{q})}
\\ +\e^{-2}\left(\widetilde{\bm{C}}_{1,q-\frac{1}{2}}(1-\re)^{\bar{p}}+\,(q-1)\kappa_{\re}+\sqrt{\sigma_{0}}\left(\sqrt{\sigma_1} \Lambda_{q}^{(2)}(\delta)-\sqrt{\sigma_0}\right)\right) \|\nabla_{v} h\|_{L^{2}_{v,x}(\m_{q+\frac{1}{2}})}^{2}\,,
\end{multline*}
where we used that the contribution to $\mathcal{J}_{v}$ of the term $C_{1}(\nabla_{v}h)$ vanishes. 
Hence, combining this estimate with \eqref{eq:Bal-L2} and \eqref{eq:BaldxL2} and using that $\e \leq 1$, one obtains, for any $\tau >0$,
\begin{multline*}
\mathcal{I} + \mathcal{J}_{x}+ \eta\,\mathcal{J}_{v} 
\leq \e^{-2} \bigg(\left[-\nu_{2,0,0,q} + \eta\left(C_{\delta}+\widetilde{\bm{C}}_{1,q-\frac{1}{2}}(1-\re)^{\bar{p}}\right)\right]\|h\|_{L^{2}_{v,x}(\m_{q+\frac{1}{2}})}^{2} \\ 
-\left(\nu_{2,0,0,q}- \frac{1}{4\tau}\eta\right)\|\nabla_{x}h\|_{L^{2}_{v,x}(\m_{q+\frac{1}{2}})}^{2} \\
+ \eta\left[\widetilde{\bm{C}}_{1,q-\frac{1}{2}}(1-\re)^{\bar{p}}+\,(q-1)\kappa_{\re}+\sqrt{\sigma_{0}}\left(\sqrt{\sigma_1}\Lambda_{q}^{(2)}(\delta)-\sqrt{\sigma_0}\right)+\tau\right] \|\nabla_{v} h\|_{L^{2}(\m_{q+\frac{1}{2}})}^{2}\bigg)
\end{multline*}
where we used Young's inequality to estimate  the mixed term 
$$\|\nabla_x h\|_{L^{2}_{v,x}(\m_q)}\,\|\nabla_{v}h\|_{L^{2}_{v,x}(\m_{q})} \leq \frac{1}{4\tau}\|\nabla_{x}h\|_{L^{2}_{v,x}(\m_{q})}^{2}+\tau\,\|\nabla_{v}h\|_{L^{2}_{v,x}(\m_{q})}^{2}, \qquad \tau >0.$$
There exist $\re^{\dagger}_{2,1,1,q} >0$ and $\delta_{2,1,1,q}^{\dagger} >0$ and $\tau >0$ small enough so that for any  $\re \in (\re^{\dagger}_{2,1,1,q},1)$ and any $\delta \in (0,\delta^{\dagger}_{2,1,1,q})$, 
$$\left(\widetilde{\bm{C}}_{1,q-\frac{1}{2}}(1-\re)^{\bar{p}}+\,(q-1)\kappa_{\re}+\sqrt{\sigma_{0}}\left(\Lambda_{q}^{(2)}(\delta)-\sqrt{\sigma_0}\right)\right) +\tau <0 .$$
One chooses then $\eta>0$ small enough such that 
$$
\nu_{2,0,0,q}-\eta\, \max\bigg(\frac{1}{4\tau},\sup_{\delta \in (0,\delta^\dagger_{2,1,1,q})}C_\delta +\widetilde{\bm{C}}_{1,q-\frac{1}{2}}(1-\re^\dagger_{2,1,1,q})^{\bar{p}}\bigg) >0,
$$ 
we finally obtain that there exists $\nu_{2,1,1,q}>0$ such that for $\re \in (\re^{\dagger}_{2,1,1,q},1)$ and $\delta \in (0,\delta^{\dagger}_{2,1,1,q})$,
\begin{equation*}\begin{split}
\mathcal{I} + \mathcal{J}_{x}+ \eta\,\mathcal{J}_{v}
&\leq  -\e^{-2} \nu_{2,1,1,q}\left [\|h\|_{L^{2}_{x.v}(\m_{q+\frac{1}{2}})}^{2} + \|\nabla_{x}h\|_{L^{2}_{v,x}(\m_{q+\frac{1}{2}})}^{2}
+ \eta\,  \|\nabla_{v} h\|_{L^{2}_{v,x}(\m_{q+\frac{1}{2}})}^{2}\right]\\
&\leq -\e^{-2}\nu_{2,1,1,q} \vertiii{h}^{2}.
\end{split}\end{equation*}
This proves that ${\mathcal B}_{\re,\e}^{(\delta)}+\e^{-2} \nu_{2,1,1,q}$ is hypo-dissipative in $\W^{1,2}_{v,x}(\m_{q})$. We prove the result for higher order derivatives in the same way considering now the norm 
$$\vertiii{h}^{2}=\sum_{|\bm{\beta}_{1}|+|\bm{\beta}_{2}| \leq k}\eta^{|\bm{\beta}_{1}|}\left\|\nabla^{|\bm{\beta}_{1}|}_{v}\nabla_{x}^{|\bm{\beta}_{2}|}h\right\|_{L^{2}_{v}L^{2}_{x}(\m_{q})}^{2}$$
for some $\eta>0$ to be chosen sufficiently small.\end{proof}

\begin{nb}\label{nb:hypo} It is important to notice that the equivalent norms constructed in the Proposition \ref{prop:hypo} are \emph{independent of} $\e$. This means that the hypo-dissipativity  of $\mathcal{B}_{\re,\e}^{(\delta)}+\e^{-2}{\nu}_{j,\ell,s,q}$ on $ {\W^{s,j}_{v}\W^{\ell,2}_{x}(\m_{q})}$, $j=1,2$ can be re-written as
$$\|(\lambda-\e^{-2}{\nu}_{j,\ell,s,q}-\B_{\re,\e}^{(\delta)})g\|_{ {\W^{s,j}_{v}\W^{\ell,2}_{x}(\m_{q})}} \geq C\,\lambda\|g\|_{ {\W^{s,j}_{v}\W^{\ell,2}_{x}(\m_{q})}}, \qquad j=1,2$$ 
for any $\lambda >0$, $g \in \D(\B_{\re,\e}^{(\delta)})$, and some constant $C>0$ depending on $j,\ell,s,q$ but not on $\e$.
\end{nb}

We now state some semigroup generation result.
\begin{prop}\label{prop:Bree}
For $j=1,2$, for any $\ell \geq s \geq 0$, $q >q_j$, $\re \in (\re^{\dagger}_{j,\ell,s,q},1), \delta \in (0,\delta^{\dagger}_{j,\ell,s,q})$ and $\e >0$, the operator 
$$\mathcal{B}_{\re,\e}^{(\delta)}\::\:\D(\mathcal{B}_{\re,\e}^{(\delta)}) \subset  {\W^{s,j}_{v}\W^{\ell,2}_{x}}(\m_{q}) \longrightarrow {\W^{s,j}_{v}\W^{\ell,2}_{x}}(\m_{q})$$ is the generator of a $C_{0}$-semigroup $\{\cS_{\re,\e}^{(\delta)}(t)\;;\;t \geq 0\}$ in $ {\W^{s,j}_{v}\W^{\ell,2}_{x}}(\m_{q})$ and there exist $0 < \nu_{*} < {\nu}_{j,\ell,s,q}$ and $C_{j,\ell,s,q} >0$  such that
\begin{equation}\label{eq:cSale}
\left\|\,\cS_{\re,\e}^{(\delta)}(t)\,\right\|_{\mathscr{B}( {\W^{s,j}_{v}\W^{\ell,2}_{x}}(\m_{q}))} \leq C_{j,\ell,s,q}\exp(-\e^{-2}\nu_{*}\,t)\,, \qquad \forall \, t \geq 0.
\end{equation}
As a consequence, 
$$\G_{\re,\e}\::\:\D(\G_{\re,\e}) \subset  {\W^{s,j}_{v}\W^{\ell,2}_{x}}(\m_{q}) \longrightarrow  {\W^{s,j}_{v}\W^{\ell,2}_{x}}(\m_{q})$$
is the generator of a $C_{0}$-semigroup $\left\{\mathcal{V}_{\re,\e}(t)\;;\;t\geq 0\right\}$
in $\W^{s,j}_{v}\W^{\ell,2}_{x}(\m_{q})$. 
\end{prop}
\begin{proof} The fact that $\mathcal{B}_{\re,\e}^{(\delta)}$ is a generator of a $C_{0}$-semigroup in ${\W^{s,1}_{v}\W^{\ell,2}_{x}}(\m_{q})
$ is proven in Appendix~\ref{appen:G1e}.  Since we already proved that  $\mathcal{B}_{\re,\e}^{(\delta)}+\e^{-2}\nu_{j,\ell,s,q}$ is hypo-dissipative, we deduce directly~\eqref{eq:cSale}. Finally, because $\A_{\e}^{(\delta)}$ is a bounded operator in $\W_{v}^{s,1}\W^{\ell,2}_{x}(\m_{q})$, we deduce from the bounded perturbation theorem that $\G_{\re,\e}=\A_{\e}^{(\delta)}+\mathcal{B}_{\re,\e}^{(\delta)}$ generates a $C_{0}$-semigroup in $\W_{v}^{s,1}\W^{\ell,2}_{x}(\m_{q})$.
\end{proof}

\subsection{The elastic semigroup}

The spectral analysis of $\G_{1,\e}$ and the generation of its associated semigroup has been performed in Theorem 2.1 of \cite{bmam}. We need a slightly more precise estimate on the decay of the semigroup independently of $\e$. Our main result concerning $\G_{1,\e}$ is the following whose proof is postponed to Appendix \ref{appen:G1e}: 
\begin{theo}\label{theo:G1e}
There exists $\e_{0} \in (0,1)$ such that, for all $\ell,s \in \N$ with $\ell \geq s$  and $q >q^\star$ and any $\e \in (0,\e_{0})$, the full transport operator $\G_{1,\e}$ generates a $C_{0}$-semigroup $\{\mathcal{V}_{1,\e}(t)\;;\;t \geq 0\}$ on  {$\W^{s,2}_{v}\W^{\ell,2}_{x}(\m_{q})$}. Moreover, there exist $C_{0} >0$ and $\mu_{\star} >0$ (both independent of $\e$) such that, 
\begin{multline}\label{eq:decay}
\big\|\mathcal{V}_{1,\e}(t)\left[h-\mathbf{P}_{0}h\right]\big\|_{ {\W^{s,2}_{v}\W^{\ell,2}_{x}}(\m_{q})} \\
\leq C_{0}\exp(-\mu_{\star}t)\,\|h-\mathbf{P}_{0}h\|_{ {\W^{s,2}_{v}\W^{\ell,2}_{x}}(\m_{q})}\,, \qquad \forall \, t \geq 0\,,
\end{multline}
holds true for any $h \in  {\W^{s,2}_{v}\W^{\ell,2}_{x}}(\m_{q})$, where $\mathbf{P}_{0}$ is the spectral projection onto $\mathrm{Ker}(\G_{1,\e})=\mathrm{Ker}(\mathscr{L}_{1})$ which is \emph{independent of} $\e$ and given by
\begin{equation}\label{eq:P0}
\mathbf{P}_{0}h=\sum_{i=1}^{d+2}\left(\int_{\T^{d}\times\R^{d}}h\,\Psi_{i}\,\d x\d v\right)\,\Psi_{i}\,\M\end{equation}
where $\Psi_{1}(v)=1$, $\Psi_{i}(v)=\frac{1}{\sqrt{\en_{1}}}v_{i-1}$ $(i=2,\ldots,d+1)$ and $\Psi_{d+2}(v)=\frac{|v|^{2}-d\en_{1}}{\en_{1}\sqrt{2d}}$ $(v \in \R^{d}).$
\end{theo}
\begin{nb}  \label{nb:G1e} Theorem \ref{theo:G1e} is known to be true on the Hilbert space $\W^{\ell,2}_{v,x}(\M^{-\frac{1}{2}})$, see \cite[Theorems~2.1 and~2.4]{briant}. Notice that Theorem 2.1 from~{\cite{briant}} only provides an exponential decay of a norm of the solution which depends on $\e$. The introduction of a new norm which is equivalent to the usual $\W^{\ell,2}_{v,x}$ norm uniformly in $\e$ then allows the author to recover in Theorem~2.4 a uniform in $\e$ exponential decay of the solution to the whole nonlinear problem. One can of course proceed similarly to obtain a uniform in $\e$ exponential decay of the semigroup. In the present context of polynomial weighted spaces, a similar result was obtained in {\cite[Theorem~2.1]{bmam}} with the important difference that the estimate \eqref{eq:decay} was shown only for $t > t_{\star} >0$. This actually comes from the use of a general enlargement theorem from {\cite{GMM}} which yields 
\begin{equation} \label{eq:badrate}
\left\|\mathcal{V}_{1,\e}(t)h-\mathbf{P}_{0}h\right\|_{{\W^{s,2}_{v}\W^{\ell,2}_{x}}(\m_{q})} \leq C_{0}\frac{t^{N}}{\e^{N(2+s)}}\exp(-\mu\,t)\,\|h-\mathbf{P}_{0}h\|_{{\W^{s,2}_{v}\W^{\ell,2}_{x}}(\m_{q})}, \quad t\geq0
\end{equation}
for some $N \in \N$ and $s >0$ and $\mu <\mu_{\star}$. It is important for the rest of our analysis to be able to remove this strong dependence on $\e$ in the decay estimate of $\,\mathcal{V}_{1,\e}(t)(\mathbf{Id-P}_{0})$. This is done in Appendix~\ref{appen:G1e}. The key point is that, in our case, the enlargement argument is developed with~$\W^{\ell,2}_{v,x}(\M^{-1/2})$ as the small space and $\W^{s,2}_v\W^{\ell,2}_x(\m_q)$ as the larger one. One can remark that in this context, we do not need any gain of regularity in the space variable. The fact that the rate degenerates for small times in~\eqref{eq:badrate} actually comes from the use of an averaging lemma to gain regularity in $x$, which is no longer necessary in our framework. Note that Theorem~\ref{theo:G1e} also holds in $\W^{s,1}_v\W^{\ell,2}_x(\m_q)$ for $q>2$ for the same reasons.
\end{nb}
\begin{nb} \label{nb:G1e2} Notice that, unfortunately, it seems that the above result is \emph{not true} on the natural space $\W^{s,1}_{v}\W^{\ell,1}_{x}(\m_{q})$: indeed, whereas \eqref{eq:badrate} still holds in such a space, it seems that the ``initial layer'' dependence on $\e$ is not removable in this case because $\A_{\e}$ \emph{has no regularizing properties on the $x$-variable.} This is the main reason why we need in our approach to deal with spaces built upon~$L^{2}_{x}$.  
\end{nb}

An important consequence of the Theorem \ref{theo:G1e} is the following proposition. 
\begin{prop}\label{prop:resG1e} Let $\ell,s \in \N$ with $\ell \geq s$  and  {$q >q^{\star}$.} There exists $C_{1} >0$ such that, 
$$\|\Rs(\lambda,\G_{1,\e})\|_{\mathscr{B}( {\W^{s,2}_{v}\W^{\ell,2}_{x}}(\m_{q}))} \leq C_{1}\,\max\bigg(\frac{1}{|\lambda|},\frac{1}{\mathrm{Re}\lambda+\mu_{\star}}\bigg)\,, \quad \forall \, \lambda \in \C_{\mu_{\star}}^{\star}\,, \;\, \forall \,\e \in (0,\e_{0})\,,$$
where $\e_{0}$ and $\mu_{\star}$ have been defined in Theorem \ref{theo:G1e}, $C_{1}$ being independent of $\e$.
\end{prop}
\begin{proof} 
On the space $ {\W^{s,2}_{v}\W^{\ell,2}_{x}}(\m_{q})$, the spectrum of $\G_{1,\e}$ satisfies
$$\mathfrak{S}(\G_{1,\e}) \cap \left\{z \in \mathbb{C}\;;\;\mathrm{Re}z >-\mu_{\star}\right\}=\{0\}$$
and the above projection $\mathbf{P}_{0}$ is nothing but the spectral projection of $\mathfrak{S}(\G_{1,\e})$ associated to the zero eigenvalue given by
$$\mathbf{P}_{0}=\frac{1}{2i\pi}\oint_{\gamma_{r}}\Rs(z,\G_{1,\e})\d z\,, \qquad \gamma_{r}:=\{z \in \mathbb{C}\;;\;|z|=r\}\,, \qquad r < \mu_{\star}.$$
Notice also 
$$\mathrm{dim}\left(\mathrm{Range}(\mathbf{P}_{0})\right)=\mathrm{dim}\,\mathrm{Ker}(\G_{1,\e})=d+2\,,$$
which means that the algebraic multiplicity of the zero eigenvalue coincides with its geometrical multiplicity and, as such, $0$ is a \emph{simple} pole of the resolvent $\Rs(\cdot,\G_{1,\e})$ (see \cite[III.5]{kato}).  Denote by $\|\cdot\|$ the operator norm in $\mathscr{B}\left( {\W^{s,2}_{v}\W^{\ell,2}_{x}}(\m_{q})\right)$ and fix $\mu \in (0,\mu_{\star})$. Since $\Rs(\lambda,\G_{1,\e})=\Rs(\lambda,\G_{1,\e})\mathbf{P}_{0}+\Rs(\lambda,\G_{1,\e})(\mathbf{Id-P}_{0})$ and $\mathbf{P}_{0}$ commutes with $\G_{1,\e}$, we only need to estimate independently
$$\|\Rs(\lambda,\G_{1,\e})\mathbf{P}_{0}\| \qquad \text{ and } \qquad \|\Rs\left(\lambda,\G_{1,\e}\right)\left[\mathbf{Id-P}_{0}\right]\|$$
for any $\lambda \in \C_{\mu}^{\star}.$ Since the multiplicity of the pole $0$ is one, one has
$\Rs(\lambda,\G_{1,\e})\mathbf{P}_{0}=\frac{1}{\lambda}\mathbf{P}_{0}$
and
$$\|\Rs(\lambda,\G_{1,\e})\mathbf{P}_{0}\| \leq \frac{\|\mathbf{P}_{0}\|}{|\lambda|}, \qquad \lambda \in \C_{\mu_{\star}}^{\star}\,.$$ 
On the other hand, since for any $\lambda \in \C_{\mu_{\star}}$ 
$$\Rs\left(\lambda,\G_{1,\e} \right)\left[\mathbf{Id-P}_{0}\right]=\int_{0}^{\infty}e^{-\lambda\,t}\mathcal{V}_{1,\e}(t)\left[\mathbf{Id-P}_{0}\right]\d t\,,$$
one deduces from Theorem \ref{theo:G1e} that
$$\|\Rs\left(\lambda,\G_{1,\e} \right)\left[\mathbf{Id-P}_{0}\right]\| \leq C_{0}\int_{0}^{\infty}e^{-\mathrm{Re}\lambda\,t}e^{-\mu_{\star}\,t}\|\mathbf{Id-P}_{0}\|\d t\,,$$
which gives that
$$\|\Rs\left(\lambda,\G_{1,\e} \right)\left[\mathbf{Id-P}_{0}\right]\| \leq C_{0}\|\mathbf{Id-P}_{0}\|\frac{1}{\mathrm{Re}\lambda+\mu_{\star}}, \qquad \forall \lambda \in \C_{\mu_{\star}}.$$
This gives the desired estimate with $C_{1}=\|\mathbf{P}_{0}\|+C_{0}\|\mathbf{Id-P}_{0}\|$ independent of $\e$ and~$\mu$.\end{proof}

\section{Linear theory in the weakly inelastic regime} \label{sec:elas}

We start this part by giving some results on the spectrum of the homogeneous operator $\LL$ in the weakly inelastic regime.
\begin{prop}\label{prop:specLL}
{For $\alpha$ close enough to $1$, on the spaces
$$L^{1}_{v}(\m_{q}) \cap \Big\{ h \in L^{1}_{v}(\m_{1})\;;\;\int_{\R^{d}}h(v)\d v=\int_{\R^{d}}h(v)\,v\,\d v=0\Big\}, \qquad q > 2$$
and
$$L^{2}_{v}(\m_{q}) \cap \Big\{ h \in L^{1}_{v}(\m_{1})\;;\;\int_{\R^{d}}h(v)\d v=\int_{\R^{d}}h(v)\,v\,\d v=0\Big\}, \qquad q > q^\star,$$}
the spectrum of $\LL$ is such that there exists $\overline{\mu} >0$ such that
\begin{equation}\label{eq:spectLL}
\mathfrak{S}(\LL) \cap \{\lambda \in \mathbb{C}\;;\;\mathrm{Re}\lambda > -\overline{\mu}\,\}=\{-\mu_{\re} \}\end{equation}
where  $\mu_{\re}$ is a simple eigenvalue of $\LL$ with
\begin{equation}\label{eq:mualpha}
\mu_{\re}=(1-\re) + \mathrm{O}((1-\re)^{2}) \qquad \text{ as } \quad \re \to 1.\end{equation}
Moreover,  denoting by $\phi_{\re}$ the unique associated eigenfunction such that $\|\phi_{\re}\|_{L^{1}_{v}(\m_{2})}= 1$ and $\phi_{\re}(0) < 0$, 
it holds
\begin{equation}\label{eq:limphialpha}
\lim_{\re\to 1}\phi_{\re}(v)=c_{0}\left(|v|^{2}-d\en_{1}\right)\M=:\phi_{1}(v).\end{equation}
\end{prop}
\begin{nb}  
In~{\cite{MiMo3}}, the authors obtained the exact same result but only in exponentially weighted~$L^1_v$-spaces. But it is an easy matter to enlarge and shrink the space in which this type of result holds thanks to the enlargement and shrinkage arguments developed in~{\cite{MiMoFP}} and thanks to the splitting of~$\LL$ exhibited in~\eqref{eq:splitLalpha}. Indeed, one can check that the assumptions of~ {\cite[Theorem~2.2]{MiMoFP}} are satisfied thanks to Lemma~\ref{prop:hypo1} and Proposition~\ref{prop:hypo}.
\end{nb}

The final goal of this section is to prove Theorem \ref{cor:mu} in both families of spaces considered there. {Let $\kappa>\frac{d}{2}$ be fixed.} For simplicity, in this section, we use the following notation:
$$\Y:=\W^{ {s,2}}_{v}\W^{\ell,2}_{x} (\m_{q}),\qquad \ell\in \N, \, \, s\in \N^*, \quad {\ell \geq s+1}, \quad {q >q^{\star}+\kappa+2}.$$

We will actually prove that the Theorem \ref{cor:mu} holds true first in the space $\Y$ for any choice of the parameter {$q >q^{\star}+\kappa+2$} and then will use a factorization argument to deduce it also holds in the spaces $\mathcal{E}$ introduced in Theorem~\ref{cor:mu}. We introduce then the two spaces
$$\Y_{-1}:=\W^{s-1,2}_{v}\,\W^{\ell,2}_{x}(\m_{q -\kappa-2}), \qquad \Y_{1}:=\W^{s+1,2}_{v}\W^{\ell,2}_{x}(\m_{q +\kappa+2})$$
and will also use the notation $\Y_0:=\Y$. 
We recall that, in the space $\Y$, the full linearized operator is given by
$$\mathcal{G}_{\re,\e}h=\e^{-2} \LL(h) - \e^{-1} v \cdot \nabla_{x}h, \qquad {\forall \re \in (0,1]}\,,$$
with domain $\D(\G_{\re,\e})=\W^{s+1,2}_{v} \W^{\ell+1,2}_{x}(\m_{q+1})\,.$

\noindent
Clearly, any spatially homogeneous eigenfunction of $\LL $ associated to an eigenvalue $\lambda \in \mathbb{C}$ is an eigenfunction to $\G_{\re,\e}$ with associated eigenvalue $\e^{-2} \lambda.$ In particular
$$\mathrm{Ker}(\LL) \subset \mathrm{Ker}(\G_{\re,\e}).$$
Notice that, in contrast to {\cite{briant,bmam}}, it is not clear whether such kernels agree. We deduce in particular from Proposition \ref{prop:specLL} that, on the space {$L^{2}_{v,x}(\m_{q})$}, 
$$-\e^{-2} \mu_{\re} \in \mathfrak{S}(\G_{\re,\e})$$
with associated eigenfunction $\phi_{\re}$, that is,
$$\G_{\re,\e}\phi_{\re}=-\e^{-2}\,\mu_{\re} \phi_{\re}.$$
For the eigenvalue $-\e^{-2}\,\mu_{\re}$ to stay sufficiently close to $0$, we assume that $\re=\re(\e)$ satisfies Assumption \ref{hyp:re} and write 
$$\G_{\e}=\G_{\re(\e),\e}\,,$$
and keep the notation $\G_{1,\e}$ for the elastic operator.  Similarly, for all the operators introduced in Section \ref{sec:ll} the double subscript $(\re,\e)$ will be replaced by $\e$ except when $\re=1$.   More precisely, to fix notations, we have
$$\G_{\e}h=\e^{-2} \LLe h - \e^{-1}v \cdot \nabla_{x}h\,,$$
with 
$$\LLe h=\mathbf{L}_{\re(\e)}h -\kappa_{\re(\e)}\nabla_{v}\cdot (vh)\,,$$
and 
$$\mathbf{L}_{\re(\e)}h=\Q_{\re(\e)}(h,G_{\re(\e)})+\Q_{\re(\e)}(G_{\re(\e)},h).$$
In the sequel, since $j,\ell,s,q$ are fixed, we set
$$\delta^{\dagger}:=\min\big\{\delta_{2,\ell,s-1,q-\kappa-2}^{\dagger}\, , \,\delta_{2,\ell,s,q}^{\dagger}\,,\,\delta_{2,\ell,s+1,q+\kappa+2}^{\dagger}\big\},
$$ 
and
$$\re^{\dagger}:=\max\big\{\re_{2,\ell,s-1,q-\kappa-2}^{\dagger}\, , \,\re_{2,\ell,s,q}^{\dagger}\,,\,\re_{2,\ell,s+1,q+\kappa+2}^{\dagger}\big\}\,,$$
so that, for $\delta \in (0,\delta^{\dagger})$ and $\re \in (\re^{\dagger},1)$, the results of the previous section hold in all the spaces~$\Y_j$.  Moreover, we denote by $\e^{\dagger} >0$ the \emph{unique} solution  to 
$$\re(\e^{\dagger})=\re^{\dagger}.$$
We consider $\delta \in (0,\delta^{\dagger})$, $\e \in (0,\e^{\dagger})$ (which implies $\re(\e) \in (\re^{\dagger},1)$), and write
$$\A_{\e}=\A_{\e}^{(\delta)}, \qquad \B_{\e}=\B_{\re,\e}^{(\delta)}.$$
Our scope here is to obtain a result similar to {Lemma 2.16 of \cite{Tr}} regarding the invertibility of $\G_{\e}.$ We actually drastically simplify the proof given there by exploiting the fact that the difference operator $\G_{\e}-\G_{1,\e}$ does not involve any spatial derivatives. Precisely, one starts with the following estimate for this difference:
\begin{lem}\label{lem:diffG} There exists some positive constant $C_{0}$ such that, for any $\e \in (0,\e^{\dagger})$
\begin{equation}\label{eq:diffG}
\max\left(\left\|\G_{\e}-\G_{1,\e}\right\|_{\mathscr{B}(\Y_{0},\Y_{-1})},\left\|\G_{\e}-\G_{1,\e}\right\|_{\mathscr{B}(\Y_{1},\Y_{0})}\right) \leq C_{0}\,\frac{1-\re(\e)}{\e^{2}}.\end{equation}
\end{lem}
\begin{proof} Observe that, in the difference $\G_{\e}-\G_{1,\e}$, the transport term $v\cdot\nabla_{x}$ vanishes so that
$$\G_{\e}-\G_{1,\e}=\e^{-2} \left[\LLe-\mathscr{L}_{1}\right].$$
Let us only prove the estimate in $\mathscr{B}(\Y_0,\Y_{-1})$, the other being the same (changing only the value of $s$ and $q$). For a given $h \in \Y_0$, one has, since $\nabla_{x}^{k}$ commutes with both $\LLe$ and $\mathscr{L}_{1}$, 
\begin{equation*}\begin{split}
\|\left(\G_{\e}-\G_{1,\e}\right)h\|_{\Y_{-1}}^{2}&=\e^{-4}\sum_{\substack {0\leq k \leq \ell, \, 0 \leq r \leq s-1, \\ r+k \leq \ell}} 
\int_{\R^{d}} \| \nabla_x^{k} \nabla_v^{r} \left[\LLe-\mathscr{L}_{1}\right]h\|_{L^{2}_{x}}^{2} \, \m_{q-\kappa-2}^{2}(v)\d v\\
&=\e^{-4} \sum_{\substack {0\leq k \leq \ell}} \int_{\T^{d}}\left\|\left[\LLe-%
\mathscr{L}_{1}\right]\nabla_{x}^{k}h\right\|_{\W^{s-1,2}_{v}(\m_{q-\kappa-2})}^{2}\d x\\
\end{split}\end{equation*}
According to \eqref{eq:llXk-L2}, we deduce that there is $C(s,q) >0$ such that
\begin{multline*}
\left\|\left(\G_{\e}-\G_{1,\e}\right)h\right\|_{\Y_{-1}} \leq C(s,q)\frac{1-\re(\e)}{\e^{2}}\left(\sum_{\substack {0\leq k \leq \ell}}\int_{\T^{d}}\|\nabla_{x}^{k}h\|_{\W^{s,2}_{v}(\m_{q})}^{2}\d x\right)^{\frac{1}{2}}\\
=C(s,q)\frac{1-\re(\e)}{\e^{2}}\|h\|_{\Y_0}\end{multline*}
which is the desired result. \end{proof}
\begin{nb} The above result does seem to be true if one replace $\Y$ with $\W^{s,1}_{v}\W^{\ell,2}_{x}(\m_{q})$ here even if $\mathscr{L}_{\re}-\mathscr{L}_{1}$ satisfies nice estimates on $\W^{s,1}_{v}(\m_{q})$. The use of Fubini Theorem is fundamental here to be able to get our rate.
\end{nb}
\begin{nb}\label{nb:diffQXY}  Notice that the exact same argument together with Lemma \ref{lem:els} show that for $\ell$, $s \in \N$ and $q \geq 0$,
\begin{multline*}
\|\Q_{\re(\e)}(g,h)-\Q_{1}(g,h)\|_{\W^{s,1}_v\W^{\ell,2}_x(\m_q)}\\
\leq C(1-\re(\e))\|g\|_{\W^{s+1,2}_v \W^{\ell,2}_x (\m_{q+2\kappa+2})}\|h\|_{\W^{s+1,2}_v \W^{\ell,2}_x (\m_{q+2\kappa+2})}
\end{multline*}
with $\kappa>\frac{d}{2}$.
\end{nb}

One can now adapt Lemma 2.16 of {\cite{Tr}}:
\begin{prop}\label{lem:inverse} For all $\lambda \in \C_{\mu_{\star}}^{\star}$, let
$$\mathcal{J}_{\e}(\lambda)=\left(\G_{\e} -\G_{1,\e}\right)\Rs(\lambda,\G_{1,\e})\A_{\e}\,\Rs(\lambda,\B_{\e}).$$
Then, $\mathcal{J}_{\e}(\lambda) \in \mathscr{B}(\Y)$.  Furthemore, for any $\mu \in (0,\mu_{\star})$ and $$\lambda \in \C_{\mu}\setminus \mathbb{D}(\mu_{\star}-\mu)=\{z\in \C\;;\;\mathrm{Re}z >-\mu\,,\,|z| >\mu_{\star}-\mu\}\,,$$
it holds that
\begin{equation}\label{eq:Jalk}
\left\|\mathcal{J}_{\e}(\lambda)\right\|_{\mathscr{B}(\Y)} \leq \frac{C}{\mu_{\star}-\mu}\,\frac{1-\re(\e)}{\e^{2}}
\end{equation}
for a universal constant $C >0$. In addition, there exists ${\e}^{\star} \in (0,\e^{\dagger})$ such that $\mathbf{Id}-\mathcal{J}_{\e}(\lambda)$ and $\lambda-\G_{\e}$ are invertible in $\Y$ with
\begin{equation}\label{eq:reso}
\Rs(\lambda,\G_{\e})=\Gamma_{\e}(\lambda)(\mathbf{Id}-\mathcal{J}_{\e}(\lambda))^{-1}, \qquad \lambda \in \C_{\mu} \setminus \mathbb{D}(\mu_{\star}-\mu)\,,\quad \e \in (0,{\e}^{\star})\,,
\end{equation}
where $\Gamma_{\e}(\lambda)=\Rs(\lambda,\B_{\e})+\Rs(\lambda,\G_{1,\e})\A_{\e}\,\Rs(\lambda,\B_{\e})$.  Finally, there exists some constant $C >0$ such that
\begin{equation}\label{eq:estimR}
\|\Rs(\lambda,\G_{\e})\|_{\mathscr{B}(\Y)} \leq \frac{C}{\mu_{\star}-\mu}, \qquad \forall \lambda \in \C_{\mu} \setminus \mathbb{D}(\mu_{\star}-\mu)\,, \quad \e \in (0,{\e}^{\star}).\end{equation}
\end{prop}
\begin{proof} We adapt the method of \cite{Tr} but simplifies it in several aspects. For $\mathrm{Re}\lambda >-{\mu}_{\star}$, $\lambda\neq 0$, one knows from Proposition \ref{prop:resG1e} that $\Rs(\lambda,\G_{1,\e}) \in \mathscr{B}(\Y_{1})$ {since we assumed $\ell\geq s+1$} and there is $C_{1} >0$ such that, for any $\e \in (0,\e^{\dagger})$, it holds that
$$\left\|\Rs(\lambda,\G_{1,\e})\right\|_{\mathscr{B}(\Y_{1})} \leq C_{1}\max\left(\frac{1}{|\lambda|},\,\frac{1}{\mathrm{Re}\lambda+\mu_{\star}}\right), \qquad \lambda \in \C_{\mu_{\star}}^{\star}.$$ 
Moreover, from Proposition \ref{prop:hypo}, there is $\bm{\nu} >0$ such that $\B_{\e}+\e^{-2}\bm{\nu}$ is hypo-dissipative in $\Y$. In particular (see Remark \ref{nb:hypo}) there exists $C_{2} >0$, independent of $\e$, such that
$$\|\Rs(\lambda,\B_{\e})\|_{\mathscr{B}(\Y)} \leq \frac{C_{2}}{\mathrm{Re}\lambda+\e^{-2}\bm{\nu}}, \qquad \forall \, \mathrm{Re}\lambda >-\mu_{\star}\,.$$
Therefore, as soon as $\e^{-2}\bm{\nu} > 2\mu_{\star}$, one gets 
\begin{equation}\label{eq:estiRsB}
\|\Rs(\lambda,\B_{\e})\|_{\mathscr{B}(\Y)} \leq \frac{C_{2}\,\e^{2}}{\e^{2}\mathrm{Re}\lambda+\bm{\nu}} \leq C_{3}\,\e^{2}, \qquad \forall \, \mathrm{Re}\lambda >-\mu_{\star},\end{equation}
with $C_{3}=2C_{2}/\bm{\nu}.$  A similar estimate holds true if $\Y$ is replaced with $\Y_{\pm 1}$. Notice that the regularization properties of $\A_{\e}$ in both velocity regularity and tail behaviour implies that there exists $C  >0$ (independent of $\e$) such that 
\begin{equation}\label{eq:Aeij}\|\A_{\e}\|_{\mathscr{B}(\Y_{i},\Y_{j})} \leq C \e^{-2}, \qquad i \leq j, \quad i,j \in \{-1,0,1\}\end{equation}
from which in particular,
$$\|\A_{\e}\Rs(\lambda,\B_{\e})\|_{\mathscr{B}(\Y,\Y_{1})} \leq C_{4}\,,\qquad \qquad \forall\, \mathrm{Re}\lambda >-\mu_{\star},$$
with $C_{4}=C_{3}\,C.$  We deduce with this that, for any $\mathrm{Re}\lambda >-\mu_{\star}$, $\lambda\neq 0$, the operator $\mathcal{J}_{\e}(\lambda) \in \mathscr{B}(\Y)$ is well-defined and, for any $r \in (0,\mu_{\star})$
\begin{equation*}\begin{split}
\left\|\mathcal{J}_{\e}(\lambda)\right\|_{\mathscr{B}(\Y)} &\leq \left\|\G_{\e}-\G_{1,\e}\right\|_{\mathscr{B}(\Y_{1},\Y)}\,\|\Rs(\lambda,\G_{1,\e})\|_{\mathscr{B}(\Y_{1})}\,\left\|\A_{\e}\,\Rs(\lambda,\mathcal{B}_{\e})\right\|_{\mathscr{B}(\Y,\Y_{1})}\\
&\leq C_{5}\frac{1-\re(\e)}{\e^{2}}\,\,\max\left(\frac{1}{|\lambda|},\,\frac{1}{\mathrm{Re}\lambda+\mu_{\star}}\right), \qquad \lambda \in \C_{\mu_{\star}}^{\star}\,,
\end{split}\end{equation*}
with $C_{5}:=C_{0}C_{1}C_{4} >0$ independent of $\e$. Then, for $\mu \in (0,\mu_{\star})$ it holds
\begin{equation}\label{eq:Jel}
\left\|\mathcal{J}_{\e}(\lambda)\right\|_{\mathscr{B}(\Y)} \leq C_{5}\frac{1-\re(\e)}{\e^{2}}\,\,\max\left(\frac{1}{|\lambda|},\,\frac{1}{\mu_{\star}-\mu}\right), \qquad \lambda \in \C_{\mu}^{\star}\,,\end{equation}
which gives \eqref{eq:Jalk}. With this, under Assumptions \ref{hyp:re}, one can choose $\e^{\star}$ small enough, depending on the difference $|\mu_{\star}-\mu|$, so that
\begin{equation}\label{eq:def-re}
\rho(\e)=\frac{C_{5}}{\mu_{\star}-\mu}\,\frac{1-\re(\e)}{\e^{2}} < 1, \qquad \forall \e \in (0,\e^{\star}).\end{equation}
Under such an assumption, one sees that, for all $\lambda \in \C_{\mu}\setminus \mathbb{D}(\mu_{\star}-\mu)$,  $\mathbf{Id}-\mathcal{J}_{\e}(\lambda)$ is invertible in~$\Y$  with 
$$(\mathbf{Id}-\mathcal{J}_{\e}(\lambda))^{-1}=\sum_{p=0}^{\infty}\left[\mathcal{J}_{\e}(\lambda)\right]^{p}, \qquad \forall \e \in (0,\e^{\star}).$$
Let us fix then $\e \in (0,\e^{\star})$ and $\lambda \in \C_{\mu}\setminus \mathbb{D}(\mu_{\star}-\mu)$. The range of $\Gamma_{\e}(\lambda)$ is clearly included in $\D(\B_{\e})=\D(\G_{1,\e})$. Then, writing $\G_{\e}=\A_{\e}+\B_{\e}$ we easily get that
$$(\lambda-\G_{\e})\Gamma_{\e}(\lambda)=\mathbf{Id}-\mathcal{J}_{\e}(\lambda)$$
i.e. $\Gamma_{\e}(\lambda)(\mathbf{Id}-\mathcal{J}_{\e}(\lambda))^{-1}$ is a right-inverse of $(\lambda-\G_{\e}).$ To prove that $\lambda-\G_{\e}$ is invertible, it is therefore enough to prove that it is one-to-one. Consider the eigenvalue problem
$$\G_{\e} h=\lambda\,h, \qquad h \in \D(\G_{\e}),$$
 Writing this as $(\lambda-\G_{1,\e})h=\G_{\e} h-\G_{1,\e}h$, there is a positive constant $C_{6} >0$ independent of $\e$ such that
\begin{equation}\label{eq:hx1} 
\|h\|_{\Y}=\|\Rs(\lambda,\G_{1,\e})(\G_{\e}-\G_{1,\e})h\|_{\Y}  \leq C_{6}\frac{1-\re(\e)}{\e^{2}}\,\,\|h\|_{\Y_{1}}\end{equation}
where we used  Proposition \ref{prop:resG1e} to estimate $\|\Rs(\lambda,\G_{1,\e})\|_{\mathscr{B}(\Y)}$ on $\C_{\mu}\setminus\mathbb{D}(\mu_{\star}-\mu)$ and \eqref{eq:diffG} for the difference $(\G_{\e}-\G_{1,\e})h$. Let us now estimate $\|h\|_{\Y_{1}}$. Since $\G_{\e} h=\lambda\,h$, one has $(\lambda-\B_{\e})h=\A_{\e} h$ and
$h=\Rs(\lambda,\B_{\e})\A_{\e} h$,
so that, thanks to \eqref{eq:estiRsB},
$$\|h\|_{\Y_{1}} \leq \|\Rs(\lambda,\B_{\e})\|_{\mathscr{B}(\Y_{1})}\,\|\A_{\e} h\|_{\Y_{1}} \leq C_{3}\e^{2}\|\A_{\e} h\|_{\Y_{1}}\leq \bar{C}_{3}\|h\|_{\Y}$$
where we used \eqref{eq:Aeij}. Combining this with the above estimate \eqref{eq:hx1}, we end up with
$$\|h\|_{\Y} \leq   C_{7}\frac{1-\re(\e)}{\e^{2}}\,\,\|h\|_{\Y}$$
with $C_{7}:=C_{6}\bar{C}_{3}$ independent of $\e$. One sees that, up to reducing $\e^{\star}$, one can assume that  
$C_{7}\frac{1-\re(\e)}{\e^{2}}< 1$ for $\e \in (0,\e^{\star})$ which implies that $h=0.$ This proves that $\lambda-\G_{\e}$ is one-to-one and its right-inverse is, actually, its inverse. Thus, for $\e\in (0,\e^{\star})$, $\C_{\mu} \setminus \mathbb{D}(\mu_{\star}-\mu)$ belongs to the resolvent set of $\G_{\e}$ and this shows \eqref{eq:reso}. To estimate now $\|\Rs(\lambda,\G_{\e})\|_{\mathscr{B}(\Y)}$ one simply notices that
\begin{equation}\label{eq:I-Jel}
\|(\mathbf{Id}-\mathcal{J}_{\e}(\lambda))^{-1}\|_{\mathscr{B}(\Y)} \leq \sum_{p=0}^{\infty}\|\mathcal{J}_{\e}(\lambda)\|_{\mathscr{B}(\Y)}^{p}\leq \frac{1}{1-\rho(\e)}, \quad \forall\, \lambda \in \C_{\mu}\setminus \mathbb{D}(\mu_{\star}-\mu)\end{equation}
from which, as soon as $\lambda \in \C_{\mu}\setminus\mathbb{D}(\mu_{\star}-\mu)$,
$$\|\Rs(\lambda,\G_{\e})\|_{\mathscr{B}(\Y)}\leq \frac{1}{1-\rho(\e)}\,\|\Gamma_{\e}(\lambda)\|_{\mathscr{B}(\Y)}\,.$$
One checks, using the previous computations, that for $\lambda\in \C_{\mu}\setminus\mathbb{D}(\mu_{\star}-\mu)$,
\begin{equation}\label{eq:normGe}\|\Gamma_{\e}(\lambda)\|_{\mathscr{B}(\Y)} \leq C_{3}\e^{2}+C_{3}\|\A\|_{\mathscr{B}(\Y)}\|\Rs(\lambda,\G_{1,\e})\|_{\mathscr{B}(\Y)}\end{equation}
and deduces \eqref{eq:estimR}.  This achieves the proof.
\end{proof}
%
\begin{nb} Of course, the above result is relevant mainly for $\tfrac{1}{2}\mu_{\star} < \mu < \mu_{\star}$ for which $\mathbb{D}(\mu_{\star}-\mu) \subset \C_{\mu}$, see Figure 1. Notice also that, in previous statement, the parameter $\e^{\star}$ is depending only on the gap
$$\chi:=\mu_{\star}-\mu\,.$$
From \eqref{eq:def-re} we consider $\e$ for which
$$\lambda_{0}:=\lim_{\e\to0^{+}}\frac{1-\re(\e)}{\e^{2}} \ll \chi\,,$$
therefore, $\lambda_{0}$ is a fraction of $\chi$.
\end{nb}
A first obvious consequence of Proposition \ref{lem:inverse} is that, for any $\mu \in (0,\mu_{\star})$, there is $\e^{\star} \in (0,\e^{\dagger})$ depending only on $\chi=\mu_{\star}-\mu$ such that, on $\Y$
$$\mathfrak{S}(\G_{\e}) \cap \{\lambda\in \mathbb{C}\,;\,\mathrm{Re}\lambda > -\mu\} \subset \{z \in \mathbb{C}\,;\,|z| \leq  \mu_{\star}-\mu\}, \qquad \forall\, \e \in (0,\e^{\star}).$$
We denote by $\mathbf{P}_{\e}$ the spectral projection associated to the set
$$\mathfrak{S}_{\e}:=\mathfrak{S}(\G_{\e}) \cap \C_{\mu}=\mathfrak{S}(\G_{\e}) \cap  \mathbb{D}(\mu_{\star}-\mu).$$
One can deduce then the following lemma whose proof is similar to \cite[Lemma 2.17]{Tr}. 
%
\begin{lem}\phantomsection\label{lem:PaP0}\phantomsection For any $\mu \in (0,\mu_{\star})$ there is some $\e_{0}^{\star} \in (0,\e^{\star})$ depending only on $\mu_{\star}-\mu$ and such that
$$\left\|\mathbf{P}_{\e}-\mathbf{P}_{0}\right\|_{\mathscr{B}(\Y)} < 1, \qquad \forall\, \e \in (0,\e_{0}^{\star}).$$
In particular,
\begin{equation}\label{eq:dim}
\mathrm{dim}\,\mathrm{Range}(\mathbf{P}_{\e})=\mathrm{dim}\,\mathrm{Range}(\mathbf{P}_{0})=d+2, \qquad \forall\, \e \in (0,\e_{0}^{\star}).\end{equation}
\end{lem}
\begin{proof} Let $\frac{\mu_{\star}}{2} < \mu < \mu_{\star}$ and $0 < r < \chi:=\mu_{\star}-\mu$. Recall that $\e^{\star}$ depends only on $\chi$. One has $\mathbb{D}(r) \subset \C_{\mu}^{\star}.$ We set $\gamma_{r}:=\{z \in \C\;;\;|z|=r\}$. Recall that
$$\mathbf{P}_{\e}:=\frac{1}{2i\pi}\oint_{\gamma_{r}}\Rs(\lambda,\G_{\e})\d\lambda, \qquad \mathbf{P}_{0}:=\frac{1}{2i\pi}\oint_{\gamma_{r}}\Rs(\lambda,\G_{1,\e})\d\lambda.$$
For $\lambda \in \gamma_{r}$, set
$$\mathcal{Z}_{\e}(\lambda)=\Rs(\lambda,\G_{1,\e})\A_{\e}\Rs(\lambda,\B_{\e})$$
so that $\Gamma_{\e}(\lambda)=\Rs(\lambda,\B_{\e})+\mathcal{Z}_{\e}(\lambda)$. Recall from \eqref{eq:reso} that, for $\lambda \in \gamma_{r}$,
\begin{multline*}
\Rs(\lambda,\G_{\e})=\Rs(\lambda,\B_{\e})(\mathbf{Id}-\mathcal{J}_{\e}(\lambda))^{-1}+\mathcal{Z}_{\e}(\lambda)(\mathbf{Id}-\mathcal{J}_{\e}(\lambda))^{-1}\\
=\Rs(\lambda,\B_{\e})+\Rs(\lambda,\B_{\e})\mathcal{J}_{\e}(\lambda)(\mathbf{Id}-\mathcal{J}_{\e}(\lambda))^{-1}+\mathcal{Z}_{\e}(\lambda)(\mathbf{Id}-\mathcal{J}_{\e}(\lambda))^{-1}\end{multline*}
where we wrote $(\mathbf{Id}-\mathcal{J}_{\e}(\lambda))^{-1}=\mathbf{Id}+\mathcal{J}_{\e}(\lambda)(\mathbf{Id}-\mathcal{J}_{\e}(\lambda))^{-1}.$ In the same way, one sees that
\begin{multline*}
\Rs(\lambda,\G_{1,\e})=\Rs(\lambda,\B_{1,\e})+\Rs(\lambda,\G_{1,\e})\A_{\e}\Rs(\lambda,\B_{1,\e})\\
=\Rs(\lambda,\B_{1,\e})+\Rs(\lambda,\G_{1,\e})\A_{\e}\left[\Rs(\lambda,\B_{1,\e})-\Rs(\lambda,\B_{\e})\right]+\mathcal{Z}_{\e}(\lambda).\end{multline*}
Since the mappings $\lambda \in \mathbb{D}(r) \mapsto \Rs(\lambda,\B_{\e})$ and $\lambda \in \mathbb{D}(r) \mapsto \Rs(\lambda,\B_{1,\e})$ are analytic, one has
$$\oint_{\gamma_{r}}\Rs(\lambda,\B_{\e})\d \lambda=\oint_{\gamma_{r}}\Rs(\lambda,\B_{1,\e})\d \lambda=0\,,$$
so that
$$\mathbf{P}_{\e}=\frac{1}{2i\pi}\oint_{\gamma_{r}}\Rs(\lambda,\B_{\e})\mathcal{J}_{\e}(\lambda)(\mathbf{Id}-\mathcal{J}_{\e}(\lambda))^{-1}\d\lambda
+\frac{1}{2i\pi}\oint_{\gamma_{r}}\mathcal{Z}_{\e}(\lambda)(\mathbf{Id}-\mathcal{J}_{\e}(\lambda))^{-1}\d\lambda\,,$$
whereas 
$$\mathbf{P}_{0}=\frac{1}{2i\pi}\oint_{\gamma_{r}}\Rs(\lambda,\G_{1,\e})\A_{\e}\left[\Rs(\lambda,\B_{1,\e})-\Rs(\lambda,\B_{\e})\right]\d\lambda
+\frac{1}{2i\pi}\oint_{\gamma_{r}}\mathcal{Z}_{\e}(\lambda)\d\lambda.$$
 Consequently, one easily obtains that
\begin{multline*}\mathbf{P}_{\e}-\mathbf{P}_{0}=\frac{1}{2i\pi}\oint_{\gamma_{r}}\Gamma_{\e}(\lambda)\mathcal{J}_{\e}(\lambda)(\mathbf{Id}-\mathcal{J}_{\e}(\lambda))^{-1}\d\lambda\\
+\frac{1}{2i\pi}\oint_{\gamma_{r}}\Rs(\lambda,\G_{1,\e})\A_{\e}\left[\Rs(\lambda,\B_{\e})-\Rs(\lambda,\B_{1,\e})\right]\d\lambda.\end{multline*}
Using \eqref{eq:Jel}, \eqref{eq:I-Jel}, and \eqref{eq:estimR}, one notices that there exists $C >0$ independent of $\e$ such that
$$\left\|\Gamma_{\e}(\lambda)\mathcal{J}_{\e}(\lambda)(\mathbf{Id}-\mathcal{J}_{\e}(\lambda))^{-1}\right\|_{\mathscr{B}(\Y)} \leq \frac{C}{r^{2}(1-\rho(\e))}\frac{1-\re(\e)}{\e^{2}}, \qquad \forall \lambda \in \gamma_{r}\,,$$
where we used that $0< r < \mu_{\star}-\mu$ and noticed that $\|\Gamma_{\e}(\lambda)\|_{\mathscr{B}(\Y)} \leq C/r$ by virtue of \eqref{eq:normGe}. Moreover, from Proposition \ref{prop:resG1e}, it follows that
$$\left\|\Rs(\lambda,\G_{1,\e})\A_{\e}\left[\Rs(\lambda,\B_{\e})-\Rs(\lambda,\B_{1,\e})\right]\right\|_{\mathscr{B}(\Y)} \leq \frac{C_{1}}{r}\,\left\|\A_{\e}\Rs(\lambda,\B_{\e})-\A_{\e}\Rs(\lambda,\B_{1,\e})\right\|_{\mathscr{B}(\Y)}$$
for any $\lambda \in \gamma_{r}$, from which
\begin{multline*}
\left\|\mathbf{P}_{\e}-\mathbf{P}_{0}\right\|_{\mathscr{B}(\Y)} \leq \frac{C_{0}}{r}\left(\frac{1}{r(1-\rho(\e))}\frac{1-\re(\e)}{\e^{2}}+\sup_{\lambda\in\gamma_{r}}\left\|\A_{\e}\Rs(\lambda,\B_{\e})-\A_{\e}\Rs(\lambda,\B_{1,\e})\right\|_{\mathscr{B}(\Y)}\right)\end{multline*}
for some positive constant $C_{0} >0$ independent of $\e$. We only need to estimate 
$$\left\|\A_{\e}\Rs(\lambda,\B_{\e})-\A_{\e}\Rs(\lambda,\B_{1,\e})\right\|_{\mathscr{B}(\Y)}$$ for $\lambda \in \gamma_{r}$. {Observe that, for $\lambda \in \gamma_{r}$,
$$\A_{\e}\Rs(\lambda,\B_{\e})-\A_{\e}\Rs(\lambda,\B_{1,\e})=\A_{\e}\Rs(\lambda,\B_{\e})\left[\B_{\e}-\B_{1,\e}\right]\Rs(\lambda,\B_{1,\e})$$
and, with the notations of the proof of Proposition \ref{lem:inverse}, 
\begin{multline*}
\left\|\A_{\e}\Rs(\lambda,\B_{\e})-\A_{\e}\Rs(\lambda,\B_{1,\e})\right\|_{\mathscr{B}(\Y)}\\
\leq \|\A_{\e}\Rs(\lambda,\B_{\e})\|_{\mathscr{B}(\Y_{-1},\Y_{0})}\,\|\B_{\e}-\B_{1,\e}\|_{\mathscr{B}(\Y_{0},\Y_{-1})}\,\|\Rs(\lambda,\B_{1,\e})\|_{\mathscr{B}(\Y_{0})}.\end{multline*}
Now, as in Proposition \ref{lem:inverse} (see \eqref{eq:estiRsB}) there is a positive constant $C >0$ independent of $\e$ such that
$$\,\|\Rs(\lambda,\B_{\e})\|_{\mathscr{B}(\Y_{-1})} \leq C\e^{2} \qquad \|\Rs(\lambda,\B_{1,\e})\|_{\mathscr{B}(\Y_{0})} \leq C\,,\quad \lambda \in \gamma_{r}\,,$$
where we used the hypo-dissipativity of $\B_{\e}+\e^{-2}\nu_{0}$ in $\Y_{-1}$ thanks to Proposition \ref{prop:hypo}. Now, using \eqref{eq:Aeij} with $i=-1,j=0$, we deduce from \eqref{eq:diffG} that
$$\left\|\A_{\e}\Rs(\lambda,\B_{\e})-\A_{\e}\Rs(\lambda,\B_{1,\e})\right\|_{\mathscr{B}(\Y)} \leq C\frac{1-\re(\e)}{\e^{2}}, \qquad \forall \lambda \in \gamma_{r}.$$
}
Gathering the previous estimates, it follows that, for any $0<r<\chi=\mu_{\star}-\mu$,
\begin{equation}\label{eq:PePo}
\|\mathbf{P}_{\e}-\mathbf{P}_{0}\|_{\mathscr{B}(\Y)} \leq \frac{C}{r}\,\frac{1-\re(\e)}{\e^{2}}\left(\frac{1}{r(1-\rho(\e))}+1\right):=\ell(\e)
\end{equation}
and, thanks to Assumption \ref{hyp:re}, one can find $\e_{\star}$ depending only on $\chi$ such that $\ell(\e) < 1$ for any $\e \in (0,\e^{\star})$. In particular, we deduce \eqref{eq:dim} from {\cite[Paragraph I.4.6]{kato}}.
\end{proof}
\begin{figure}\label{fig:1}
\begin{tikzpicture}[line cap=round,line join=round,x=1.0cm,y=1.0cm, scale=0.9]
\draw[->,color=black] (-6,0) -- (4,0);
\draw[->,color=black] (0,-3.3) -- (0,3.3);
\draw[color=black] (+5.2pt,-5.3pt) node  {\tiny O} ;
\draw(0,0) circle (1cm);
\draw[fill](-5,0) circle (0.6mm);
\draw[fill](-4,0) circle (0.6mm);
\draw (1.6,-0.8) node {\small $\mathbb{D}(\mu_{\star}-\mu)$};
\draw[dashed] (-4,-3)--(-4,3);
\draw[dashed] (-5,-3)--(-5,3);
\draw (0,0)--(0.71,0.71)node[midway, above]{$\chi$};
\draw[<->] (-5,-3.3)--(-4,-3.3)node[midway, below]{$\chi$};
\draw (-5,-0.3) node[left] {$-\mu_{\star}$};
\draw (-4,-0.3) node[right] {$-\mu$};
\draw (4.3,0) node[below] {\small $\mathrm{Re}\lambda$};
\draw (0,3.3) node[left] {\small $\mathrm{Im}\lambda$};
\draw[fill] (-0.3,0) circle (0.6mm);
\draw (-0.1,-0.3) node[left] {\small $-\bar{\lambda}_{\e}$};
\draw (-2,2.3) node {\small $\C_{\mu}\setminus\mathbb{D}(\mu_{\star}-\mu)$};
\end{tikzpicture}
\caption{The set $\C_{\mu} \setminus\mathbb{D}(\mu_{\star}-\mu)$ and the eigenvalue $-\bar{\lambda}_{\e}$.}
\end{figure}
With Lemma \ref{lem:PaP0} we can prove  Theorem \ref{cor:mu}:
\begin{proof}[Proof of Theorem \ref{cor:mu}] We prove the result first in the space 
$$\E=\Y=\W^{s,2}_{v}\W^{\ell,2}_{x}(\m_{q}), \quad \quad \ell \geq s+1, \qquad  {q > q^{\star}+\kappa+2}$$ 
where we recall that $\kappa>\frac{d}{2}$. 
The structure of the spectrum of $\mathfrak{S}(\G_{\e}) \cap \C_{\mu}$ in the space $\Y$ comes directly from Lemma \ref{lem:PaP0} together with Proposition \ref{lem:inverse}. To describe more precisely the spectrum, one first recalls that
$$\mathfrak{S}(\LLe) \cap \{z \in \mathbb{C}\;;\,\mathrm{Re}z > -\mu\} \subset \mathfrak{S}(\G_{\e})  \cap \{z \in \mathbb{C}\;;\,\mathrm{Re}z>-\mu\}.$$
Since, for $\e$ small enough, the spectral projection $\Pi_{\LLe}$ associated to $\mathfrak{S}(\LLe) \cap \C_{\mu}$ satisfies
$$\mathrm{dim(Range}(\Pi_{\LLe}))=\mathrm{dim(Range}(\Pi_{\mathscr{L}_{1}}))=d+2=\mathrm{dim(Range}(\mathbf{P}_{\e}))\,,$$
we get that
\begin{equation}\label{eq:spectr=}
\mathfrak{S}(\LLe) \cap \C_{\mu} = \mathfrak{S}(\G_{\e})  \cap \C_{\mu}\,,
\end{equation}
that is, the eigenvalues $\lambda_{j}(\e)$ are actually eigenvalues of $\LLe$.  In particular, one has that 
$$\lambda_{d+2}(\e)=-\e^{-2} \mu_{\re(\e)}=-\frac{1-\re(\e)}{\e^{2}}+\mathrm{O}\left(\left(\frac{1-\re(\e)}{\e}\right)^{2}\right)\,, \qquad \text{ for } \e \simeq 0\,,$$
according to \eqref{eq:spectLL} and \eqref{eq:mualpha}. We set 
$$\bar{\lambda}_{\e}:=-\lambda_{d+2}(\e) >0, \qquad \lambda_{\e} \simeq -\e^{-2}(1-\re(\e)).$$
For the other eigenvalues, one notices that 
$$\int_{\R^{d}}\LLe \varphi(v)\d v=0\,, \qquad \forall \,\varphi \in \D(\LLe) \subset \Y\,.$$ 
Of course, the spatial variable $x$ plays no role here since $\LLe$ is local in $x$. We begin with understanding the eigenfunctions in 
$${\overline{Y}=L^{2}_{v}L^{2}_{x}(\m_{q+\kappa}).}$$
Recall that 
$$\int_{\R^{d}}\LLe \varphi(v)\d v=0\,, \qquad \forall \,\varphi \in \D(\LLe) \subset \overline{Y}\,,$$
which implies that
$$\m_{q+\kappa}^{-1} \in \D(\LLe^{\star}) \quad \text{ with } \quad \LLe^{\star}(\m_{q+\kappa}^{-1})=0,$$ that is, $0$ is an eigenvalue of the adjoint $\LLe^{\star}$ in $\overline{Y}^{\star}$ and therefore an eigenvalue of $\LLe$ in $\overline{Y}.$ With the same reasoning, since  
$$\int_{ \R^{d}}\LLe \varphi(v)\,v_{i}\d v =-\e^{-2}\kappa_{\re(\e)}\int_{\R^{d}}v_{i}\nabla \cdot (v\varphi(v))\d v=\e^{-2}\kappa_{\re(\e)}\int_{\R^{d}}v_{i}\,\varphi(v)\d v$$
one sees that, for any $i=1,\ldots,d$, $m_{i}^{\star}(v):=v_{i}\m_{q+\kappa}^{-1}(v) \in \D(\LLe^{\star})$ satisfies 
$$\LLe^{\star}m_{i}^{\star}=\e^{-2}\kappa_{\re(\e)}m_{i}^{\star}\,,$$
that is, $\e^{-2}\kappa_{\re(\e)}$ is an eigenvalue of $\LLe^{\star}$ of multiplicity $d$ and, as such, an eigenvalue of $\LLe$ with same multiplicity in the space $\overline{Y}.$ With this, we found $d+1$ eigenvalues of $\LLe$ in the space~$\overline{Y}$. To prove that these $d+1$ eigenvalues are still eigenvalues of $\LLe$ in the smaller space~$\Y$, we proceed as follows. Let $\bar{g}$ be an eigenfunction of $\LLe$ in $\overline{Y}$ associated to the $\e^{-2}\kappa_{\re(\e)}$ eigenvalue, i.e
$$\LLe\,\bar{g}=\e^{-2}\kappa_{\re(\e)}\bar{g}, \qquad \bar{g} \in \D(\LLe) \cap \overline{Y}.$$
With the splitting $\LLe=\B_{\re}^{\delta}+\A^{(\delta)}$, where we recall $\re=\re(\e)$ and $\delta$ is sufficiently small, one deduces from this that
$$\left(\e^{-2}\kappa_{\re(\e)}-\B_{\re}^{\delta}\right)\bar{g}=\A^{(\delta)}\bar{g}.$$
Using the fact that for $\e^{-2}\kappa_{\re(\e)} < \mu_{\star} - \mu < \nu_{\ast}$ the operator $\e^{-2}\kappa_{\re(\e)}-\B_{\re}^{\delta}$ is invertible in both~$\Y$ and $\overline{Y}$ thanks to Proposition \ref{prop:Bree} and 
$$\bar{g}=\Rs(\e^{-2}\kappa_{\re(\e)}\,,\,\B_{\re}^{(\delta)})\A^{(\delta)}\bar{g}\,.$$
Because $\bar{g}$ is depending on the velocity only, using the regularizing effect of $\A^{(\delta)}$ and the {hypo-dissipativity} property of the operator $\e^{-2}\kappa_{\re(\e)}-\B_{\re}^{\delta}$ one concludes that $\A^{(\delta)}\bar{g} \in \Y$ and, by previous identity, so is $\bar{g}$.  Therefore, any eigenfunction of $\LLe$ associated to the eigenvalue $\e^{-2}\kappa_{\re(\e)}$ in $\overline{Y}$ lies in $\Y$ as well, consequently, it is an eigenvalue of $\LLe$ in $\Y$. It has the same multiplicity $d$ as in $\overline{Y}$ since the reasoning is valid for \emph{any} eigenfunction $\bar{g}$. In the same way, we prove that $0$ is a simple eigenvalue of $\LLe$ in $\Y$. We just found exactly $d+1$ eigenvalues and exhausted $\mathfrak{S}(\LLe) \cap \C_{\mu}=\mathfrak{S}(\LLe) \cap \mathbb{D}(\mu_{\star}-\mu)$ under the assumption that $\e^{-2}\kappa_{\re(\e)} < \mu_{\star}-\mu$ which gives the desired result. \\

\noindent Now, let us prove the result in the space
$$\E:=\W^{s,1}_{v}\W^{\ell,2}_{x}(\m_{q}), \qquad q >2, \qquad {\ell \geq s}.$$
As mentioned earlier, we will resort to Theorem 2.1 of \cite{GMM}. We observe that, since $\kappa >\frac{d}{2}$, 
{$$E:=\W^{s,2}_{v}\W^{\ell+1,2}_{x}(\m_{q+q^\star+\kappa}) \hookrightarrow \E$$}
with a continuous embedding, with of course $E$ dense in $\E$. {We observe that $q+q^\star+\kappa > q^{\star}+\kappa+2$ and $\ell+1 \geq s+1$.} From the first part of the proof, the conclusion of Theorem \ref{cor:mu} holds in $E$. It follows then from a simple application of  Theorem 2.1 in \cite{GMM} that the part of the spectrum of  $ \G_{\e}$ lying in the half plance $\{z \in \mathbb{C}\;;\,\mathrm{Re}z \geq -\mu\}$ coincides in both the spaces $\E$ and $E$, i.e. Theorem \ref{cor:mu} holds in $\E$. Notice that all assumptions of Theorem 2.1 in \cite{GMM} are met for the pair of Banach spaces $(E,\E)$ in a straightforward way due to the splitting
$$\G_{\e}=\A_{\e}+\B_{\e}$$
and using Lemma \ref{prop:hypo1}, Propositions \ref{prop:hypo}--\ref{prop:Bree}. In particular, due to the regularizing effect of $\A_{\e}$ in velocity, Hypothesis (H2) (iii) of {\cite[Theorem 2.1]{GMM}} is satisfied with $n=1$. {The extension to spaces of type $\W^{s,2}_{v}\W^{\ell,2}_{x}(\m_{q})$ with $q >q^\star$ and $\ell \geq s$ is similar.}
\end{proof} 
\begin{nb} {It is an interesting open question to determine whether the conclusion of Theorem~\ref{cor:mu} remains true in spaces of the type
$\W^{s,1}_{v}\W^{\ell,1}_{x}(\m_{q}).$
The obstacle here is of course the fact that the regularizing operator $\A_{\e}$ acts in velocity only and do not induce any gain of integrability in the space variable. }  
\end{nb}

\section{Nonlinear analysis}\label{sec:non}

We now apply the results obtained so far to the study of Eq. \eqref{BE}. In all this section we assume that 
\begin{equation}\label{eq:Ekmq}
\E= {\W^{k,1}_{v}\W^{m,2}_{x}(\m_{q})}\end{equation} 
with 
\begin{equation}\label{eq:mkq}
 m  > d, \qquad m-1 \geq {k   \geq 0}, \qquad  {q   \geq 3}\,,
\end{equation}
and introduce also the Hilbert space on which $\mathscr{L}_{1}$ is symmetric 
$$\H:=\W_{v,x}^{m,2}\left(\M^{-1/2}\right).$$ 
We recall here that $\M$ is the steady state of $\mathscr{L}_{1}$ whereas $\H$ is a Hilbert space on which the elastic Boltzmann equation is well-understood {\cite{briant}}.

We also denote
\begin{multline}\label{E1E2}
{\E_{1}:={\W^{k,1}_{v}\W^{m,2}_{x}(\m_{q+1})}}, \qquad \qquad \E_{2}:={\W^{k+1,2}_{v}\W^{m,2}_{x}(\m_{q+{2\kappa}+2}), \qquad \kappa > \frac{d}{2}}\\
 \qquad \qquad \H_{1}:=\W_{v,x}^{m,2}\left(\M^{-1/2}\langle \cdot \rangle^{1/2}\right) \qquad \text{ and } \quad \E_{-1}:=\W^{k,1}_{v}\W^{m,2}_{x} (\m_{q-1})\end{multline}
where $k,m,q$ satisfy \eqref{eq:mkq}.

\medskip
\noindent
The analysis of the elastic case in {\cite{bmam,briant}} holds in $\W^{\beta,2}_{v,x}(\M^{-1/2})$ for $\beta >d$.  We need, however, the $\H$-norm to control the $\E$-norm, which constrains $\beta \geq m$.  At the same time, it is needed that $\mathcal{A}_{\e} \in \mathscr{B}(\E,\H)$ and,  because $\mathcal{A}_{\e}$ has no regularisation effect on the spatial variable, we are forced to choose $\beta \leq m$. {This explains the choice of $\beta=m.$  {Moreover, we need the constraint $m>d$ to carry out our nonlinear analysis, more precisely, we use that the embedding $\W^{m/2,2}_{x}(\T^{d}) \hookrightarrow L^{\infty}_{x}(\T^{d})$ is continuous if $m>d$ which provides us an algebra structure. Notice that the analysis of~{\cite{briant}} is also valid under this condition.}}
{Taking $q \geq3$ allows us to control the dissipation of kinetic energy $\int_{\R^{d}}\Q_{\re}(f,f)|v|^{2}\d v$ and to apply the results of Section \ref{sec:elas}.} Finally, the restriction $k \leq m-1$ in~\eqref{eq:mkq} implies the continuous embedding $\H \hookrightarrow \E_{2}$.
\medskip
\noindent

For $\bm{A,B}> 0$, we will indicate in the sequel $\bm{A} \lesssim \bm{B}$
whenever there is a positive constant $C>0$ depending on the mass and energy of the $h(0)$, but not on parameters like $t,\e$ or $\Mo$, such that
$\bm{A} \leq C\,\bm{B}.$

\medskip
\noindent

We adapt the approach of {\cite{bmam}} and decompose the solution $h_{\e}$ into
$$h_{\e}(t,x,v)=\ho(t,x,v)+\hu(t,x,v)$$
where $\ho=\ho_{\e} \in \E$ and $\hu=\hu_{\e} \in \H$ are the solutions to the following system of equations 
\begin{equation}\label{eq:h0}
\hspace{-.4cm}\left\{\begin{array}{ccl}
\partial_{t} \ho&=&\!\!\!\B_{\re(\e),\e}\ho + \e^{-1}\Q_{\re(\e)}(\ho,\ho) + \e^{-1}\Big[\Q_{\re(\e)}(\ho,\hu)+\Q_{\re(\e)}(\hu,\ho)\Big] \\ [10pt]
&& + {\Big[\G_{\e}\hu-\G_{1,\e}\hu\Big] + \e^{-1}\Big[\Q_{\re(\e)}(\hu,\hu)-\Q_{1}(\hu,\hu)\Big]} \,,\\ [10pt]
\ho(0,x,v)&=&\!\!\!h^{\e}_{\mathrm{\mathrm{in}}}(x,v) \in \mathcal{E}\,.
\end{array}\right.
\end{equation}
and 
\begin{equation}\label{eq:h1}
\left\{
\begin{array}{ccl}
\partial_{t} \hu&=& {\G_{1,\e}\hu} + \e^{-1}\Q_{1}(\hu,\hu) + \A_\e \ho
\,,\\[10pt]
\hu(0,x,v)&=&0\,.
\end{array}\right.
\end{equation}
In this section, we omit the dependence on $\e$ for $\ho$ and $\hu$.  We recall that 
$$\int_{\T^{d}\times\R^{d}} F_{\mathrm{in}}^{\e}(x,v)\left(\begin{array}{c}1\\v\end{array}\right)\d v\d x=0 \Longrightarrow \int_{\T^{d}\times\R^{d}} f_{\e}(t,x,v)\left(\begin{array}{c}1\\v\end{array}\right)\d v\d x=\left(\begin{array}{c}1\\0\end{array}\right)$$
and, in particular, the fluctuation $h_{\e}(t,x,v)$ also satisfies
\begin{equation}\label{eq:massmomH}
\int_{\T^{d}\times\R^{d}}h_{\e}(t,x,v) \left(\begin{array}{c}1\\v\end{array}\right)\d v\d x=\left(\begin{array}{c}0 \\0\end{array}\right).\end{equation}
Recalling the definition of $\mathbf{P}_{0}$ in Theorem \ref{theo:G1e}, we define
\begin{equation}\label{eq:PP0}
\mathbb{P}_{0}h=\sum_{i=1}^{d+1}\left(\int_{\T^{d}\times\R^{d}}h\,\Psi_{i}\,\d v\d x\right)\,\Psi_{i}\,\M\,, \quad \Pi_{0}h=\left(\int_{\T^{d}\times\R^{d}}h\Psi_{d+2}\,\d v\d x\right)\,\Psi_{d+2}\,\M\,,
\end{equation}
where recall that
$$\Psi_{1}=1\,,\quad \Psi_{i}=\frac{v_{i-1}}{\sqrt{\en_{1}}} \; \text{ for } \; i=2,\ldots,d+1\,, \quad \text{and} \quad \Psi_{d+2}=\frac{1}{\en_{1}\sqrt{2d}}(|v|^{2}-d\en_{1}).$$
Of course, see \eqref{eq:P0}, one has $\mathbb{P}_{0}=\mathbf{P}_{0}-\Pi_{0}$.  Recall that the eigenfunctions $\Psi_{j}$ are such that
$$\int_{\R^{3}}\Psi_{i}(v)\Psi_{j}(v)\M(v)\d v=\delta_{i,j} \qquad i,j=1,\ldots,d+2\,,$$
which in particular implies that, in the Hilbert space $\H$\footnote{Recall here that, on the space $L^{2}_{v}(\M^{-\frac{1}{2}})$ the inner product is $\langle f,g\rangle=\int_{\R^{d}}f(v)g(v)\M^{-1}(v)\d v.$}, one has $\mathbf{Id}-\mathbf{P}_{0}=\mathbf{P}_{0}^{\perp}$. We begin with two basic observations. The first one is related to $\mathbb{P}_{0}$:
\begin{lem}\label{lemobservables}
For $i=1,\ldots,d+1$, it holds that
\begin{equation*}
\bigg|\int_{\T^{d}\times\R^{d}} \hu(t,x,v)  \Psi_{i}(v) \d  v \d x \bigg| \leq \max\left(1,\frac{1}{\sqrt{\en_{1}}}\right)\| \ho(t) \|_{\E}\,.
\end{equation*}
As a consequence,
\begin{equation*}
\| \mathbb{P}_{0} \hu(t)\|_{\E}\leq C\| \ho(t) \|_{\E}\,,
\end{equation*}
for some constant $C>0$ depending only on {$\M$}.\end{lem}
\begin{proof}
Note that total mass and momentum conservation leads to
\begin{equation*}\begin{split}
0=\int_{\T^{d}\times\R^{d}} h(t,x,v) \Psi_{i}(v)\d v\d x &= \int_{\T^{d}\times\R^{d}} \ho(t,x,v)  \Psi_{i}(v)\d v\d x \\
&\qquad + \int_{\T^{d}\times\R^{d}} \hu(t,x,v) \Psi_{i}(v)\d v\d x\,, \quad i=1,\ldots,d+1.
\end{split}\end{equation*}
Thus, for any  $i=1,\ldots,d+1$, 
\begin{equation*}
\bigg|\int_{\T^{d}\times\R^{d}} \hu(t,x,v)   \Psi_{i}(v)\d v\d x \bigg| = \bigg|\int_{\T^{d}\times\R^{d}} \ho(t,x,v) \Psi_{i}(v)\d v\d x \bigg| \leq \max\bigg(1,\frac{1}{\sqrt{\en_{1}}}\bigg)\| \ho(t) \|_{\E} \,
\end{equation*}
 {thanks to Cauchy-Schwarz inequality and}
since $|\Psi_{i}(v)| \leq \max\big(1,\frac{1}{\sqrt{\en_{1}}}\big)\m_{q}(v)$ for any $i=1,\ldots,d+1.$ Regarding the estimate for the projection, it follows from the previous inequality and \eqref{eq:PP0} by taking for example
$C:=\max_{i=1,\ldots,d+1}\|\Psi_{i}\M\|_{\E}$.\end{proof}
A second observation regards the action of $\Pi_{0}$ on the linearized operator $\G_{\e}$:
\begin{lem} For any solution $h=h(t,x,v)$ to \eqref{BE}, one has
\begin{equation}\label{pi0Ge}
\Pi_{0}\left[\G_{\e}h(t)\right]=-\bar{\lambda}_{\e}\left(1+r_{\e}\right)\Pi_{0}h(t) + {s}_{\e}(t)\phi_{1}, \qquad t \geq0 \end{equation}
where $\phi_{1}$ is defined in \eqref{eq:limphialpha}, $r_{\e} \in \R$ (independent of $\,t$) and  $s_{\e}(t) \in \R$ are such that
\begin{equation}\label{eq:stimerese}
\left|r_{\e}\right| \leq C_{b}\left(1-\alpha(\e)\right), \qquad \left|{s}_{\e}(t)\right| \leq C_{b}\frac{1-\re(\e)}{\e^{2}}\left\|\left(\mathbf{I-P}_{0}\right)h(t)\right\|_{L^{1}_{x,v}(\varpi_{3})}\end{equation}
for some positive constant $C_{b}$ independent of $\e$ and $t\geq0.$ 
\end{lem}
\begin{proof} The proof is by direct inspection. One first notices that 
$$\Pi_{0}h=\frac{1}{2dc_{0}\vartheta_{1}^{2}}\left(\int_{\R^{d}\times\T^{d}}h\left(|v|^{2}-d\vartheta_{1}\right)\d v\d x\right)\phi_{1}(v)=:\beta_{0}[h]\phi_{1}.$$
Notice that, since $\int_{\R^{d}}\G_{\e}h\d v=0$ and $\int_{\T^{d}}v \cdot \nabla_{x}h \d x=0$ one has
\begin{multline*}
\Pi_{0}\left[\G_{\e}h(t)\right]=\frac{1}{2d\e^{2}c_{0}\vartheta_{1}^{2}}\left(\int_{\T^{d}\times\R^{d}}\LLe h(t,x,v)|v|^{2}\d v\d x\right)\phi_{1}\\
=\frac{1}{2d\e^{2}c_{0}\vartheta_{1}^{2}}\left(\int_{\T^{d}\times \R^{d}}\left(\mathbf{L}_{\alpha(\e)}h(t,x,v)-\kappa_{\alpha(\e)}\nabla_{v}\cdot (vh(t,x,v))\right)|v|^{2}\d v\d x\right)\phi_{1}\\
=\frac{1}{2d\e^{2}c_{0}\vartheta_{1}^{2}}\left(\int_{\T^{d}\times \R^{d}}\left(\mathbf{L}_{\alpha(\e)}h(t,x,v)+2\kappa_{\alpha(\e)}h(t,x,v))\right)|v|^{2}\d v\d x\right)\phi_{1}
\end{multline*}
Now, as in~\eqref{eq:Dre}, one can check that
$$\int_{\T^{d}\times \R^{d}}\mathbf{L}_{\alpha(\e)}h(t,x,v)|v|^{2}\d v\d x=-\frac{1-\alpha^{2}(\e)}{2}\gamma_{b}\int_{\T^{d}\times \R^{2d}}h(t,x,v)G_{\alpha(\e)}(\vet)|v-\vet|^{3}\d\vet\d v\d x$$
which, writing first $h=\Pi_{0}h+\left(\mathbf{I}-\Pi_{0}\right)h=\beta_{0}[h]\phi_{1}+\left(\mathbf{I}-\Pi_{0}\right)h$ and then $G_{\alpha(\e)}=\M_{1}+(G_{\alpha(\e)}-\M_{1})$ gives
$$\Pi_{0}\left[\G_{\e}h(t)\right]=a_{\e}\beta_{0}[h]\phi_{1} + s_{\e}(t)\phi_{1}=a_{\e}\Pi_{0}h+s_{\e}(t)\phi_{1}$$
where 
$$s_{\e}(t):=
-\frac{1-\alpha^{2}(\e)}{4d\e^{2}\,c_{0}\vartheta_{1}^{2}}\gamma_{b}\left(\int_{\T^{d}\times \R^{2d}}\left[\mathbf{I}-\Pi_{0}\right]h(t,x,v)G_{\alpha(\e)}(\vet)|v-\vet|^{3}\d\vet\d v\d x\right)$$
and 
\begin{multline*}
a_{\e}:=\frac{1}{\e^{2}}\bigg\{2\kappa_{\alpha(\e)}-\frac{1-\alpha^{2}(\e)}{4dc_{0}\vartheta^{2}_{1}}\gamma_{b}\int_{\R^{2d}}\phi_{1}(v)\M_{1}(\vet)|v-\vet|^{3}\d\vet\d v\\
-\frac{1-\alpha^{2}(\e)}{4dc_{0}\vartheta^{2}_{1}}\gamma_{b}\int_{\R^{2d}}\phi_{1}(v)\left[G_{\alpha(\e)}(\vet)-\M_{1}(\vet)\right]|v-\vet|^{3}\d\vet\d v\bigg\}.\end{multline*}
One has (see {\cite[Lemma 5.19, Eq. (5.10)]{MiMo3}}) 
$$\frac{\gamma_{b}}{4}\int_{\R^{2d}}\phi_{1}(v)\M_{1}(\vet)|v-\vet|^{3}\d v\d\vet=\frac{3d}{2}c_{0}\vartheta_{1}^{2}$$
which, recalling that $\kappa_{\alpha}=1-\alpha$, results easily in 
\begin{multline*}
a_{\e}=\frac{1-\alpha(\e)}{\e^{2}}\Big\{2-\frac{3}{2}(1+\alpha(\e)) \\-\frac{1+\alpha(\e)}{4dc_{0}\vartheta^{2}_{1}}\gamma_{b}\int_{\R^{2d}}\phi_{1}(v)\left[G_{\alpha(\e)}(\vet)-\M_{1}(\vet)\right]|v-\vet|^{3}\d\vet\d v\Big\}.
\end{multline*}
Writing simply $1+\alpha(\e)=2-(1-\alpha(\e))$ one sees that
$$a_{\e}=\frac{1-\alpha(\e)}{\e^{2}}\left(-1 + \tilde{r}_{\e}\right),$$
with 
$$|\tilde{r}_{\e}| \leq \frac{3}{2}\left(1-\alpha(\e)\right)+ \frac{\gamma_{b}}{2dc_{0}\vartheta_{1}^{2}}\|\phi_{1}\|_{L^{1}_{v}(\m_{3})}\left\|G_{\alpha(\e)}-\M_{1}\right\|_{L^{1}_{v}(\m_{3})} \leq C(1-\alpha(\e))$$
thanks to Lemma \ref{prop:psi}. 
The bound on $s_{\e}(t)$ is also obvious since, for solution $h$ to \eqref{BE}, conservation of mass and momentum implies that $\Pi_{0}h(t)=\mathbf{P}_{0}h(t).$ Then, since 
$$-\bar{\lambda}_{\e}=-\frac{1-\alpha(\e)}{\e^{2}}+\mathrm{O}\left(\left(\frac{1-\alpha(\e)}{\e^{2}}\right)^{2}\right)$$
we get the desired result.
\end{proof}

\begin{nb} If we denote by $\Pi_{\e}$ the spectral projection associated to $\G_{\e}$ and its eigenvalue $-\bar{\lambda}_{\e}$, it may appear at first sight preferable to rather deal with the projection $\Pi_{\e}$ (since $\Pi_{\e}\G_{\e}h=-\bar{\lambda}_{\e}\Pi_{\e}h$) but we face then two different problems: first, $\Pi_{\e}$ is not fully explicit whereas $\Pi_{0}$ is; second, applying $\Pi_{\e}$ to the equation satisfied by $h$
$$\partial_{t} h= \G_{\e}h + \e^{-1}\Q_{\re(\e)}(h,h),$$
nothing guarantees that $\Pi_{\e}\left[\e^{-1}\Q_{\re(\e)}(h,h)\right]$ remains of order $1$ with respect to $\e$ whereas we will see later on (see Lemma \ref{lem:energy}) that, due to the dissipation of kinetic energy,  $\Pi_{0}\left[\e^{-1}\Q_{\re(\e)}(h,h)\right]$ is actually of order $\e$.\end{nb}
In all the sequel, we will denote
$$\lae:=\bar{\lambda}_{\e}\left(1+r_{\e}\right), \quad \lae >0.$$
Notice that $\lae$ does not exactly corresponds to the eigenvalue $\bar{\lambda}_{\e}$ of $\G_{\e}$ but we observe that 
$$\lae \underset{\e \to 0}\sim \bar{\lambda}_{\e}$$
with $\lim_{\e\to0}\lae=\lim_{\e\to 0}\bar{\lambda}_{\e}=\lambda_{0}$ where we recall $\lambda_{0}$ is defined in Assumption \ref{hyp:re}. In the rest of this Section, we estimate separately $\ho$ and $\hu$.

\subsection{Estimating $\ho$} For the part of the solution $\ho(t)$ in $\E$ we have the following estimate.
\begin{prop}\label{prop:h0} Assume that {$\ho\in \E$, {$\hu \in \H$}} are such that
\begin{equation*}
\sup_{t\geq0}\big(\|\ho (t)\|_{\E}  + \|\hu(t)\|_{{{\H}}} \big) \leq \Mo<\infty\,.
\end{equation*}
Let $\nu_{0}:=\min\{\nu_{1,m,k,q},\nu_{1,m,k,q+1}\}$ given in Proposition \ref{prop:hypo}. Then, for $\mu_{0} \in (0,\nu_{0})$ there exists an explicit {$\e_1 >0$ (that can be chosen less than $\e_0$ defined in Theorem \ref{theo:G1e})} such that:
\begin{equation}\begin{split}\label{eq:estimh0}
\|\ho(t)\|_{\E} &\,{\lesssim} \,\|\ho(0)\|_{\E}\,e^{-\frac{\mu_0}{\e^2}t} + {\lambda}_{\e}\int^{t}_{0}e^{-\frac{\mu_0}{\e^2}(t-s)}\|\hu(s)\|_{\E_{2}}\,\d s\\
&\phantom{+++++} +\e{\lambda}_{\e}\int^{t}_{0}e^{-\frac{\mu_0}{\e^2}(t-s)}\|\hu(s)\|^{2}_{\E_{2}}\,\d s\,,\quad {\forall\, \e\in(0,\e_1)}\,.
\end{split}\end{equation}
As a consequence, {for any $\e \in (0,\e_1)$}, 
\begin{equation}
\begin{split}\label{eq:estimh02}
\|\ho(t)\|^{2}_{\E} &\,{\lesssim} \, \|\ho(0)\|^2_{\E}\,e^{-\frac{2\mu_0}{\e^2}t} + \big(\e\, {\lambda}_{\e}\big)^{2}\int^{t}_{0}e^{-\frac{\mu_0}{\e^2}(t-s)}\|\hu(s)\|^{2}_{\E_{2}}\,\d s\\
&\phantom{++++++} +\big( \e^2\, {\lambda}_{\e} \big)^{2}\int^{t}_{0}e^{-\frac{\mu_0}{\e^2}(t-s)}\|\hu(s)\|^{4}_{\E_{2}}\,\d s\,.
\end{split}
\end{equation}
\end{prop}
 \begin{proof} 
{In the subsequent proof, we denote by $\|\cdot\|_{\E_{1}}$ and $\|\cdot\|_{\E}$ the norms on $\E_{1}$ and $\E$ that are equivalent to the standard ones (with multiplicative constants independent of $\e$) and that make $\e^{-2}\nu_{0}+\B_{\re(\e),\e}$ \emph{dissipative}. \footnote{ {More precisely, we shall use here norms such that
$$\text{ if } \quad \partial_{t}g=\B_{\re(\e),\e}g \quad \text{ then } \quad  \frac{\d}{\d t}\|g(t)\|_{\E} \leq -\e^{-2}\nu_{0}\|g(t)\|_{\E_{1}}$$
Notice that such norm exists from Proposition \ref{prop:hypo}.}}
The conclusion with standard norms will simply follows by equivalence.} We first observe that
 \begin{multline*}
\dfrac{\d}{\d t}\|\ho(t)\|_{\E} \leq -\frac{\nu_{0}}{\e^2}\|\ho(t)\|_{\E_{1}} + \e^{-1}\Big(\|\Q_{\re(\e)}(\ho(t),\ho(t))\|_{\E} + \|\Q_{\re(\e)}(\ho(t),\hu(t))\|_{\E}\\
+\|\Q_{\re(\e)}(\hu(t),\ho(t))\|_{\E}\Big)+\big\|\G_{\e}\hu(t) - \G_{1,\e}\hu(t)\big\|_{\E} \\
+ \e^{-1}\big\|\Q_{\re(\e)}(\hu(t),\hu(t))-\Q_{1}(\hu(t),\hu(t))\big\|_{\E}\,.
\end{multline*}
Using classical estimates for $\Q_{\re(\e)}$ and $\Q_1$, (see {\cite{ACG,AG}}), there exist  $C >0$ independent of $\e$ such that
\begin{multline*}
\|\Q_{\re(\e)}(\ho(t),\ho(t))\|_{\E} + \|\Q_{\re(\e)}(\ho(t),\hu(t))\|_{\E}\\
+\|\Q_{\re(\e)}(\hu(t),\ho(t))\|_{\E} \leq C\Big(\|\ho(t)\|_{\E}+\|\hu(t)\|_{\E_{1}}\Big)
\|\ho(t)\|_{\E_{1}}\,,
\end{multline*}
and, {thanks to Remark \ref{nb:diffQXY}}
\begin{multline*}
{\left\|\G_{\e}\hu(t)-\G_{1,\e}\hu(t)\right\|_{\E} + \e^{-1}\left\|\Q_{\re(\e)}(\hu(t),\hu(t))-\Q_{1}(\hu(t),\hu(t))\right\|_{\E}} \\
\leq C(1-\re(\e))\|\hu(t)\|_{\E_{2}}\Big(\e^{-2}+\e^{-1}\|\hu(t)\|_{\E_{2}}\Big).
\end{multline*}
Notice that such estimate is exactly what motivated the definition of $\E_{2}.$
We conclude that
\begin{align*}
\dfrac{\d}{\d t}\|\ho(t)\|_{\E} \leq -&\e^{-2}\Big(\nu_{0}-\e\,C\big(\|\ho(t)\|_{\E}+\|\hu(t)\|_{\E_{1}}\big)\Big)\|\ho(t)\|_{\E_{1}}\\
&+ {C(1-\re(\e))\e^{-2}\|\hu(t)\|_{\E_{2}} + C(1-\re(\e))\e^{-1}\|\hu(t)\|_{\E_{2}}^{2}}\,.
\end{align*}
For any $\mu_0\in(0,\nu_{0})$, we pick {$\e_1 \in (0,\e_0)$ as $\nu_{0}-\e_1\,C\,\Mo \ge \mu_{0}$}.  Therefore,
\begin{equation*}
\nu_{0}-\e\,C\big(\|\ho(t)\|_{\E}+\|\hu(t)\|_{\E_{1}}\big) \geq  \mu_{0}\,, \qquad {\forall\,\e\in(0,\e_{1})}\,.
\end{equation*} 
Consequently, we obtain that
\begin{equation}
\begin{split}\label{imp:h0}
\dfrac{\d}{\d t}\|\ho(t)\|_{\E} &\leq -\frac{\mu_{0}}{\e^2}\,\|\ho(t)\|_{\E_{1}} + C(1-\re(\e))\e^{-2}\|\hu(t)\|_{\E_{2}} \\
&\hspace{5cm}+ C(1-\re(\e))\e^{-1}\|\hu(t)\|_{\E_{2}}^{2},\\
&\leq -\frac{\mu_{0}}{\e^{2}}\,\|\ho(t)\|_{\E_{1}} + C{\lambda}_{\e}\|\hu(t)\|_{\E_{2}} + C\e\,{\lambda}_{\e}\|\hu(t)\|_{\E_{2}}^{2}\,,
 \qquad \forall\, t \geq0,\end{split}
\end{equation}
where we used that $\e^{2}\lae \simeq \e^{2}\bar{\lambda}_{\e} \simeq 1-\re(\e)$ which gives \eqref{eq:estimh0}  after integration. To prove \eqref{eq:estimh02}, we use the fact that  {by Cauchy-Schwarz inequality, for any nonnegative mapping $t \mapsto \zeta(t)$ and $\beta >0$, we have that for any $r \in (0,1)$,
\begin{equation}\label{eq:integsquare2}
\begin{split}
\bigg(\int^{t}_{0}e^{-\beta\,(t-s)}\zeta(s)\d s\bigg)^{2} &\leq \left(\int^{t}_{0}e^{-2r\beta\,(t-s)}\,\d s\right)\left(\int^{t}_{0}e^{-2(1-r)\beta\,(t-s)}\zeta(s)^{2}\,\d s\right)\\
&\leq\frac{1}{2 r \beta}\int^{t}_{0}e^{-2(1-r)\beta\,(t-s)} \zeta(s)^{2}\,\d s\,,\qquad \forall\, t \geq 0\,.
\end{split}\end{equation}
This inequality applied with $r=\frac12$ gives the result.}
\end{proof}
\subsection{Estimating $\mathbf{P}_{0}\hu$}
One has the following fundamental estimate for $\mathbf{P}_{0}\hu(t)$.
\begin{lem}\label{lem:energy}
We have that
\begin{multline}\label{eq:lemenergy}
\| \mathbf{P}_{0} \hu(t)  \|_{\E_{-1}}  \lesssim  \| \Pi_{0} h(0) \|_{\E}e^{-\lae t}  + \|\ho(t)\|_{\E} \\
+\e \lae\int^{t}_{0}e^{-\lae (t-s)}\Big(\|\hu(s)\|^{2}_{\E} + \|\ho(s)\|^{2}_{\E}\Big) \d s \\
+\lae \int_{0}^{t}e^{-\lae (t-s)}\Big(\left\|\ho(s)\right\|_{\E}+\left\|\left(\mathbf{I-P}_{0}\right)\hu(s)\right\|_{\E}\Big)\d s
\end{multline}
for any $t \geq0$.
\end{lem}
\begin{proof}
The equation for $h$ is given by
\begin{equation*}
\partial_{t} h= \G_{\e}h + \e^{-1}\Q_{\re(\e)}(h,h)\,.
\end{equation*}
 Thus, applying the projection $\Pi_{0}$ and using \eqref{pi0Ge}
\begin{equation*}
\partial_{t}\big(\Pi_{0} h\big)=-\bar{\lambda}_{\e}\left(1+r_{\e}\right)\Pi_{0}h + s_{\e}(t)\phi_{1}+ \e^{-1}\Pi_{0} \Q_{\re(\e)}(h,h)\,,\end{equation*}
so that
\begin{align}\label{e1-pert-en}
\Pi_{0}  h(t)= \Pi_{0}  h(0)\,e^{-\lae\,t} + \int^{t}_{0}e^{-\lae\,(t-s)}\left(\e^{-1}\Pi_{0} \Q_{\re(\e)}(h(s),h(s)) + s_{\e}(s)\phi_{1}\right) \d s\,,
\end{align}
where  $\lae=\bar{\lambda}_{\e}\left(1+r_{\e}\right)$.
Notice that, according to \eqref{eq:PP0}, $\Pi_{0}\Q_{\re(\e)}$ is explicit with
\begin{equation*}
 \big\| \Pi_{0} \Q_{\re(\e)}(h(s),h(s)) \big\|_{\E_{-1}} = (1-\re^{2}(\e) )\Big|\mathcal{D}_{\re(\e)}(h(s),h(s)) \Big| \big\|\Psi_{d+2}\M \big\|_{\E_{-1}}\,,
 \end{equation*}
 where $\mathcal{D}_{\re}(g,g)$ denotes 
  the normalized energy dissipation associated to $\Q_{\re}$, namely,
\begin{equation*}\begin{split}
\mathcal{D}_{\re}(g,g)&=-\frac{1}{1-\re^{2}}\int_{\T^{d}\times \R^{d}}\Psi_{d+2}(v)\Q_{\re}(g,g)\d v\d x\\
&=\gamma_{b}\int_{\T^{d}}\d x\int_{\R^d \times \R^d}g(x,v)g(x,\vb)|v-\vb|^{3}\d \vb\d v
\end{split}\end{equation*}
for some nonnegative $\gamma_{b}$ independent of $\re$, see~\eqref{eq:Dre}. 
 {Now, one clearly has
$$|\mathcal{D}_{\re}(h(s),h(s))| \leq C\,\int_{\T^{d}}\left[\int_{\R^{d}}\m_{3}(v)|h(s,x,v)|\d v\right]^{2}\d x$$
and, using Minkowski's integral inequality, we deduce that
$$\left|\mathcal{D}_{\re}(h(s),h(s))\right| \leq C\left[\int_{\R^{d}}\m_{3}(v)\left(\int_{\T^{d}}|h(s,x,v)|^{2}\d x\right)^{\frac{1}{2}}\d v\right]^{2}=C\|h(s)\|_{L^{1}_{v}L^{2}_{x}(\m_{3})}^{2}.$$
}
Therefore,
$$\big\| \Pi_{0} \Q_{\re(\e)}(h(s),h(s)) \big\|_{\E_{-1}} \lesssim (1-\re(\e))\| h(s) \|^{2}_{\E}\,$$
because $\m_{q}(v) \geq \langle v\rangle^{3}$ for any $v\in\R^{3}$. 
\smallskip
\noindent 
Thus, applying the $\|\cdot\|_{\E_{-1}}$-norm in \eqref{e1-pert-en}, one obtains
\begin{multline}\label{eq:Pi0h}
\big\| \Pi_{0}  h(t) \big\|_{\E_{-1}} \lesssim \| \Pi_{0}  h(0) \|_{\E}\,e^{-\lae\,t} + \frac{1-\re(\e)}{\e}\,\int_{0}^{t}e^{-\lae(t-s)}\|h(s)\|^{2}_{\E}\,\d s \\
+	\lae\int_{0}^{t}e^{-\lae(t-s)}\left\|\left(\mathbf{I-P}_{0}\right)h(s)\right\|_{\E}\d s
\end{multline}
where we used \eqref{eq:stimerese} to estimate $\|s_{\e}(s)\phi_{1}\|_{\E_{-1}}$. 
As already observed, according to \eqref{eq:massmomH}, one has 
$$\mathbf{P}_{0}h(t)=\Pi_{0}h(t)$$ for any  $t\geq0.$ 
Since $\mathbf{P}_{0} \hu(t) =\mathbf{P}_{0} h(t) - \mathbf{P}_{0} \ho(t)$, we can reformulate the above \eqref{eq:Pi0h} in terms of the relevant functions $\hu$ and $\ho$ to obtain the desired estimate recalling that $\e\lae \simeq \frac{1-\re(\e)}{\e}$.
\end{proof}
%
\noindent
We make more precise our estimates of $\mathbf{P}_{0}\hu(t)$ in the following
\begin{prop}\label{prop:precisePo}
There exists an explicit $\e_2 \in (0,\e_1)$ such that for any $\e \in (0,\e_2)$ and $t \geq 0$, it holds that
\begin{multline*}
\|\mathbf{P}_{0}\hu(t)\|_{\E_{-1}} \lesssim \Big( \| \Pi_{0} h(0) \|_{\E} + \| \ho(0) \|_{\E} +  \e^{3}\lambda_{\e}\| \ho(0) \|^{2}_{\E} \Big)e^{-\lambda_{\e}t}\\ 
+ \lambda_{\e}\int^{t}_{0}e^{-\frac{\mu_0}{\e^2}(t-s)}\|\hu(s)\|_{\E_{2}}\,\d s +\e\lambda_{\e}\int^{t}_{0}e^{-\frac{\mu_0}{\e^2}(t-s)}\|\hu(s)\|^{2}_{\E_{2}}\,\d s\\
+\lambda_{\e}\int_{0}^{t}e^{-\lae(t-s)}\left\|\left(\mathbf{I-P}_{0}\right)\hu(s)\right\|_{\E}\d s+  {\e^{2}\lambda_{\e}^2} \int_{0}^{t}e^{-\lambda_{\e}(t-s)}\|\hu(s)\|_{\E}\d s\\
+ \e\lambda_{\e}\int^{t}_{0}e^{-\lambda_{\e} (t-s)}\|\hu(s)\|^{2}_{\E} \d s
+ \e^7\,\lambda_{\e}^{3}\int_{0}^{t}e^{-\lambda_{\e}(t-s)}\|\hu(s)\|^{4}_{\E_{2}}\,\d s\,.
\end{multline*}
\end{prop}
\begin{proof} We insert the bound for $\|\ho(t)\|^{i}_{\E}$ for $i=1,2$ in \eqref{eq:estimh0} and \eqref{eq:estimh02} in the estimate of Lemma~\ref{lem:energy}. Assuming $\mu_{0} \geq 2\e^{2}\lambda_{\e}$ and recalling that $1-\re(\e) \simeq \e^{2}\lambda_{\e}$, we first deduce from~\eqref{eq:estimh0}--\eqref{eq:lemenergy}  that 
\begin{multline}\label{eq:P0h1-1}
\| \mathbf{P}_{0} \hu(t)  \|_{\E_{-1}}  \lesssim   \left(\| \Pi_{0} h(0) \|_{\E}+\|\ho(0)\|_{\E}\right)e^{-\lambda_{\e}t} 
+ \lambda_{\e}\int^{t}_{0}e^{-\frac{\mu_0}{\e^2}(t-s)}\|\hu(s)\|_{\E_{2}}\,\d s \\ +\e\lambda_{\e}\int^{t}_{0}e^{-\frac{\mu_0}{\e^2}(t-s)}\|\hu(s)\|^{2}_{\E_{2}}\,\d s 
+ \e\lambda_{\e}\int^{t}_{0}e^{-\lambda_{\e} (t-s)}\|\hu(s)\|^{2}_{\E} \d s\\
+\ {\lambda_\e}\|\ho(0)\|_{\E}\int_{0}^{t}e^{-\lambda_{\e}(t-s)}e^{-\frac{\mu_{0}}{\e^{2}}s}\d s +  {\lambda_{\e}^2}\int_{0}^{t}e^{-\lae(t-s)}\d s\int_{0}^{s}e^{-\frac{\mu_{0}}{\e^{2}}(s-\tau)}\|\hu(\tau)\|_{\E_{2}}\d\tau\\
+{\e\lambda_{\e}^2}\int_{0}^{t}e^{-\lae(t-s)}\d s\int_{0}^{s}e^{-\frac{\mu_{0}}{\e^{2}}(s-\tau)}\|\hu(\tau)\|_{\E_{2}}^{2}\d\tau
+\lae \int_{0}^{t}e^{-\lae(t-s)}\left\|\left(\mathbf{I-P}_{0}\right)\hu(s)\right\|_{\E}\d s\\
+ \e\lambda_{\e}\int^{t}_{0}e^{-\lambda_{\e} (t-s)}\|\ho(s)\|^{2}_{\E}\,\d s\,.
\end{multline}
Now, using \eqref{eq:estimh02} for the last integral, we obtain
\begin{multline*}
\int^{t}_{0}e^{-\lambda_{\e} (t-s)}\|\ho(s)\|^{2}_{\E}\,\d s \lesssim \|\ho(0)\|^2_{\E}\int_{0}^{t}e^{-\frac{2\mu_0}{\e^2}s-\lambda_{\e}(t-s)}\,\d s \\
+ \big(\e\,\lambda_{\e}\big)^{2}\int_{0}^{t}e^{-\lambda_{\e}(t-s)}\,\d s\int^{s}_{0}e^{-\frac{\mu_0}{\e^2}(s-\tau)}\|\hu(\tau)\|^{2}_{\E_{2}}\d \tau\\
 + \big( \e^2\,\lambda_{\e} \big)^{2}\int_{0}^{t}e^{-\lambda_{\e}(t-s)}\,\d s\int^{s}_{0}e^{-\frac{\mu_0}{\e^2}(s-\tau)}\|\hu(\tau)\|^{4}_{\E_{2}}\,\d \tau\,.
\end{multline*}
Using that, for any $\beta >\re >0$ and nonnegative mapping $t\mapsto \zeta(t)$
\begin{equation}\begin{split}\label{eq:integdouble}
\int_{0}^{t}e^{-\re(t-s)}\,\d s\int_{0}^{s}e^{-\beta(s-\tau)}\zeta(\tau)\d\tau
&=e^{-\re\,t}\int_{0}^{t}e^{\beta\tau}\zeta(\tau)\d\tau\int_{\tau}^{t}e^{-(\beta-\re)s}\,\d s\\
&\leq \frac{1}{\beta-\re}\int_{0}^{t}e^{-\re(t-\tau)}\zeta(\tau)\d\tau
\end{split}\end{equation}
we have, for $\mu_{0} \geq 2\e^{2}\lambda_{\e}$, that
\begin{equation}\label{eq:doublehuk}
{\int_{0}^{t}e^{-\lambda_{\e}(t-s)}\,\d s\int^{s}_{0}e^{-\frac{\mu_0}{\e^2}(s-\tau)}\|\hu(\tau)\|^{k}_{\E_{2}}\d \tau \leq \frac{2\e^{2}}{\mu_{0}}\int_{0}^{t}e^{-\lambda_{\e}(t-s)}\|\hu(s)\|_{\E_{2}}^{k}\,\d s, \quad k=1,2,4\,,}\end{equation}
so that
\begin{multline*}
\int^{t}_{0}e^{-\lambda_{\e} (t-s)}\|\ho(s)\|^{2}_{\E}\,\d s \lesssim \e^{2}\|\ho(0)\|^2_{\E}e^{-\lambda_{\e}\,t}
+  (\e^{2}\lambda_{\e})^{2}\int_{0}^{t}e^{-\lambda_{\e}(t-s)}\|\hu(s)\|^{2}_{\E_{2}}\,\d s\\
 +  \big( \e^3\,\lambda_{\e} \big)^{2}\int_{0}^{t}e^{-\lambda_{\e}(t-s)}\|\hu(s)\|^{4}_{\E_{2}}\,\d s\,.
\end{multline*}
{Using again the above \eqref{eq:doublehuk} to estimate the sixth and seventh terms of \eqref{eq:P0h1-1} and keeping only the dominant terms, we get the desired estimate.}
\end{proof}
\begin{nb}\label{nb:P0h1-2}
We will also need an estimate for $\|\mathbf{P}_{0}\hu(t)\|^{2}_{\E_{-1}}$. {Using~\eqref{eq:integsquare2}, we obtain that for any~$r \in (0,1)$, 
\begin{multline*}
\|\mathbf{P}_{0} \hu(t)  \|^2_{\E_{-1}} \lesssim \Big( \| \Pi_{0} h(0) \|^2_{\E} + \| \ho(0) \|^2_{\E} +   (\e^{3}\lambda_{\e})^2 \| \ho(0) \|^{4}_{\E} \Big)e^{-2\lambda_{\e}t}\\
\qquad +  {(\e\,\lambda_{\e})^2} \int^{t}_{0}e^{-\frac{\mu_0}{\e^2}(t-s)}\| \hu(s) \|^2_{\E_{2}}\,\d s +  (\e^2\lambda_{\e})^{2} \int^{t}_{0}e^{-\frac{\mu_0}{\e^2}(t-s)}\| \hu(s) \|^{4}_{\E_{2}}\,\d s\\
 +\frac{\lambda_{\e}}{r}\int_{0}^{t}e^{-2(1-r)\lae(t-s)}\left\|\left(\mathbf{I-P}_{0}\right)\hu(s)\right\|_{\E}^{2}\d s +  {\frac{\e^{4}\lambda_{\e}^3}{r}} \int_{0}^{t}e^{-2(1-r)\lambda_{\e}(t-s)}\|\hu(s)\|_{\E_{2}}^{2}\d s\\
\hspace{2cm}+\,\frac{\e^2\lambda_{\e}}{r}\int^{t}_{0}e^{-2(1-r)\lambda_{\e}(t-s)}\| \hu(s) \|^{4}_{\E_{2}}\,\d s +\frac{\e^{4}\big( \e^{2}\lambda_{\e}\big)^{5}}{r}\int^{t}_{0}e^{-2(1-r)\lambda_{\e}(t-s)}\| \hu(s) \|^{8}_{\E_{2}}\,\d s\,.
\end{multline*}
where the multiplicative constant does not depend on $r$. }
\end{nb}

\subsection{Estimating the complement $(\mathbf{Id-P}_0)\hu$} Let us focus on an estimate on $\mathbf{P}_0^\perp \hu(t)$ with $\mathbf{P}_0^\perp=\mathbf{Id-P}_{0}$, the orthogonal projection onto $\left(\mathrm{Ker}(\G_{1,\e})\right)^{\perp}$ in the Hilbert space $L^2_{v,x}(\M^{-1/2})$.  The same notation for the operator $\G_{1,\e}$ in the spaces $\E$ and $\H$ is used.  
\smallskip
\noindent
We begin with the following lemma where, we recall that $\Sigma_{\M}$ is defined in \eqref{eq:SigmaM}.
\begin{lem}\label{prop:huut} With the notations of  Theorem \ref{theo:G1e}, let {$\e \in (0,\e_0)$}, $\mu \in (0,\mu_{\star})$ and assume that
$$\sup_{t\geq0}\left(\| \ho(t)\|_{\E}+\|\hu(t)\|_{\H}\right) \leq \Mo$$
with {$\Mo\leq 1$} small enough so that
\begin{equation}\label{eq:nu}
\nu:=\frac{2\mu}{\sigma_{0}^{2}}-c_{0}\Mo^{2} >0\end{equation}
where $\sigma_{0}:=\inf_{\xi\in\R^{d}}\Sigma_{\M}(\xi) >0$ and $c_{0} >0$ is a universal constant depending only on $\M$ {defined in~\eqref{eq:hutH}}. Set
$$\huu(t)=\hu(t)-\mathbf{P}_{0}\hu(t)\,, \qquad \forall\, t \geq 0.$$
Then, there exists $C_{0} >0$ independent of $\e >0$ such that
\begin{multline}\label{eq:huut1-0}
\|\huu(t)\|_{\H}^{2} \leq 
C_{0}\int_{0}^{t}e^{-\nu(t-s)}\|\mathbf{P}_{0}\hu(s)\|_{\E_{-1}}^{4}\,\d s+\frac{C_{0}}{\e^{2}}\int_{0}^{t}e^{-\nu(t-s)}\|\hu(s)\|_{\H}\,\|\ho(s)\|_{\E}\,\d s\end{multline}
for any $t \geq0.$ In particular,
\begin{multline}\label{eq:huut1}
\|\huu(t)\|_{\H}^{2} \leq 
C_{0}\Mo^{2}\int_{0}^{t}e^{-\nu(t-s)}\|\hu(s)\|_{\H}^{2}\,\d s+\frac{C_{0}}{\e^{2}}\int_{0}^{t}e^{-\nu(t-s)}\|\hu(s)\|_{\H}\,\|\ho(s)\|_{\E}\,\d s\end{multline}
for any $t \geq0.$
\end{lem}
\begin{proof}  We start by recalling that $h^1(0)=0$ so that $\Psi(0)=0$. One checks from \eqref{eq:h1} that
\begin{align*}
\partial_t \huu &=\G_{1,\e}\huu + \mathbf{P}_0^\perp\left(\e^{-1} \Q_1(\hu,\hu)+\A_\e \ho\right) =\G_{1,\e}\huu + \e^{-1} \Q_1(\hu,\hu) + \mathbf{P}_0^\perp\A_\e \ho\,,
\end{align*}
where for the later we used that $\mathbf{P}_0 \Q_1(\hu,\hu)=0$. {Hereafter, we denote by $\|\cdot\|_\H$ a hypocoercive norm which is equivalent to the usual one independently of $\e$ and which allows us to write nice energy estimates. It is worth mentioning that such a norm has been exhibited in~\cite[Theorem~2.4]{briant}.} Using {\cite[Theorem 4.7]{bmam}}, one obtains as in {\cite[Eq.~(4.8)]{bmam}} that, for any $\mu \in (0,\mu_{\star})$ there is some positive constant $C >0$ such that
\begin{equation}\label{eq:bmh1}
\dfrac{\d}{\d t}\|\huu(t)\|_{\H}^{2} \leq -\frac{2\mu}{\sigma_{0}^{2}}\|\huu(t)\|_{\H_{1}}^2 + C\|\hu(t)\|_{\H}^2\,\|\hu(t)\|_{\H_{1}}^2 + \|\huu(t)\|_{\H}\,\|\mathbf{P}_0^\perp\A_{\e} \ho(t)\|_{\H}.
\end{equation}
Writing $\hu=\mathbf{P}_0 \hu+\huu$, we obtain
\begin{equation*}\begin{split}
\|\hu(t)\|_{\H}^2\,\|\hu(t)\|_{\H_{1}}^2 &\leq 2\|\hu(t)\|_{\H}^2\Big( \|\mathbf{P}_0 \hu(t)\|_{\H_{1}}^2 + \|\huu(t)\|_{\H_{1}}^2\Big)\,\\
&\leq 2\Mo^{2}\|\huu(t)\|_{\H_{1}}^{2}+4\|\mathbf{P}_{0}\hu(t)\|_{\H_{1}}^{2}\left(\|\mathbf{P}_{0}\hu(t)\|_{\H}^{2}+\|\huu(t)\|_{\H}^{2}\right).
\end{split}\end{equation*}
In particular, since there exists a positive constant $c >0$ depending only on $\M$ such that 
$$\|\mathbf{P}_{0}\hu(t)\|_{\H}^{2} \leq \|\mathbf{P}_{0}\hu(t)\|_{\H_{1}}^{2}\leq  c\|\mathbf{P}_{0}\hu(t)\|_{\E_{-1}}^{2}\,,$$
we deduce that
\begin{equation}\label{eq:hutH}
\|\hu(t)\|_{\H}^{2}\|\hu(t)\|_{\H_{1}}^2 \leq c_{0}\|\mathbf{P}_{0}\hu(t)\|_{\E_{-1}}^{4}+c_{0}\Mo^{2}\|\huu(t)\|_{\H_{1}}^{2}\end{equation}
for some universal constant $c_{0} >0$ depending only on $\M$. 
 Therefore, assuming that $\Mo$ is small enough so that
$$\nu:=\frac{2\mu}{\sigma_{0}^{2}}-c_{0}\Mo^{2} >0$$
we deduce that
$$\dfrac{\d}{\d t}\|\huu(t)\|_{\H}^{2} \leq -\nu\,\|\huu(t)\|_{\H_{1}}^2 + C\|\mathbf{P}_{0}\hu(t)\|_{\E_{-1}}^{4}  + \|\huu(t)\|_{\H}\,\|\mathbf{P}_0^\perp\A_{\e} \ho(t)\|_{\H}, \qquad \forall \, t \geq 0$$
for some $C >0$ independent of $t$ and $\e$. Moreover, we also have that
 $$\|\mathbf{P}_0^\perp\A_{\e} \ho(t)\|_{\H} \lesssim \frac{1}{\e^2}\|\ho(t)\|_{\E}\,,\quad \|\huu(t)\|_{\H} \lesssim \|\hu(t)\|_{\H}\,, \qquad \forall \,t \geq 0\,,$$
from which we get the desired estimate \eqref{eq:huut1-0} after integration of the previous differential inequality. We deduce then \eqref{eq:huut1} from \eqref{eq:huut1-0} we use that using the estimate
\begin{equation}\label{eq:estimNorm}
{\| \cdot \|_{\E_{-1}}}\lesssim \| \cdot \|_{\H}\,, \quad \text{and}\quad \| \hu \|^{i+2}_{\H}\leq \Mo^{i}\,\| \hu \|^{2}_{\H}\,,\qquad \qquad \forall i \geq 0\,,
\end{equation}
together with the fact that $\|\mathbf{P}_{0}\hu\|_{\H} \leq \|\hu\|_{\H}.$ 
\end{proof}
To complete the estimate of $\|\huu(t)\|_{\H}^{2}$ we need to estimate the last integral in \eqref{eq:huut1}:
\begin{lem}\label{lem:huut1}
With the notation of Lemma \ref{prop:huut}, there is {an explicit $\e_3\in (0,\e_2)$} such that for any $\delta >0$, {$\e \in (0,\e_3)$}, and $t \geq 0$
\begin{multline*}
{1 \over \e^{2}}\int_{0}^{t}e^{-\nu(t-s)}\|\hu(s)\|_{\H}\,\|\ho(s)\|_{\E}\,\d s \, {\lesssim} \,\frac{\delta}{\e^{2}}\int_{0}^{t}e^{-\nu(t-s)-\frac{\mu_{0}}{\e^{2}}s}\|\hu(s)\|_{\H}^{2}\,\d s
+\frac{1}{\delta}\|\ho(0)\|_{\E}^{2}e^{-\nu\,t}\\
+ \left(\delta+\frac{\lambda_{\e}^{2}}{\delta}
\right)\int_{0}^{t}e^{-\nu(t-s)}\|\hu(s)\|_{\H}^{2}\,\d s\,.\end{multline*}
\end{lem}
\begin{proof} We use the estimate of $\|\ho(s)\|_{\E}$ provided in \eqref{eq:estimh0} which gives
$$\e^{-2}\int_{0}^{t}e^{-\nu(t-s)}\|\hu(s)\|_{\H}\,\|\ho(s)\|_{\E}\,\d s \lesssim I_{1}(t)+I_{2}(t)+I_{3}(t)$$
with
$$I_{1}(t)=\e^{-2}\int_{0}^{t}e^{-\nu(t-s)}\|\hu(s)\|_{\H}\,\|\ho(0)\|_{\E}e^{-\frac{\mu_{0}}{\e^{2}}s}\,\d s,$$
$$I_{2}(t)=\e^{-2}\lambda_{\e}\int_{0}^{t}e^{-\nu(t-s)}\|\hu(s)\|_{\H}\,\d s\int^{s}_{0}e^{-\frac{\mu_0}{\e^2}(s-\tau)}\|\hu(\tau)\|_{\E_{2}}\d \tau\,,$$
and
$$I_{3}(t)=\e^{-1}\lambda_{\e}\int_{0}^{t}e^{-\nu(t-s)}\|\hu(s)\|_{\H}\,\d s\int_{0}^{s}e^{-\frac{\mu_0}{\e^2}(s-\tau)}\|\hu(\tau)\|^{2}_{\E_{2}}\d \tau.$$
Using Young's inequality, for any $\delta >0$ it holds that
$$ \|\hu(s)\|_{\H}\|\ho(0)\|_{\E} \leq  \delta\, \|\hu(s)\|_{\H}^{2}+\frac{1}{4\delta}\|\ho(0)\|_{\E}^{2}\,,$$
so that, since $\mu_{0}-\e^{2}\nu \geq \tfrac{\mu_{0}}{2}$, 
\begin{equation*}\begin{split}
I_{1}(t) &\leq \delta\e^{-2}\int_{0}^{t}e^{-\nu(t-s)-\frac{\mu_{0}}{\e^{2}}s}\|\hu(s)\|_{\H}^{2}\,\d s +\frac{1}{4\delta\,\e^{2}}\|\ho(0)\|_{\E}^{2}\int_{0}^{t}e^{-\nu(t-s)-\frac{\mu_{0}}{\e^{2}}s}\,\d s\\
&\leq \delta\e^{-2}\int_{0}^{t}e^{-\nu(t-s)-\frac{\mu_{0}}{\e^{2}}s}\|\hu(s)\|_{\H}^{2}\,\d s + \frac{1}{2\delta\mu_{0}}\|\ho(0)\|_{\E}^{2}e^{-\nu\,t}.\end{split}\end{equation*} 
Similarly, Young's inequality implies, for any $\delta >0$, that
\begin{align*}
I_{2}(t) \leq \delta \int_{0}^{t}&e^{-\nu(t-s)}\|\hu(s)\|_{\H}^{2}\,\d s \\
&+ \frac{(\e^{-2}\lambda_{\e})^{2}}{4\delta}\int_{0}^{t}e^{-\nu(t-s)}\,\d s
\left(\int_{0}^{s}e^{-\frac{\mu_{0}}{\e^{2}}(s-\tau)}\|\hu(\tau)\|_{\E_{2}}\d \tau\right)^{2},
\end{align*}
and,  using \eqref{eq:integsquare2} with $r=\frac12$ and \eqref{eq:integdouble} to estimate the square of the last integral, we get for $\mu_{0} \geq 2\e^{2}\nu$ that
$$I_{2}(t)  \leq \delta \int_{0}^{t}e^{-\nu(t-s)}\|\hu(s)\|_{\H}^{2}\,\d s + \frac{1}{2\delta\mu_{0}^{2}}\lambda_{\e}^{2}\int_{0}^{t}e^{-\nu(t-s)}\|\hu(s)\|_{\E_{2}}^{2}\,\d s.$$
In the same way, it follows that
$$I_{3}(t)  \leq \delta \int_{0}^{t}e^{-\nu(t-s)}\|\hu(s)\|_{\H}^{2}\,\d s + \frac{1}{2\delta\mu_{0}^{2}}(\e\lambda_{\e})^{2}\,\int_{0}^{t}e^{-\nu(t-s)}\|\hu(s)\|_{\E_{2}}^{4}\,\d s.$$
Combining these estimates yields 
\begin{multline}\label{eq:estimNorm1}
\e^{-2}\int_{0}^{t}e^{-\nu(t-s)}\|\hu(s)\|_{\H}\,\|\ho(s)\|_{\E}\,\d s \, {\lesssim} \, \delta\int_{0}^{t}e^{-\nu(t-s)}\|\hu(s)\|_{\H}^{2}\,\d s +\frac{1}{\delta}\|\ho(0)\|_{\E}^{2}e^{-\nu\,t}\\
+\frac{\delta}{\e^{2}}\int_{0}^{t}e^{-\nu(t-s)-\frac{\mu_{0}}{\e^{2}}s}\|\hu(s)\|_{\H}^{2}\,\d s+ 
 \frac{\lambda_{\e}^{2}}{\delta} \int_{0}^{t}e^{-\nu(t-s)}\|\hu(s)\|_{\E_{2}}^{2}\,\d s\\
+ \frac{(\e\lambda_{\e})^{2}}{\delta}\,\int_{0}^{t}e^{-\nu(t-s)}\|\hu(s)\|_{\E_{2}}^{4}\,\d s\,.
\end{multline}
We conclude thanks to \eqref{eq:estimNorm}.\end{proof}
 
We deduce from the previous the following main estimate for $\|\huu(t)\|_{\H}$ 
\begin{prop} Under the Assumptions of Lemma \ref{prop:huut},  there  exist $\e_4 \in (0,\e_3)$,  {$\lambda_4 >0$}, $c >0$ a positive universal constant  that depends on $\mu_0$ and $\nu$ such that, for any $\delta >0$, $t \geq0$, $\e \in (0,\e_4)$,  {$\lambda_\e \in (0, \lambda_4)$}, 
\begin{multline}\label{est:Ps}
\frac{1}{c}\|\huu(t)\|_{\H}^{2} \leq  \frac{1}{\delta}\|\ho(0)\|_{\E}^{2}e^{-\nu\,t}
+\bm{C}_{2}\int_{0}^{t}e^{-\nu(t-s)}\|\hu(s)\|_{\H}^{2}\d s\\
+\frac{\delta}{\e^{2}}\int_{0}^{t}e^{-\nu(t-s)-\frac{\mu_{0}}{\e^{2}}s}\|\hu(s)\|_{\H}^{2}\d s\,,\end{multline}
with $\bm{C}_{2}=\bm{C}_{2}(\delta,\e,\Mo):=\left[\delta+\Mo^{2}+
\frac{\lae^{2}}{\delta} 
\right].$ In particular,  {there is a positive constant $c_{1}$ depending only on $\mu_{0}$ and $\nu$ such that for any $r \in (0,1)$,
\begin{multline}\label{eq:intPsi}
\frac{1}{c_{1}}\int_{0}^{t}e^{-2(1-r)\lae(t-s)}\|\huu(s)\|_{\H}^{2}\d s \\
\leq  \frac{1}{\delta}\|\ho(0)\|_{\E}^{2}e^{-2(1-r)\lae t}+  \bm{C}_{2}\int_{0}^{t}e^{-2(1-r)\lae(t-s)}\|\hu(s)\|_{\H}^{2}\d s \\
+\frac{\delta}{\e^{2}}\int_{0}^{t}e^{-2(1-r)\lae(t-s)}e^{-\frac{\mu_{0}}{\e^{2}}s}\|\hu(s)\|^{2}_{\H}\d s
\end{multline}}
holds for any $t \geq 0$, $\delta >0$.
\end{prop}

\begin{proof} Inserting the estimate  obtained in Lemma  \ref{lem:huut1}  into \eqref{eq:huut1} we get directly \eqref{est:Ps}. Using \eqref{eq:integdouble} twice, we deduce then easily \eqref{eq:intPsi} from \eqref{est:Ps} after integration {by choosing $\lambda_5$ small enough such that $\nu \geq 4 \lambda_\e$ and thus $\nu - 2(1-r)\lambda_\e \geq \nu - 2 \lambda_\e \geq \frac{\nu}{2}$ for $\lambda_\e \in (0,\lambda_4)$.}
\end{proof}
\begin{nb}\label{nb:lamba} We wish to clarify here the role of the parameter $\lambda_4$ in the above result (and in several similar others in the sequel). Of course, for any choice of $\e$, the parameter $\overline{\lambda}_{\e}$ (or here $\lambda_\e$) is fixed as it is the energy eigenvalue. Asking it to be smaller than the threshold value $\lambda_4$ has to be however understood as a constraint on the parameter $\lambda_{0}$ appearing in \eqref{eq:scaling}. The above result applies to any $\lambda_0 \geq 0$ such that $ {\lambda}_\e < \lambda_4.$ \end{nb}

We deduce from the above the following
\begin{prop}\label{prop:hufin}
Under the assumptions of Lemma \ref{prop:huut}, {there exist $\e_{5} \in (0,\e_{4})$ and $c_{2} >0$  that depends on $\mu_0$ and $\nu$}   such that, for any $\delta \in (0,1)$, $t \geq0$, $\e \in (0,\e_{5})$, {$\lambda_\e \in (0,\lambda_4)$ and $r \in (0,1)$},
\begin{multline}\label{eq:hufin}
\frac{1}{c_{2}}\|\hu(t)\|_{\H}^{2} \leq  {1 \over {\delta r}}\, \mathcal{K}_{0}\, {e^{-2(1-r)\lambda_{\e}t}} +  \left(\delta+\frac{\lambda_{\e}^{2}}{\delta}+\Mo^{2}\right)\int_{0}^{t}e^{-\nu(t-s)}\|\hu(s)\|_{\H}^{2}\d s\\
+ {\frac{\lae}{r}}\left(\delta+\frac{\lambda_{\e}^{2}}{\delta}+\Mo^{2}\right)\,\int^{t}_{0} {e^{-2(1-r)\lambda_{\e}(t-s)}}\| \hu(s) \|^{2}_{\H}\,\d s \\
+ \,  \frac{\delta}{\e^{2}}\left(1+ {\frac{\lae}{r}}\right) \int_{0}^{t} {e^{-2(1-r)\lae(t-s)-\frac{\mu_{0}}{\e^{2}}s}}\|\hu(s)\|_{\H}^{2}\,\d s
\end{multline}
where
$$
\mathcal{K}_{0}:=
\|\ho(0)\|_{\E}^{2}+
\|\Pi_{0}h(0)\|_{\E}^{2}+ \|\ho(0)\|^4_\E$$
depends only on $h(0)$.
\end{prop}
{
\begin{proof} The previous Proposition gives an estimate of $\|\huu(t)\|_{\H}^{2}$ in terms of $\|\hu(t)\|_{\H}^{2}$. Adding $\|\mathbf{P}_{0}\hu(t)\|_{\H}^{2}$ to both sides of \eqref{est:Ps} and, since 
$$\|\hu(t)\|_{\H}^{2} \leq \|\mathbf{P}_{0}\hu(t)\|_{\H}^{2}+\|\huu(t)\|_{\H}^{2} \lesssim \|\mathbf{P}_{0}\hu(t)\|_{\E_{-1}}^{2}+\|\huu(t)\|_{\H}^{2}\,,$$
we only need to estimate  $\|\mathbf{P}_{0}\hu(t)\|_{\E_{-1}}^{2}$ in terms of $\|\hu(t)\|_{\H}^{2}.$ We invoke the estimate in Remark~\ref{nb:P0h1-2} which, using that  {$\frac{\mu_{0}}{\e^{2}} \geq 2(1-r)\lae$} and assuming $\e$ small enough and keeping only the dominant terms reads simply
\begin{multline}\label{eq:Pohufin}
\|\mathbf{P}_{0} \hu(t)  \|^2_{\E_{-1}} \lesssim \Big( \| \Pi_{0} h(0) \|^2_{\E} + \| \ho(0) \|^2_{\E} +  \| \ho(0) \|^{4}_{\E} \Big)e^{-2\lambda_{\e}t}\\
 +{\frac{\e^2\lae}{r}}\int_{0}^{t} {e^{-2(1-r)\lae(t-s)}}\|\hu(s)\|_{\H}^{2} \ds 
 +{\frac{\lambda_{\e}}{r}}\int_{0}^{t} {e^{-2(1-r)\lae(t-s)}}\left\|\Psi(s)\right\|_{\E}^{2}\d s \,
\end{multline}
where we used \eqref{eq:estimNorm} repeatedly. Then, using now \eqref{eq:intPsi} to estimate the last integral and adding \eqref{eq:Pohufin} to \eqref{est:Ps}, there is $c_{0} >0$ depending only on $\mu_{0},\nu$ such that 
\begin{multline*}
\frac{1}{c_{0}}\|\hu(t)\|_{\H}^{2} \leq  {{1 \over {\delta r}} \, \mathcal{K}_{0}}\, {e^{-2(1-r)\lae t}} +\bm{C}_{2}\int_{0}^{t}e^{-\nu(t-s)}\|\hu(s)\|_{\H}^{2}\d s\\
+ {\lae \over r}\left( {\e^2}+ \bm{C}_{2}\right)\int_{0}^{t} {e^{-2(1-r)\lae(t-s)}}\|\hu(s)\|_{\H}^{2}\d s
\\
+\dfrac{\delta}{\e^{2}}\left(1+\ {\frac{\lae}{r}}\right)\int_{0}^{t} {e^{-2(1-r)\lae(t-s)-\frac{\mu_{0}}{\e^{2}}s}}\|\hu(s)\|_{\H}^{2}\d s\end{multline*}
{where we used that $\lae \lesssim 1$ to obtain the first bound in the right-hand-side of the inequality and $\nu \geq 2(1-r)\lae$ for the last term.}
This gives the desired conclusion once we noticed that, for $\e$ small enough (independent of $\delta$), it holds
${\lae\left({\e^{2}}+  \bm{C}_{2}\right)} \lesssim \lae\left(\delta + \frac{\lae^{2}}{\delta}+\Mo^{2}\right)$.
\end{proof}
}
 {We derive from this the following decay rate for $\|\hu(t)\|_{\H}$:}
\begin{cor}\label{h1-relaxation}
 {Let $r \in (0,1)$.} There exist $\e_{6} \in (0,\e_5)$,  {$\lambda_6 \in (0,\lambda_4)$} and $C>0$ depending on $\nu,\mu_{0},\Mo$  {and $r$} such that
\begin{equation}\label{exponential-e1}
\| \hu(t)\|^{2}_{\H} \leq C \,{\mathcal{K}_{0}}\,\exp\left(- 2(1-r)\overline{\lambda}_{\e}\,t\right)\,, \qquad \forall\, t \geq0\,,
\end{equation}
for any $\e \in (0,\e_{6})$, $\lambda_\e \in (0,\lambda_6)$ and where $-\overline{\lambda}_{\e}$ denotes the eigenvalue of $\G_{\re(\e),\e}$ obtained in Theorem \ref{cor:mu}, Eq. \eqref{eq:LambdaE} and  $\mathcal{K}_0$ is given in Proposition \ref{prop:hufin}.
\end{cor}
\begin{proof} Set for simplicity {$r_{0}:=\frac{r}{2}$} and
$$\bm{x}(t) := {e^{2(1-r_{0})\lambda_{\e}t}}\|\hu(t)\|_{\H}^{2}, \qquad t \geq 0.$$  {We choose $\lambda_6>0$ small enough so that $\lambda_\e \leq r_{0}$ for any $\lambda_\e \in (0,\lambda_6)$.} 
Inequality \eqref{eq:hufin}  {(applied with $r_{0}$ instead of $r$)} then yields
\begin{align*}
\frac{1}{{c_{2}}}\bm{x}(t) \leq  {{1 \over {\delta r_{0}}}} \, \mathcal{K}_{0} &+ \left(\delta+\frac{\lambda_{\e}^{2}}{\delta}+\Mo^{2}\right)\int_{0}^{t} {e^{-\left(\nu-2(1-r_{0})\lae\right)(t-s)}}\bm{x}(s)\d s\\
&\qquad+ {\frac{\lae}{r_{0}}}\left(\delta+\frac{\lambda_{\e}^{2}}{\delta}+\Mo^{2}\right)\,\int^{t}_{0}\bm{x}(s)\,\d s +\frac{\delta}{\e^{2}}\int_{0}^{t}e^{-\frac{\mu_{0}}{\e^{2}}s}\bm{x}(s)\,\d s
\end{align*}
{where $c_2$ is defined in the previous proposition}.  
We use a Gronwall type argument to prove the result. For notational simplicity introduce 
\begin{multline*}
C_{\delta}(\lae,\Mo)=c_{2}\left(\delta+\frac{\lambda_{\e}^{2}}{\delta}+\Mo^{2}\right), \quad   {A_0=\frac{c_{2}\mathcal{K}_{0}}{\delta r_{0}}},\quad  {\xi_{\delta}}(t)= {\frac{\lae}{r_{0}}}\,C_{\delta}(\lae,\Mo)+\frac{c_{2}\delta}{\e^{2}}e^{-\frac{\mu_{0}}{\e^{2}}t}, \quad t \geq 0\,,
\end{multline*}  
from which one obtains that
\begin{equation}\label{eq:x(t)}
0 \leq \bm{x}(t) \leq A_0 + C_{\delta}(\lambda_\e,\Mo)\int_{0}^{t}{e^{-(\nu-2(1-r_{0})\lambda_{\e})(t-s)}}\bm{x}(s)\d s + \int_{0}^{t} {\xi_{\delta}}(s)\bm{x}(s)\d s=:\Upsilon(t)\,.
\end{equation}
Thus,
\begin{equation*}\begin{split}
\frac{\d}{\d t}\Upsilon(t)&=- {(\nu-2(1-r_{0})\lambda_{\e})}C_{\delta}(\lambda_\e,\Mo)\int_{0}^{t} {e^{-(\nu-2(1-r_{0})\lambda_{\e})(t-s)}}\bm{x}(s)\d s \\
&\quad + \left(C_{\delta}(\lambda_\e,\Mo)+ {\xi_{\delta}}(t)\right)\bm{x}(t)\\
&=-{(\nu-2(1-r_{0})\lambda_{\e})}\left(\Upsilon(t) - A_0 - \int_{0}^{t} {\xi_{\delta}}(s)\bm{x}(s)\d s\right) + \left(C_{\delta}(\lambda_\e,\Mo)+ {\xi_{\delta}}(t)\right)\bm{x}(t).
\end{split}\end{equation*}
Using \eqref{eq:x(t)}, which reads $\bm{x}(t) \leq \Upsilon(t)$, we deduce that
\begin{align*}
\frac{\d}{\d t}\Upsilon(t) &\leq  {(\nu-2(1-r_{0})\lambda_{\e})}A_0 + \left[C_{\delta}(\lambda_\e,\Mo)- {(\nu-2(1-r_{0})\lambda_{\e})} +  {\xi_{\delta}}(t)\right]\Upsilon(t) \\
&\quad +  {(\nu-2(1-r_{0})\lambda_{\e})}\int_{0}^{t} {\xi_{\delta}}(s)\Upsilon(s)\d s\,.
\end{align*}
Clearly, it is possible to choose {$\delta \in (0,1)$ and $\Mo$ sufficiently small depending on $\mu_0$ and $\nu$ and then $\lambda_\star$ sufficiently small depending on $\mu_0$, $\nu$, $\delta$  {and $r$}} such that
$$
C_{ {\delta}}(\lambda_\e,\Mo)- {(\nu-2(1-r_{0})\lambda_{\e})} +  {\xi_{\delta}}(t)  \leq -\frac{\nu}{2} + \frac{ {\delta}\,c_{2}}{\e^{2}}e^{-\frac{\mu_{0}}{\e^{2}}t}, \qquad \forall\, t\geq 0,
$$
 {and
\begin{equation}\label{eq:lambdae}
e^{\frac{ {\delta}\,c_{2}}{\mu_{0}}} c_{2}\left( {\delta}+\Mo^{2}+\frac{\lae^{2}}{\delta}\right)\leq r_{0}^2 
\end{equation}}
hold true for any $\lambda_{\e} \in (0,\lambda_{\star})$.  Fixing $\lambda_{\e}$ in this range, we introduce
$$\bm{z}(t)=-\frac{\nu}{2}t+\frac{ {\delta}\,c_{2}}{\e^{2}}\int_{0}^{t}e^{-\frac{\mu_{0}}{\e^{2}}s}\d s$$
to deduce that
\begin{align*}
\dfrac{\d}{\d t}\left(e^{-\bm{z}(t)}\Upsilon(t)\right) &\leq {(\nu-2(1-r_{0})\lambda_{\e})}\,A_0\,e^{-\bm{z}(t)} + {(\nu-2(1-r_{0})\lambda_{\e})}e^{-\bm{z}(t)}\int_{0}^{t} {\xi_{\delta}}(s)\Upsilon(s)\d s\\
&\leq \nu\,A_0\,e^{-\bm{z}(t)} + \nu\,e^{-\bm{z}(t)}\int_{0}^{t} {\xi_{\delta}}(s)\Upsilon(s)\d s\,
\end{align*}
{where $A_{0}$ is defined as before. }
Integration of this differential inequality yields (recalling that $\Upsilon(0)=A_0$)
\begin{align*}
\Upsilon (t) &\leq A_0e^{\bm{z}(t)} + \nu\,A_0\,\int^{t}_0e^{\bm{z}(t)-\bm{z}(s)}\d s  + \nu\int_{0}^{t}e^{\bm{z}(t)-\bm{z}(s)}\d s\int_{0}^{s} {\xi_{\delta}}(\tau)\Upsilon(\tau)\d \tau\\
&= A_0e^{\bm{z}(t)} + \nu\,A_0\,\int^{t}_0 e^{\bm{z}(t)-\bm{z}(s)}\d s  + \nu \int_{0}^{t} {\xi_{\delta}}(\tau)\Upsilon(\tau)\bigg(\int^{t}_{\tau}e^{\bm{z}(t)-\bm{z}(s)}\d s\bigg)\d \tau\,.
\end{align*}
Notice that $\bm{z}(t)-\bm{z}(s) \leq \frac{ {\delta}\,c_{2}}{\mu_{0}} -\frac{\nu}{2}(t-s)$, for $0\leq s \leq t$, from which we conclude that
\begin{equation*}
\int_{\tau}^{t}e^{\bm{z}(t)-\bm{z}(s)}\d s \leq \frac{2}{\nu}\,e^{\tfrac{ {\delta}\,c_{2}}{\mu_{0}}}\,,\qquad 0 \leq \tau \leq t\,.
\end{equation*}
Consequently,
\begin{align*}
\Upsilon(t) \leq 3A_0\,e^{\frac{ {\delta}\,c_{2}}{\mu_{0}}} + 2 e^{\frac{ {\delta}\,c_{2}}{\mu_{0}}}\int_{0}^{t} {\xi_{\delta}}(s)\Upsilon(s)\d s\,,
\end{align*}
which, thanks to Gronwall lemma, implies that 
$$\Upsilon(t) \leq 3A_0\,e^{\frac{{\delta}\,c_{2}}{\mu_{0}}}\exp\left(2e^{\frac{ {\delta}\,c_{2}}{\mu_{0}}}\int_{0}^{t} {\xi_{\delta}}(s)\d s\right).$$
Notice that
\begin{equation*}
\int_{0}^{t} {\xi_{\delta}}(s)\d s \leq c_{2} {\frac{\lae}{r_{0}}}\left( {\delta}+\Mo^{2}+\frac{\lae^{2}}{ {\delta}}\right)\,t + \frac{c_{2} {\delta}}{\mu_{0}}\,.
\end{equation*}
Moreover, {\eqref{eq:lambdae} implies that}
$$
2 e^{\frac{ {\delta}\,c_{2}}{\mu_{0}}} c_{2} {\frac{\lae}{r_{0}}}\left( {\delta}+\Mo^{2}+\frac{\lae^{2}}{ {\delta}}\right)\leq  {{r} \lae}\quad \text{for any} \quad \e\in(0,\e_\star)\,.
$$
Consequently,
$$\Upsilon(t) \leq C A_0\,e^{ {r \lae t}}$$
for some positive constant $C$ depending on $\nu$, $\mu_0$, $\Mo$  {and $r$}.  Recalling the definition of $A_0$, such estimate combined with \eqref{eq:x(t)} gives that
\begin{equation}\label{exponential-e1/2}
\| \hu(t)\|^{2}_{\H} \leq C \,\mathcal{K}_{0}\,\exp\left(- {2(1-r)\lae}t\right)\,, \qquad \forall\, t \geq0\,.
\end{equation}
 {This estimate yields the final result since it can be proven for any $r \in (0,1)$ and $\lae \sim_{\e \to 0} \overline \lae$. } \end{proof}
 
Estimate \eqref{exponential-e1} leads to the main result of this section.
\begin{theo}\label{theo:h-relaxation}  {Let $r \in (0,1)$.} There exist {$\e^{\dagger} \in (0,\e_6)$, $\lambda^{\dagger} \in( 0, \lambda_6)$}, and $C>0$ depending on  {$\nu,\mu_{0},\Mo,r$}  such that
\begin{equation*}
\| h(t)\|^{2}_{\E} \leq C \,\Big(  \| h(0)\|^{2}_{\E} + \| h(0)\|^{4}_{\E}\Big)\,\exp\left(- {2(1-r)\overline \lambda_{\e}}t\right)\,, \qquad \forall\, t \geq0\,,
\end{equation*}
where $\overline \lambda_\e$ is defined in~\eqref{eq:LambdaE}, for any $\e \in (0,\e^{\dagger})$ and $\lambda_{\e} \in (0,\lambda^{\dagger})$.
\end{theo}

\begin{proof} 
Using estimate \eqref{exponential-e1} in estimate \eqref{eq:estimh02}, it follows that
\begin{equation}\label{h0-decay}
\|\ho(t)\|^{2}_{\E} \leq 2\|\ho(0)\|^2_{\E}\,e^{-\frac{2\mu_0}{\e^2}t} +  C\mathcal{K}_0 \e^{2}\big(\e\,\lambda_{\e}\big)^{2} {e^{-{2(1-r)\overline \lambda_{\e}}t}}\leq C\,\mathcal{K}_{0}\, {e^{-{2(1-r)\overline \lambda_{\e}}t}}\,.
\end{equation}
Consequently,
\begin{equation}\label{h-decay}
\begin{split}
\|h(t)\|^{2}_{\E} \leq 2\big( \|h^{0}(t)\|^{2}_{\E} + \| \hu(t)\|^{2}_{\E}\big)\leq C\big( \|h^{0}(t)\|^{2}_{\E} + \| \hu(t)\|^{2}_{\H}\big)\leq C\,\mathcal{K}_{0}\, {e^{-2(1-r)\overline\lambda_{\e}t}}\,.
\end{split}
\end{equation}
Recalling that $h^{0}(0)=h(0)$ in the definition of $\mathcal{K}_0$, estimate \eqref{h-decay} gives the result.\end{proof}
 \begin{nb} \label{nb:1/2}
In the following sections, for sake of simplicity, we will use this result for the special value $r=\frac{1}{2}$ and denote by $\e^\dagger$ and $\lambda^\dagger$ the threshold values associated to $r=\frac{1}{2}$.
\end{nb} 
We also point out the gain of decay in $h$ in the following corollary.

\smallskip
\begin{cor}\label{cor:h-relaxation}
Under the same conditions of Theorem \ref{theo:h-relaxation} it follows that
\begin{equation*}
\int^{t}_{0}\|  h (\tau)\|_{ {\E_{1}} } \d \tau \leq C\sqrt{ \mathcal{K}_0 }\min\left\{ 1 + t, 1 + {\tfrac{1}{ \overline \lambda_{\e} }} \right\}\,,\qquad \forall\, t>0\,.
\end{equation*}
In particular, $\|  h (\cdot)\|_ {{\E_{1}}} $ is integrable and exists $a.e.$ in $(0,T)$ for any $T>0$.
\end{cor}
\begin{proof}
After performing time integration of equation \eqref{imp:h0} in $[0,t]$ one finds that
\begin{equation}\label{aprioriho}
\begin{split}
\|\ho(t)\|_{\E} + \frac{\mu_{0}}{\e^2}\,&\int^{t}_{0}\|\ho(\tau)\|_{{\E_{1}}}\d \tau \\
&\leq \|\ho(0)\|_{\E} + C\int^{t}_{0}\Big(\lambda_{\e}\|\hu(\tau)\|_{ {\E_{1}}} + \e\,\lambda_{\e}\|\hu(\tau)\|_{ {\E_{1}}}^{2}\Big)\d \tau \\
&\leq C\Big( \sqrt{ \mathcal{K}_0 } + \e\,\mathcal{K}_{0} \big)\,, \qquad \forall\, t > 0\,,
\end{split}
\end{equation}
where we used estimate \eqref{exponential-e1} in the latter inequality.  Thus,
\begin{equation*}
\begin{split}
\int^{t}_{0}\|  h (\tau)\|_{ {\E_{1}} } \d \tau &\leq \int^{t}_{0}\|  h^{0} (\tau)\|_{  {\E_{1}} } \d \tau + \int^{t}_{0}\|  h^{1} (\tau)\|_{  {\E_{1}} } \d \tau \\
&\leq C\big( \sqrt{ \mathcal{K}_0 } + \e\,\mathcal{K}_{0} \big) + C\sqrt{ \mathcal{K}_0 }\int^{t}_{0} {e^{-(1-r) \overline \lambda_{\e}\tau}}\d \tau\,,\\
\end{split}
\end{equation*}
which gives the result.
\end{proof}
\begin{nb} Of course, for a fixed $\e >0$, one can replace $\min\big\{1+t,1+{\frac{1}{\overline \lambda_{\e}}}\big\}$ with $1+ {\frac{1}{\overline \lambda_{\e}}}$ and the above estimate shows that $h(t)=h_{\e}(t) \in L^{1}([0,\infty),{\E_{1}}).$ However, in the case in which {$\lim_{\e\to0}\overline \lambda_{\e}=0$} then the bound is not uniform with respect to $\e$. In practice, two situations occur according to the value of $\lambda_{0}$ in Assumption \ref{hyp:re}:
\begin{enumerate}[a)]
\item If {$\lambda_0 >0$}, then the family $\{h_{\e}(t)\}_{\e\geq0}$ is bounded in $L^{1}([0,\infty), {\E_{1}})$,\\

\item If {$\lambda_0=0$} then for any $T >0$, the  family $\{h_{\e}(t)\}_{\e\geq0}$ is bounded in $L^{1}([0,T], {\E_{1}})$.
\end{enumerate}
\end{nb}

\section{Cauchy Theory}\label{sec:Cauchy}

The scope of this Section is to prove the well-posedness of the system \eqref{eq:h0}-\eqref{eq:h1} thanks to the \emph{a priori} estimates derived in the previous section. We namely aim to prove the following precise version of Theorem \ref{theo:main-cauc} and  we will use the functional spaces introduced at the beginning of Section~\ref{sec:non}. 
\begin{theo}\label{theo:main-cauc1}
Under Assumption \ref{hyp:re}, let
\begin{equation*}
    {m  > d}, \qquad m-1 \geq  k \geq 0, \qquad q   \geq 3,
\end{equation*}
be fixed. 
There exists a triple $(\varepsilon^{\dagger},\lambda^{\dagger},\mathcal{K}^{\dagger}_{0})$ depending only on the mass and energy of $F_{\mathrm{in}}^\e$ and $m,k,q$ such that, for $\e\in(0,\e^{\dagger})$, {$\overline{\lambda}_{\e}\in(0,\lambda^{\dagger})$}, and $\mathcal{K}_0\in(0,\mathcal{K}^{\dagger}_{0})$, if
$$\|F_{\mathrm{in}}^{\e}-G_{\re(\e)}\|_{ {\W^{k,1}_{v}\W^{m,2}_{x}}(\langle v\rangle^{q})} \leq \e\,\sqrt{\mathcal{K}_{0}}$$
then the inelastic Boltzmann equation \eqref{BE0} has a unique solution 
$$f_{\e} \in\mathcal{C}\big([0,\infty); {\W^{k,1}_{v}\W^{m,2}_{x}}(\langle v\rangle^{q})\big)$$ satisfying for $t>0$
\begin{multline*}
\left\|f_{\e}(t)-G_{\re(\e)}\right\|_{ {\W^{k,1}_{v}\W^{m,2}_{x}}(\langle v\rangle^{q})}\leq C\e\sqrt{ \mathcal{K}_0 }\,\exp\left(-\overline{\lambda}_{\e}\,t\right),\\
\quad\text{and}\quad \int^{t}_{0}\left\|f_{\e}(\tau)-G_{\re(\e)}\right\|_{ {\W^{k,1}_{v}\W^{m,2}_{x}}(\langle v\rangle^{q+1})} \d \tau \leq C\e\sqrt{ \mathcal{K}_0 }\min\Big\{ 1 + t, 1 + \tfrac{1}{ \overline{\lambda}_{\e} } \Big\}\,,
\end{multline*}
for some positive constant $C >0$ independent of $\e$ and where we recall that $\overline{\lambda}_{\e} \simeq \frac{1-\re(\e)}{\e^{2}}$ is the energy eigenvalue of the linearized operator derived in Theorem \ref{cor:mu}.
\end{theo}    
\begin{nb} As in Remark \ref{nb:lamba}, the threshold value $\lambda^{\dagger}$ has to be  understood as a constraint on the parameter $\lambda_{0}$ appearing in \eqref{eq:scaling}. The above result applies to any $\lambda_0 \geq 0$ such that $\overline{\lambda}_\e < \lambda^\dagger.$ \end{nb}


\subsection{Iteration scheme}  Let us follow the iteration scheme of \cite[Section 3]{Tr} with suitable modifications.  We are seeking to approximate the solution to the inelastic Boltzmann equation using the iteration scheme
\begin{equation}\label{CTe1}
\begin{cases}
\partial_{t}h_{n+1}(t)  &= \G_{\e}h_{n+1}(t) + \e^{-1}\Q_{\re(\e)}(h_{n}(t),h_{n}(t))\,, \quad n \geq 1\\[3pt]
\partial_{t}h_{1}(t) &=\G_{\e}h_{1}(t)\,,\\
h_{n}(0) &= h(0)\in\E\,,\quad \quad n \geq 1,
\end{cases}
\end{equation} 
where the initial perturbation $h(0)$ has zero mass and momentum.  This is done using the decomposition of previous section.  More precisely, writing $h_{n}= h^{0}_{n} + h^{1}_{n}$ we consider solutions with the coupled system
\begin{equation}\label{CTe2}
\hspace{-.4cm}\left\{\begin{array}{ccl}
\!\!\partial_{t} \ho_{n+1}\!\!\!\!&=&\!\!\!\B_{\re(\e),\e}\ho_{n+1} + \e^{-1}\Q_{\re(\e)}(\ho_{n},\ho_{n}) + \e^{-1}\Big[\Q_{\re(\e)}(\ho_{n},\hu_{n})+\Q_{\re(\e)}(\hu_{n},\ho_{n})\Big] \\ [10pt]
&& + \Big[\G_{\e}\hu_{n+1}-\G_{1,\e}\hu_{n+1}\Big] + \e^{-1}\Big[\Q_{\re(\e)}(\hu_{n},\hu_{n})-\Q_{1}(\hu_{n},\hu_{n})\Big]\,,\\ [10pt]
\!\!\ho_{n+1}(0)\!\!\!\!&=&\!\!\!h^{0}(0) \in \mathcal{E}\,,
\end{array}\right.
\end{equation}
and
\begin{equation}\label{CTe3}
\left\{
\begin{array}{ccl}
\partial_{t} \hu_{n+1}&=& \G_{1,\e}\hu_{n+1} + \e^{-1}\Q_{1}(\hu_{n},\hu_{n}) + \A_\e \ho_{n+1} \,,\\[10pt]
\hu_{n+1}(0)&=&\hu(0)\in \H\,.
\end{array}\right.
\end{equation}
Motivated by the \textit{a priori} estimates of Section \ref{sec:non}, we introduce the following norms
{$$\vertiii{g}_{0}:=\sup_{ t\geq0}\Big(\| g(t) \|_{\E} + \e^{-2}\int^{t}_{0}\| g(\tau) \|_{\E_{1}} \d \tau\Big), \qquad g\in \mathcal{C}([0,\infty),\E)\,,$$
and 
$$\vertiii{g}_{1}:=\sup_{t\geq 0}\Big(\| g(t) \|^{2}_{\H} + \int^{t}_{0}\| g(\tau) \|^{2}_{\H_{1}} \d \tau \Big)^{\frac{1}{2}}, \qquad
g \in \mathcal{C}([0,\infty),\H)\,,$$}
where 
we recall that $\E, \E_{1}$, $\H$ and $\H_{1}$ are defined in \eqref{E1E2}.   
Notice that $\big(\mathcal{C}([0,\infty),\E)\,;\,\vertiii{\cdot}_{0}\big)$ and $\big(\mathcal{C}([0,\infty),\H)\,;\,\vertiii{\cdot}_{1}\big)$ are Banach spaces. In particular, the space 
$$\mathbb{B}:=\mathcal{C}([0,\infty),\E) \times \mathcal{C}([0,\infty),\H)$$
endowed with the norm 
$$\vertiii{(g,h)}:=\vertiii{g}_{0}+\vertiii{h}_{1} \qquad \text{ for  } \quad (g,h) \in \mathbb{B},$$
is a Banach space. Define then
\begin{multline}\label{CTe4}
\mathcal{X}_0 = \Big\{ \ho\in\mathcal{C}\big([0,\infty);\E\big) \;\Big|\; \vertiii{\ho}_{0} \leq C\sqrt{\mathcal{K}_0} \Big\}\,,\\
\mathcal{X}_1=\Big\{ \hu\in\mathcal{C}\big( [0,\infty); \H\big) \;\Big|\; \vertiii{\hu}_{1} \leq C\sqrt{ \mathcal{K}_0 } \Big\}\,,
\end{multline}
{for some positive constant $C >0$ which can be explicitly estimated from the subsequent computations.}  The system \eqref{CTe2}-\eqref{CTe3} is a simplified coupled version of the system \eqref{eq:h0}-\eqref{eq:h1} with all nonlinear terms as sources. Notice however that the coupling between $\ho_{n+1}$ and $\hu_{n+1}$ in the system makes it \emph{nonlinear}. However, because $\G_{\e}$ is the generator of a $C_{0}$-semigroup in $\E$, equation \eqref{CTe1} is well-posed and 
$$h_{n+1}(t)=\mathcal{V}_{\e}(t)h(0)+\e^{-1}\int_{0}^{t}\mathcal{V}_{\e}(t-s)\Q_{\re(\e)}(h_{n}(s),h_{n}(s))\d s$$
{where $\left\{\mathcal{V}_{\e}(t)\,;\,t\geq0\right\}$ is the $C_{0}$-semigroup in $\E$ generated by $\G_{\e}$ (i.e., with the notations of Prop. \ref{prop:Bree}, $\mathcal{V}_{\e}(t)=\mathcal{V}_{\re(\e),\e}(t)$, $t\geq0$). With this at hands}, substitute in \eqref{CTe2} the term $h^{1}_{n+1}$ by $h_{n+1}-h^{0}_{n+1}$ and look at $h_{n+1}(t)$ as an additional source term.   In the same way for \eqref{CTe3}, the system \eqref{CTe2}-\eqref{CTe3} becomes linear (in terms of $\ho_{n+1}$ and $\hu_{n+1}$) and admits, for any $n\in \N$, a unique solution. One can use a slight modification of the ideas of Section \ref{sec:non} to check that the iteration scheme is stable, that is, the mapping 
$$\big(\ho_{n},\hu_{n}\big)\in \mathcal{X}_0\times \mathcal{X}_{1} \mapsto \big(\ho_{n+1},\hu_{n+1}\big)\, \in \mathcal{X}_0\times \mathcal{X}_{1}$$
is well defined.  Indeed, existence of the scheme is guaranteed by the linear theory as the iteration scheme is based on the linear equation.  Moreover, note that \eqref{CTe1} preserves the conservation laws: mass conservation and vanishing momentum, which were essential for the \textit{a priori} estimates related to $\mathbf{P}_0h^{1}$.   Thus, stability holds true under the conditions of the \textit{a priori} estimates, that is, for $\e\in(0,\e^{\dagger})$ (where $\e^\dagger$ is defined in Theorem~\ref{theo:h-relaxation}-Remark~\ref{nb:1/2}) and
\begin{equation*}
\sup_{t\geq0}\left(\|\hu_{n}(t)\|_{\H}+\| \ho_{n}(t)\|_{\E}\right) \leq C\sqrt{\mathcal{K}_0} \leq \Mo\,,\qquad n \in \N\,.
\end{equation*}
This latter condition is possible by taking $\mathcal{K}_0$ smaller than a threshold depending only {on the initial mass and energy $E_{\mathrm{in}}$ \footnote{{since all the threshold values appearing here are prescribed by the choice of the initial mass and energy, see Remark \ref{nb:init}.}} 
$$\mathcal{K}_0\leq (\Mo/C)^{2}=:\mathcal{K}^{\dagger}_{0}.$$}  We leave the details to the reader and focus in the next subsections on the convergence of the scheme.
\subsection{Estimating $\|\ho_{n+1}-\ho_{n}\|_{\E}$ and $\|\hu_{n+1}-\hu_{n}\|_{\H}$}  To prove the convergence of the scheme, we define for $n \in \N$
\begin{equation*}
d^{0}_{n+1} = \ho_{n+1} - \ho_{n}, \qquad \qquad d^{1}_{n+1}=\hu_{n+1}-\hu_{n}.
\end{equation*}
Then, one deduces from \eqref{CTe2} and \eqref{CTe3}
\begin{equation}\label{CTe5}
\hspace{-.4cm}\left\{\begin{array}{ccl}
\partial_{t} d^{0}_{n+1} & = & \B_{\re(\e),\e}d^{0}_{n+1} + \Big[\G_{\e}d^{1}_{n+1}-\G_{1,\e}d^{1}_{n+1}\Big] + \e^{-1}\mathcal{F}^{0}_{n}\,,\\ [10pt]
d^{0}_{n+1}(0) & = & 0\,,
\end{array}\right.
\end{equation}
and
\begin{equation}\label{CTe6}
\left\{
\begin{array}{ccl}
\partial_{t} d^{1}_{n+1}&=& \G_{1,\e}d^{1}_{n+1} + \A_\e d^{0}_{n+1} + \e^{-1}\mathcal{F}^{1}_{n} \,,\\[10pt]
d^{1}_{n+1}(0)&=&0\,.
\end{array}\right.
\end{equation}
The sources $\mathcal{F}^{i}_{n}$, for $i\in\{0,1\}$, correspond to the bilinear terms and depend only on the previous iterations $\{h^{i}_{n},h^{i}_{n-1}\}$, for $i\in\{0,1\}$ and $n \geq 2$ (see \eqref{F0} and \eqref{F1} for the precise expression). We introduce
\begin{equation*}\begin{cases}
\Psi_{n}^{1}(t)=&  {\|h^{0}_{n}(t)\|_{\E_{1}}+\|h^{0}_{n-1}(t)\|_{\E_{1}}}\\
\Psi_{n}^{\infty}(t)=& \|h^{0}_{n}(t)\|_{\E}+\|h^{0}_{n-1}(t)\|_{\E} + \|h^{1}_{n}(t)\|_{\H}+\|h^{1}_{n-1}(t)\|_{\H}\,,\end{cases}
\end{equation*}
which satisfy
\begin{equation}\label{eq:Psi1inf}
 {\sup_{ t\geq0 }\left(\Psi_{n}^{\infty}(t) + \e^{-2} \int^{t}_{0} \Psi_{n}^{1}(\tau) \d \tau\right)} \leq C\sqrt{ \mathcal{K}_0 }\,,\qquad n\geq	2\,,
\end{equation}
for $\ho_{n},\ho_{n-1} \in \mathcal{X}_{0}$, and $\hu_{n},\hu_{n-1} \in \mathcal{X}_{1}$.  Consequently, the following estimate for $d^{0}_{n+1}$ follows under suitable modifications of the arguments leading to Proposition \ref{prop:h0} (keep in mind that $\| \cdot \|_{\E_{2}} \lesssim \| \cdot \|_{\H}$). 
\begin{lem}\label{prop:d0} Let $\e\in(0,\e^{\dagger})$ and $\mathcal{K}_0\leq \mathcal{K}^{\dagger}_{0}$.  Then, we have that
\begin{equation}\label{eq:estimd0}
\begin{split}
\|d^{0}_{n+1}(t)\|_{\E} &\lesssim \lambda_{\e}\int^{t}_{0}e^{-\frac{{\mu_{0}}}{\e^2}(t-s)}\|d^{1}_{n+1}(s)\|_{\H}\,\d s\\
&\qquad + \e^{-1}\int^{t}_{0}e^{-\frac{{\mu_{0}}}{\e^2}(t-s)}\Psi^{1}_{n}(s)\Big( {\|d^{0}_{n}(s)\|_{\E}}+ \|d^{1}_{n}(s)\|_{\H}\Big)\,\d s\\
&\quad + \int^{t}_{0}e^{-\frac{{\mu_{0}}}{\e^2}(t-s)}\Psi^{\infty}_{n}(s)\Big(\e^{-1}\|d^{0}_{n}(s)\|_{\E_{1}}+ \e\lambda_{\e}\|d^{1}_{n}(s)\|_{\H}\Big)\,\d s\,.
\end{split}
\end{equation}
\end{lem}
\begin{proof} {Here again, as in the proof of Proposition \ref{prop:h0}, we denote by $\|\cdot\|_{\E_{1}}$ and $\|\cdot\|_{\E}$ the norms on $\E_{1}$ and $\E$ that are equivalent to the standard ones (with multiplicative constants independent of $\e$) and that make $\e^{-2}\nu_{0}+\B_{\re(\e),\e}$ \emph{dissipative} so that}
\begin{equation*}\begin{split}
\dfrac{\d}{\d t}\|d^{0}_{n+1}(t)\|_{\E} &\leq -\frac{\nu_{0}}{\e^{2}}\|d^{0}_{n+1}(t)\|_{\E_{1}} +\e^{-1}\|\mathcal{F}^{0}_{n}(t)\|_{\E}
+\left\|\G_{\e}d^{1}_{n+1}(t)-\G_{1,\e}d^{1}_{n+1}(t)\right\|_{\E}\\
&\leq -\frac{\nu_{0}}{\e^{2}}\|d^{0}_{n+1}(t)\|_{\E_{1}} +\e^{-1}\|\mathcal{F}^{0}_{n}(t)\|_{\E}+
C\lambda_{\e}\|d^{1}_{n+1}(t)\|_{\H}.\end{split}\end{equation*}
We need to estimate $\|\mathcal{F}_{n}^{0}(t)\|_{\E}$. One has, 
\begin{equation}\label{F0}\begin{split}
\mathcal{F}_{n}^{0}&=\Q_{\re(\e)}(d^{0}_{n},h^{0}_{n})+\Q_{\re(\e)}(h^{0}_{n-1},d^{0}_{n})+2\widetilde{\Q}_{\re(\e)}(d^{0}_{n},h^{1}_{n})+2\widetilde{\Q}_{\re(\e)}(h^{0}_{n-1},d^{1}_{n})\\
&\phantom{++++} +\left(\Q_{\re(\e)}(d^{1}_{n},h^{1}_{n})-\Q_{1}(d^{1}_{n},h^{1}_{n})\right)+\left(\Q_{\re(\e)}(h^{1}_{n-1},d^{1}_{n})-\Q_{1}(h^{1}_{n-1},d^{1}_{n})\right).
\end{split}\end{equation}
Therefore, since $1-\re(\e) \lesssim \e^{2}\,\lambda_{\e}$, using  {Remark \ref{nb:diffQXY}} and the usual estimates for $\Q_{\re}$ and~$\Q_{1}$:{
\begin{multline*}
\|\mathcal{F}_{n}^{0}\|_{\E} \lesssim \|d^{0}_{n}\|_{\E_{1}}\left(\|h^{0}_{n}\|_{\E}+\|h^{0}_{n-1}\|_{\E}\right) + \|d^{0}_{n}\|_{\E}\left(\|h^{0}_{n}\|_{\E_{1}}+\|h^{0}_{n-1}\|_{\E_{1}}\right)\\
+\|d^{0}_{n}\|_{\E}\|h^{1}_{n}\|_{\E_{1}}+\|d^{0}_{n}\|_{\E_{1}}\|h^{1}_{n}\|_{\E} + \|d^{1}_{n}\|_{\E_{1}}\,\|h^{0}_{n-1}\|_{\E}\\
+ \|d^{1}_{n}\|_{\E}\,\|h^{0}_{n-1}\|_{\E_{1}}+\e^{2}\lambda_{\e}\|d^{1}_{n}\|_{\E_{2}}\left(\|h^{1}_{n}\|_{\E_{2}}+\|h^{1}_{n-1}\|_{\E_{2}}\right).
\end{multline*}}
Using that $\|\cdot\|_{\E_{2}} \lesssim \|\cdot\|_{\H}$ we get
$$\|\mathcal{F}_{n}^{0}(t)\|_{\E} \lesssim \|d^{0}_{n}(t)\|_{\E_{1}}\Psi^{\infty}_{n}(t)+ \Psi^{1}_{n}(t)\left( \|d^{0}_{n}(t)\|_{\E}
+\|d^{1}_{n}(t)\|_{\H}\right) + \e^{2}\lambda_{\e}\|d^{1}_{n}(t)\|_{\E_{2}}\Psi^{\infty}_{n}(t).$$
This leads to the desired estimate {since $\mu_0<\nu_0$ (see the proof of Proposition~\ref{prop:h0})}.
\end{proof} 

Regarding the projection $ \mathbf{P}_{0} d^{1}_{n+1}(t)$, since the difference $h_{n+1}-h_{n}=d^{0}_{n+1}+d_{n+1}^{1}$ has zero mass and momentum, one can follow the line of proof of Lemma \ref{lem:energy} to deduce that
\begin{align*}
\| \mathbf{P}_{0} d^{1}_{n+1}(t)  \|_{\E_{-1}}  &\lesssim   \|d^{0}_{n+1}(t)\|_{\E} + \e\lambda_{\e}\int^{t}_{0}e^{-\lambda_{\e} (t-s)}\Psi^{\infty}_{n}(s)\Big(\|d^{0}_{n}(s)\|_{\E} +  {\|d^{1}_{n}(s)\|_{\H}}\Big) \d s\\
& {+ \lambda_\e \int_0^te^{-\lambda_\e(t-s)} \left(\| d^0_{n+1}(s)\|_\E + \|\mathbf{P}_{0}^\perp d^1_{n+1}(s)\|_\E\right) \, \d s \, .}
\end{align*}
Consequently, plugging \eqref{eq:estimd0} in the second term in the right side and recalling that $$\| \mathbf{P}_{0} d^{1}_{n+1}(t)  \|_{\H} \lesssim \| \mathbf{P}_{0} d^{1}_{n+1}(t)  \|_{\E_{-1} } \lesssim \| \mathbf{P}_{0} d^{1}_{n+1}(t)  \|_{\H}$$ we obtain the following lemma.
\begin{lem}\label{projection}
For any $t \ge 0$, we have that
\begin{align*}
\| \mathbf{P}_{0} d^{1}_{n+1}(t)&  \|_{\H} \lesssim   \lambda_{\e}\int^{t}_{0}e^{-\frac{\mu_0}{\e^2}(t-s)}\|d^{1}_{n+1}(s)\|_{\H}\,\d s\\
& + \e^{-1}\int^{t}_{0}e^{-\frac{\mu_0}{\e^2}(t-s)}\Psi^{1}_{n}(s)\Big(\|d^{0}_{n}(s)\|_{\E}+ \|d^{1}_{n}(s)\|_{\H}\Big)\,\d s\\
&\quad + \int^{t}_{0}e^{-\frac{\mu_0}{\e^2}(t-s)}\Psi^{\infty}_{n}(s)\Big(\e^{-1}\|d^{0}_{n}(s)\|_{\E_{1}}+ \e\lambda_{\e}\|d^{1}_{n}(s)\|_{\H}\Big)\,\d s\\
&\qquad + \e\lambda_{\e}\int^{t}_{0}e^{-\lambda_{\e} (t-s)}\Psi^{\infty}_{n}(s)\Big(\|d^{0}_{n}(s)\|_{\E} + \|d^{1}_{n}(s)\|_{\H}\Big) \d s \\
&{+ \lambda_\e \int_0^te^{-\lambda_\e(t-s)} \left(\|d^0_{n+1}(s)\|_\E + \|\mathbf{P}_{0}^\perp d^1_{n+1}(s)\|_\H\right) \, \d s \, .}
\end{align*}
\end{lem}
Let us focus on estimating $\mathbf{P}^{\perp}_{0} d^{1}_{n+1}(t)$. To do so, we introduce the functions $\Phi^{1}_{n}$ and $\Phi^{\infty}_{n}$ defined by
\begin{equation*}
\Phi^{1}_{n}(t) = \| h^{1}_{n}(t) \|^{2}_{\H_{1}} +  {\| h^{1}_{n-1}(t) \|^{2}_{\H_{1}}}\quad\text{and}\quad\Phi^{\infty}_{n} (t)= \| h^{1}_{n}(t)\|^{2}_{\H} +  {\| h^{1}_{n-1}(t) \|^{2}_{\H}}
\end{equation*}
which satisfy
\begin{equation}\label{eq:Phi1inf}
\sup_{ t\geq0 }\Big(\Phi^{\infty}_{n}(t) +  {\int^{t}_{0} \Phi^{1}_{n}(\tau) \d \tau}\Big) \leq C \mathcal{K}_0\,,\quad n\geq2.
\end{equation}
One has the following lemma.
\begin{lem}\label{perp:projection}
Let $\e\in(0,\e^{\dagger})$ and $\mathcal{K}_0\leq \mathcal{K}^{\dagger}_{0}$.  Then,
\begin{multline*}
\|\mathbf{P}^{\perp}_{0} d^{1}_{n+1}(t) \|_{\H}^{2} \lesssim \int^{t}_{0}e^{-\nu(t-s)}\Phi^{1}_{n}(s) \| d^{1}_{n}(s) \|^{2}_{\H}\d s\\
+\int^{t}_{0}e^{-\nu(t-s)}\Phi^{\infty}_{n}(s) \| d^{1}_{n}(s) \|^{2}_{\H_{1}}\d s
+{{\lambda_{\e}}} \int^{t}_{0}e^{-\nu(t-s)} \| d^{1}_{n+1}(s) \|^{2}_{\H} \d s \\
+{\e^{-1}}\left(\sup_{s\geq0}\|d^{1}_{n+1}(s)\|_{\H}\right) 
\int^{t}_{0}e^{-\nu(t-\tau)} \, \Psi^{1}_{n}(\tau)\,\Big( \|d^{0}_{n}(\tau)\|_{\E} +  \|d^{1}_{n}(\tau)\|_{\H}\Big)\,\d \tau\\
+\left(\sup_{s\geq0}\|d^{1}_{n+1}(s)\|_{\H}\right)\int^{t}_{0}e^{-\nu(t-\tau)} \, \Psi^{\infty}_{n}(\tau)\,\Big(\e^{-1}\|d^{0}_{n}(\tau)\|_{\E} + \e\, \lambda_{\e}\|d^{1}_{n}(\tau)\|_{\H}\Big)\,\d \tau.
\end{multline*}
\end{lem}
\begin{proof} {One deduces from \eqref{CTe6} that $\mathbf{P}_{0}^{\perp}d_{n+1}^{1}(t)$ is such that
\begin{equation*}
\partial_{t}\mathbf{P}_{0}^{\perp}d_{n+1}^{1}(t)=\G_{1,\e}\mathbf{P}_{0}^{\perp}\,d^{1}_{n+1}(t)+\mathbf{P}_{0}^{\perp}\A_{\e}d^{0}_{n+1}(t)+\e^{-1}\F_{n}^{1}
\end{equation*}
where
\begin{equation}\label{F1}
\mathcal{F}_{n}^{1}=\Q_{1}(\hu_{n},\hu_{n})-\Q_{1}(\hu_{n-1},\hu_{n-1})=\Q_{1}(d^{1}_{n},\hu_{n})+\Q_{1}(\hu_{n-1},d^{1}_{n})\,.
\end{equation}
Following the argument leading to inequality \eqref{eq:bmh1} (see also \cite[Lemma 4.6, Theorem 4.7]{bmam}) one deduces   that}
\begin{multline}\label{eq:d1}
\|\mathbf{P}^{\perp}_{0} d^{1}_{n+1}(t) \|_{\H}^{2} \lesssim \int^{t}_{0}e^{-\nu(t-s)}\Phi^{1}_{n}(s) \| d^{1}_{n}(s) \|^{2}_{\H}\d s+\int^{t}_{0}e^{-\nu(t-s)}\Phi^{\infty}_{n}(s) \| d^{1}_{n}(s) \|^{2}_{\H_{1}}\d s\\
\qquad+\int^{t}_{0}e^{-\nu(t-s)} \| d^{1}_{n+1}(s) \|_{\H}\,\|\mathcal{A}_{\e}d^{0}_{n+1}(s)\|_{\H} \d s\,.
\end{multline}
The latter term in the right side of \eqref{eq:d1} can be estimated using \eqref{eq:estimd0} and recalling that $\|\mathcal{A}_{\e}d^{0}_{n+1}\|_{\H}\lesssim\e^{-2}\,\|d^{0}_{n+1}\|_{\E}$.  Thus, 
\begin{equation*}
\int^{t}_{0}e^{-\nu(t-s)} \| d^{1}_{n+1}(s) \|_{\H}\,\|\mathcal{A}_{\e}d^{0}_{n+1}(s)\|_{\H} \d s \lesssim \sum^{3}_{i=1}\mathcal{T}_{i}\,,
\end{equation*}
with
\begin{equation*}\begin{split}
\mathcal{T}_{1} &= \frac{\lambda_{\e}}{\e^{2}}\int^{t}_{0}e^{-\nu(t-s)}  \| d^{1}_{n+1}(s) \|_{\H}\bigg(\int^{s}_{0}e^{-\frac{\mu_0}{\e^2}(s-\tau)}\|d^{1}_{n+1}(\tau)\|_{\H}\,\d \tau\bigg) \d s,\\
\mathcal{T}_{2} &= \e^{-3}\int^{t}_{0}e^{-\nu(t-s)} \| d^{1}_{n+1}(s) \|_{\H}\bigg[\int^{s}_{0}e^{-\frac{\mu_0}{\e^2}(s-\tau)}\Psi^{1}_{n}(\tau)\Big( \|d^{0}_{n}(\tau)\|_{\E} +  \|d^{1}_{n}(\tau)\|_{\H}\Big) \d\tau\bigg]\d s\\
\mathcal{T}_{3} &= {\e^{-2}}\int^{t}_{0}e^{-\nu(t-s)} \| d^{1}_{n+1}(s) \|_{\H}\bigg[\int^{s}_{0}e^{-\frac{\mu_0}{\e^2}(s-\tau)}\Psi^{\infty}_{n}(\tau)\Big(\e^{-1}\|d^{0}_{n}(\tau)\|_{\E_{1}} +  \e\lambda_{\e}\|d^{1}_{n}(\tau)\|_{\H}\Big) \d\tau\bigg] \d s.\end{split}\end{equation*}
{It is easy to check, using \eqref{eq:integdouble}, that}
$$\mathcal{T}_{2} \leq \frac{2}{\mu_{0}\e}\left(\sup_{s\geq0}\|d_{n+1}^{1}(s)\|_{\H}\right)\,\int^{t}_{0}e^{-\nu(t-\tau)} \, \Psi^{1}_{n}(\tau)\,\Big( \|d^{0}_{n}(\tau)\|_{\E} +  \|d^{1}_{n}(\tau)\|_{\H}\Big)\,\d \tau$$
and
$$\mathcal{T}_{3} \leq \frac{2}{\mu_{0}}\left(\sup_{s\geq0}\|d_{n+1}^{1}(s)\|_{\H}\right)\,\int^{t}_{0}e^{-\nu(t-\tau)} \, \Psi^{\infty}_{n}(\tau)\,\Big( \e^{-1}\|d^{0}_{n}(\tau)\|_{\E_{1}} + \e\, \lambda_{\e}\|d^{1}_{n}(\tau)\|_{\H}\Big)\,\d \tau.$$ 
{The estimate for $\mathcal{T}_{1}$ is a bit more involved. Thanks to Cauchy-Schwarz inequality one first has
$$
\mathcal{T}_{1} \leq \frac{\lambda_{\e}}{\e^{2}}\left(\int^{t}_{0}e^{-\nu(t-s)}  \| d^{1}_{n+1}(s) \|_{\H}^{2}\d s\right)^{\frac{1}{2}}\\
\left(\int_{0}^{t}e^{-\nu(t-s)}Y^{2}(s)\d s \right)^{\frac{1}{2}}$$
where 
$$Y(s):=\int_{0}^{s}e^{-\frac{\mu_{0}}{\e^{2}}(s-\tau)}\|d^{1}_{n+1}(\tau)\|_{\H}\d \tau, \qquad s \in (0,t).$$
Thanks to \eqref{eq:integsquare2}  {applied with $r=\frac12$}, 
$$Y^{2}(s) \leq \frac{\e^{2}}{\mu_{0}}\int_{0}^{s}e^{-\frac{\mu_{0}}{\e^{2}}(s-\tau)}\|d^{1}_{n+1}(\tau)\|_{\H}^{2}\d \tau$$
and, using now \eqref{eq:integdouble} for $\mu_{0} > 2\e^{2}\nu$,
\begin{multline*}
\int_{0}^{t}e^{-\nu(t-s)}Y^{2}(s)\d s  \leq \frac{\e^{2}}{\mu_{0}}\int_{0}^{t}e^{-\nu(t-s)}\d s\int_{0}^{s}e^{-\frac{\mu_{0}}{\e^{2}}(s-\tau)}\|d^{1}_{n+1}(\tau)\|_{\H}^{2}\d \tau\\
\leq \frac{\e^{2}}{\mu_{0}\left(\frac{\mu_{0}}{\e^{2}}-\nu\right)}\int_{0}^{t}e^{-\nu(t-s)}\|d^{1}_{n+1}(s)\|_{\H}^{2}\d s
\leq \frac{2\e^{4}}{\mu_{0}^{2}}\int_{0}^{t}e^{-\nu(t-s)}\|d^{1}_{n+1}(s)\|_{\H}^{2}\d s.\end{multline*}
We deduce finally that
$$\mathcal{T}_{1} \leq \frac{\sqrt{2}\lambda_{\e}}{\mu_{0}}\int_{0}^{t}e^{-\nu(t-s)}\|d^{1}_{n+1}(s)\|_{\H}^{2}\d s$$
and this, together with the estimates for $\mathcal{T}_{2}$ and $\mathcal{T}_{3}$, gives the result.}
\end{proof}
Introducing now the quantities
\begin{equation*}\begin{split}
\sE^{0}_{n}&=\sup_{ t\geq0}\Big(\| d^{0}_{n}(t) \|_{\E} +  {\e^{-2}\int^{t}_{0}\| d^{0}_{n}(\tau) \|_{\E_{1}} \d \tau}\Big)\,,\\
\sE^{1}_{n}&=\sup_{t\geq0}\Big(\| d^{1}_{n}(t) \|^{2}_{\H} + {\int^{t}_{0}\| d^{1}_{n}(\tau) \|^{2}_{\H_{1}} \d \tau} \Big)^{\frac{1}{2}}\,,\qquad n \geq 2\,,
\end{split}\end{equation*}
we can gather the three  previous lemmas and use \eqref{eq:Psi1inf} to obtain the following result.

\begin{prop}\label{prop:sE1}
{For any $n \in \N$ and $t \geq0$
\begin{equation}\label{do-final}
\|d^{0}_{n+1}(t)\|_{\E} \lesssim {\e^{2}\lambda_{\e}}\, \sE^{1}_{n+1} + \e\,\sqrt{ \mathcal{K}_0 }\,\big( \sE^{0}_{n} + \sE^{1}_{n}\big)\,,
\end{equation}
while
\begin{equation}\label{d1-final}
\| d^{1}_{n+1}(t)\|_{\H} {\, \lesssim \sqrt{{\lambda_{\e}+\e}} \; \sE^{1}_{n+1} + \sqrt{\e \mathcal{K}_0 }\,\sE^{0}_{n} + \sqrt{ \mathcal{K}_0 }\,\sE^{1}_{n}\,,}
\end{equation}
as long as $\e\in(0,\e^{\dagger})$, $\mathcal{K}_0\leq\mathcal{K}^{\dagger}_{0}$.}
\end{prop}
\begin{proof}  First, we claim that 
\begin{equation}\label{d1-perp-b}
\| \mathbf{P}^{\perp}_{0}d^{1}_{n+1}(t)\|_{\H} {\, \lesssim \sqrt{{\lambda_{\e}+\e}} \; \sE^{1}_{n+1} + \sqrt{\e \mathcal{K}_0 }\,\sE^{0}_{n} + \sqrt{ \mathcal{K}_0 }\,\sE^{1}_{n}\,,}
\end{equation}
Indeed, from Lemma \ref{perp:projection}, we have that
\begin{multline*}
\| \mathbf{P}^{\perp}_{0}d^{1}_{n+1}(t)\|_{\H}^{2} \lesssim {{\lambda_{\e}}}\left[\sE^{1}_{n+1}\right]^{2}
+\left[\sE^{1}_{n}\right]^{2}\left(\int_{0}^{t}e^{-\nu(t-s)}\Phi_{n}^{1}(s)\d s+\sup_{s\geq0}\Phi_{n}^{\infty}(s)\right)\\
+\e^{-1}\sE^{1}_{n+1}\left(\sE^{0}_{n}+\sE^{1}_{n}\right)
\int^{t}_{0}e^{-\nu(t-\tau)} \, \Psi^{1}_{n}(\tau)\d\tau\\
+\sE^{1}_{n+1}\left(\sup_{s\geq0}\Psi_{n}^{\infty}(s)\right)
{\left( \e \, \sE^0_n+ \e \lambda_\e \sE^1_{n}\right)}\,.
\end{multline*}
We can thus invoke \eqref{eq:Psi1inf} and~\eqref{eq:Phi1inf} to deduce that
$$
\| \mathbf{P}^{\perp}_{0}d^{1}_{n+1}(t)\|_{\H}^{2} \lesssim 
\mathcal{K}_{0}\left[\sE^{1}_{n}\right]^{2}+{{\lambda_{\e}}}\left[\sE^{1}_{n+1}\right]^{2}+{\e}\sqrt{\mathcal{K}_{0}}\left(\sE^{0}_{n}+\sE^{1}_{n}\right)\sE^{1}_{n+1}$$where we used that $\e\lambda_{\e} < \e.$ From Young's inequality, we 
deduce that
$$
\| \mathbf{P}^{\perp}_{0}d^{1}_{n+1}(t)\|_{\H}^{2} \lesssim {(\lae+\e)\left[\sE^{1}_{n+1}\right]^{2}} 
{+\mathcal{K}_{0} \left[\sE^{1}_{n}\right]^{2}
+\e\mathcal{K}_{0}\left[\sE^{0}_{n}\right]^{2}},
$$
which proves \eqref{d1-perp-b}. 
In the same way, the estimate~\eqref{do-final} is easily deduced from Lemma \ref{prop:d0}. To end the proof, it remains to prove that 
\begin{equation}\label{d1-pro-a}
{\| \mathbf{P}_{0}d^{1}_{n+1}(t)\|_{\H} \lesssim \sqrt{\lambda_{\e}+\e}\, \sE^{1}_{n+1} + \sqrt{ \e \, \mathcal{K}_0 }\,\sE^{0}_{n} + \sqrt{\mathcal{K}_0} \sE^{1}_{n} \,.}
\end{equation}
This inequality is a consequence of Lemma \ref{projection} combined with \eqref{eq:Psi1inf},~\eqref{do-final} and~\eqref{d1-perp-b}.\end{proof}
\begin{prop}\label{prop:average} For any $t \geq 0,$ we have that
\begin{equation}\label{eq:averd0}
\Big(\frac{\mu_0}{\e^{2}}-\nu\Big)\int^{t}_{0}e^{-\nu(t-s)}\|d^{0}_{n+1}(s)\|_{\E_{1}}\d s  \lesssim  \lambda_{\e}\,\sE^{1}_{n+1} + \e\,\sqrt{\mathcal{K}_{0}}\,\big( \sE^{0}_{n} + \sE^{1}_{n}\big)\,,\end{equation}
and
\begin{equation}\label{int-d1-final}
\begin{aligned}
\bigg(\int^{t}_{0}e^{-\nu(t-\tau)}\| d^{1}_{n+1}(\tau)\|^{2}_{\H_{1}}\d\tau\bigg)^{\frac 12} &\lesssim \sqrt{\lambda_{\e}+\e}\, \sE^{1}_{n+1} + \sqrt{ \e \, \mathcal{K}_0 }\,\sE^{0}_{n} + \sqrt{\mathcal{K}_0} \sE^{1}_{n}\,.
\end{aligned}
\end{equation}
\end{prop}
\begin{proof} {To prove \eqref{eq:averd0}, we follow the argument that led to Lemma \ref{prop:d0} and thus in the subsequent proof, we again denote by $\|\cdot\|_{\E_{1}}$ and $\|\cdot\|_{\E}$ the norms on $\E_{1}$ and $\E$ that are equivalent to the standard ones independently of $\e$ and that make $\e^{-2}\nu_{0}+\B_{\re(\e),\e}$ \emph{dissipative} so that}
we can write
\begin{equation*}\frac{\d}{\d t}\|d^{0}_{n+1}(t)\|_{\E}\leq -\frac{\nu_{0}}{\e^{2}}\|d^{0}_{n+1}(t)\|_{\E_{1}} +\e^{-1}\|\mathcal{F}^{0}_{n}(t)\|_{\E}+
C\lambda_{\e}\|d^{1}_{n+1}(t)\|_{\H}\,,
\end{equation*}
which implies that,
\begin{multline*}
\frac{\d}{\d t}\|d^{0}_{n+1}(t)\|_{\E}+\nu\,\|d^{0}_{n+1}(t)\|_{\E_{1}} \leq -\Big(\frac{\mu_{0}}{\e^{2}} - \nu\Big)\|d^{0}_{n+1}(t)\|_{\E_{1}}\\
+\e^{-1}\|\mathcal{F}^{0}_{n}(t)\|_{\E}+
C\lambda_{\e}\|d^{1}_{n+1}(t)\|_{\H}\end{multline*}
where we used that $\mu_0<\nu_0$. 
After integration over $[0,t]$, using that $d^{0}_{n+1}(0)=0$, we get that
\begin{multline*}
\|d^{0}_{n+1}(t)\|_{\E} \leq -\Big(\frac{\mu_{0}}{\e^{2}} - \nu\Big)\int_{0}^{t}e^{-\nu(t-s)}\|d^{0}_{n+1}(s)\|_{\E_{1}}\d s
+\e^{-1}\int_{0}^{t}e^{-\nu(t-s)}\|\mathcal{F}_{n}^{0}(s)\|_{\E}\d s \\
+C\lambda_{\e}\int_{0}^{t}e^{-\nu(t-s)}\|d^{1}_{n+1}(s)\|_{\H}\d s\,,
\end{multline*}
and, recalling that $\mathcal{F}_{n}^{0}$ is given by \eqref{F0}, we 
estimate $\|\mathcal{F}_{n}^{0}(s)\|_{\E}$ as in Lemma \ref{prop:d0} to obtain that
\begin{multline*}
\Big(\frac{\mu_0}{\e^{2}}-\nu\Big)\int^{t}_{0}e^{-\nu(t-s)}\|d^{0}_{n+1}(s)\|_{\E_{1}}\d s \lesssim \lambda_{\e}\int^{t}_{0}e^{-\nu(t-s)}\|d^{1}_{n+1}(s)\|_{\H}\,\d s\\
+ \e^{-1}\int^{t}_{0}e^{-\nu(t-s)}\Psi^{1}_{n}(s)\Big(\|d^{0}_{n}(s)\|_{\E}+ \|d^{1}_{n}(s)\|_{\H}\Big)\,\d s\\
+ \int^{t}_{0}e^{-\nu(t-s)}\Psi^{\infty}_{n}(s)\Big(\e^{-1}\|d^{0}_{n}(s)\|_{\E_{1}}+ \e\lambda_{\e}\|d^{1}_{n}(s)\|_{\H}\Big)\,\d s.
\end{multline*}
This yields \eqref{eq:averd0}. In the same way, we adapt the proof of Lemma \ref{perp:projection} to get that
\begin{multline*}
\Big(\frac{2\mu}{\sigma^{2}_{0}}  - \nu\Big)\int^{t}_{0}e^{-\nu(t-\tau)}\|\mathbf{P}^{\perp}_{0} d^{1}_{n+1}(\tau) \|^{2}_{ \H_{1}}\d \tau
\lesssim \int^{t}_{0}e^{-\nu(t-s)}\Phi^{1}_{n}(s) \| d^{1}_{n}(s) \|^{2}_{\H}\d s\\
+\int^{t}_{0}e^{-\nu(t-s)}\Phi^{\infty}_{n}(s) \| d^{1}_{n}(s) \|^{2}_{\H_{1}}\d s + \int^{t}_{0}e^{-\nu(t-s)} \| d^{1}_{n+1}(s) \|_{\H}\,\|\mathcal{A}_{\e}d^{0}_{n+1}(s)\|_{\H} \d s
\end{multline*}
{where $2\mu/\sigma_0^2 - \nu = c_0 \Delta_0^2$ (see~\eqref{eq:nu})}. 
This estimate is similar to \eqref{eq:d1} and therefore we can resume both the proofs of Lemma~\ref{perp:projection} and Proposition \ref{prop:sE1} to obtain that
$$
\begin{aligned}
\Delta_0^2 \bigg(\int^{t}_{0}e^{-\nu(t-\tau)}\| \mathbf{P}_0^\perp d^{1}_{n+1}(\tau)\|^{2}_{\H_{1}}\d\tau\bigg)^{\frac 12} &\lesssim \sqrt{\lambda_{\e}+\e}\, \sE^{1}_{n+1} + \sqrt{ \e \, \mathcal{K}_0 }\,\sE^{0}_{n} + \sqrt{\mathcal{K}_0} \sE^{1}_{n}\,.
\end{aligned}
$$
 To conclude, we recall that $\| \mathbf{P}_{0}d^{1}_{n+1}\|_{\H_{1}} \lesssim \| \mathbf{P}_{0}d^{1}_{n+1}\|_{\H}$ so that a simple integration of \eqref{d1-pro-a} gives that
$$
\begin{aligned}
\bigg(\int^{t}_{0}e^{-\nu(t-\tau)}\| \mathbf{P}_0d^{1}_{n+1}(\tau)\|^{2}_{\H_{1}}\d\tau\bigg)^{\frac 12} &\lesssim \sqrt{\lambda_{\e}+\e}\, \sE^{1}_{n+1} + \sqrt{ \e \, \mathcal{K}_0 }\,\sE^{0}_{n} + \sqrt{\mathcal{K}_0} \sE^{1}_{n}\,.
\end{aligned}
$$
Adding these two estimates, one deduces \eqref{int-d1-final}.
\end{proof}
\subsection{Convergence of the iteration scheme} {
We are now in position to conclude our analysis by proving the convergence of the iteration scheme.  {In the sequel, we indicate with a same letter $C$ a positive constant depending on $\mu_{0}$, $\nu$ and $\Mo$ that may change from line to line.} Suitably adding \eqref{do-final} and \eqref{eq:averd0} and taking the supremum in time, one has that
\begin{equation}\label{E0-final-final}
\sE^{0}_{n+1} \lesssim \lambda_{\e}\,\sE^{1}_{n+1} + \e\,\sqrt{\mathcal{K}_{0}}\,\big( \sE^{0}_{n} + \sE^{1}_{n}\big)
\end{equation}
{where we used that $\mu_0 \geq 2 \e^2 \nu$}. 
Similarly, adding \eqref{d1-final} and \eqref{int-d1-final} and taking the supremum in time it holds that
{\begin{equation}\label{E1-final-final}
\sE^{1}_{n+1}\lesssim \sqrt{\lambda_{\e}+\e} \; \sE^{1}_{n+1} +  \sqrt{\e \mathcal{K}_0 }\,\sE^{0}_{n} + \sqrt{ \mathcal{K}_0 }\,\sE^{1}_{n}\,.\end{equation}}
Let us define $\mathscr{E}_{n}=\sE^{0}_{n}+\sE^{1}_{n}$, for $n \geq 2$.  {Adding the estimates \eqref{E0-final-final} and \eqref{E1-final-final}, we conclude that there exists $C>0$ such that  {$\mathscr{E}_{n+1} \leq C  \sqrt{\lambda_\e+\e} \, \mathscr{E}_{n+1} + C \sqrt{\mathcal{K}_0} \, \mathscr{E}_{n}$.} 
Thus, choosing $\e$ sufficiently small such that {$C\sqrt{\lambda_{\e}+\e} < \frac{1}{2}$}, we get that 
$\mathscr{E}_{n+1} \leq C  \sqrt{\mathcal{K}_{0}}\,\mathscr{E}_{n}$ from which
\begin{equation*}
\mathscr{E}_{n+1} \leq \big(C\sqrt{\mathcal{K}_0}\big)^{n-1}\,\mathscr{E}_{2}\,,\quad \qquad \forall \, n\geq 2\,.
\end{equation*}
{Choosing  $\mathcal{K}_0\leq\mathcal{K}^{\dagger}_{0} < {C^{-2}}$ so that
$$\theta:=C  \sqrt{\mathcal{K}_0} < 1$$
we deduce that,} in the Banach space $(\mathbb{B},\vertiii{\cdot})$, one has for $m > n \geq 1$,
\begin{equation*}
\vertiii{(h^{0}_{m},h^{1}_{m}) - (h^{0}_{n},h^{1}_{n})} \leq \sum^{m-1}_{i=n}\mathscr{E}_{i+1}\leq \mathscr{E}_{2}\,\frac{\theta^{n-1}}{1-\theta}\,,\qquad \theta:=C  \sqrt{\mathcal{K}_0}\,.
\end{equation*}
Whence the sequence $\big\{(h^{0}_{n},h^{1}_{n})\big\}_{n}\subset\mathcal{X}_0\times\mathcal{X}_{1} \subset \mathbb{B}$ is a Cauchy sequence and  it converges in $(\mathbb{B},\vertiii{\cdot})$ to a limit $(h^{0},h^{1})\in \mathcal{X}_0\times\mathcal{X}_{1}$.  Of course, such limit satisfies equations \eqref{eq:h0} and \eqref{eq:h1}.  Thus, $h=h^{0}+h^{1}$ is a solution to the inelastic Boltzmann problem \eqref{BE}.  Such solution is unique {in the class of functions that we consider} since, at essence, we proved that the problem is a contraction on $\mathcal{X}_0\times\mathcal{X}_{1}$.  This completes the proof of Theorem \ref{theo:main-cauc1} recalling that $f_\e=G_{\re(\e)}+\e h_\e$.

\section{Hydrodynamic limit}\label{sec:hydr}
 
 {In this last section, we will once again specify that $h$, $h^0$ and $h^1$ depend on $\e$ by noting $h=h_\e$, $h^0=\ho_{\e}$, $h^1=\hu_{\e}$. On the other hand, to lighten notations, we will write $\alpha$ for~$\alpha(\e)$ but recall that $\alpha=\alpha(\e)$ satisfies Assumption~\ref{hyp:re}. Finally, we will consider $m$, $k$ and~$q$ satisfying
 $$
 m>d, \quad m-1 \geq k \geq 1, \quad q \geq 5. 
 $$
as well as the corresponding spaces $\mathcal{E}$ and $\mathcal{E}_1$ defined in~\eqref{eq:Ekmq}-\eqref{E1E2}.} We prove here the following precised version of 
Theorem \ref{theo:CVNS-int} in the Introduction:
\begin{theo}\label{theo:CVNS}
Under the Assumptions of Theorem \ref{theo:main-cauc}, set
$$f_{\e}(t,x,v)=G_{\re(\e)} + \e\,h_{\e}(t,x,v)\,,$$
with $h_{\e}(0,x,v)=h_{\mathrm{in}}^{\e}(x,v)=\e^{-1}\left(F^{\e}_{\mathrm{in}}-G_{\re(\e)}\right)$ such that
$$\lim_{\e\to0}\left\|\bm{\pi}_{0}h_{\mathrm{in}}^{\e}-h_{0}\right\|_{{L^1_v\W^{m,2}_{x}}}=0\,,$$
where $\bm{\pi}_{0}$ stands for the projection over the elastic linearized Boltzmann operator (see \eqref{pi0} for a precise definition)
$$h_{0}(x,v)=\left(\varrho_{0}(x)+u_{0}(x)\cdot v + \tfrac{1}{2}\vE_{0}(x)(|v|^{2}-d\en_{1})\right)\M(v)\,,$$
with $\M$ being the Maxwellian distribution introduced in \eqref{eq:max} and
$$(\varrho_{0},u_{0},\vE_{0}) \in \mathscr{W}_{m}\,,$$
where we set $\mathscr{W}_{\ell}:=\ \left(\W^{\ell,2}_{x}(\T^{d})\right)^{d+2}$  for any $\ell \in \N$. 

\smallskip
\noindent
Then, for any $T >0$, $\left\{ h_{\e} \right\}_{\e}$ converges in some weak sense to a limit $\bm{h}=\bm{h}(t,x,v)$ which is such that 
\begin{equation*}
\bm{h}(t,x,v)=\left(\varrho(t,x)+u(t,x)\cdot v + \frac{1}{2}\vE(t,x)(|v|^{2}-d\en_{1})\right)\M(v)\,,
\end{equation*}
where 
$$(\varrho,u,\vE) \in \mathcal{C}([0,T];\,\mathscr{W}_{m-2}) \cap L^{1}\left((0,T);\,\mathscr{W}_{m}\right),
$$
is solution to the following  \emph{incompressible Navier-Stokes-Fourier system with forcing}
\begin{equation*} 
\begin{cases}
\partial_{t}u-{\frac{{\bm \nu}}{\en_1}}\,\Delta_{x}u + {\en_{1}}\,u\cdot \nabla_{x}\,u+\nabla_{x}p=\lambda_{0}u\,,\\[6pt]
\partial_{t}\,\vE-\frac{\gamma}{\en_{1}^{2}}\,\Delta_{x}\vE{+}\en_{1}\,u\cdot \nabla_{x}\vE=\dfrac{\lambda_{0}\,\bar{c}}{2(d+2)}\sqrt{\en_{1}}\,\vE\,,\\[8pt]
\mathrm{div}_{x}u=0, \qquad \varrho + \en_{1}\,\vE = 0\,,
\end{cases}
\end{equation*}
subject to initial conditions $(\varrho_{\mathrm{in}},u_{\mathrm{in}},\vE_{\mathrm{in}})$ given by
\begin{equation*}
u_{\mathrm{in}}=u(0)=\mathcal{P}u_{0}, \quad \vE_{\mathrm{in}}=\vE(0)=\frac{d}{d+2}\theta_{0}-\frac{2}{(d+2)\en_{1}}\varrho_{0}, \quad
\varrho_{\mathrm{in}}=\varrho(0)=-\en_{1}\vE_{\mathrm{in}}\,,
\end{equation*}
where $\mathcal{P}u_{0}$ is the Leray projection of $u_0$ on divergence-free vector fields. The viscosity $\bm{\nu} >0$ and heat conductivity $\gamma >0$ are explicit and $\lambda_{0} >0$ is the parameter appearing in \eqref{eq:scaling}. The parameter $\bar{c} >0$ is depending on the collision kernel $b(\cdot)$.
\end{theo}

\subsection{Compactness and convergence}

We start this section recalling the expression for the spectral projection $\bm{\pi}_{0}$ onto the kernel $\mathrm{Ker}(\mathbf{L}_{1})$ of the linearized collision operator $\mathbf{L}_{1}$ \emph{seen as an operator acting in velocity only} on the space $L^{2}_{v}(\M^{-\frac{1}{2}})$. We recall that, with the notations of Theorem \ref{theo:G1e},
\begin{equation}\label{pi0}
\bm{\pi}_{0}(g) : =\sum_{i=1}^{d+2}\left(\int_{ \R^{d} }g\,\Psi_{i}\,\d v \right)\,\Psi_{i}\,\M\,,
\end{equation}
where $\Psi_{1}(v)=1$, $\Psi_{i}(v)=\frac{1}{\sqrt{\en_{1}}}v_{i-1}$ $(i=2,\ldots,d+1)$ and $\Psi_{d+2}(v)=\frac{|v|^{2}-d\en_{1}}{\en_{1}\sqrt{2d}}$.  Note that the difference with respect to the spectral projection $\mathbf{P}_{0}$ for the operator $\G_{1,\e}$ in \eqref{eq:P0} is that no spatial integration is performed. 

 {Consider now $h_\e=\ho_{\e}+\hu_{\e}$ the solution constructed in Section~\ref{sec:Cauchy}.}
One can prove the following estimate for time-averages of $(\mathbf{Id}-\bm{\pi}_{0})h_\e(\tau)$ in spaces which do not involve derivatives in the $v$-variable:
\begin{prop}\label{lem:equi} For any $0\leq \beta \leq m - 1$ there exists $C >0$ independent of $\e$ such that
\begin{equation}\label{strong-hyde2}
\int^{t_{2}}_{t_{1}}\| (\mathbf{Id}-\bm{\pi}_{0})h_\e(\tau) \|_{{L^{1}_{v}\W^{\beta,2}_{x}(\m_{q})}}\d \tau \leq C\,\e\,\sqrt{\mathcal{K}_0}\,	{\max\left\{\sqrt{t_{2} - t_{1}}, t_2-t_1\right\}}\end{equation}
holds true for any $0\leq t_{1} \leq  t_{2}$.
\end{prop} 
\begin{proof} For a given $0\leq \beta \leq m-1$, we introduce the hierarchy of Hilbert spaces $$\tilde{H}_{s}=L^{2}_{v}\W^{\beta,2}_{x}(\M^{-\frac{1}{2}}\langle v\rangle^{\frac{s}{2}}), \qquad s \in \R\,,$$
setting simply $\tilde{H}:=\tilde{H}_{0}.$ Recall that $-\mathbf{L}_{1}$ is (better than) coercive on $(\mathbf{Id}-\bm{\pi_{0}})\tilde{H}$ (see {\cite{briant}} for instance) and denote by $\tilde{\mu}_{1}$ the coercivity constant, namely
$$-\langle \mathbf{L}_{1}(\mathbf{Id}-\bm{\pi}_{0})g\,,(\mathbf{Id}-\bm{\pi}_{0})g\rangle_{\tilde{H}} \geq \tilde{\mu}_{1}\|(\mathbf{Id}-\bm{\pi}_{0})g\|_{\tilde{H}_{1}}^{2}, \qquad g \in \tilde{H}_{1}.$$ 
In the space $\tilde{H}$, we can compute the inner product between $\partial_{t}\hu_\e$ and $(\mathbf{Id}-\bm{\pi}_{0})\hu_{\e}$ where we recall that $\hu_\e$ solves \begin{equation*}\label{eq:huG1e}
\partial_{t} \hu_\e =  {\G_{1,\e}\hu_\e} + \e^{-1}\Q_{1}(\hu_\e,\hu_\e) + \A_\e \ho_\e \,.
\end{equation*} We obtain, thanks to Cauchy-Schwarz inequality, that
\begin{align*}
&\frac{1}{2}\frac{\d}{\d t}\left\| (\mathbf{Id} - \bm{\pi}_{0})\hu_{\e} \right\|^2_{\tilde{H}}  + \frac{\tilde{\mu}_{1}}{\e^{2}}\| (\mathbf{Id}-\bm{\pi}_{0})\hu_{\e} \|^{2}_{ \tilde{H}_{1} } \\
&\quad \leq \langle \e^{-1}\big(\Q_{1}(\hu_{\e},\hu_{\e})  - (\mathbf{Id}-\bm{\pi}_{0})(v\cdot\nabla_{x}\hu_\e)\big)+  (\mathbf{Id}-\bm{\pi}_{0})(\A_\e \ho_\e) , (\mathbf{Id}-\bm{\pi}_{0})\hu_{\e}\rangle_{\tilde{H}}\\
&\quad \leq \e^{-1}\Big( \| \Q_{1}(\hu_{\e},\hu_{\e})\|_{\tilde{H}_{-1}} + \| v\cdot \nabla_{x} \hu_{\e} \|_{\tilde{H}_{-1}} \Big)\|  (\mathbf{Id}-\bm{\pi}_{0})\hu_{\e} \|_{\tilde{H}_{1}}\\
&\hspace{4.1cm} + \| \A_\e \ho_\e \|_{\tilde{H}}\|  (\mathbf{Id}-\bm{\pi}_{0})\hu_{\e} \|_{\tilde{H}}\,.
\end{align*}
We deduce easily then with a simple use of Young's inequality on the right-hand-side of this inequality that there is $C >0$ independent of $\e$ such that
\begin{multline*}
\frac{1}{2}\frac{\d}{\d t}\left\| (\mathbf{Id}-\bm{\pi}_{0})\hu_{\e} \right\|^2_{\tilde{H}}  + \frac{\tilde{\mu}_{1}}{2\e^{2}}\| (\mathbf{Id}-\bm{\pi}_{0})\hu_{\e} \|^{2}_{ \tilde{H}_{1} } \\
\leq C\big(\| \Q_{1}(\hu_{\e},\hu_{\e})\|^{2}_{\tilde{H}_{-1}} + \| v\cdot\nabla_{x} \hu_{\e} \|^{2}_{\tilde{H}_{-1}} +\e^{2} \| \A_\e \ho_\e \|^{2}_{\tilde{H}}\big)\,.\end{multline*}
Thus, for some different $C >0$, one has
\begin{multline*}
\frac{\d}{\d t}\left\| (\mathbf{Id}-\bm{\pi}_{0})\hu_{\e}(t) \right\|^2_{\tilde{H}}  + \frac{\tilde{\mu}_{1}}{\e^{2}}\| (\mathbf{Id}-\bm{\pi}_{0})\hu_{\e}(t) \|^{2}_{ \tilde{H}_{1} }\\ \leq C\big( \| \hu_\e(t) \|^{4}_{\H} + \| \hu_\e(t) \|^{2}_{\H} + \| \ho_\e(t)\|^{2}_{\E} \big)
\leq C\mathcal{K}_0\,\end{multline*}
where the last estimate comes from the results obtained in Sections \ref{sec:non} and~\ref{sec:Cauchy},  {$\mathcal{K}_0 \lesssim 1$} and also $\|\cdot\|_{\tilde{H}} \lesssim \|\cdot\|_{\H}.$ We integrate this inequality over $(t_{1},t_{2})$ to get
\begin{equation}\label{strong-comp}\begin{split}
\frac{\tilde{\mu}_{1}}{\e^{2}}\int^{t_{2}}_{t_{1}}\| (\mathbf{Id}-\bm{\pi}_{0})\hu_{\e}(t) \|^{2}_{ \tilde{H}_{1} }\d t &\leq \|(\mathbf{Id}-\bm{\pi}_{0})\hu_{\e}(t_{1}) \|^2_{\tilde{H}} + C\mathcal{K}_0(t_{2}-t_{1})\\
&\leq C\mathcal{K}_0\,\max(1,t_{2}-t_{1})\,,
\end{split}\end{equation}
where we used \eqref{exponential-e1/2}. 
Introduce now the space  $\tilde{E}=L^{1}_{v}\W^{\beta,2}_{x}(\m_{q})$.  Noticing that 
$$\| \cdot \|_{\tilde{E}}\lesssim\| \cdot\|_{\E} \qquad \text{ and } \quad \| \cdot \|_{\tilde{E}}\lesssim\| \cdot\|_{ \tilde{H} }$$
and writing that $h_\e(\tau)=\hu_\e(\tau)+\ho_\e(\tau)$, one has
\begin{equation}\label{eq:noT}
\int^{t_{2}}_{t_{1}}\| (\mathbf{Id}-\bm{\pi}_{0})h_{\e}(\tau) \|_{\tilde{E} } \, \d \tau \leq \int^{t_{2}}_{t_{1}}\Big(\| (\mathbf{Id}-\bm{\pi}_{0})\hu_\e(\tau) \|_{ \tilde{E} } + \| (\mathbf{Id}-\bm{\pi}_{0})\ho_{\e}(\tau) \|_{ \tilde{E} }\Big)\d \tau\,.
\end{equation}
Using Cauchy-Schwarz inequality
\begin{align*}
\int^{t_{2}}_{t_{1}}\| (\mathbf{Id}-\bm{\pi}_{0})&h_{\e}(\tau) \|_{\tilde{E} }\d \tau\\
&\lesssim \sqrt{t_2 - t_1}\, \Bigg(\bigg(\int^{t_2}_{t_1}\| (\mathbf{Id}-\bm{\pi}_{0})\hu_{\e}(\tau) \|^{2}_{ \tilde{H} }\, \d \tau \bigg)^{\frac 12} +  \bigg(\int^{t_{2}}_{t_{1}}\| \ho_\e(\tau) \|^{2}_{ \E }\,\d \tau\bigg)^{\frac 12}\Bigg)\,.
\end{align*}
From \eqref{strong-comp}, the first integral involving $\hu_\e$ is such that
$$\bigg(\int^{t_2}_{t_1}\| (\mathbf{Id}-\bm{\pi}_{0})\hu_{\e}(\tau) \|^{2}_{ \tilde{H} }\, \d \tau \bigg)^{\frac 12} \leq C \e \sqrt{{\mathcal{K}_{0}}}
\max\big\{ 1 , \sqrt{t_{2}-t_{1}} \big\}$$
whereas, to estimate the integral involving $\ho_\e$ we use that $\ho_\e \in \mathcal{X}_{0}$ as defined in Section \ref{sec:Cauchy} to get
\begin{equation*}
\begin{split}
\int^{t_{2}}_{t_{1}}\| \ho_\e(\tau) \|^{2}_{ \E }\,\d \tau &\leq \sup_{t_{1}\leq \tau\leq t_{2}}\|\ho_\e(\tau)\|_{\E}\int_{t_{1}}^{t_{2}}\|\ho_\e(\tau)\|_{\E_{1}}\d\tau\\
&\leq \e^{2}\vertiii{\ho_\e}_{0}^{2} \leq C\e^{2}\,\mathcal{K}_{0}.
\end{split}\end{equation*}
This proves the result.
\end{proof}

{\begin{nb} Notice that, if we are not interested in introducing a modulus of continuity in time for the above integral, we can directly deduce from \eqref{strong-comp} and \eqref{eq:noT} that
\begin{multline*}
\int_{0}^{T}\left\|\left(\mathbf{Id}-\bm{\pi}_{0}\right)h_{\e}(t)\right\|_{L^{1}_{v}\W^{\beta,2}_{x}\,(\m_{q+1})}\d t \\
\lesssim
\int_{0}^{T}\left\|\left(\mathbf{Id}-\bm{\pi}_{0}\right)\ho_{\e}(t)\right\|_{\E_{1}}\d t + \int_{0}^{T}\left\|\left(\mathbf{Id}-\bm{\pi}_{0}\right)\hu_{\e}(t)\right\|_{\tilde{H}_{1}}\d t \\
\lesssim \e\vertiii{\ho_{\e}}_{0} + \sqrt{T}\left(\int_{0}^{T}\left\|\left(\mathbf{Id}-\bm{\pi}_{0}\right)\hu_{\e}(t)\right\|_{\tilde{H}_{1}}^{2}\d t \right)^{\frac{1}{2}}\end{multline*}
which results in
\begin{equation}\label{eq:weight+1}
\int_{0}^{T}\left\|\left(\mathbf{Id}-\bm{\pi}_{0}\right)h_{\e}(t)\right\|_{L^{1}_{v}\W^{\beta,2}_{x}(\m_{q+1})}\d t \lesssim \e {\sqrt{\mathcal{K}_0}(1 + \sqrt{T})}\end{equation}
for any $0\leq \beta \leq m-1$.
\end{nb}}
We deduce the following convergence result:
\begin{theo}[Weak convergence]\label{theo:strong-conv} Fix $T >0$, and let 
$$\left\{h_{\e}\right\}_{\e}\subset L^1\big((0,T);{L^{1}_{v}\W_{x}^{m,2}(\m_{q})}\big)$$ be a sequence of solutions to the inelastic Boltzmann equation \eqref{BE}.  Then, with the splitting $h_{\e}=\ho_{\e}+\hu_{\e}$, up to extraction of a subsequence, one has
\begin{equation}\begin{cases}\label{eq:mode-conv}
\left\{\ho_{\e}\right\}_{\e} \text{converges to $0$ strongly  in } L^{1}((0,T)\,;\,{\E_1}) \\
\\
\left\{\hu_{\e}\right\}_{\e} \text{which converges to $\bm{h}$ weakly in } L^{2}\left((0,T)\,;\,L^{2}_{v}\W^{m,2}_x\big(\M^{-\frac{1}{2}}\big)\right)\end{cases}\end{equation}
where $\bm{h}=\bm{\pi}_{0}(\bm{h}).$ In particular, there exist 
$$\varrho \in L^{2}\left((0,T);\,\W^{m,2}_{x}(\T^{d})\right), \qquad \vE \in L^{2}\left((0,T);\,\W^{m,2}_x(\T^{d})\right),$$
$$u \in L^{2}\left((0,T);\;\left(\W^{m,2}_{x}(\T^{d})\right)^{d}\right)$$
such that
 \begin{equation}\label{eq:hlim}
\bm{h}(t,x,v)=\left(\varrho(t,x)+u(t,x)\cdot v + \frac{1}{2}\vE(t,x)(|v|^{2}-d\en_{1})\right)\M(v)\end{equation}
where $\M$ is the Maxwellian distribution introduced in \eqref{eq:max}.
\end{theo}
\begin{proof} Let $T >0$ be fixed. We use the notations of Proposition \ref{lem:equi}. The estimates obtained in Section \ref{sec:Cauchy}, using  the splitting $h_{\e}=\ho_{\e}(t)+\hu_{\e}(t)$  imply the following properties of the sequences of time-dependent vector-valued mappings $\{\hu_{\e}\}_{\e}, \{\ho_{\e}\}_{\e}$ and $\{h_{\e}\}_{\e}$:
\begin{equation}\label{eq:sequ}
\{\hu_{\e} \} \subset \big( L^{1} \cap L^{\infty} \big)\big((0,T);\H\big)\qquad \text{ is bounded } \end{equation}
\begin{equation}\label{eq:seqo}
\int_{0}^{T}\|\ho_{\e}(t)\|_{\E_{1}}\d t \lesssim \e^{2}\end{equation}
From \eqref{eq:sequ} and since $\|\cdot\|_{L^{2}_{v}\W^{m,2}_{x}(\M^{-\frac{1}{2}})} \lesssim \|\cdot\|_{\H}$, we deduce that
$$\{h_{\e}^{1}\} \text{ is bounded in } L^{2}\left((0,T)\,;\,L^{2}_{v}\W^{m,2}_{x}(\M^{-\frac{1}{2}})\right)$$
and therefore, admits a subsequence, say $\left\{\hu_{\e'}\right\}_{\e'}$ which converges weakly to some $\bm{h}$ in the space $L^{2}\big((0,T);L^{2}_{v}\W^{m,2}_{x}(\M^{-\frac{1}{2}})\big)$. This, combined with \eqref{eq:seqo} gives \eqref{eq:mode-conv}. From \eqref{strong-comp} we also have, for that subsequence, 
$$\lim_{\e'\to0}\int_{0}^{T}\left\|\left(\mathbf{Id}-\bm{\pi}_{0}\right)\hu_{\e'}(t)\right\|_{L^{2}_{v}\W^{m-1,2}_{x}(\M^{-\frac{1}{2}})}^{2}\d t=0$$
so that $\left(\mathbf{Id}-\bm{\pi}_{0}\right)\bm{h}=0$. This gives the result.\end{proof}
 \begin{nb}\label{nb:mode}
As observed in the previous proof, the convergence \eqref{eq:mode-conv} can be made even more precise since we also have
\begin{equation*}
\left\{\left(\mathbf{Id}-\bm{\pi}_{0}\right)\hu_{\e}\right\} \text{ converges strongly to $0$ in } L^{2}\left((0,T)\,;\,L^{2}_{v}\W^{m-1,2}_{x}(\M^{-\frac{1}{2}})\right).\end{equation*}
This means somehow that the only part of $h_{\e}$ which prevents the strong convergence towards $\bm{h}$ is~$\left\{\bm{\pi}_{0}\hu_{\e}\right\}_{\e}$. 
\end{nb}

Because of Theorem~\ref{theo:strong-conv} and for simplicity sake, from here on, we will write that our sequences converge even if it is true up to an extraction.

{The above mode of convergence implies the following convergence of velocity averages of $h_{\e}$.  For any function $f=f(t,x,v)$ we denote the velocity average by
$$\la f\,\ra=\int_{\R^{d}}f(t,x,v)\d v$$
recalling of course that this is a function depending on $(t,x)$. We have then the following:
\begin{lem}\label{lem:mode}
Let $\{h_{\e}\}$ be converging to $\bm{h}$ in the sense of Theorem \ref{theo:strong-conv}. Then, for any function $\psi=\psi(v)$ such that
$$|\psi(v)| \lesssim \m_{q}(v)$$
one has
\begin{equation}\label{eq:distr}
\la \psi\,h_{\e}\ra \longrightarrow \la \psi\,\bm{h}\ra \quad \text{ in } \mathscr{D}'_{t,x}\end{equation}
whereas
\begin{equation}\label{eq:distrQ1}
\la \psi\,\Q_{1}^{\bm{r}}(h_{\e},h_{\e})\ra \longrightarrow 0 \quad \text{ in } \mathscr{D}'_{t,x}\end{equation}
where we set $\Q_{1}^{\bm{r}}(h_{\e},h_{\e})=\Q_{1}(h_{\e},h_{\e})-\Q_{1}\left(\bm{\pi}_{0}h_{\e},\bm{\pi}_{0}h_{\e}\right).$
\end{lem}
}\begin{proof} Let $\psi$ be such that $|\psi(v)| \lesssim \m_{q}(v)$ and let $\varphi=\varphi(t,x) \in \mathcal{C}_{c}^{\infty}((0,T) \times \T^{d})$ be given. One computes
$$I_{\e}:=\int_{0}^{T}\d t\int_{\T^{d}}\varphi(t,x)\left(\la \psi\,h_{\e}\ra - \la \psi\,\bm{h}\ra\right)\d x=I_{\e}^{0}+I_{\e}^{1}$$
where
\begin{equation*}
I^{0}_{\e}=\int_{0}^{T}\d t\int_{\T^{d}}\varphi(t,x) \la \psi\,\ho_{\e}\ra\d x, \qquad 
I^{1}_{\e}=\int_{0}^{T}\d t\int_{\T^{d}}\varphi(t,x)\left(\la \psi\,\hu_{\e}\ra - \la \psi\,\bm{h}\ra\right)\d x.\end{equation*}
Because $|I_{\e}^{0}| \lesssim \|\varphi\|_{L^{\infty}_{t,x}}\int_{0}^{T}\|\ho_{\e}(t)\|_{L^{1}_{v,x}(\m_{q})}\d t   \lesssim \|\varphi\|_{L^{\infty}_{t,x}}\int_{0}^{T}\|\ho_{\e}(t)\|_{L^{1}_{v}L^{2}_{x}(\m_{q})}\d t$, we deduce from~\eqref{eq:mode-conv} that
$\lim_{\e\to0}I_{\e}^{0}=0.$ In the same way, one has
$$I_{\e}^{1}=\int_{0}^{T}\d t\int_{\T^{d}\times \R^{d}}\left(\psi(v)\,\M(v)\,\varphi(t,x)\right)\,\left(\hu_{\e}(t,x,v)-\bm{h}(t,x,v)\right)\d x\,\M^{-1}(v)\d v$$
and, since the mapping 
\begin{equation}\label{eq:varphiL2}
(t,x,v) \longmapsto \psi(v)\,\M(v)\,\varphi(t,x) \quad \text{ belongs to } L^{2}((0,T)\,;\,L^{2}_{v}\W^{m,2}_{x}(\M^{-\frac{1}{2}})),\end{equation} we deduce from \eqref{eq:mode-conv} that $\lim_{\e \to 0}I_{\e}^{1}=0.$ This proves \eqref{eq:distr}. 
To prove \eqref{eq:distrQ1}, one sets$$J_{\e}:=\int_{0}^{T}\d t\int_{\T^{d}}\varphi(t,x)\la \psi\,\Q_{1}^{\bm{r}}(h_{\e},h_{\e})\ra\d x.$$
One writes $J_{\e}=J^{1}_{\e}+J_{\e}^{2}$ where
\begin{multline*}
J_{\e}^{1}=\int_{0}^{T}\d t\int_{\T^{d}}\varphi(t,x) \la \psi\,\Q_{1}((\mathbf{Id}-\bm{\pi}_{0})h_{\e},(\mathbf{Id}-\bm{\pi}_{0})h_{\e})\ra \d x\\
J_{\e}^{2}=2\int_{0}^{T}\d t\int_{\T^{d}}\varphi(t,x) \la \psi\,\widetilde{\Q}_{1}((\mathbf{Id}-\bm{\pi}_{0})h_{\e},\bm{\pi}_{0}h_{\e})\ra \d x\\
\end{multline*}
where we recall that $\widetilde\Q_1$ is defined in~\eqref{def:Qalphatilde}. 
One has
\begin{align*}|J_{\e}^{1}| &\lesssim \|\varphi\|_{L^{\infty}_{t,x}}\int_{0}^{T}\left\|\Q_{1}\left((\mathbf{Id}-\bm{\pi}_{0})h_{\e},(\mathbf{Id}-\bm{\pi}_{0})h_{\e}\right)\right\|_{L^{1}_{v,x}(\m_{q})}\d t\\
&\lesssim \|\varphi\|_{L^{\infty}_{t,x}}\int_{0}^{T}\left\|\Q_{1}\left((\mathbf{Id}-\bm{\pi}_{0})h_{\e},(\mathbf{Id}-\bm{\pi}_{0})h_{\e}\right)\right\|_{L^{1}_{v}L^2_x(\m_{q})}\d t.
\end{align*}
Noticing that 
$$
\left\|\Q_{1}\left((\mathbf{Id}-\bm{\pi}_{0})h_{\e},(\mathbf{Id}-\bm{\pi}_{0})h_{\e}\right)
\right\|_{L^{1}_{v}L^2_{x}(\m_{q})} \lesssim \|h_\e\|_{\mathcal{E}} \|(\mathbf{Id}-\bm{\pi}_0)h_\e\|_{L^{1}_{v}\W^{m-1,2}_{x}(\m_{q+1})}.
$$
we deduce from \eqref{eq:weight+1} and the fact that $\sup_{t\in (0,T)}\|h_{\e}(t)\|_{\E} < \infty$ that
$$\lim_{\e \to0}|J_{\e}^{1}|=0.$$
We prove exactly in the same way that
$$\lim_{\e \to 0}|J_{\e}^{2}|=0.$$
This proves the result.\end{proof}

\noindent
Regarding the characterisation \eqref{eq:hlim} of the limit $\bm{h}(t)$, note that
$$\varrho(t,x)=\int_{\R^{d}}\bm{h}(t,x,v)\d v, \qquad u(t,x)=\frac{1}{\en_{1}}\int_{\R^{d}}v\,\bm{h}(t,x,v)\d v\,,$$
and
$$\varrho(t,x)+\en_{1}\vE(t,x)=\frac{1}{d\en_{1}}\int_{\R^{d}}|v|^{2}\bm{h}(t,x,v)\d v\,.$$

\begin{cor}\label{cor:boussi} With the notations of Theorem \ref{theo:strong-conv}, for any $T >0$, the limit $\bm{h}(t,x,v)$ given by \eqref{eq:hlim} satisfies the \emph{incompressibility condition}
\begin{equation}\label{eq:incomp}
\mathrm{div}_{x} u(t,x)=0\,, \qquad t \in (0,T)\,,
\end{equation}
and \emph{Boussinesq relation}
\begin{equation}\label{eq:boussi}
\nabla_{x}\left(\varrho+\en_{1}\vE\right)\,=0\,.
\end{equation}
As a consequence, introducing 
$$E(t)= \int_{\T^{d}}\,\vE(t,x)\d x, \qquad t \in (0,T)\,,$$
one has \emph{strengthened Boussinesq relation}
\begin{equation}\label{eq:boussi2}
\varrho(t,x) + \en_{1}\left(\vE(t,x)- E(t)\right)=0\,, \qquad \text{ for a.e } (t,x) \in (0,T)\times \T^{d}.\end{equation}
\end{cor}
\begin{proof} Set
$$\varrho_{\e}(t,x)=\int_{\R^{d}}h_{\e}(t,x,v)\d v, \qquad \bm{u}_{\e}(t,x)={\frac{1}{\en_{1}}}\int_{\R^{d}}v\,h_{\e}(t,x,v)\d v\,,$$
and, multiplying \eqref{BE} with $1$ and $v$ and integrating in velocity, we get
\begin{equation}
\label{mass0}
\e\partial_{t}\varrho_{\e}+\,{\en_{1}}\mathrm{div}_{x}\left(\bm{u}_{\e}\right)=0\,,
\end{equation}
\begin{equation}\label{bulk0}
\e\,\partial_{t}\bm{u}_{\e}+
\mathrm{Div}_{x}\left(\bm{J}_{\e}\right)=\frac{\kappa_{\re}}{\e}\bm{u}_{\e}\,,
\end{equation}
where $\bm{J}_{\e}(t,x)$ denotes the tensor
$$\bm{J}_{\e}(t,x):={\frac{1}{\en_{1}}}\int_{\R^{d}}v \otimes v\,h_{\e}(t,x,v)\d v\,,$$
since both $\mathbf{L}_{\re}$ and $\Q_{\re}$ conserve mass and momentum. The proof of \eqref{eq:incomp} is straightforward since $\e\partial_{t}\varrho_{\e} \to 0$ and $\mathrm{div}_{x}(\bm{u}_{\e})\to \mathrm{div}_{x}u$ in the distribution sense. Let us give the detail for the sake of completeness.
Multiplying \eqref{mass0} with a function $\varphi \in \mathcal{C}_{c}^{\infty}((0,T) \times \T^{d})$  and integrating over $(0,T) \times \T^{d}$ we get that
$$
-\int_{0}^{T}\d t\int_{\T^{d}}\nabla_{x}\varphi(t,x)\cdot \bm{u}_{\e}(t,x)\d x=\e\int_{0}^{T}\d t\int_{\T^{d}}\varrho_{\e}(t,x)\partial_{t}\varphi(t,x)\d x\,,
$$
which, taking the limit $\e \to 0$ and because $\varrho_{\e} \to \varrho$ and $\bm{u}_{\e} \to u$  in {$\mathscr{D}'_{t,x}$}, yields
$$\int_{0}^{T}\d t\int_{\T^{d}}\nabla_{x}\varphi(t,x)\cdot u(t,x)\d x=0\,, \qquad \forall \, \varphi \in \mathcal{C}_{c}^{\infty}((0,T) \times \T^{d}).$$
Since $u(t,x) \in {L^{2}((0,T)\,;\,(\W^{m,2}_{x}(\T^{d})^d))}$, the incompressibility condition \eqref{eq:incomp} holds true. In the same way, for any $i=1,\ldots,d $ and $\varphi \in \mathcal{C}_{c}^{\infty}((0,T) \times \T^{d})$, noticing that
$$\lim_{\e\to0^{+}}\e\int_{0}^{T}\bm{u}_{\e}^{i}\,\partial_{t}\varphi(t,x)\d x=\lim_{\e\to0^{+}}\frac{\kappa_{\re}}{\e}\int_{0}^{T}\d t\int_{\T^{d}}\bm{u}_{\e}^{i}(t,x)\varphi(t,x)\d x=0\,,$$
because $\kappa_{\re}=1-\re \leq C\e^{2}$ we get that 
$$0=\lim_{\e \to 0^{+}}\sum_{j=1}^{d}\int_{0}^{T}\d t \int_{\T^{d}}\bm{J}_{\e}^{i,j}(t,x)\partial_{x_{j}}\varphi(t,x)\d x=\sum_{j=1}^{d}\int_{0}^{T}\d t \int_{\T^{d}}\bm{J}_{0}^{i,j}(t,x)\partial_{x_{j}}\varphi(t,x)\d x\,,$$
where
$$\bm{J}_{0}^{i,j}(t,x)={\frac{1}{\en_{1}}}\int_{\R^{d}}v_{i}\,v_{j}\,\bm{h}(t,x,v)\d v=\left(\varrho(t,x)+\en_{1}\vE(t,x)\right)\delta_{ij}, \qquad i,j=1,\ldots,d.$$
Therefore, for any $i=1,\ldots,d$,
$$\int_{0}^{T}\d t\int_{\T^{d}}\left(\varrho(t,x)+\en_{1}\vE(t,x)\right)\,\partial_{x_{i}}\varphi(t,x)\d x=0\,, \qquad \forall\, \varphi \in \mathcal{C}_{c}^{\infty}((0,T) \times \T^{d}).$$
As before, this gives the Boussinesq relation \eqref{eq:boussi}. To show that Boussinesq relation can be strengthened, one notices that
$$\lim_{\e\to0^{+}}\int_{\T^{d}}\varrho_{\e}(t,x)\d x=\int_{\T^{d}}\varrho(t,x)\d x\, \quad \text{ in } \mathscr{D}'_{t}$$
from which we deduce, from the conservation of mass for \eqref{BE1}, that
$$\int_{\T^{d}}\varrho(t,x)\d x=0\,, \qquad \text{ for a.e. } t >0.$$
With the definition of $E(t)$, this implies that
$$\int_{\T^{d}}(\varrho(t,x)+\en_{1}\left(\vE(t,x)-E(t)\right))\d x=0\,, \qquad \text{ for a.e. } t >0\,,$$
and, this combined with \eqref{eq:boussi} yields the strengthened form \eqref{eq:boussi2}.\end{proof}

\begin{nb} Using Boussinesq relation together with \eqref{eq:hlim}, one checks without major difficulty that
\begin{equation}\label{eq:vnah} 
v \cdot \nabla_{x}\bm{h}\,=\,\M(v \otimes v):\nabla_{x}u + \frac{1}{2}\M\,\left(|v|^{2}-(d+2)\en_{1}\right)\,v\cdot \nabla_{x}\vE\,.
\end{equation} 
Then, using the incompressibility condition \eqref{eq:incomp} it holds that
$$\int_{\R^{d}}\Psi_{j}(v)\,v\cdot\nabla_{x}\bm{h}\,\d v=0, \qquad \forall j=1,\ldots,d+2\,,$$
that is, $\bm{\pi}_{0}(v\cdot \nabla_{x}\bm{h})=0$.  In particular, $v\cdot \nabla_{x}\bm{h} \in \mathrm{Range}(\mathbf{Id}-\bm{\pi}_{0})  \subset \mathrm{Range}(\mathbf{L}_{1})$ (see \cite[Eq. (6.34), p. 180]{kato}).\end{nb}

\subsection{Identification of the limit}\label{sec:hydro}

We aim here to fully characterise the limit $\bm{h}(t,x,v)$ obtained in Theorem \ref{theo:strong-conv}. To do so, we identify the limit equation satisfied by the macroscopic quantities
$(\varrho,u,\vE)$  in \eqref{eq:hlim} following the path of {\cite{BaGoLe2,golseSR}} and exploiting the fact that {the mode of convergence in Theorem \ref{theo:strong-conv} is stronger than the one of {\cite{BaGoLe2,golseSR}}}.  The regime of weak inelasticity is central in the analysis. 

We denote by $\{h_{\e}\}$ \emph{any} subsequence which converges to $\bm{h}$ in the above Theorem \ref{theo:strong-conv}.  We will see in the sequel, under some strong convergence assumption on the initial datum that all subsequences will share the same limit and, as such, the whole sequence will be convergent.

\medskip
\noindent
Recall \eqref{BE}
\begin{equation}\label{BE1}
\e\partial_{t}h_{\e}+v\cdot \nabla_{x}h_{\e} + \e^{-1}\kappa_{\re}\nabla_{v}\cdot (vh_{\e})=\e^{-1}\mathbf{L}_{\re}h_{\e}+\Q_{\re}(h_{\e},h_{\e})\,,
\end{equation}
under the scaling hypothesis that $\re=1-\lambda_{0}\e^{2}+o(\e^{2})$, $\lambda_{0}\geq0$ (see Assumption~\ref{hyp:re}).  Multiplying \eqref{BE1} respectively with $1$, $v$, $\tfrac{1}{2}|v|^{2}$, we observe that the quantities
$$\la h_{\e}\ra,\quad \la vh_{\e}\ra,\quad \la \tfrac{1}{2}|v|^{2}\ra, \quad \la \tfrac{1}{2}|v|^{2}v\,h_{\e}\ra\,, \quad \text{and} \quad \la v\otimes v\,h_{\e}\ra\,,$$
are important.  As in the classical elastic case, we write
 $$\la v\otimes v\,h_{\e}\ra=\la \bm{A}\,h_{\e}\ra + p_{\e}\mathbf{Id}, \qquad p_{\e}=\la \frac{1}{d}|v|^{2}\,h_{\e}\ra\,,$$
where we introduce the traceless tensor 
$$\bm{A}=\bm{A}(v)=v \otimes v -\frac{1}{d}|v|^{2}\mathbf{Id}.$$
Properties of this tensor are established in Appendix \ref{sec:hydro1}. In a more precise way, one obtains, after integrating \eqref{BE1} against $1$, $v_{i}$, $\frac{1}{2}\,|v|^{2}$,
\begin{subequations}\label{moments}
\begin{equation}\label{mass1}
\partial_{t}\la h_{\e}\ra +\frac{1}{\e}\mathrm{div}_{x}\la v\,h_{\e}\ra=0\,,
\end{equation}
\begin{equation}\label{bulk1}
\partial_{t} \la v\,h_{\e}\ra +\frac{1}{\e}
\mathrm{Div}_{x}\la \bm{A}\,h_{\e} \ra +  \frac{1}{\e}\nabla_{x}p_{\e}=\frac{\kappa_{\re}}{\e^{2}}\la v\,h_{\e}\ra\,,
\end{equation}
\begin{equation}\label{energy1}
\partial_{t}\la \tfrac{1}{2}|v|^{2}h_{\e}\ra+\frac{1}{\e}\mathrm{div}_{x}\,\la \tfrac{1}{2}|v|^{2}v\,h_{\e}\ra\,=\frac{1}{\e^{3}}\mathscr{J}_{\re}(f_{\e},f_{\e})+\frac{2\kappa_{\re}}{\e^{2}}\la \tfrac{1}{2}|v|^{2}h_{\e}\ra\,,
\end{equation}
where
$$\mathscr{J}_{\re}(f,f)=\int_{\R^{d}}\left[\Q_{\re}(f,f)-\Q_{\re}(G_{\re},G_{\re})\right]\,|v|^{2}\d v.$$
\end{subequations} 
Notice that, using \eqref{eq:hlim} as well as Corollary \ref{cor:boussi}, 
\begin{multline*}
\mathrm{div}_{x}\la v\,h_{\e}\ra \longrightarrow {\en_{1}}\mathrm{div}_{x}u=0,\qquad \la \tfrac{1}{2}|v|^{2}h_{\e}\ra \longrightarrow \frac{d\en_{1}}{2}\left(\varrho+\en_{1}\vE\right),\\
\nabla_{x}p_{\e} \longrightarrow \frac{1}{d}\nabla_{x}\la |v|^{2}\bm{h}\ra =\en_{1}\nabla_{x}(\varrho+\en_{1}\vE)=0\,,\\
\la \bm{A}\,h_{\e}\ra \longrightarrow \la \bm{A}\,\bm{h}\ra=0\,, \\
\la \tfrac{1}{2}|v|^{2}v_{j}\,h_{\e}\ra\,\longrightarrow \la \tfrac{1}{2}|v|^{2}v_{j}\,\bm{h}\ra=\tfrac{1}{2}u_{j}\la |v|^{2}v_{j}^{2}\M\ra=\frac{d+2}{2}\en_{1}^{2}u_{j}, \qquad j=1,\ldots,d\,,
\end{multline*}
where all the limits hold in {$\mathscr{D}'_{t,x}$} and where $\la \bm{A}\bm{h}\ra=0$ since $\bm{h} \in \mathrm{Ker}(\mathbf{L}_{1})$ and $\bm{A} \in \mathrm{Range}(\mathbf{I}-\bm{\pi}_{0})$. Moreover, under the above scaling 
$$\frac{\kappa_{\re}}{\e^{2}}\la v\,h_{\e}\ra \longrightarrow {\en_{1}}\lambda_0u\,, \quad \text{in} \quad {\mathscr{D}'_{t,x}}\,,$$
since $\lambda_0=\lim_{\e \to 0^{+}}\e^{-2}\kappa_{\re}$. The limit of $\e^{-3}\mathscr{J}_{\re}(f_{\e},f_{\e})$ is handled in the following lemma.
\begin{lem}\label{lemJ0}
It holds that  
$$\frac{1}{\e^{3}}\mathscr{J}_{\re}(f_{\e},f_{\e}) \longrightarrow \mathcal{J}_{0}\qquad \text{ in } \mathscr{D}'_{t,x}\,,$$
where
$$\mathcal{J}_{0}(t,x)=-\lambda_{0}\,\bar{c}\,\en_{1}^{\frac{3}{2}}\left(\varrho(t,x)+\frac{3}{4}\en_{1}\,\vE(t,x)\right)$$
for some positive constant $\bar{c}$ depending only on the angular kernel $b(\cdot)$ and $d$. In particular, 
$$\mathcal{J}_{0} =-\lambda_{0}\,\bar{c}\,\en_{1}^{\frac{5}{2}}\left(E(t)-\frac{1}{4}\vE(t,x)\right).$$
\end{lem}
\begin{proof} We recall, see \eqref{eq:Dre}, that 
$${\int_{\R^{d}}|v|^{2}\Q_{\re}(g,f)\d v=-(1-\re^{2})\frac{\gamma_{b}}{4}\int_{\R^{d}\times\R^{d}}f(v)g(v_{\ast})\,|v-v_{\ast}|^{3}\d v\d v_{\ast}}$$
where $\gamma_{b}=\frac{1}{2}\int_{\S^{d-1}}\left(1-\widehat{u}\cdot \sigma\right)b(\widehat{u}\cdot\sigma)\d \sigma$.
Thus, for $f_{\e}=G_{\re}+\e\,h_{\e}$ we obtain
\begin{multline}\label{eq:Jalp}
\frac{1}{\e^{3}}\mathscr{J}_{\re}(f_{\e},f_{\e})=-\frac{\gamma_{b}}{4}\frac{1-\re^{2}}{\e^{2}}\bigg(\int_{\R^{d}\times \R^{d}}\left[h_{\e}(v)G_{\re}(v_{\ast})+h_{\e}(v_{\ast})G_{\re}(v)\right]|v-v_{\ast}|^{3}\d v\d v_{\ast}\\
+\e\int_{\R^{d}\times\R^{d}}h_{\e}(v)h_{\e}(v_{\ast})|v-v_{\ast}|^{3}\d v\d v_{\ast}\bigg)\,.
\end{multline}
Recall that $\lim_{\e\to0^{+}}\frac{1-\re}{\e^{2}}=\lambda_{0}$. It is clear {from Minkowski's integral inequality} that the $\W^{m,2}_{x}(\T^{d})$ norm of the last term in the right-side is controlled by $\|h_{\e}\|_{\E}^{2}$.  
Theorem \ref{theo:main-cauc1} implies that the last term in \eqref{eq:Jalp} is converging to $0$ in $L^{1}((0,T);\W^{m,2}_{x}(\T^{d}))$. One handles the first term in the right-side using Theorem \ref{theo:strong-conv} and the fact that $G_{\re}\to \M$ strongly.  Details are left to the reader.
We then easily obtain the convergence of $\e^{-3}\mathscr{J}_{\re}(f_{\e},f_{\e})$ towards
$$\mathcal{J}_{0}:=-\lambda_{0}\,\gamma_{b}\,\int_{\R^{d}\times\R^{d}}\bm{h}(t,x,v)\M(v_{\ast})|v-v_{\ast}|^{3}\d v\d v_{\ast}.$$
The expression of $\mathcal{J}_{0}$ is then obtained by direct inspection from \eqref{eq:hlim} with 
$$\bar{c}=\gamma_{b}\,a,\qquad a=\frac{2\sqrt{2}}{(2\pi)^{\frac{d}{2}}}\int_{\R^{d}}\exp\left(-\frac{1}{2}|v|^{2}\right)|v|^{3}\d v\,,$$ 
where
 \begin{multline*}
\int_{\R^{2d}}\M(v)\M(v_{\ast})|v-v_{\ast}|^{3}\d v\d v_{\ast}=\en_{1}^{\frac{3}{2}}a\,,\\
\int_{\R^{2d}}\M(v)\M(v_{\ast})|v|^{2}|v-v_{\ast}|^{3}\d v\d v_{\ast}= \frac{2d+3}{2}\,\en_{1}^{\frac{5}{2}}\,a.
\end{multline*}
We refer to \cite[Lemma A.1]{MiMo3} for these identities. The second part of the lemma follows from the strengthened Boussinesq relation \eqref{eq:boussi2}.
\end{proof}

\subsection{About the equations of motion and temperature} \label{subsec:mottemp}
We give here some preliminary result aiming at deriving the equations satisfied by the bulk velocity $u(t,x)$ and $\vE(t,x)$. As in {\cite{BaGoLe2,golseSR}}, in order to investigate the limiting behaviour of the system \eqref{moments} as $\e\to0^{+}$, we need to investigate the limit in the distributional sense of 
 \begin{equation}\label{eq:Avh}
\e^{-1}\mathrm{Div}_{x}\la \bm{A}\,h_{\e} \ra=-\e^{-1}\mathrm{Div}_{x}\la \phi\,\mathbf{L}_{1}h_{\e} \ra \end{equation}
and
\begin{equation}\label{eq:bvh}
\e^{-1}\mathrm{div}_{x}\la \bm{b}\,h_{\e}\ra=-\e^{-1}\mathrm{div}_{x}\la \psi\,\mathbf{L}_{1}h_{\e}\ra\end{equation}
where $\phi$ and $\psi$ are defined in Lemma~\ref{lem:phipsi} and where we used that $\mathbf{L}_1$ is selfadjoint in~$L^2_v(\mathcal{M}^{-1/2})$. \\

{Since the limiting vector-field $u$ is divergence-free, it turns out enough to investigate only the limit of $\mathcal{P} \mathrm{Div}_{x}\la \e^{-1}\bm{A}\,h_{\e}\ra$ where we recall that $\mathcal{P}$ is the Leray projection on divergence-free vector fields\footnote{Recall that, for a vector field $\bm{u}$, $\mathcal{P}\bm{u}=\bm{u}-\nabla\,\Delta^{-1}(\nabla \cdot \bm{u})$. On the torus, it can be defined via Fourier expansion, if $\bm{u}=\sum_{k \in \Z^{d}}\bm{a}_{k}e^{i k \cdot x}$, $\bm{a}_{k} \in \C^{d}$, then $\mathcal{P}\bm{u}=\sum_{k \in \Z^{d}}\left(\mathbf{I}_{d}-\frac{k\otimes k}{|k|^{2}}\right)\bm{a}_{k}e^{i k \cdot x}.$} .}
We begin with a strong compactness result 
  
\begin{lem}\label{prop:strongue} Introduce
\begin{multline*}
{u_\e(t,x)=\exp\left(-t\frac{\kappa_{\re}}{\e^{2}}\right)\mathcal{P}\bm{u}_{\e}(t,x)}\,\\
\text{and}  \, \, \, \vartheta_{\e}(t,x)=\la \tfrac{1}{2}\big( |v|^{2}-(d+2)\en_{1} \big)h_{\e}\ra\,, \, \, t \in (0,T),\,\;x \in \T^{d}
.\end{multline*}
Then, {$\{\partial_{t}u_{\e}\}_{\e}$ and $\{\partial_{t}\vartheta_{\e}\}_{\e}$ are bounded in $L^{1}\left((0,T)\,;\, {\W^{m-2,2}_{x}(\T^{d})}\right)$.}
Consequently, up to the extraction of a subsequence,
\begin{equation}\label{eq:Puest}
\lim_{\e \to 0}\int_{0}^{T}\left\|\mathcal{P}\bm{u}_{\e}(t)-u(t)\right\|_{ {\W^{m-2,2}_{x}(\T^{d})}}\d t=0\end{equation}
and
\begin{equation}\label{eq:varthstr}
\lim_{\e\to0^{+}}\int_{0}^{T}\left\|\vartheta_{\e}(t,\cdot)-\vartheta_{0}(t,\cdot)\right\|_{ {\W^{m-2,2}_{x}(\T^{d})}}\d t=0\end{equation}
where 
$$\vartheta_{0}(t,x)=\la\tfrac{1}{2}(|v|^{2}-(d+2)\en_{1})\bm{h}\ra=\frac{d\en_{1}}{2}\left(\varrho(t,x)+\en_{1}\vE(t,x)\right)-\frac{d+2}{2}\en_{1}\varrho(t,x)\,.$$ In other words, $\{\mathcal{P}\bm{u}_{\e}\}_{\e}$ converges strongly to $u$ in $L^{1}\left((0,T)\,;\, {\W^{m-2,2}_{x}(\T^{d})}\right)$ and $\{\vartheta_{\e}\}_{\e}$ converges strongly to $\vartheta_{0}$ in $L^{1}\left((0,T)\,;\, {\W^{m-2,2}_{x}(\T^{d})}\right).$ 
\end{lem}
\begin{proof} We begin with the proof of \eqref{eq:Puest}. We apply the  Leray projection $\mathcal{P}$   to~\eqref{bulk1} to eliminate the pressure gradient term. Then, 
we have that
\begin{equation*}\label{bulk1P}
\partial_{t} u_{\e}=-\exp\left(-t\frac{\kappa_{\re}}{\e^{2}}\right)\mathcal{P}\left(\en_{1}^{-1}\mathrm{Div}_{x}\la \tfrac{1}{\e}\bm{A}\,h_{\e} \ra\right)\,.
\end{equation*}
Notice that, since $\{h_{\e}\}_{\e}$ is bounded in $L^{1}((0,T)\,;\,\E)$ {by Minkowski's integral inequality}, one has that  $$ {\{u_{\e}\}_{\e} \text{ is bounded in } L^{1}\left((0,T)\,;\,\W^{m,2}_{x}(\T^{d})\right).}$$
Moreover, since $\bm{A}\,h_{\e}=\bm{A}\,\left(\mathbf{Id}-\bm{\pi}_{0}\right)h_{\e}$ we deduce from Proposition \ref{lem:equi}  {and Minkowski's integral inequality} that
$$\sup_{\e}\int_{0}^{T}\left\|\mathcal{P}\left(\mathrm{Div}_{x}\la \tfrac{1}{\e}\bm{A}\,h_{\e}\ra\right)\right\|_{ {\W^{m-2,2}_{x}(\T^{d})}} \d t < \infty.$$
In particular
$$\{\partial_{t}u_{\e}\}_{\e} \text{ is bounded in } L^{1}\left((0,T)\,;\, {\W^{m-2,2}_{x}(\T^{d})}\right).$$
Applying \cite[Corollary 4]{simon} with $X= {\W^{m,2}_{x}(\T^{d})}$ and $B=Y= {\W^{m-2,2}_{x}(\T^{d})}$ (so that the embedding of $X$ into $B$ is compact by Rellich-Kondrachov Theorem \cite[Proposition~3.4, p.~330]{taylor}), we deduce that $\{u_{\e}\}_{\e}$ is relatively compact in $L^{1}\left((0,T)\,;\, {\W^{m-2,2}_{x}(\T^{d})}\right)$. The result of strong convergence follows easily since we already now that $\mathcal{P}u_{\e}$ converges to {$u$} in $\mathscr{D}'_{t,x}$ (see Lemma~\ref{lem:mode} and recall $u=\mathcal{P}u$ since $u$ is divergence-free). 

The proof of \eqref{eq:varthstr} is similar.   We begin with observing that, multiplying \eqref{mass1} with~$-\frac{d+2}{2}\en_{1}$ and add it to \eqref{energy1} we obtain the evolution of $\vartheta_{\e}(t,x)$
\begin{equation}\label{eq:trick}
\partial_{t}\vartheta_{\e}+\frac{1}{\e}\mathrm{div}_{x}\la \bm{b}\,h_{\e}\ra=\frac{1}{\e^{3}}\mathscr{J}_{\re}(f_{\e},f_{\e})+\frac{2\kappa_{\re}}{\e^{2}}\la \tfrac{1}{2}|v|^{2}h_{\e}\ra\,.
\end{equation}
Notice that $\{\vartheta_{\e}\}_{\e}$ is bounded in $L^{1}\left((0,T)\,;\, {\W^{m,2}_{x}(\T^{d})}\right)$
while, because $\bm{b}\,h_{\e}=\bm{b}\,\left(\mathbf{Id}-\bm{\pi}_{0}\right)h_{\e}$ we deduce from Proposition \ref{lem:equi}  {by Minkowski's integral inequality} that
$$\sup_{\e}\int_{0}^{T}\left\|\left(\mathrm{div}_{x}\la \tfrac{1}{\e}\bm{b}\,h_{\e}\ra\right)\right\|_{ {\W^{m-2,2}_{x}(\T^{d})}} \d t < \infty.$$
It is easy to see that the right-hand side of \eqref{eq:trick} is also bounded in $L^{1}\left((0,T)\,,\, {\W^{m,2}_{x}(\T^{d})}\right)$
so that
$\{\partial_{t}\vartheta_{\e}\}_{\e}$ is bounded in $L^{1}\left((0,T)\,; {\W^{m-2,2}_{x}(\T^{d})}\right).$
Using again \cite[Corollary 4]{simon} together with  Rellich-Kondrachov Theorem, we deduce as before that $\{\vartheta_{\e}\}_{\e}$ is relatively compact in $L^{1}\left((0,T)\,;\, {\W^{m-2,2}_{x}(\T^{d})}\right)$. Since we already know that $\vartheta_{\e}$ converges in the distributional sense to $\vartheta_{0}$ (see Lemma \ref{lem:mode}), we get the result of strong convergence. 
\end{proof}
\begin{nb} We will see later that the convergence of $\{\mathcal{P}\bm{u}_{\e}\}_{\e}$ and $\{\vartheta_{\e}\}_{\e}$ can actually be strenghten for well-prepared initial datum (see Proposition \ref{prop:ascoli}). \end{nb}

A first consequence of the above Lemma is the following which regards \eqref{eq:Avh}
\begin{lem}\label{prop:limitA} In the distributional sense, 
\begin{equation}\label{eq:Ah}\lim_{\e\to0^{+}}\mathcal{P}\mathrm{Div}_{x}\left(\la {\e}^{-1}\bm{A}\,h_{\e}\ra-\la \phi \,\Q_{1}\left(\bm{\pi}_{0}h_{\e},\bm{\pi}_{0}h_{\e}\right)\ra\right)=-{\bm \nu}\,\Delta_{x}u\end{equation}
where ${\bm \nu}$ is defined in Lemma~\ref{lem:phipsi}. 
 \end{lem} 
\begin{proof} When compared to the elastic case, $\mathbf{L}_{1}h_{\e}$ does not appear in \eqref{BE1}.  We add it, as well as the quadratic elastic Boltzmann operator, by force and rewrite the latter as
\begin{multline}\label{BE2}
\e\partial_{t}h_{\e}+\,v\cdot \nabla_{x}h_{\e} -\e^{-1}\mathbf{L}_{1}h_{\e} = \Q_{1}(h_{\e},h_{\e}) - \e^{-1}\kappa_{\re}\nabla_{v}\cdot (vh_{\e}) +\\\e^{-1}\left(\mathbf{L}_{\re}h_{\e}-\mathbf{L}_{1}h_{\e}\right)+\Q_{\re}(h_{\e},h_{\e})-\Q_{1}(h_{\e},h_{e}).
\end{multline}
We interpret the last three terms as a source term
\begin{equation} \label{eq:source}\bm{S}_{\e}:=\e^{-1}\left(\mathbf{L}_{\re}h_{\e}-\mathbf{L}_{1}h_{\e}\right)+\Q_{\re}(h_{\e},h_{\e})-\Q_{1}(h_{\e},h_{\e})-\e^{-1}\kappa_{\re}\nabla_{v}\cdot (vh_{\e}).\end{equation}
Then, multiplying \eqref{BE2} by $\phi$ and integrating over $\R^{d}$, we get using \eqref{eq:Avh} that, for any~$i,\,j=1,\ldots,d,$
\begin{multline}\label{eq:phiijh}
\e\partial_{t}\la \phi^{i,j} h_{\e}\ra+ \mathrm{div}_{x}\la v\,\phi^{i,j} h_{\e}\ra -\e^{-1}\la \phi^{i,j}\mathbf{L}_{1}h_{\e}\ra\\
=\la\,\phi^{i,j}\,\Q_{1}(h_{\e},h_{\e})\ra+ \la \phi^{i,j}\,\bm{S}_{\e}\ra\,.\end{multline}
{One writes
$$\Q_{1}(h_{\e},h_{\e})=\Q_{1}\left(\bm{\pi}_{0}h_{\e},\bm{\pi}_{0}h_{\e}\right) + \Q_{1}^{\bm{r}}(h_{\e},h_{\e})$$
so that \eqref{eq:phiijh} becomes
\begin{multline*}
\e\partial_{t}\la \phi^{i,j} h_{\e}\ra+ \mathrm{div}_{x}\la v\,\phi^{i,j} h_{\e}\ra -\e^{-1}\la \phi^{i,j}\mathbf{L}_{1}h_{\e}\ra \\
=\la \phi^{i,j} \,\Q_{1}\left(\bm{\pi}_{0}h_{\e},\bm{\pi}_{0}h_{\e}\right)\ra\,+\,\la\,\phi^{i,j}\,\Q_{1}^{\bm{r}}(h_{\e},h_{\e})\ra + \la \phi^{i,j}\,\bm{S}_{\e}\ra\,.\end{multline*}
}
According to Lemma \ref{lem:mode} we have that
\begin{equation*}\begin{split}
\e\partial_{t}\la \phi^{i,j}\,h_{\e}\ra &\longrightarrow 0, \qquad \mathrm{div}_{x}\la v\,\phi^{i,j}\,h_{\e}\ra\longrightarrow \mathrm{div}_{x}\la v\,\phi^{i,j}\,\bm{h}\ra\,,\\
\la \phi^{i,j} \Q_{1}^{\bm{r}}(h_{\e},h_{\e})\ra &\longrightarrow 0\,, \qquad
\la\,\phi^{i,j}\,\bm{S}_{\e}\ra \longrightarrow 0\,,
\end{split}
\end{equation*}
where the limits are all meant in the distributional sense and where the last limit is deduced from the strong convergence of $\bm{S}_{\e}$ to $0$ in $L^{1}((0,T); {L^{1}_{v}L^{2}_{x}(\m_{q-1}))}$ (see Lemma~\ref{lem:source}).

From Lemma \ref{lem:phinu} in Appendix \ref{sec:hydro1}, one has
$$\la v_{\ell}\,\phi^{i,j}\,\bm{h}\ra =\begin{cases} \qquad\quad {\bm \nu}\,u_{j}  \quad &\text{if}\;\, i \neq j, \;\, \ell=i,\\
\qquad\quad{\bm \nu}\,u_{i}  \qquad &\text{if}\;\, i \neq j, \;\, \ell=j,\\
- \frac{2}{d}{\bm \nu}\,u_{\ell}+2{\bm \nu}\,u_{i}\delta_{i\ell} \quad &\text{if}\;\, i=j,\\
\qquad\quad 0 \quad &\text{else}.
\end{cases}$$
Therefore, using the incompressibility condition, 
$$\mathrm{div}_{x}\la v\,\phi^{i,j}\,\bm{h}\ra={\bm \nu}\left(\partial_{x_{j}}u_{i}+\partial_{x_{i}}u_{j}\right).$$
We deduce that 
\begin{equation*}\label{eQ:limit}
\lim_{\e\to0^{+}}\left(\e^{-1}\la \phi^{i,j}\,\mathbf{L}_{1}h_{\e}\ra {+}\la \phi^{i,j} \,\Q_{1}\left(\bm{\pi}_{0}h_{\e},\bm{\pi}_{0}h_{\e}\right)\ra\right)={\bm \nu}(\partial_{x_{j}}u_{i}+\partial_{x_{i}}u_{j}),
\end{equation*}
in the distributional sense. Applying the $\mathrm{Div}_{x}$ operator one deduces that, in $\mathscr{D}'_{t,x}$
$$\lim_{\e\to0^{+}}\mathrm{Div}_{x}^{i}\left(\e^{-1}\la \phi\,\mathbf{L}_{1}h_{\e}\ra{+}\la \phi \,\Q_{1}\left(\bm{\pi}_{0}h_{\e},\bm{\pi}_{0}h_{\e}\right)\ra\right)={\bm \nu} \Delta_{x}u_{i}\,,$$
where we use the incompressibility condition to deduce that $\mathrm{Div}_{x}^{i}\left(\partial_{x_{j}}u_{i}+\partial_{x_{i}}u_{j}\right)=\Delta_{x}u_{i}$. This proves the result.\end{proof}
In the same spirit, we have the following which now regards \eqref{eq:bvh}.
\begin{lem}\label{lem:temvE} In the distributional sense, 
\begin{equation}\label{eq:limpsiL1}
\lim_{\e\to 0^{+}}\left(\e^{-1}\mathrm{div}_{x}\la \bm{b}\,h_{\e}\ra + \mathrm{div}_{x} \la \psi\,\Q_{1}(\bm{\pi}_{0}h_{\e},\bm{\pi}_{0}h_{\e}\ra\right)=-\frac{d+2}{2}\,\gamma\,\Delta_{x}\vE\,.
\end{equation}
\end{lem}
\begin{proof} We recall that 
$$\frac{1}{\e}\mathrm{div}_{x}\la \bm{b}\,h_{\e}\ra=-\frac{1}{\e}\mathrm{div}_{x}\la \mathbf{L}_{1}(h_{\e})\psi\ra.$$
Multiply \eqref{BE2} with $\psi_{i}$ (recall that $\psi$ is defined by \eqref{eq:L1chi}). As previously, it holds that 
$$
\e\partial_{t}\la \psi_{i} h_{\e}\ra+ \mathrm{div}_{x}\la v\,\psi_{i} h_{\e}\ra -\e^{-1}\la \psi_{i}\mathbf{L}_{1}h_{\e}\ra=\la\,\psi_{i}\,\Q_{1}(h_{\e},h_{\e})\ra+ \la \psi_{i}\,\bm{S}_{\e}\ra\,,$$
and 
$$\e\partial_{t}\la \psi_{i} h_{\e}\ra \longrightarrow 0, \qquad \la \psi_{i}\,\bm{S}_{\e}\ra \longrightarrow 0\,,$$
in the distributional sense.  Splitting again $\Q_{1}(h_{\e},h_{\e})=\Q_{1}^{\bm{r}}(h_{\e},h_{\e})+\Q_{1}\left(\bm{\pi}_{0}h_{\e},\bm{\pi}_{0}h_{\e}\right)$, one has $\la\,\psi_{i}\,\Q_{1}^{\bm{r}}(h_{\e},h_{\e})\ra$ converges to $0$ in $\mathscr{D}'_{t,x}$ so that, in the distributional sense, it follows that
$$\lim_{\e\to 0^{+}}\left(\e^{-1}\la \psi_{i}\mathbf{L}_{1}h_{\e}\ra{+}\la \psi_{i}\Q_{1}(\bm{\pi}_{0}h_{\e},\bm{\pi}_{0}h_{\e}\ra \right)=\mathrm{div}_{x}\la v\,\psi_{i}\bm{h}\ra=\frac{d+2}{2}\,\gamma\,\partial_{x_{i}}\vE\,$$
thanks to Lemma \ref{lem:psiQ} in Appendix \ref{sec:hydro1}. This gives the result.\end{proof}

\subsection{Convergence of the nonlinear terms}

To determine the distributional limit of \eqref{eq:Avh} and \eqref{eq:bvh}, we ``only'' need now to explicit the limit of 
$$\mathcal{P}\mathrm{Div}_{x}\la \phi\,\Q_{1}(\bm{\pi}_{0}h_{\e},\bm{\pi}_{0}h_{\e})\ra \quad \text{ and } \quad \mathrm{div}_{x}\la \psi\Q_{1}(\bm{\pi}_{0}h_{\e},\bm{\pi}_{0}h_{\e})\ra$$ respectively. Writing
$$\bm{\pi}_{0}h_{\e}=\left(\varrho_{\e}(t,x)+\bm{u}_{\e}(t,x) \cdot v + \frac{1}{2}\vE_{\e}(t,x)\left(|v|^{2}-d\en_{1}\right)\right)\M(v)$$
we first observe that, according to Lemma \ref{lem:Q1hh} and Lemma \ref{lem:psiQ} in Appendix \ref{sec:hydro1},
$$\la \phi \,\Q_{1}\left(\bm{\pi}_{0}h_{\e},\bm{\pi}_{0}h_{\e}\right)\ra=\en_{1}^{2}\left[\bm{u}_{\e}\otimes \bm{u}_{\e} -\frac{2}{d}|\bm{u}_{\e}|^{2}\mathbf{Id}\right]$$
and 
$$ \la \psi\,\Q_{1}(\bm{\pi}_{0}h_{\e},\bm{\pi}_{0}h_{\e})\ra=\frac{d+2}{2}\en_{1}^{3}\,(\vE_{\e}\,\bm{u}_{\e})\,.$$
Therefore,
$$\mathcal{P}\mathrm{Div}_{x}\la \phi\,\Q_{1}\left(\bm{\pi}_{0}h_{\e},\bm{\pi}_{0}h_{\e}\right)\ra=\en_{1}^{2}\mathcal{P}\mathrm{Div}_{x}\left(\bm{u}_{\e} \otimes \bm{u}_{\e}\right)$$
since $\mathrm{Div}_{x}\left(|\bm{u}_{\e}|^{2}\mathbf{Id}\right)$ is a gradient term and
$$\mathrm{div}_{x}\la \psi\,\Q_{1}(\bm{\pi}_{0}h_{\e},\bm{\pi}_{0}h_{\e})\ra=\frac{d+2}{2}\en_{1}^{3}\,\mathrm{div}_{x}\left(\vE_{\e}\,\bm{u}_{\e}\right).$$
One has the following whose proof is adapted from \cite[Corollary~5.7]{golseSR}.
\begin{lem}\label{lem:convect}
In the distributional sense (in $\mathscr{D}'_{t,x}$), one has
$$\lim_{\e\to 0^{+}}\mathcal{P}\mathrm{Div}_{x}\la \phi\,\Q_{1}\left(\bm{\pi}_{0}h_{\e},\bm{\pi}_{0}h_{\e}\right)\ra=\en_{1}^{2} \mathcal{P}\mathrm{Div}_{x}(u\otimes u)$$
and
$$\lim_{\e \to 0^{+}}\mathrm{div}_{x}\la \psi\,\Q_{1}(\bm{\pi}_{0}h_{\e},\bm{\pi}_{0}h_{\e})\ra=\frac{d+2}{2}\en_{1}^{3}u \cdot \nabla_{x}\vE.$$
In particular
\begin{equation}\label{eq:limphiL1-2}\lim_{\e \to 0^{+}}\mathcal{P}\mathrm{Div}_{x}\la {\e}^{-1}\bm{A}\,h_{\e}\ra=-\bm{\nu}\Delta_{x}u + \en_{1}^{2}\mathcal{P}\mathrm{Div}_{x}(u \otimes u) \qquad \text{ in } \mathscr{D}'_{t,x}\end{equation}
while
\begin{equation}\label{eq:limpsiL1-2}
\lim_{\e\to0^{+}}\mathrm{div}_{x}\la \e^{-1}\bm{b}\,h_{\e}\ra=-\frac{d+2}{2}\left(\gamma\,\Delta_{x}\vE - \en_{1}^{3}u \cdot \nabla_{x}\vE\right) \qquad \text{ in } \mathscr{D}'_{t,x}.\end{equation}
\end{lem}
\begin{proof} 
We write $\bm{u}_{\e}=\mathcal{P}\bm{u}_{\e}+(\mathbf{Id}-\mathcal{P})\bm{u}_{\e}$. Due to the strong convergence of $\mathcal{P}\bm{u}_{\e}$ towards $u$ in $L^{1}\left((0,T)\,;\, {\W^{m-2,2}_{x}(\T^{d})}\right)$ (see Lemma \ref{prop:strongue}) and the weak convergence of $\bm{u}_{\e}$ (see Lemma~\ref{lem:mode}), we see that
$$\mathcal{P}\mathrm{Div}_{x}\left(\bm{u}_{\e} \otimes \bm{u}_{\e}- (\mathbf{Id}-\mathcal{P})\bm{u}_{\e} \otimes (\mathbf{Id}-\mathcal{P})\bm{u}_{\e}\right)
\longrightarrow  \mathcal{P}\mathrm{Div}_{x}\left(u \otimes u\right) \text{ in } \mathscr{D}'_{t,x}.$$
So, to prove the first part of the Lemma, we only need to prove that
\begin{equation}\label{eq:limNaUe}
\mathcal{P}\mathrm{Div}_{x}\left( (\mathbf{Id}-\mathcal{P})\bm{u}_{\e} \otimes (\mathbf{Id}-\mathcal{P})\bm{u}_{\e}\right) \longrightarrow 0\end{equation}
in $\mathscr{D}'_{t,x}.$ Moreover, as in \cite[Corollary 5.7]{golseSR}, we set
$$\bm{\beta}_{\e}:=\frac{1}{d\,\en_{1}}\la |v|^{2}h_{\e}\ra=\varrho_{\e}+\en_{1}\vE_{\e}$$
which is such that $\vE_{\e}=\frac{2}{(d+2)\en_{1}}\left(\bm{\beta}_{\e}+\frac{1}{\en_{1}}\vartheta_{\e}\right)$
and
\begin{multline*}
\mathrm{div}_{x}(\theta_{\e}\bm{u}_{\e})=\frac{2}{(d+2)\en_{1}}\left(\mathrm{div}_{x}\left(\bm{\beta}_{\e}\bm{u}_{\e}+\frac{1}{\en_{1}} \bm{u}_{\e}\vartheta_{\e}\right)\right)\\
=\frac{2}{(d+2)\en_{1}}\mathrm{div}_{x}\left(\bm{\beta}_{\e}\left(\mathbf{Id}-\mathcal{P}\right)\bm{u}_{\e}\right)
+\frac{2}{(d+2)\en_{1}}\left[\mathrm{div}_{x}\left(\bm{\beta}_{\e}\mathcal{P}\bm{u}_{\e}+\frac{1}{\en_{1}}\bm{u}_{\e}\vartheta_{\e}\right)\right].
\end{multline*}
Therefore, using the strong convergence of $\vartheta_{\e}$ towards $\vartheta_{0}$ in $L^{1}\left((0,T)\,; {\W^{m-2,2}_{x}(\T^{d})}\right)$ given by Lemma \ref{prop:strongue} together with the weak convergence of $\bm{u}_{\e}$ to $u$ from Lemma~\ref{lem:mode}, we get
$$\frac{2}{(d+2)\en_{1}^{2}}\mathrm{div}_{x}(\bm{u}_{\e}\vartheta_{\e}) \longrightarrow \frac{2}{(d+2)\en_{1}^{2}}\mathrm{div}_{x}(u\,\vartheta_{0}) \qquad \text{in } \mathscr{D}'_{t,x}$$
whereas the strong convergence of $\mathcal{P}\bm{u}_{\e}$ to $u$ with the weak convergence of $\bm{\beta}_{\e}$ towards $\varrho+\en_{1}\vE$ we get
$$\mathrm{div}_{x}(\bm{\beta}_{\e}\mathcal{P}\bm{u}_{\e}) \longrightarrow \mathrm{div}_{x}\left(u\left(\varrho+\en_{1}\vE\right)\right)=0 \qquad \text{in } \mathscr{D}'_{t,x}$$
where we used both the incompressiblity condition \eqref{eq:incomp} together with   Boussinesq relation \eqref{eq:boussi}. Notice that, thanks to \eqref{eq:incomp}, it holds
$$\frac{2}{(d+2)\en_{1}^{2}}\mathrm{div}_{x}(u\vartheta_{0})=\frac{2}{(d+2)\en_{1}^{2}}u\cdot \nabla_{x}\vartheta_{0}=u \cdot \nabla_{x}\vE$$
where we used the expression of $\vartheta_{0}$ together with Bousinesq relation \eqref{eq:boussi}. This shows that
$$\mathrm{div}_{x}(\theta_{\e}\bm{u}_{\e})-\frac{2}{(d+2)\en_{1}}\mathrm{div}_{x}\left(\bm{\beta}_{\e}\left(\mathbf{Id}-\mathcal{P}\right)\bm{u}_{\e}\right) \longrightarrow u \cdot\nabla_{x}\vE \qquad \text{ in } \mathscr{D}'_{t,x}$$
and, to get the second part of the result, we need to prove that
\begin{equation}\label{eq:limBeUe}
\mathrm{div}_{x}\left(\bm{\beta}_{\e}\left(\mathbf{Id}-\mathcal{P}\right)\bm{u}_{\e}\right) \longrightarrow 0 \qquad \text{in } \mathscr{D}'_{t,x}.
\end{equation}
Let us now focus on the proof of \eqref{eq:limNaUe} and \eqref{eq:limBeUe}. One observes that, Equation \eqref{bulk1} reads
\begin{equation}\label{eq:bulkP}
\e\,\partial_{t}\bm{u}_{\e} + \nabla_{x}\bm{\beta}_{\e}=\frac{\kappa_{\re}}{\e}\bm{u}_{\e} - \en_{1}^{-1}\mathrm{Div}_{x}\la \bm{A}\,h_{\e}\ra\end{equation}
whereas \eqref{energy1} can be reformulated as 
\begin{equation}\label{energy21}
\e\partial_{t}\bm{\beta}_{\e}+\mathrm{div}_{x}\,\la \tfrac{1}{d\en_{1}}|v|^{2}v\,h_{\e}\ra\,=\frac{2}{d\en_{1}\e^{2}}\mathscr{J}_{\re}(f_{\e},f_{\e})+\frac{2\kappa_{\re}}{\e}\bm{\beta}_{\e}\,
\end{equation}
where we check easily that 
\begin{equation*}\begin{split}\mathrm{div}_{x}\la \tfrac{1}{d\en_{1}}|v|^{2}v\,h_{\e}\ra&=\frac{2}{d\en_{1}}\mathrm{div}_{x}\la \bm{b}\,h_{\e}\ra + \frac{d+2}{d}\en_{1}\mathrm{div}_{x}\bm{u}_{\e}\\
&=\frac{2}{d\en_{1}}\mathrm{div}_{x}\la \bm{b}\,h_{\e}\ra+ \frac{d+2}{d}\en_1\mathrm{div}_{x}\left(\mathbf{Id}-\mathcal{P}\right)\bm{u}_{\e}.
\end{split}\end{equation*}
{Recall that from Theorem~\ref{theo:main-cauc1}, $h_\e \in L^\infty\left((0,T);\mathcal{E}\right)$ so that  by Minkowski's integral inequality, $\bm{\beta}_\e \in L^\infty\left((0,T); {\W^{m,2}_{x}(\T^{d})}\right)$ and using \cite[Proposition 1.6, p. 33]{majda}), we can write
$$(\mathbf{Id}-\mathcal{P})\bm{u}_{\e}=\nabla_{x}\bm{U}_{\e}$$
with $\bm{U}_{\e} \in {L^{\infty}}\left((0,T);\left( {\W^{m-1,2}_{x}(\T^{d})}\right)^{d}\right)$.} After applying $(\mathbf{Id}-\mathcal{P})$ to \eqref{eq:bulkP} and reformulating~\eqref{energy21}, we obtain that $\bm{U}_{\e}$ and $\bm{\beta}_{\e}$ satisfy
\begin{equation}\label{eq:system}\begin{cases}
\e\partial_{t}\nabla_{x}\bm{U}_{\e} + \nabla_{x}\bm{\beta}_{\e}=\bm{F}_{\e}\\
\\
\e \partial_{t}\bm{\beta}_{\e} + \frac{d+2}{d}\en_{1}\Delta_{x}\bm{U}_{\e}=\bm{G}_{\e}
\end{cases}\end{equation}
with 
\begin{multline*}
\bm{F}_{\e}:=\frac{\kappa_{\re}}{\e}\nabla_{x}\bm{U}_{\e}-\en_{1}^{-1}(\mathbf{Id}-\mathcal{P})\mathrm{Div}_{x}\la \bm{A}\,h_{\e}\ra\\
\bm{G}_{\e}:=-\frac{2}{d\en_{1}}\mathrm{div}_{x}\la \bm{b}\,h_{\e}\ra+\frac{2}{d\en_{1}\e^{2}}\mathscr{J}_{\re}(f_{\e},f_{\e}) + \frac{2\kappa_{\re}}{\e}\bm{\beta}_{\e}.\end{multline*} 
It is easy to see that 
$$
\|\bm{F}_{\e}\|_{L^{1}((0,T)\,;\, {\W^{m-2,2}_{x}(\T^{d})})} \lesssim \e, \qquad \text{ and } \qquad \|\bm{G}_{\e}\|_{L^{1}((0,T)\,;\, {\W^{m-2,2}_{x}(\T^{d})})} \lesssim \e.
$$
so that both $\bm{F}_{\e}$ and $\bm{G}_{\e}$ converge \emph{strongly} to $0$ in $L^{1}((0,T);L^2_{x}(\T^{d}))$ and 
$$\bm{U}_\e \in L^\infty((0,T);(\W^{1,2}_x(\T^d))^{d}), \qquad {\bm \beta}_\e \in L^\infty((0,T);L^2_x(\T^d)).$$ Then, according to the \emph{compensated compactness} argument of {\cite{lions-masm}} recalled in  Proposition \ref{prop:LM} in Appendix \ref{sec:hydro1}, we deduce that~\eqref{eq:limNaUe} and \eqref{eq:limBeUe} hold true and this achieves the proof. The proofs of \eqref{eq:limphiL1-2} and~\eqref{eq:limpsiL1-2} follow then from an application of Lemmas \ref{prop:limitA} and \ref{lem:temvE}.
\end{proof}
 
Coming back to the system of equations \eqref{moments} and with the preliminary results of Section \ref{sec:hydro}, we get the following
{where we wrote $\mathcal{P}\mathrm{Div}_{x}(u \otimes u)=\mathrm{Div}_{x}(u \otimes u) + \en_1^{-1}\nabla_{x} p$, see \cite[Proposition 1.6]{majda}}.

\begin{prop}\label{lem:temp}
The limit velocity $u(t,x)$ in \eqref{eq:hlim} satisfies
\begin{equation}\label{eq:bulk2}
\partial_{t}u-{\frac{{\bm \nu}}{\en_1}}\,\Delta_{x}u + {\en_{1}}\mathrm{Div}_{x}\left(u\otimes u\right)+\nabla_{x}p=\lambda_0u
\end{equation}
while the limit temperature $\vE(t,x)$ in \eqref{eq:hlim} satisfies 
\begin{equation*}
\partial_{t}\vE 
-\frac{\gamma}{\en_{1}^{2}}\,\Delta_{x}\vE{+}\en_{1}\,u\cdot \nabla_{x}\vE
=\frac{2}{(d+2)\en_{1}^{2}}\mathcal{J}_{0}+\frac{2d\lambda_0}{d+2}E(t)+\frac{2}{d+2}\frac{\d}{\d t}E(t)\,,\end{equation*}
where
$$E(t)= \int_{\T^{d}}\vE(t,x)\d x, \qquad t \geq0\,.$$
\end{prop} 
Notice that, due to \eqref{eq:incomp}, $\mathrm{Div}_{x}(u\otimes u)=\left(u \cdot \nabla_{x}\right)u$ and \eqref{eq:bulk2} is nothing but a \emph{damped} Navier-Stokes equation associated to a divergence-free source term given by $\lambda_0u$.
\begin{proof} The proof of \eqref{eq:bulk2} is a straightforward consequence of the previous limit. To obtain investigate the evolution of $\vE$, we recall that $\vartheta_{\e}$ satisfies \eqref{eq:trick}. 
We notice that
$$\frac{1}{\e^{3}}\mathscr{J}_{\re}(f_{\e},f_{\e})+\frac{2\kappa_{\re}}{\e^{2}}\la \tfrac{1}{2}|v|^{2}h_{\e}\ra \longrightarrow \mathcal{J}_{0}+d\en_{1}\lambda_0\left(\varrho+\en_{1}\vE\right)\,,$$
whereas
$$\vartheta_{\e}\longrightarrow \la \tfrac{1}{2}(|v|^{2}-(d+2)\en_{1})\bm{h}\ra=\frac{d\en_{1}}{2}\left(\varrho+\en_{1}\vE\right)-\frac{d+2}{2}\en_{1}\varrho\,,$$
where the convergence is meant in {$\mathscr{D}'_{t,x}$}.  We deduce from \eqref{eq:limpsiL1-2}, performing the distributional limit of \eqref{eq:trick}, that
\begin{multline}\label{eq:energy2}
\frac{d\en_{1}}{2}\partial_{t}\left(\varrho+\en_{1}\vE\right)-\frac{d+2}{2}\en_{1}\partial_{t}\varrho
-\frac{d+2}{2}\gamma\,\Delta_{x}\vE{+}\frac{d+2}{2}\en_{1}^{3}\,u\cdot \nabla_{x}\vE\\
=\mathcal{J}_{0}+d\en_{1}\lambda_0\left(\varrho+\en_{1}\vE\right).
\end{multline}
Using the strengthened Boussinesq relation \eqref{eq:boussi2}, we see that 
$$\partial_{t}\left(\varrho+\en_{1}\vE\right)=\en_{1}\frac{\d}{\d t}E(t), \quad \text{and} \quad \partial_{t}\varrho=-\en_{1}\left(\partial_{t}\vE-\frac{\d}{\d t}E(t)\right)\,,$$ 
and get the result.
\end{proof}
\begin{prop}\label{prop:temp} For any $t \geq 0$, it follows that
$$E(t)= \int_{\T^{d}}\vE(t,x)\d x=0\,,$$
consequently, the limiting temperature $\vE(t,x)$ in \eqref{eq:hlim} satisfies
\begin{equation}\label{eq:energy3}
\partial_{t}\,\vE-\frac{\gamma}{\en_{1}^{2}}\,\Delta_{x}\vE{+}\en_{1}\,u\cdot \nabla_{x}\vE=\frac{\lambda_{0}\,\bar{c}}{2(d+2)}\sqrt{\en_{1}}\,\vE. 
\end{equation}
Moreover, the strong Boussinesq relation
\begin{equation}\label{eq:boussi3}
\varrho(t,x)+\en_{1}\vE(t,x)=0\,, \qquad x \in \T^{d}\,,
\end{equation}
holds true.
\end{prop}
\begin{proof} To capture the evolution of the temperature $E(t)$, we average equation \eqref{eq:energy2} over~$\T^{d}$ and using the incompressibility condition \eqref{eq:incomp} we deduce get that
\begin{equation}\label{eq:dEt}
\dfrac{\d}{\d t}E(t)=\frac{2}{d\en_{1}^{2}}\int_{\T^{d}}\mathcal{J}_{0}(t,x)\d x+2\lambda_0E(t).\end{equation}
And, from Lemma \ref{lemJ0}, it holds that
$$\dfrac{\d}{\d t}E(t)=\bar{c}_{0}\,E(t), \qquad \bar{c}_{0}:=2\lambda_0-\frac{3}{2d}\lambda_{0}\bar{c}\,\sqrt{\en_{1}}\,,$$
so that,
$$E(t)=E(0)\exp(\bar{c}_{0}t)\,, \qquad t \geq 0.$$
Now, return to the original equation \eqref{BE} and recall that the solution $f_{\e}(t,x,v)$ is given by
\begin{equation*}
f_{\e}(t) = G_{\re} + \e\,h_{\e}(t) = \M + \big(G_{\re} - \M\big) + \e\,h_{\e}(t)\,,\qquad t\geq0\,,
\end{equation*}
where $\M$ has the same \textit{global} mass, momentum and energy as the initial datum $f_{\e}(0)=F^{\e}_{\mathrm{in}}$, independent of $\e>0$.  For any test-function $\phi=\phi(x,v)$ we get that
\begin{equation*}
\e^{-1}\int_{\T^{d}}\la \big(f_{\e}(t) - \M)\,\phi \ra  \d x - \e^{-1}\int_{\T^{d}} \la \big(G_{\re} - \M\big)\,\phi \ra \d x = \int_{\T^{d}} \la h_{\e}(t)\,\phi \ra \d x\,,\quad t\geq0\,.
\end{equation*}
Using this equality for $\phi(x,v)=\tfrac{1}{2}|v|^{2}$ and $t=0$ one is led to
\begin{equation*}
-\e^{-1}\int_{\T^{d}} \la \tfrac{1}{2}\big(G_{\re} -\M\big)\,|v|^{2} \ra \d x = \int_{\T^{d}} \la \tfrac{1}{2}h_{\e}(0)\,|v|^{2}\ra \d x\,.
\end{equation*}
We recall that $\| G_{\re} - \M\|_{L^{1}_v(\m_2)}\leq C(1-\re)\lesssim \e^{2}$, consequently
\begin{equation*}
E(0) =  \int_{\T^d}\theta(0,x)\d x= \lim_{\e\rightarrow0} \int_{\T^{d}} \la h_{\e}(0)\tfrac{1}{2}|v|^{2}\ra \d x =0\,.
\end{equation*}
Therefore, $E(t)\equiv0$ for all $t>0$.  This observation and \eqref{eq:energy3} lead us to the equation for the energy and Boussinesq relation \eqref{eq:boussi3}.
\end{proof}
\subsection{About the initial conditions}
Before going into the proof of Theorem~\ref{theo:CVNS-int} and handle the problem of initial datum of our limit system, we begin by proving that our limits $u$ and~$\theta$ in~\eqref{eq:hlim} are actually continuous on~$(0,T)$.
{\begin{lem} \label{lem:cont}Consider the sequences $\{u_\e\}_{\e}$ and~$\{\vartheta_\e\}_\e$ defined in Lemma~\ref{prop:strongue}. 
The time-depending mappings 
$$
t \in [0,T] \longmapsto   \left\|\vartheta_{\e}(t)\right\|_{ {\W^{m-2,2}_{x}(\T^{d})}}
\quad \text{and} \quad 
t \in [0,T] \longmapsto  \left\|u_{\e}(t)\right\|_{ {\W^{m-2,2}_{x}(\T^{d})}}
$$
are H\"older continuous uniformly in $\e$. 
As a consequence, the limiting mass $\varrho$, velocity $u$ and temperature $\vE$ in \eqref{eq:hlim} are continuous on $(0,T)$.
\end{lem}} 
\begin{proof}
Recall that we set 
$$\vartheta_{\e}(t,x)=\la \tfrac{1}{2}\big( |v|^{2}-(d+2)\en_{1} \big)h_{\e}\ra.$$ 
For any test-function $\varphi=\varphi(x) \in \mathcal{C}_{c}^{\infty}(\T^{d})$ and multi-index $\beta$ with $|\beta| \leq m-2$, multiplying~\eqref{eq:trick} with $\partial_{x}^{\beta}\varphi$ and integrating in time and space, one deduces that for any~$0 \leq t_1\leq t_2$, 
\begin{multline}\label{eq:equivartheta}
\int_{\T^{d}}\left[\partial_{x}^{\beta}\vartheta_{\e}(t_{2},x)-\partial_{x}^{\beta}\vartheta_{\e}(t_{1},x)\right]\varphi(x)\d x=
\int_{t_{1}}^{t_{2}}\d t\int_{\T^{d}}\mathrm{div}_{x}\la \e^{-1}\bm{b}\,\partial_{x}^{\beta}h_{\e}(t)\ra \varphi(x)\d x\\
+\int_{t_{1}}^{t_{2}}\d t\int_{\T^{d}}\e^{-3}\partial_{x}^{\beta}\mathscr{J}_{\re}(f_{\e},f_{\e})\varphi(x)\d x\\
+\frac{2\kappa_{\re}}{\e^{2}}\int_{t_{1}}^{t_{2}}\d t \int_{\T^{d}}\la \tfrac{1}{2}|v|^{2}\partial_{x}^{\beta}h_{\e}\ra \varphi(x)\d x.
\end{multline}
 {Notice that
$$\left|\int_{\T^{d}}\la \tfrac{1}{2}|v|^{2}\partial_{x}^{\beta}h_{\e}\ra \varphi(x)\d x\right| \leq \|\varphi\|_{L^{2}_{x}}\,\left\|\la \tfrac{1}{2}|v|^{2}\partial_{x}^{\beta}h_{\e}\ra\right\|_{L^{2}_{x}} \leq \|\varphi\|_{L^{2}_{x}}\left\|\tfrac{1}{2}|v|^{2}\partial_{x}^{\beta}h_{\e} \right\|_{L^{1}_{v}L^{2}_{x}}$$
thanks to Minkowski's integral inequality.}
Clearly, since $\e^{-2}\kappa_{\re} \to \lambda_0$, there is $C>0$ such that
\begin{equation*}\begin{split}
\frac{2\kappa_{\re}}{\e^{2}}\left|\int_{t_{1}}^{t_{2}}\d t\int_{\T^{d}}\la \tfrac{1}{2}|v|^{2}\partial_{x}^{\beta}h_{\e}\ra \varphi(x)\d x\right| &\leq C\|\varphi\|_{L^{2}}\int_{t_{1}}^{t_{2}}\|h_{\e}(t)\|_{ {L^{1}_{v}\W^{m-1,2}_{x}(\m_{2})}}\d t \\
&\leq C\sqrt{\mathcal{K}_{0}}(t_{2}-t_{1})\end{split}\end{equation*}
from the general estimate in Theorem \ref{theo:main-cauc1}. In the same way, since
$$\partial_{x}^{\beta}\mathscr{J}_{\re}(f_{\e},f_{\e})=\mathscr{J}_{\re}(\partial_{x}^{\beta}f_{\e},f_{\e})+\mathscr{J}_{\re}(f_{\e},\partial_{x}^{\beta}f_{\e})\,,$$ 
with $f_{\e}=G_{\re}+\e\,h_{\e}$, one deduces again from Theorem \ref{theo:main-cauc1} that 
\begin{multline*}
\left|\int_{t_{1}}^{t_{2}}\d t\int_{\T^{d}}\e^{-3}\partial_{x}^{\beta}\mathscr{J}_{\re}(f_{\e},f_{\e})\varphi(x)\d x\right|\\
\leq C\|\varphi\|_{L^{2}_{x}}\,\int_{t_{1}}^{t_{2}}\|h_{\e}(t)\|_{ {L^{1}_{v}\W^{m-1,2}_{x}(\m_{3})}}\left(1+\|h_{\e}(t)\|_{ {L^{1}_{v}\W^{m-1,2}_{x}(\m_{3})}}\right)\d t \leq C\sqrt{\mathcal{K}_{0}}(t_{2}-t_{1}).
\end{multline*}
Moreover, noticing that  $\la \bm{b}h_{\e}(t)\ra=\la \bm{b}(\mathbf{Id}-\bm{\pi}_{0})h_{\e}(t)\ra$ for any $t \geq0$, one deduces easily from Proposition \ref{lem:equi} that
$${\left|\int_{t_{1}}^{t_{2}}\e^{-1}\mathrm{div}_{x}\la \bm{b}\,h_{\e}(t) \ra\d t\right| \leq C\,\sqrt{t_{2}-t_{1}}}$$
for any $0 \leq t_{2}-t_{1} \leq 1.$ Since $\partial_{x}^{\beta}$ commutes with $\bm{\pi}_{0}$ we deduce easily that there is $C >0$ independent of $\e$ such that  {for any $0 \leq \beta \leq m-2$,}
\begin{equation}\label{eq:equicon2}
 {\left|\int_{t_{1}}^{t_{2}}\e^{-1}\mathrm{div}_{x}\la \bm{b}\,\partial^\beta_x h_{\e}(t) \ra\d t\right| \leq C\,\sqrt{t_{2}-t_{1}}}
\end{equation}
for any $0 \leq t_{2}-t_{1} \leq1.$ 
We conclude with \eqref{eq:equivartheta} that
$$\left|\int_{\T^{d}}\left[\partial_{x}^{\beta}\vartheta_{\e}(t_{2},x)-\partial_{x}^{\beta}\vartheta_{\e}(t_{1},x)\right]\varphi(x)\d x \right|\leq C\,\|\varphi\|_{L^{2}_{x}}\sqrt{\mathcal{K}_{0}}\,\sqrt{t_{2}-t_{1}}$$
for some positive constant independent of $\e$ and $0\leq t_{2}-t_{1}\leq 1$. {Since $\mathscr{C}_{c}^{\infty}(\T^{d})$ is dense in $L^{2}(\T^{d})$, the previous estimate is true for any $\varphi \in L^{2}(\T^{d})$ and, taking the supremum over all $\varphi \in L^{2}(\T^{d})$}, we deduce that
\begin{equation}\label{eq:Ascoli0}\left\|\partial_{x}^{\beta}\vartheta_{\e}(t_{2})-\partial_{x}^{\beta}\vartheta_{\e}(t_{1})\right\|_{L^{2}_{x}} \leq C\,\sqrt{\mathcal{K}_{0}}\,\sqrt{t_{2}-t_{1}}\end{equation}
and, the time-depending mappings 
$t \in [0,T] \longmapsto  {\left\|\vartheta_{\e}(t)\right\|_{ {\W^{m-2,2}_{x}(\T^{d})}}}$
 are thus H\"older continuous uniformly in $\e$.
Recall also that $\vartheta_{\e}(t)$ converges in $L^{1}((0,T)\;;\, {\W^{m-2,2}_{x}(\T^{d})})$ towards $\vartheta_{0}(t)=\frac{d\en_{1}}{2}\left(\varrho+\en_{1}\vE\right)-\frac{d+2}{2}\en_{1}\varrho\,$ from Lemma~\ref{prop:strongue}. {As a consequence, there exists a subsequence $(\vartheta_{\e'})_{\e'}$ such that $\|\vartheta_{\e'}(t) - \vartheta_0(t)\|_{ {\W^{m-2,2}_{x}(\T^{d})}}$ converges towards $0$ for almost every $t \in [0,T]$. Using then the uniform in $\e$ H\"older continuity obtained above, we can deduce that $\vartheta_0$ is H\"older continuous on~$(0,T)$.}
Recalling that $E(0)=0$ according to Proposition~\ref{prop:temp}, the strong Boussinesq relation~\eqref{eq:boussi3} holds true and 
$$\vartheta_{0}(t)=\frac{d+2}{2}\en_{1}^{2}\vE(t)=-\frac{d+2}{2}\en_{1}\varrho(t)\,,$$
which gives the regularity of both $\varrho$ and $\vE$.

We recall that, setting
$$u_{\e}(t,x)=\frac{1}{\en_{1}}\exp\left(-t\frac{\kappa_{\re}}{\e^{2}}\right)\mathcal{P}\la vh_{\e}\ra,$$
we have that
\begin{equation*}\label{bulk1P}
\partial_{t} u_{\e} +\exp\left(-t\frac{\kappa_{\re}}{\e^{2}}\right)\mathcal{P}\left(\mathrm{Div}_{x}\la \tfrac{1}{\en_{1}\e}\bm{A}\,h_{\e} \ra\right) =0\,.
\end{equation*}
we multiply this identity by $\partial^{\beta}_{x}\varphi$ and integrate in both time and space to get
\begin{multline*}
\int_{\T^{d}}\left[\partial_{x}^{\beta}u_{\e}(t_{2},x)-\partial_{x}^{\beta}u_{\e}(t_{1},x)\right]\varphi(x)\d x\\
=\int_{t_{1}}^{t_{2}}\exp\left(- \frac{\kappa_{\re}}{\e^{2}}t\right)\d t \int_{\T^{d}}\mathcal{P}\left(\mathrm{Div}_{x}\la \tfrac{1}{\en_{1}\e}\bm{A}\,\partial_{x}^{\beta}h_{\e} \ra\right)\varphi(x)\d x.
\end{multline*}
Arguing as in the proof of \eqref{eq:equicon2}, we see that there is $C >0$ independent of $\e$ such that {for any $0 \leq |\beta| \leq m-2$,}
\begin{equation*}\label{eq:equicon}
 {\left|\int_{t_{1}}^{t_{2}}\e^{-1}\mathrm{Div}_{x}\la \bm{A}\,\partial^\beta_x h_{\e}(t) \ra\d t\right| \leq C\,\sqrt{t_{2}-t_{1}}}
\end{equation*}
for any $0 \leq t_{2}-t_{1} \leq1.$  This gives easily
$$\left|\int_{\T^{d}}\left[\partial_{x}^{\beta}u_{\e}(t_{2},x)-\partial_{x}^{\beta}u_{\e}(t_{1},x)\right]\varphi(x)\d x\right| \leq C\|\varphi\|_{ {L^{2}_{x}}}\sqrt{t_{2}-t_{1}}\sqrt{\mathcal{K}_{0}}\,,$$
from which, as before, the time-depending mapping
$t \in [0,T] \longmapsto  \left\|u_{\e}(t)\right\|_{ {\W^{m-2,2}_{x}(\T^{d})}}$
is H\"older continuous uniformly in $\e$.  We deduce the result of regularity on $u$ as previously done for $\varrho$ and $\theta$ noticing that the limit of $u_{\e}$ is $\exp(-t\lambda_{0})\mathcal{P}u=\exp(-t\lambda_{0})u$.
\end{proof}
Recall that, in Theorem \ref{theo:strong-conv}, the convergence of ${h}_{\e}$ to the solution $\bm{h}(t,x)$ given by \eqref{eq:hlim} is known to hold   only for a subsequence and, in particular, different subsequences could converge towards different initial datum and therefore $(\varrho,u,\vE)$ could be different solutions to the Navier-Stokes system. We aim here to \emph{prescribe} the initial datum by ensuring the convergence of the initial datum $h_{\mathrm{in}}^{\e}$ towards a \emph{single} possible limit.

\smallskip
\noindent
Recall that the initial datum for \eqref{BE2} is denoted by $h_{\mathrm{in}}^{\e}$.  We write
$h_{\mathrm{in}}^{\e}=\bm{\pi}_{0}h^{\e}_{\mathrm{in}}+(\mathbf{Id}-\bm{\pi}_{0})h_{\mathrm{in}}^{\e}$ and introduce the following assumption.
\begin{hyp}\label{hyp:initial}
Assume that there exists $${(\varrho_{0},u_{0},\vE_{0}) \in \W^{m,2}_{x}(\T^{d}) \times \left(\W^{m,2}_{x}(\T^{d})\right)^{d}\times \W^{m,2}_{x}(\T^{d})}\,,$$
such that
$$\lim_{\e\to0}\left\|\bm{\pi}_{0}h_{\mathrm{in}}^{\e}-h_{0}\right\|_{ {L^{1}_{v}\W^{m,2}_{x}(\m_{q})}}=0\,,$$
where
$$h_{0}(x,v)=\left(\varrho_{0}(x)+u_{0}(x)\cdot v + \tfrac{1}{2}\vE_{0}(x)(|v|^{2}-d\en_{1})\right)\M(v)\,.$$
\end{hyp}
Under this assumption we can prescribe the initial value of the solution $(\varrho,u,\vE)$ and strengthen the convergence. 
\begin{prop}\label{prop:ascoli} We define the initial data for $(\varrho,u,\vE)$ as
\begin{multline}\label{eq:iniNS}
u_{\mathrm{in}}=u(0):=\mathcal{P}u_{0}\,, \qquad \vE_{\mathrm{in}}=\vE(0)=\frac{d}{d+2}\theta_{0}-\frac{2}{(d+2)\en_{1}}\varrho_{0}\,,\\ \varrho_{\mathrm{in}}=\varrho(0):=-\en_{1}\vE_{\mathrm{in}}\,,
\end{multline}
where we recall that $\mathcal{P}u_{0}$ is the Leray projection on divergence-free vector fields.
Then, as a consequence, for any $T >0$, one has that
$$
\vartheta_{\e}(t)=\la \tfrac{1}{2}\big( |v|^{2}-(d+2)\en_{1} \big)h_{\e}\ra \longrightarrow \frac{d+2}{2}\en_{1}^{2}\vE\,, \quad
 \text{ in } \quad \mathcal{C}\left([0,T]\,,\, {\W^{m-2,2}_{x}(\T^{d})}\right)\,,$$
and 
$${\frac{1}{\en_{1}}\mathcal{P}}\la v\,h_{\e}(t)\ra \longrightarrow u\,, \qquad \text{ in } \quad \mathcal{C}\left([0,T]\,,\,\left( {\W^{m-2,2}_{x}(\T^{d})}\right)^{d}\right).$$
\end{prop}
\begin{proof} According to Lemma~\ref{lem:cont}, we already have that the family of time-depending mappings
\begin{equation}\label{eq:Ascoli}
\left\{t \in [0,T] \longmapsto  \left\|\vartheta_{\e}(t)\right\|_{ {\W^{m-2,2}_{x}(\T^{d})}}\right\}_{\e}\end{equation}
is equicontinuous. 
At time $t=0$ according to Assumption \ref{hyp:initial},
$$\vartheta_{\e}(0,x)=\la \tfrac{1}{2}\big( |v|^{2}-(d+2)\en_{1} \big)h_{\mathrm{in}}^{\e}\ra \longrightarrow \en_{1}\left[\frac{d\en_{1}}{2}\vE_{0}(x)-\varrho_{0}(x)\right]\,,$$
and, by definition of $\varrho(0,x),\vE(0,x)$, we get that
$$\lim_{\e\to0^{+}}\|\vartheta_{\e}(0,\cdot)-\vartheta_{0}(0,\cdot)\|_{ {\W^{m-2,2}_{x}(\T^{d})}}=0.$$
In particular, the family $\{\|\vartheta_{\e}(0)\|_{ {\W^{m-2,2}_{x}(\T^{d})}}\}_{\e}$ is bounded and, since the family~\eqref{eq:Ascoli} is uniformly in $\e$ H\"older continuous, for any $t \in [0,T]$, the family $\{\|\vartheta_{\e}(t)\|_{{\W^{m-2,2}_{x}(\T^{d})}}\}_{\e}$ is also bounded. Since it is also  equicontinuous, Arzel\`a-Ascoli Theorem implies that the convergence holds in $\mathcal{C}([0,T]\,;\, {\W^{m-2,2}_{x}(\T^{d})})$ and
$$\vartheta_{0} \in \mathcal{C}([0,T]\,;\, {\W^{m-2,2}_{x}(\T^{d})}).$$
{As in the proof of Lemma~\ref{lem:cont}, it implies the continuity on $[0,T]$ of both $\varrho$ and $\theta$.}

\smallskip
\noindent
We proceed in a similar way for the regularity of $u$. \end{proof}
\noindent
All the previous convergence results lead us to the fact the the limit
$$\bm{h}(t,x,v)=\left(\varrho(t,x)+u(t,x)\cdot v + \frac{1}{2}\vE(t,x)(|v|^{2}-d\en_{1})\right)\M(v)$$
is such that 
$$(\varrho,u,\vE) \in \mathcal{C}([0,T];\mathscr{W}_{m-2}) \cap L^{2}\left((0,T);\mathscr{W}_{m}\right)\,,$$
solve  the following \emph{incompressible Navier-Stokes-Fourier system} where the right-hand-side acts as a self-consistent forcing term
\begin{equation}
\begin{cases}
\partial_{t}u-{\frac{{\bm \nu}}{\en_1}}\,\Delta_{x}u + {\en_{1}}\,u\cdot \nabla_{x}\,u+\nabla_{x}p=\lambda_{0}u\,,\\[7pt]
\partial_{t}\,\vE-\frac{\gamma}{\en_{1}^{2}}\,\Delta_{x}\vE{+}\en_{1}\,u\cdot \nabla_{x}\vE=\dfrac{\lambda_{0}\,\bar{c}}{2(d+2)}\sqrt{\en_{1}}\,\vE\,,\\[8pt]
\mathrm{div}_{x}u=0, \qquad \varrho + \en_{1}\,\vE = 0\,,
\end{cases}
\end{equation}
subject to the initial datum $(\varrho_{\mathrm{in}},u_{\mathrm{in}},\vE_{\mathrm{in}})$.  This proves Theorem \ref{theo:CVNS-int} in full.

\subsection{About the original problem in the physical variables}\label{sec:orig} The above considerations allow us to get a quite precise description of the asymptotic behaviour for the original physical problem \eqref{Bol-e}. Indeed, recalling the relations \eqref{eq:ScalING} together with Theorem \ref{theo:CVNS-int} one has\begin{align*}
F_{\e}(t,x,v)&={V}_{\e}(t)^{d}f_{\e}\big(\tau_{\e}(t),x,V_{\e}(t)v\big)\\
&={V}_{\e}(t)^{d}\Big( G_{\re(\e)}(V_{\e}(t)v) + \e\,h_{\e}(\tau_{\e}(t),x,V_{\e}(t)v) \Big)\\
&={V}_{\e}(t)^{d}\Big( G_{\re(\e)}(V_{\e}(t)v) + \e\,\bm{h}(\tau_{\e}(t),x,V_{\e}(t)v)\Big) + \e\,e_{\e}(t,x,v) \,,
\end{align*}
where the error term $e_{\e}$ is given by $$e_{\e}(t,x,v)={V}_{\e}(t)^{d}\,\big( h_{\e}(\tau_{\e}(t),x,V_{\e}(t)v) - \bm{h}(\tau_{\e}(t),x,V_{\e}(t)v) \big)\,.$$  Under Assumption \ref{hyp:re}, a relevant phenomenon occurs when considering the purely dissipative case $\lambda_0>0$. In such a case, the term $e_{\e}(t,x,v)$ becomes a \emph{uniform in time} error term.  The reason is that,
when $\lambda_0>0$, the scaling $V_{\e}(t)$ increases up to infinity.  More precisely,
\begin{equation*}
V_{\e}(t) \approx \big( 1+\lambda_0\,t \big)\,,\qquad \e \ll 1.
\end{equation*}
Indeed, Lemma \ref{lemma-app-pp} guarantees that for any $a\in(0,1/2)$, up to an extraction of a subsequence if necessary,
\begin{equation}\label{error-term}
\big| \langle e_{\e}(t),|v|^{\kappa}\varphi\rangle\big| \leq C_{\varphi}\sqrt{\mathcal{K}_0}\,V_{\e}(t)^{-\kappa-a}\,,\qquad  {\varphi\in  \mathcal{C}^{1}_{v,b}L^{\infty}_{x}}\,,\;\, 0\leq\kappa\leq q-1\,,
\end{equation}
where we denoted by $\mathcal{C}^{1}_{v,b}$ the set of $\mathcal{C}^{1}$ functions in $v$ that are bounded as well as their first order derivatives.  Consequently, 
\begin{align}\label{original-approx}
\begin{split}
F_{\e}&(t,x,v) = {V}_{\e}(t)^{d}\Big( G_{\re(\e)}(V_{\e}(t)v) + \e\,\big(\varrho(\tau_{\e}(t),x)+u(\tau_{\e}(t),x)\cdot (V_{\e}(t)v) \\
&+ \frac{1}{2}\vE(\tau_{\e}(t),x)(|V_{\e}(t)v|^{2}-d\en_{1})\big)\M(V_{\e}(t)v)\Big) + \mathrm{O}\Big( V_{\e}(t)^{-\kappa-a} \Big)\,,
\end{split}
\end{align}
in the \textit{weak sense} described in \eqref{error-term}.  In particular,  {if $\varphi=1$ and $\kappa=2$,} one finds from~\eqref{original-approx} an explicit expression for \textit{Haff's law} is obtained.  That is, the optimal cooling rate of the temperature is described by
\begin{equation*}\begin{split}
\bm{T}_{\e}(t)&=\frac{1}{\big|\mathbb{T}^{d}\big|}\int_{\mathbb{T}^{d}\times\mathbb{R}^{d}}F_{\e}(t,x,v)|v|^{2}\d v\,\d x \\
&= \frac{1}{ V_{\e}(t)^{2} }\bigg(\int_{\R^{d}}G_{\alpha}(v)|v|^{2}\d v + \frac{\e}{2\big|\mathbb{T}^{d}\big|}\int_{\mathbb{R}^{d}}\big( |v|^{2} - d\en_1 \big)|v|^{2}\mathcal{M}(v)\,\d v \int_{\mathbb{T}^{d}}\vE(\tau_{\e}(t),x)\,\d x \\
&\phantom{++}+\frac{\e}{\big|\T^{d}\big|}\int_{\R^{d}}|v|^{2}\M(v)\d v\int_{\T^{d}}\varrho(\tau_{\e}(t),x)\d x\bigg)+ \mathrm{O}\Big( V_{\e}(t)^{-2 - a} \Big)\\
& {\approx} \frac{d\en_{1}}{V_{\e}(t)^{2}}\bigg(1 + \frac{\e}{\big|\T^{d}|}\left(d\en_{1}\int_{\T^{d}}\vE(\tau_{\e}(t),x)\d x + 2\int_{\T^{d}}\varrho(\tau_{\e}(t),x)\d x\right)\bigg)\,,\qquad t \gg \frac{1}{\lambda_0}\,.
\end{split}\end{equation*}
Recalling that the fluctuation $h_{\e}$ is such that the average mass and temperature both vanish at all times, we deduce the precised Haff's law
$$\bm{T}_{\e}(t) \approx \frac{d\en_{1}}{V_{\e}(t)^{2}}, \qquad t \gg \frac{1}{\lambda_0}.$$
In the Appendix \ref{Appendix-PP} we complement this discussion and, in particular, show that the Haff's law holds uniformly \textit{locally} in space due to the boundedness of the solutions that we treat here.  This is not expected in a general context.

\appendix
\section{About granular gases in the spatial homogeneous setting}\label{appen:homog}
\subsection{The collision operator}
We collect several results about the Boltzmann collision operator $\Q_{\alpha}$ for granular gases. 
We shall consider a collision operator with more general collision kernel than the hard-spheres case considered in the paper, more precisely, a collision kernel $B(u,\sigma)$ of the form
\begin{equation}
\label{Bu} {B}(u,\sigma)=\Phi(|u|)b(\widehat{u} \cdot \sigma).
\end{equation}
The kinetic potential $\Phi(\cdot)$ is a suitable nonnegative function in $\R^{d}$ and the \textit{angular kernel} $b(\cdot)$ is assumed in $L^{1}(-1,1)$.  The associated collision operator $\Q_{B,\re}$ is defined through the weak formulation
\begin{equation}\label{co:weak-A}
\int_{\mathbb{R}^{d}}\Q_{B,\re}(f,f)(v)\psi(v)\d v=\frac{1}{2}\int_{\mathbb{R}^{d} \times \mathbb{R}^d}f(v)f(\vb)\mathcal{A}_{B,\re}[\psi](v,\vb)\;\d\vb\d v
\end{equation}
for any test function $\psi=\psi(v)$  where
$$\mathcal{A}_{B,\re}[\psi](v,\vb)=\int_{\mathbb{S}^{d-1}}\bigg(\psi( {v'})+\psi( {\vb'})-\psi(v)-\psi(\vb)\bigg) {B}(u, \sigma)\d\sigma$$
where
the post-collisional velocities $(v',v_{\ast}')$ are given
by
\begin{equation}\begin{split}
\label{co:transf-A}
  v'=v+\frac{1+\re}{4}\,(|u|\sigma-u),&
\qquad  v_{\ast}'=v_{\ast}-\frac{1+\re}{4}\,(|u|\sigma-u),\\
\text{where} \qquad u=v-v_{\ast},& \qquad \widehat{u}=\frac{u}{|u|}.
\end{split}
\end{equation}
{This allows to split $\Q_{B,\re}$ into positive and negative parts
$$\Q_{B,\re}=\Q^{+}_{B,\re}-\Q_{B,\re}^{-}$$
where $\Q_{\re}^{+}$ and $\Q^{-}_{\re}$ are given, in strong form, by
\begin{equation}\label{eq:Qstrongsigma}
\Q_{B,\re}^{+}(g,f)(v)=\frac{1}{\re^{2}}\,\int_{\R^{d}\times \S^{d}}\,\frac{|'u|}{|u|}B(u,\sigma \cdot \widehat{u})f('v)g('v_{\ast})\d\sigma\,\d v_{\ast}\end{equation}
and
$$\Q_{B,\re}^{-}(g,f)(v)=f(v)\int_{\R^{d}\times\S^{d-1}} g(v_{\ast})B(u,\sigma \cdot \widehat{u})\d v_{\ast}\d \sigma.$$
where, for $\sigma \in \S^{d-1}$, $'v$ and $'v_{\ast}$ denote the pre-collisional velocities
\begin{equation*}
\begin{cases}
'v&=\dfrac{v+v_{\ast}}{2}-\dfrac{1-\re}{4\re}u \,+\dfrac{1+\re}{4\re}|u|\sigma\\
\\
'v_{\ast}&=\dfrac{v+v_{\ast}}{2}+\dfrac{1-\re}{4\re}u \,-\dfrac{1+\re}{4\re}|u|\sigma, \qquad u=v-v_{\ast}.
\end{cases}\end{equation*}
}
A particularly relevant model is the one of hard-spheres corresponding to $\Phi(|u|)=|u|$ which is the model investigated in the core of the paper and, in that case, we simply denote the collision operator $\Q_{B,\re}$ by $\Q_\re$.

\subsection{Alternative representation of the velocities}\label{Sec21} As well-known, the above collision operator  is a well-accepted model that describes collisions in a system composed by a large number of granular particles which are assumed to be hard-spheres with equal mass (that we take to be $m=1$) and that undertake inelastic collisions.  The collision mechanism and the role of the \textit{coefficient of normal restitution} is easier to understand in an alternative representation of the post-collisional velocities. More precisely,  if $v$ and $\vb$  denote the velocities of two particles before collision, their respective velocities $v'$ and $\vb'$ after collision are such that
\begin{equation}\label{coef}
\big( u' \cdot n \big)=- \re\big( u\cdot n \big) \,.
\end{equation}
The unitary vector $n \in \mathbb{S}^{d-1}$  determines the impact direction, that is, $n$ stands for the unit vector that points from the $v$-particle center to the $v_{\ast}$-particle center at the moment of impact.  Here above
\begin{equation}
u=v-v_{\ast},\qquad u'=v'-v_{\ast}',
\end{equation}
denote respectively the relative velocity before and after collision.  The velocities after collision $v'$ and $\vb'$ are given, in virtue of \eqref{coef} and the conservation of momentum, by
\begin{equation}
\label{transfpre}
  v'=v-\frac{1+\re}{2}\,\big( u\cdot n \big)n,
\qquad v_{\ast}'=v_{\ast}+\frac{1+\re}{2}\, \big( u\cdot n \big)n.
\end{equation}
In particular, the energy relation and the collision mechanism can be written as
\begin{equation}\label{energ}
|v|^{2}+| \vb|^{2}=|'\!v|^{2}+|'\!\vb|^{2}-\frac{1-\re^{2}}{2}\,\big('\!u \cdot n\big)^{2}, \qquad
u\cdot n=-\re\big('\!u \cdot n\big).
\end{equation}
Pre-collisional velocities $('v,'\vb)$ (resulting in $(v,\vb)$ after collision) can be therefore introduced through the relation
\begin{equation}\label{'v'vb}
v='\!\!v-\frac{1+\re}{2}\big( '\!\!u \cdot n \big) n\,, \qquad \vb='\!\!\vb+\frac{1+\re}{2}\big( '\!\!u \cdot n \big) n, \qquad '\!\!u='\!\!v-'\!\!\vb
\end{equation}

This representation is of course equivalent to the one given in \eqref{co:transf-A}
 (so-called $\sigma$-representation) by setting, for a given pair of velocities $(v,\vb)$,
\begin{equation*}
\sigma=\widehat{u}-2 \,(\widehat{u}\cdot n)n \in \S^{d-1}.
\end{equation*}
Such a description provides an alternative parametrization of the unit sphere $\mathbb{S}^{d-1}$ in which the unit vector $\sigma$ points in the post-collisional relative velocity direction in the case of elastic collisions. In this case, the impact velocity reads
\begin{equation*}
|u\cdot\, n|=|u| \,|\widehat{u} \cdot n|=|u| \sqrt{\frac{1-\widehat{u} \cdot \sigma}{2}}.
\end{equation*}
In the $n$-representation, we can also explicit the strong form of the collision operator $\Q_{\re}.$ Namely, for a given pair of distributions $f=f(v)$ and $g=g(v)$ and a given \textit{collision kernel} the Boltzmann collision operator is defined as the difference of two nonnegative operators (gain and loss operators respectively)
\begin{equation*}
\Q_{\re}\big(g,f\big)=\Q_{\re}^{+}\big(g,f\big)-\Q_{\re}^{-}\big(g,f\big),
\end{equation*}
with
\begin{align}\label{Boltstrong}
\begin{split}
\Q_{\re}^{+}\big(g,f\big)(v)&=\frac{1}{\re^{2}}\int_{\R^{d}\times\S^{d-1}} |u \cdot n|\,b_0(\widehat{u}\cdot n)\,f('\!v)g('\!\vb)\d n\d\vb\,,\\
\Q_{\re}^{-}\big(g,f\big)(v)&=f(v)\int_{\R^{d}\times \S^{d-1}} |u\cdot n|b_0(\widehat{u}\cdot n)g(\vb)\d n\d\vb.
\end{split}
\end{align}
where the new angular collision kernel $b_{0}(\cdot)$ is related to the original one $b(\cdot)$ through the relation
\begin{equation*}
b_0\big(\widehat{u} \cdot n\big)=2^{d-1}|\widehat{u}\cdot n|^{d-2}b(\widehat{u}\cdot \sigma).
\end{equation*}
i.e.
$$b_{0}(x)=\,2^{d-1}|x|^{d-2}b(1-2x^2), \qquad x \in [-1,1].$$

\subsection{Estimates on the collision operator}
Using the above representation, we prove  Lemma~\ref{lem:els} in Section \ref{sec:ll}.


\begin{proof}[Proof of Lemma \ref{lem:els}]
Note that
\begin{equation}\label{eq:QI1I2}
\Q_{1}(g,f) - \Q_{\re}(g,f) = \mathcal{I}_{1}(g,f) + \mathcal{I}_{2}(g,f)\,,
\end{equation}
where \begin{align*}
\mathcal{I}_{1}(g,f) = -\frac{1-\re^2}{\re^2}\int_{\mathbb{R}^{d}}\int_{\mathbb{S}^{d-1}} g('\!v_{\ast ,\re})f('\!v_{\re})\, |u \cdot n| \,b_{0}(\hat{u}\cdot n) \d n\d v_{\ast}\,,
\end{align*}
and
\begin{align*}
\mathcal{I}_{2}(g,f) = \int_{\mathbb{R}^{d}}\int_{\S^{d-1}} \big(g('\!v_{\ast ,\re})f('\!v_{\re}) - g('\!v_{\ast ,1})f('\!v_{1})\big)\, |u \cdot n| \,b_{0}(\hat{u}\cdot n) \d n\d v_{\ast}\,.
\end{align*}
Here we adopt the notation $'\!v_{1}$ and $'\!v_{\re}$ for the pre-collisional velocities associate to elastic ($\re=1$) and inelastic $(0< \re <1)$ interactions respectively in the $n$-representation (see \eqref{'v'vb}).

We begin with proving the estimates in the $L^{1}_{v}$ setting. By classical means, there is a positive $c_{q} >0$ such that
\begin{equation*}
\| \mathcal{I}_{1}(g,f) \|_{ L^{1}_{v}(\m_{q}) } \leq \frac{1-\re}{\re^{2}}\,c_{q}\,\| b_{0} \|_{L^{1}(\S^{d-1})}\| g \|_{ L^{1}_{v}(\m_{q+1}) } \| f \|_{ L^{1}_{v}(\m_{q+1}) }\,.
\end{equation*}
We estimate then the difference $g('\!v_{\ast ,\re})f('\!v_{\re})- g('\!v_{\ast,1})f('\!v_{1})$ by writing 
\begin{equation*}
g('\!v_{\ast ,\re})f('\!v_{\re}) - g('\!v_{\ast,1})f('\!v_{1}) = g('\!v_{\ast,1})\big( f('\!v_{\re})- f('\!v_{1})\big) + \big( g('\!v_{\ast ,\re}) - g('\!v_{\ast,1}) \big) f('\!v_{\re})\,,
\end{equation*}
and splitting $\mathcal{I}_{2}(g,f)$ accordingly into $\mathcal{I}_{2}(g,f)=\mathcal{I}_{2}^{1}(g,f)+ \mathcal{I}_{2}^{2}(g,f)$. 

\medskip
\noindent
For the term $\mathcal{I}^{1}_{2}(g,f)$, we first notice that, according to \eqref{'v'vb},
\begin{equation*}
'\!v_{\re} = v - \frac{1+\re}{2\re}(u\cdot n)n\,,\qquad '\!v_{*,\re} = v - u + \frac{1+\re}{2\re}(u\cdot n)n\,.
\end{equation*}
Therefore, using the change of variable $w=\frac{1+\re}{2\re}u$ one has $'\!v_{\re}\rightarrow \,'\!v_{1}$ and
\begin{multline*}
\int_{\mathbb{R}^{d}}\int_{\S^{d-1}} g('\!v_{\ast,1})\, f('\!v_{\re})\, |u \cdot n| \,b_{0}(\hat{u}\cdot n) \d n\d v_{\ast}\\
=\left(\frac{2\re}{1+\re}\right)^{d+1}\int_{\mathbb{R}^{d}}\int_{\S^{d-1}} g(\tilde{v}_{\ast})\, f('\!v_{1})\, |u \cdot n| \,b_{0}(\hat{u}\cdot n) \d n\d v_{\ast}\,,
\end{multline*}
where
\begin{equation*}
\tilde{v}_{*} = v - \frac{2\re}{1+\re}u + (u\cdot n)n =\, '\!v_{\ast,1} + \frac{1-\re}{1+\re}u\,.
\end{equation*}
Consequently,
\begin{multline*}
\mathcal{I}^{1}_{2}(g,f) = \int_{\mathbb{R}^{d}}\int_{\S^{d-1}} \Big(\big(\tfrac{2\re}{1+\re}\big)^{d+1}g(\tilde{v}_{\ast})- g('\!v_{\ast,1})\Big)\,f('\!v_{1})\, |u \cdot n| \,b_{0}(\hat{u}\cdot n) \d n\d v_{\ast}\\
=\left(\left(\frac{2\re}{1+\re}\right)^{d+1}-1\right)\int_{\mathbb{R}^{d}}\int_{\S^{d-1}} g(\tilde{v}_{\ast})\,f('\!v_{1})\, |u \cdot n| \,b_{0}(\hat{u}\cdot n) \d n\d v_{\ast}\\
\qquad\qquad + \int_{\mathbb{R}^{d}}\int_{\S^{d-1}}\big(g(\tilde{v}_{\ast})- g('\!v_{\ast,1})\big)\,f('\!v_{1})\, |u \cdot n| \,b_{0}(\hat{u}\cdot n) \d n\d v_{\ast}= :\mathcal{I}^{1,1}_{2}(g,f) + \mathcal{I}^{1,2}_{2}(g,f)\,.
\end{multline*}
For the first term, thanks to the mean-value theorem, one notices that 
$$\left|\left(\frac{2\re}{1+\re}\right)^{d+1}-1\right|\leq d\,\frac{1-\re}{1+\re}\,, \qquad \forall \re \in (0,1)$$
so that using the pre-post collisional change of variable $(v_{\ast},v)\rightarrow(\tilde{v}_{\ast},\,'\!v_{1})$ (with Jacobian computed as $\frac{1+\re}{2}\big(\frac{1+\re}{2\re}\big)^{d-1}$ ) one concludes that
\begin{equation*}
\| \mathcal{I}^{1,1}_{2}(g,f) \|_{ L^{1}_{v}(\m_{q}) }\leq (1-\re)\,\big(\tfrac{1+\re}{\re}\big)^{d-1}\,c_{q}\,\| b_{0} \|_{L^{1}(\S^{d-1})}\| g \|_{ L^{1}_{v}(\m_{q+1}) } \| f \|_{ L^{1}_{v}(\m_{q+1}) }\,
\end{equation*}
for some positive $c_{q}$ depending only on $q$ and the dimension $d.$
For the second term one uses Taylor formula:
\begin{equation*}
g(\tilde{v}_{\ast}) - g('\!v_{\ast,1})  = (\tilde{v}_{\ast} - \,'\!v_{\ast,1})\cdot \int^{1}_{0} \nabla g(\tilde{v}_{\ast ,t}) \d t = \frac{1-\re}{1+\re}u\cdot \int^{1}_{0} \nabla g(\tilde{v}_{\ast ,t}) \d t\,,
\end{equation*}
where
\begin{equation*}
\tilde{v}_{\ast ,t} =  \tilde{v}_{\ast}\,t + \, '\!v_{\ast,1}\,(1-t) = '\!\!v_{\ast ,1} - \frac{1-\re}{1+\re}\,u\,t\,, \qquad \forall t \in (0,1).
\end{equation*}
Consequently,
\begin{multline*}
\int_{\mathbb{R}^{2d}}\big| \mathcal{I}^{1,2}_{2}(g,f)(v) \big|\langle v \rangle^{q}  \d v \\
\leq \frac{1-\re}{1+\re}\int^{1}_{0}\d t \int_{\S^{d-1}}b_{0}(\widehat{u}\cdot n)|\widehat{u}\cdot n|^{2}\d n\int_{\mathbb{R}^{2d}}|f('\!v_{1})|\,|\nabla g(\tilde{v}_{\ast ,t})|\, |u|^{2} \langle v \rangle^{q} \d v_{*}\d v\,. 
\end{multline*}
We can apply the pre-post collisional change of variable in the inner integral $(v_{\ast},v)\rightarrow(\tilde{v}_{\ast,t},\,'\!v_{1})$ with Jacobian
$$\frac{1+\re}{ 1+\re + (1-\re)t }\left(\frac{1+\re}{1+\re - (1-\re)t}\right)^{d-1} \leq \left( \frac{1+\re}{2\re}\right)^{d-1}\,,$$
to obtain that
\begin{equation*}
\| \mathcal{I}^{1,2}_{2}(g,f) \|_{ L^{1}_{v}(\m_{q}) }\leq \left( \frac{1+\re}{2\re}\right)^{d-1}\frac{1-\re}{1+\re}\,c_{q}\,\| b_{0} \|_{L^{1}(\S^{d-1})}\| \nabla g \|_{ L^{1}_{v}(\m_{q+2}) } \| f \|_{ L^{1}_{v}(\m_{q+2}) }\,
\end{equation*}
for some positive constant $c_{q}$. Regarding the term $\mathcal{I}_{2}^{2}(g,f)$, one invokes Taylor formula again
\begin{equation*}
g('\!v_{\ast ,\re}) - g('\!v_{\ast,1}) = ('\!v_{\ast ,\re} - \,'\!v_{\ast,1})\cdot \int^{1}_{0} \nabla g('\!v_{\ast ,t}) \d t\,,
\end{equation*}
where we recall that, according to \eqref{'v'vb},
\begin{align*}
'\!v_{\ast ,\re} - \,'\!v_{\ast,1} = \frac{1-\re}{2\re}( u\cdot n )n\,,
\end{align*}
and, for $0 < t < 1$,
\begin{align*}
'\!v_{\ast ,t} &=  '\!\!v_{\ast ,\re}\,t + '\!v_{\ast,1}\,(1-t) = '\!\!v_{\ast ,\re} - (1- t)\frac{1-\re}{2\re}( u\cdot n )n\,.
\end{align*}
Therefore,
\begin{multline*}
\int_{\mathbb{R}^{2d}}\big| \mathcal{I}^{2}_{2}(g,f)(v) \big|\langle v \rangle^{q}  \d v \\
\leq \frac{1-\re}{2\re}\int^{1}_{0}\d t \int_{\S^{d-1}}b_{0}(\widehat{u}\cdot n)|\widehat{u}\cdot n|^{2}\d n\int_{\mathbb{R}^{2d}}|f('\!\tilde{u}_{\re} + \,'\!v_{\ast ,t})|\,|\nabla g('\!v_{\ast ,t})|\, |u|^{2} \langle v \rangle^{q} \d u\d v 
\end{multline*}
with
\begin{equation*}
'\tilde{u}_{\re}:= \,'\!v_{\re} - \,'\!v_{\ast ,t} =u - 2(u\cdot n)n - (1 + t)\frac{1-\re}{2\re}(u\cdot n)n\,.
\end{equation*} 
Apply, for fixed $n$, the pre-post collisional change of variables $(v_{\ast},u)\rightarrow('v_{\ast ,t},'\tilde{u}_{\re})$ (with Jacobian $\tilde{J}_{\re}(t)=1 + (1+t)\frac{1-\re}{2\re}\geq1$).  Noticing that
\begin{equation*}
|u| \leq |'u_{\re}|\leq \frac{2}{1+t}|'\tilde{u}_{\re}|\,,\qquad |v| \leq |'v_{\re}| + | 'v_{\ast ,\re} | \leq 2|'v_{\ast ,t}| + 4|'\tilde{u}_{\re}|\,,
\end{equation*}
it follows that
\begin{multline*}
\int_{\mathbb{R}^{d}}\big| \mathcal{I}^{2}_{2}(g,f)(v) \big|\langle v \rangle^{q}  \d v 
\leq \frac{1-\re}{2\re}\,c_{q}\, \int^{1}_{0} \frac{1}{\tilde{J}_{\re}(t)}\d t\int_{\S^{d-1}}b_{0}(\widehat{u}\cdot n)|\widehat{u}\cdot n|^{2}\d n\\\phantom{++++}\int_{\mathbb{R}^{2d}}|f(v + u)\langle v + u\rangle^{q+2}|\,|\nabla g(v) \langle v \rangle^{q+2}| \d u\d v\\
\qquad\leq\frac{1-\re}{2\re}\,c_{q}\,\| b_{0} \|_{L^{1}(\S^{d-1})}\,\| f \|_{L^{1}_{v}(\m_{q+2})}\| \nabla g\|_{L^{1}_{2}(\m_{q+2})}\,.
\end{multline*}
Gathering previous estimates proves the first assertion of the Lemma. The proof for higher norms $\W^{k,1}_{v}(\m_{q})$ is simply obtained differentiating and applying the previous estimates for each suitable difference.

Let us now see how to derive the estimates in the $L^{2}_{v}$ setting. One still starts with the representation \eqref{eq:QI1I2}. To estimate $\mathcal{I}_1(g,f)$ in $L^2_v(\m_q)$, we first use Cauchy-Schwarz inequality to get
\begin{align*}
&\|\mathcal{I}_{1}(g,f) \|^2_{ L^{2}_{v}(\m_{q}) } \\
&\quad \leq C_{\kappa}\frac{(1-\re^2)^2}{\re^4} \int_{\mathbb{R}^{2d} \times \mathbb{S}^{d-1}} |g('\!v_{\ast ,\re})|^2|f('\!v_{\re})|^2\, |u \cdot n|^2 \, \langle u \rangle^{2\kappa} \,b_{0}(\widehat{u}\cdot n)  \m_{2q} (v) \d n\d v_{\ast}\d v\,,
\end{align*}
where
$$C_{\kappa}:=\sup_{v\in \R^{d}}\int_{\R^{d}\times \S^{d-1}}b_{0}(\widehat{u}\cdot n)\langle u\rangle^{-2\kappa}\d v_{\ast}\d n=\|b_{0}\|_{L^{1}(\S^{d-1})}\,\|\langle \cdot\rangle^{-2\kappa}\|_{L^{1}(\R^{d})} < \infty$$
for $\kappa > \frac{d}{2}$. We proceed then exactly as previously to obtain:
$$
\| \mathcal{I}_{1}(g,f) \|_{ L^{2}_{v}(\m_{q}) }\leq c_q \frac{1-\re}{\re^{2}}\,\| b_{0} \|_{L^{1}(\S^{d-1})}\| g \|_{ L^{2}_{v}(\m_{q+\kappa+1}) } \| f \|_{ L^{2}_{v}(\m_{q+\kappa+1}) }\,. 
$$
In a similar way,   using Cauchy-Schwarz inequality, one can prove that 
$$
\|\mathcal{I}^{1}_{2}(g,f)\|_{L^2_v(\m_q)} \leq c_q \frac{1-\re}{2\re}\,\|b_{0}\|_{L^{1}(\S^{d-1})}\,\| g \|_{L^{2}_{v}(\m_{q+\kappa+2})}\| f \|_{\W^{1,2}_{v}(\m_{q+\kappa+2})}
$$
and
$$
\|\mathcal{I}^{2}_{2}(g,f)\|_{L^2_v(\m_q)} \leq c_q \frac{1-\re}{2\re}\,\|b_{0}\|_{L^{1}(\S^{d-1})}\,\| g \|_{\W^{1,2}_{v}(\m_{q+\kappa+2})}\| f \|_{L^2_{v}(\m_{q+\kappa+2})}.
$$
which gives the result for $k=0$. The proof for higher norms is obtained as before, differentiating and applying the above estimate to each suitable difference.
\end{proof}

 {Let us now give the proof of Lemma~\ref{lem:L1Lalpha} which is based on the $\sigma$-representation. Note that it is reminiscent of the proof of~\cite[Proposition~3.2]{MiMo3} and \cite[Theorem 3.11]{CMS} but that we face some additional difficulties due to the polynomial weight. }
 {\begin{proof}[Proof of Lemma~\ref{lem:L1Lalpha}]
We only prove the first estimate on $\Q_1(\cdot,\M) - \Q_\re(\cdot,\M)$, the other one can be treated exactly in the same way. Notice first that
$$ 
\Q_1(f,\M)- \Q_\re(f,\M) = \Q_1^+(f,\M)-\Q_\re^+(f,\M). 
$$
As in the proofs of~\cite[Proposition~3.2]{MiMo3} and \cite[Theorem~3.11]{CMS}, we set $w := v+\vb$ and $\widehat w := w/|w|$ and define $\chi \in [0,\pi/2]$ through $|\!\cos \chi| := |\widehat w \cdot \sigma|$.    Let $\delta \in (0,1)$ and $R >1$ be fixed and let
 $\eta_{\delta} \in \W^{1,\infty}(-1,1)$ such that $\eta_{\delta}(s)=\eta_{\delta}(-s)$ for any $s \in (0,1)$ and
$$\eta_{\delta} (s) = \begin{cases}1 \quad &\text{ if } \quad s \in (-1+2\delta,1-2\delta)\\
 0 \quad &\text{ if } \quad  s \notin (-1+\delta,1-\delta)\end{cases}$$
with moreover
\begin{equation*}
0 \leq \eta_{\delta}(s) \leq 1 \qquad \text{ and } \qquad |\eta_{\delta}'(s)| \leq \frac{3}{\delta}\qquad \forall s \in (-1,1).\end{equation*}
Let us define also $\Theta_R(r) = \Theta(r/R)$ with $\Theta(x) = 1$ on $[0,1]$, $\Theta(x) = 1-x$ for $x \in [1,2]$ and $\Theta(x) = 0$ on $[2,\infty)$. We define the set
$$A(\delta)  := \{ \sigma \in \S^{d-1}; \,\, \sin^2 \chi \geq \delta \}$$
we split $\Q^+_\re$ into
$$\Q^+_{\re} = \Q^{+}_{B_{r},\re}  + \Q^{+}_{B_{l},\re}  + \Q^{+}_{B_{a},\re}$$
where the collision kernels $B_i(u,\widehat{u} \cdot \sigma\,)$, $i=l,r,a$, are defined by
$$B_{r}(u,\widehat{u} \cdot \sigma\,)=\eta_{\delta}(\widehat{u} \cdot \sigma\,) \,\Theta_R(u)\,|u|b(\widehat{u}\cdot \sigma), \quad B_{l}(u,\widehat{u} \cdot \sigma\,) := {\bf 1}_{A(\delta)}(\sigma) \, (1-\Theta_R(|u|))\,|u|b(\widehat{u}\cdot \sigma) $$ and
$$B_{a}(u,\widehat{u} \cdot \sigma\,)=|u|b(\widehat{u}\cdot \sigma)\,\left((1 - \eta_{\delta}(\widehat{u} \cdot \sigma\,)) \, \Theta_R(|u|) +  \, (1-\Theta_R(|u|)) \, {\bf 1}_{A^c(\delta)}\right).$$ 
This splitting corresponds to a splitting for \emph{small angles} (corresponding to the kernel $B_{a}$), \emph{large velocities} (corresponding to $B_{l}$) and the reminder term (corresponding to $B_{r}$).
The treatment of small angles and of the truncated operator is similar to the one of~\cite[Proposition~3.2]{MiMo3} and we only recall the results obtained therein: there exists a constant $C_q>0$ such that for any $\re \in (0,1]$, $\delta \in (0,1)$ and $R \in (1,\infty)$,
$$
\|\Q^{+}_{B_{a},\re}(f,\M)\|_{L^1_v(\m_q)} \leq C_q \delta \|f\|_{L^1_v(\m_{q+1})}
$$
and
$$
\|\Q^{+}_{B_{r},1}(f,\M) - \Q^{+}_{B_{r},\re}(f,\M)\|_{L^1_v(\m_q)} \leq C_q (1-\alpha) \left(\frac{R^2}{\delta}+\frac{R}{\delta^3}\right) \|f\|_{L^1_v(\m_{q})}.
$$
Let us now handle the case of large relative velocities. To this end, we will use the strong formulation of the $\sigma$-representation of our collision operator given in~\eqref{eq:Qstrongsigma}:
\begin{equation*}\begin{split}
\Q_{B_{l},\re}^{+}\left(f,\M\right)(v)&=\frac{1}{\re^{2}}\int_{\R^{d}\times \S^{d-1}}\frac{|'u|}{|u|}B_{l}(u,\widehat{u}\cdot\sigma),\M('\!v)f('\!\vb)\d\vb\d \sigma\\
&=\frac{1}{\re^{2}}\int_{\R^{d}\times\S^{d-1}}|'u|b(\widehat{u}\cdot \sigma)\, {\bf 1}_{A(\delta)}(\sigma) \, (1-\Theta_R(|u|))\,\M('\!v)f('\!\vb)\d\vb\d \sigma.
\end{split}\end{equation*}
We first observe that
\begin{equation}\label{eq:bl}
|'u|b(\widehat{u}\cdot \sigma)\, {\bf 1}_{A(\delta)}(\sigma)\left(1-\Theta_{R}(|u|)\right) \leq \frac{|u|^{2}}{\re\,R}b(\widehat{u}\cdot \sigma){\bf 1}_{A(\delta)}(\sigma)\,{\bf 1}_{\{R \leq |u| \leq 2R\}}.\end{equation}
We then remark that 
$$
'\!v = v - \frac{1+\alpha}{4\alpha} (u-|u|\sigma)= \frac{w}{2} - {1 \over 2} \left[\frac{1-\alpha}{2\alpha}u - \frac{1+\alpha}{2 \alpha} |u| \sigma\right].
$$
Using now that,  when $\sin^2 \chi \geq \delta$, then $|\widehat w \cdot \sigma | \leq \sqrt{1-\delta} \leq 1-\frac{\delta}{2}$, we have
\begin{align*}
|'\!v|^2 
&\geq {1 \over 4} (|w|^2+|u|^2) - \frac{1+ \alpha}{4 \alpha}|w||u| \left(1-{\delta \over 2}\right) - \frac{1-\alpha}{4\alpha}(|v|^2-|\vb|^2) \\
&\geq |v|^2 \left({1 \over 2} - \frac{1+ \alpha}{4 \alpha} \left(1-{\delta \over 2}\right) - \frac{1- \alpha}{4 \alpha}\right)
+|\vb|^2 \left({1 \over 2} - \frac{1+ \alpha}{4 \alpha}  \left(1-{\delta \over 2}\right) + \frac{1- \alpha}{4 \alpha}\right) \\
&\geq \delta \frac{1+\alpha}{8\alpha} |\vb|^2 + \left(\delta \frac{1+\alpha}{8\alpha} -  \frac{1-\alpha}{2\alpha}\right) |v|^2. 
\end{align*}
Then, if $\delta \geq 8 \frac{1-\alpha}{1+\alpha}$, we get:
$$
|'\!\vb|^2 \geq \delta \frac{1+\alpha}{16\alpha} (|v|^2+|\vb|^2) \geq \frac{\delta}{16} (|v|^2+|\vb|^2). 
$$
As a consequence, 
\begin{equation}\label{eq:pointwiseMa}
\M('\!v) \langle \vb \rangle^2 \m_{q+2}(v) \lesssim e^{-\delta \frac{|v|^2}{64 \en_1}}\m_{q+2}(v)e^{-\delta \frac{|\vb|^2}{64 \en_1}} \langle \vb \rangle^2\M^\frac12('\!v)\lesssim {1 \over \delta^{\frac{q+4}{2}}}\M^\frac12('\!v). 
\end{equation}
Using this, we can now estimate the norm of $\Q_{B_{l},\re}^{+}\big(f,\M\big)$ if $\delta \geq 8 \frac{1-\alpha}{1+\alpha}$ thanks to \eqref{eq:bl} to get
\begin{align*}
\|\Q_{B_{l},\re}^{+}\big(f,\M\big)\|_{L^1_v(\m_q)} &\lesssim 
\frac{1}{R}\int_{\R^d \times \R^{d}\times\S^{d-1}} |u|^{2}b(\widehat{u}\cdot \sigma)\, {\bf 1}_{A(\delta)}(\sigma)
\,\M('\!v)|f('\!\vb)| \m_q(v)\d \sigma  \d\vb\d v \\
&\lesssim \frac{1}{R}\int_{\R^d \times \R^{d}\times\S^{d-1}}\M('\!v)|f('\!\vb)| \langle \vb\rangle^2 \m_{q+2}(v)b(\widehat{u}\cdot\sigma)\d \sigma \d\vb\d v \\
&\lesssim \frac{1}{R \, \delta^\frac{q+4}{2}} \int_{\R^d \times \R^{d}\times\S^{d-1}} \M^\frac12('\!v)|f('\!\vb)|b(\widehat{u}\cdot\sigma)\d \sigma \d\vb\d v \\
&\lesssim  \frac{1}{R \, \delta^\frac{q+4}{2}} \|f\|_{L^1_v}
\end{align*} {since $b \in L^{1}(\S^{d-1})$ and $\M \in L^{\infty}(\R^{d})$}.
Gathering the previous estimates, we obtain if $\delta \geq 8 \frac{1-\alpha}{1+\alpha}$:
$$
\|\Q_{\alpha}(f,\M)-\Q_{1}(f,\M)\|_{L^{1}_v(\m_{q})} 
\lesssim \left(\delta+ (1-\alpha)\left(\frac{R^{2}}{\delta}+\frac{R}{\delta^{3}}\right)+\frac{1}{R\delta^{\frac{q+4}{2}}}\right)\|f\|_{L^{1}_v(\m_{q+1})}.
$$
Picking now $\delta=(1-\alpha)^{p}$ for some $p \in (0,1)$ (so that the condition $\delta \geq 8 \frac{1-\alpha}{1+\alpha}$ will be satisfied for $\alpha$ close enough to $1$) and $R=(1-\alpha)^{-p-\frac{q+4}{2}p}=(1-\alpha)^{-p\frac{q+6}{2}}$, we then obtain for $\alpha$ close enough to $1$ that
\begin{align*}
&\|\Q_{\alpha}(f,\M)-\Q_{1}(f,\M)\|_{L^{1}_v(\m_{q})} \\ 
&\quad \lesssim \left((1-\alpha)^{p}+(1-\alpha)^{1-p(q+7)}+(1-\alpha)^{1-p\frac{q+12}{2}}\right)\|f\|_{L^{1}_v(\m_{q+1})},
\end{align*}
and, with $p=\frac{1}{q+8}$, $1-p(q+7)=p$ and
\begin{align*}
\|\Q_{\alpha}(f,\M)-\Q_{1}(f,\M)\|_{L^{1}_v(\m_{q})} 
&\lesssim \left((1-\alpha)^{p}+(1-\alpha)^{\frac{q+4}{2}p}\right)\|f\|_{L^{1}_v(\m_{q+1})} \\
&\lesssim (1-\alpha)^{p}\|f\|_{L^{1}_v(\m_{q+1})}.
\end{align*}
 For higher norms, it is enough to differentiate and apply previous estimates for each suitable difference.
\end{proof}}
We now give the proof of Lemma \ref{prop:psi} in Section \ref{sec:ll}  {which is based on Lemma \ref{lem:els}}.

\begin{proof}[Proof of Lemma \ref{prop:psi}]  {We start with the case of $\W^{k,1}_v(\m_q)$.} To this end, we slightly modify here a strategy adopted in {\cite{CMS}} which consists in combining a {\it nonlinear} estimate for $\|G_{\re}-\M\|_{\W^{k,1}_v(\m_{q})}$ together with non-quantitative convergence. We fix $k,q$ and we divide the proof into three steps:

\smallskip
\noindent
\textit{First step: non quantitative convergence}. We prove that
\begin{equation}\label{eq:nonquant}
\lim_{\re\to1}\|\M-G_{\re}\|_{\W^{k,1}_{v}(\m_{q})}=0.\end{equation}
We argue here as in \cite[Theorem 4.1]{CMS}. We sketch only the main steps. First, as already noticed in~{\cite{MiMo3}}, there is $\re_{0} >0$ such that 
$$\sup_{\re \in (\re_{0},1)}\|G_{\re}\|_{\W^{k,1}_{v}(\m_{q})} < \infty.$$ Then, there is a sequence $(\re_{n})_{n}$ converging to $1$ such that $(G_{\re_{n}})_{n}$ converges weakly, in $\W^{k,1}_{v}(\m_{q+1})$ to some limit $\bar{G}$  (notice that, a priori, the limit function $\bar{G}$ depends on the choice of $k$ and $q$). Using the decay of $(G_{\re})$ and compact embedding for Sobolev spaces, this convergence is actually strong, i.e. $\lim_{n}\|G_{\re_{n}}-\bar{G}\|_{\W^{k,1}_{v}(\m_{q})}=0$. According to \eqref{appe1}, one necessarily has $\bar{G}=\M$ and one deduces easily that whole net $(G_{\re})_{\re}$ is converging to~$\M$. This proves \eqref{eq:nonquant}.

\smallskip
\noindent 
\textit{Second step: nonlinear estimate.} We first consider the Maxwellian 
$\M_{\re}$ with same mass, momentum and energy of $G_{\re}$ and we consider the linearized elastic collision operator around that Maxwellian
$$\mathbf{L}g=\Q_{1}(g,\M_{\re})+\Q_{1}(\M_{\re},g), \qquad g \in \W^{k,1}_{v}(\m_{q}).$$  
One simply notices that, since $\Q_{1}(\M_{\re},\M_{\re})=0$,
\begin{multline*}
\mathbf{L}(G_{\re})=\Q_{1}(G_{\re}-\M_{\re},\M_{\re} - G_{\re}) + \Q_{\re}(G_{\re},G_{\re}) + \bigg[\Q_{1}(G_{\re},G_{\re}) -\Q_{\re}(G_{\re},G_{\re})\bigg]\\
= \Q_{1}(G_{\re}-\M_{\re},\M_{\re}-G_{\re}) - (1-\re)\nabla_v\cdot (vG_{\re})+ \bigg[\Q_{1}(G_{\re},G_{\re}) -\Q_{\re}(G_{\re},G_{\re})\bigg].
\end{multline*}
Therefore, using classical estimates for $\Q_{1}$ (see {\cite{ACG,AG}}) 
\begin{multline*} \|\mathbf{L}(G_{\re}) \|_{\W^{k,1}_{v}(\m_{q})} \leq \|\Q_{1}(G_{\re}-\M_{\re},\M_{\re} - G_{\re}) \|_{\W^{k,1}_{v}(\m_{q})} + (1-\re)\,\|G_{\re}\|_{\W^{k+1,1}_{v}(\m_{q+1})} \\+ 
\|\Q_{1}(G_{\re},G_{\re}) -\Q_{\re}(G_{\re},G_{\re})\|_{\W^{k,1}_{v}(\m_{q})}\\
\leq C_1\,\|G_{\re}-\M_{\re}\|_{\W^{k,1}_v(\m_{q+1})}^2 + C(1-\alpha)\|G_{\re}\|_{\W^{k+1,1}_{v}(\m_{q+1})} + C_1(1-\alpha)\,\|G_{\re}\|_{\W^{k,1}_{v}(\m_{q+2})}^2\end{multline*}
where we used Lemma \ref{lem:els} for estimating the difference $\Q_{1}(G_{\re},G_{\re}) -\Q_{\re}(G_{\re},G_{\re})$. Since $\sup_{\re}\|G_{\re}\|_{\W^{{k+1},1}_{v}
(\m_{q+2})}<\infty$, we obtain that there is a positive constant $C_{2} >0$ such that
$$\|\mathbf{L}(G_{\re})\|_{\W^{k,1}_{v}(\m_{q})} \leq C_{2}(1-\re) + C_{2}\|G_{\re}-\M_{\re}\|_{\W^{k,1}_{v}(\m_{q+1})}^2.$$
We can write $\mathbf{L}(G_{\re})=\mathbf{L}(G_{\re}-\M_{\re})$ and, as $G_{\re}-\M_{\re}$ has zero mass, momentum and energy, there is a positive constant $c >0$ (that can be taken independent of $\re$) such that 
$$\|\mathbf{L}(G_{\re}-\M_{\re})\|_{\W^{k,1}_{v}(\m_{q})}\geq c\|G_{\re}-\M_{\re}\|_{\W^{k,1}_{v}(\m_{q+1})}.$$
{Recall that the constant $c >0$ is actually the norm of the inverse of $\mathbf{L}$ on the subspace of functions with zero mass, momentum and energy; recall that this inverse maps $\W^{k,1}_{v}(\m_{q})$ into  $\mathscr{D}(\mathbf{L})=\W^{k,1}_{v}(\m_{q+1})$.} 
Therefore, with $C_{3}=C_{2}/c$
\begin{equation}\label{eq:nonli}
\|G_{\re}-\M_{\re}\|_{\W^{k,1}_{v}(\m_{q+1})} \leq C_{3}(1-\re) + C_{3}\|G_{\re}-\M_{\re}\|_{\W^{k,1}_{v}(\m_{q+1})}^2, \qquad \re \in (\re_{0},1).
\end{equation}
\noindent \textit{Third step: conclusion}. Setting
$$\en_{\re}:=\frac{1}{d}\int_{\R^{d}}|v|^{2}\,\M_{\re}(v)\d v=\frac{1}{d}\int_{\R^{d}}|v|^{2}\,G_{\re}(v)\d v,$$
one sees easily from \eqref{appe1} that $|\en_{1}-\en_{\re}| \leq C(1-\re)$ and then, one can check without difficulty that there is some positive constant $C_{k,q} >0$ such that
\begin{equation}\label{eq:MaM1}
\left\|\M_{\re}-\M\right\|_{\W^{k,1}_{v}(\m_{q})} \leq C_{k,q}(1-\re)\,, \qquad  \re \in  {[\re_{2},1]}.
\end{equation}
Thanks to \eqref{eq:nonquant}, we can then find $ {\re_{3} }> \re_{0}$ such that
$$C_{3}\|G_{\re}-\M_{\re}\|_{\W^{k,1}_{v}(\m_{q+1})} \leq \frac{1}{2}$$
where $C_{3}>0$ is the positive constant in \eqref{eq:nonli}. Then, \eqref{eq:nonli} reads simply as
$$\|G_{\re}-\M_{\re}\|_{\W^{k,1}_{v}(\m_{q+1})} \leq 2C_{3}(1-\re)\,, \qquad \re \in [ {\re_{3} },1]\,,$$
and, using \eqref{eq:MaM1}, we end up with
$$\|G_{\re}-\M\|_{\W^{k,1}_{v}(\m_{q+1})} \leq C(1-\re)\,, \qquad \re \in [ {\re_{3} },1]\,,$$
which gives also a quantitative lower bound on $ {\re_{3} }.$

The estimate on the $\W^{k,2}_v$-norm actually comes from the $\W^{k,1}_v$-one and some Sobolev embedding. Indeed for $k'>d$, one can write that 
\begin{align*}
\|G_\re-\M_\re\|_{\W^{k,2}_{v}(\m_{q})} &\leq C_k \|G_\re-\M_\re\|_{\W^{k,1}_{v}(\m_{q})}^{1/2} \|G_\re-\M_\re\|_{\W^{k,\infty}_{v}(\m_{q})}^{1/2} \\
&\leq C_k \|G_\re-\M_\re\|_{\W^{k,1}_{v}(\m_{q})}^{1/2} \|G_\re-\M_\re\|_{\W^{k+k',1}_{v}(\m_{q})}^{1/2}
\end{align*}
which allows us to conclude thanks to the first part of the proof. 
\end{proof}

{Finally, we state  Theorem 1 of \cite{ACG} (see also \cite{AG}) followed by an immediate corollary:
\begin{theo}[Theorem 1, {\cite{ACG}}] \label{theo:Ricardo}
Consider $q \geq 1$ and $r \in [1,\infty]$ such that condition~\eqref{eq:conditionb} is satisfied. For any $f \in L^r_v(\m_{q})$ and $g \in L^{1}_{v}(\m_{q})$,
$$\|\Q_{\re}^{+}(f,g)\|_{L^{r}_{v}(\m_{q-1})} \leq C_{r}(b)\|f\|_{L^{r}_{v}(\m_{q})}\|g\|_{L^{1}_{v}(\m_{q})},$$
and
$$\|\Q_{\re}^{+}(g,f)\|_{L^{r}_{v}(\m_{q-1})}\leq \tilde{C}_{r}(b)\|f\|_{L^{r}_{v}(\m_{q})}\,\|g\|_{L^{1}_{v}(\m_{q})}$$
with 
$$C_{r}(b)=2^{\frac{q+1}{2}+\frac{1}{r}+\frac{d}{r'}}|\S^{d-2}|\int_{-1}^{1}\left(\frac{1-s}{2}\right)^{-\frac{d}{2r'}}b(s)\left(1-s^{2}\right)^{\frac{d-3}{2}}\d s$$
and
$$\tilde{C}_{r}(b)=2^{\frac{q+1}{2}+\frac{1}{r}}|\S^{d-2}|\int_{-1}^{1}\left(\frac{1+s}{2}+\frac{(1-\re)^{2}}{4}\frac{1+s}{2}\right)^{-\frac{d}{2r'}}b(s)\left(1-s^{2}\right)^{\frac{d-3}{2}}\d s$$
where $\frac{1}{r}+\frac{1}{r'}=1$.
\end{theo}
\begin{cor} \label{cor:Ricardo}
Consider $q \geq 1$ and $r \in [1,\infty]$ such that condition~\eqref{eq:conditionb} is satisfied. Then there exists $C(r,q) >0$ such that
$$\sup_{\re\in (0,1]}\left(\left\|\Q_{\re}^{+}(f,\M)\right\|_{L^{r}_{v}(\m_{q-1})} + \left\|\Q_{\re}^{+}(\M,f)\right\|_{L^{r}_{v}(\m_{q-1})}\right) \leq C(r,q)\|f\|_{L^{r}_{v}(\m_{q})}$$
and
$$\sup_{\re\in (0,1]}\left(\left\|\Q_{\re}^{+}(f,G_{\re}-\M)\right\|_{L^{r}_{v}(\m_{q-1})} + \left\|\Q_{\re}^{+}(G_{\re}-\M,f)\right\|_{L^{r}_{v}(\m_{q-1})}\right) \leq C(r,q)\|f\|_{L^{r}_{v}(\m_{q})}.$$
\end{cor}} 
\subsection{Dissipativity properties}
We finally give the remaining part of the proof of Proposition~\ref{prop:hypo}.

\begin{proof}[Proof of Proposition \ref{prop:hypo}] We need to prove the result in the space $\W^{s,1}_{v}\W^{\ell,2}_{x}(\m_{q})$.
The proof follows the same lines as the one presented earlier and in particular, we write again 
$\mathcal{B}_{\re,\e}^{(\delta)}(h)=\sum_{i=0}^{3}C_{i}(h)$. As in the $\W^{s,2}_{v}\W^{\ell,2}_{x}(\m_{q})$ case, there is no loss of generality in assuming $\ell=s$. 

\noindent$\bullet$ Assume first $\ell=0$. With obvious notations,
{\begin{equation*}
\int_{\R^d} \|h(\cdot,v)\|_{L^2_x}^{-1}\left(\int_{\T^d}  \mathcal{B}_{\re,\e}^{(\delta)}(h)(x,v)  h(x,v) \, \d x\right) \, \m_{q}(v)\, \d v  
 =:\sum_{i=0}^{3}I_{i}(h).
 \end{equation*}} {First, $I_1(h)=0$ since 
$$
\int_{\T^d} \left(v \cdot \nabla_{x} h(x,v)\right) h(x,v) \, \d x = \frac12 \int_{\T^d} v \cdot \nabla_x h^2(x,v) \, \d x = 0.  
$$
According to~\eqref{eq:B1delta}, by taking $\delta$ small enough so that $\Lambda_q^{(1)}(\delta)<1$ (which is possible since $q>2$), we have
$$I_{0}(h) \leq \e^{-2} \sigma_0\left(\Lambda_{q}(\delta)-1\right)\|h\|_{L^{1}_{v}L^{2}_{x}(\m_{q+1})}.$$
Moreover, it follows from Cauchy-Schwarz inequality and Corollary~\ref{cor:Palpha} that 
$$
I_2(h) \leq 
\e^{-2} \int_{\R^d} \|\P_{\re} h (\cdot,v)\|_{L^2_x}  \m_q(v) \, \d v \leq \e^{-2} \widetilde{\bm C}_{0,q} (1-\alpha)^{p \over 2} \|h\|_{L^1_vL^2_x(\m_{q+1})}.
$$
Finally, for $I_3$, one can compute 
\begin{align*}
&\int_{\R^d}  \|h(\cdot,v)\|_{L^2_x}^{-1}  \int_{\T^d} \nabla_v\cdot (vh(x,v)) \, h(x,v)   \, \d x  \, \m_q(v) \, \d v
 \\
&\quad= d \int_{\R^d}   \|h(\cdot,v)\|_{L^2_x}  \, \m_q(v)\, \d v +  {1 \over 2} \int_{\R^d}  \|h(\cdot,v)\|_{L^2_x}^{-1} \int_{\T^d} v \cdot \nabla_v h^2(x,v)    \, \d x  \, \m_q(v) \, \d v
 \\
 &\quad = d \int_{\R^d}   \|h(\cdot,v)\|_{L^2_x}  \, \m_q(v)\, \d v +  {1 \over 2} \int_{\R^d}  \|h(\cdot,v)\|_{L^2_x}^{-1} \, v \cdot \nabla_v \|h(\cdot,v)\|_{L^2_x}^2 \, \m_q(v) \, \d v \\
&\quad = -  \int_{\R^d} \|h(\cdot,v)\|_{L^2_x} v \cdot \nabla_v \m_q (v) \, \d v .
\end{align*}
Since $v \cdot \nabla_{v} \m_{q}(v)=q \m_{q}(v)-q\m_{q-2}(v)$  we get
\begin{equation} \label{eq:I3}
I_{3}(h) \leq q\kappa_{\re}\e^{-2} \|h\|_{L^{1}_{v}L^{2}_{x}(\m_{q+1})}.
\end{equation}
Gathering the previous estimates, one obtains 
\begin{equation}\label{eq:Bal}
\begin{split}
&\mathcal{I}:=\int_{\R^d} \|h(\cdot,v)\|_{L^2_x}^{-1} \left(\int_{\T^d}  \mathcal{B}_{\re,\e}^{(\delta)}(h)(x,v)  h(x,v) \, \d x\right) \, \m_{q}(v)\, \d v \\
&\qquad \leq \e^{-2}\,\left(\widetilde{\bm{C}}_{0,q}(1-\re)^{p \over 2} + \sigma_0(\Lambda_{q}(\delta)-1) + q\kappa_{\re}\right)\|h\|_{L^{1}_{v}L^2_{x}(\m_{q+1})}.
\end{split}
\end{equation}
Recalling that $\kappa_{\re}=1-\re$ while $\lim_{\delta\to 0}(\Lambda_{q}(\delta)-1)=-\frac{q-2}{q+2} <0$ we can pick $\delta_{1,0,0,q}^{\dagger}$ small enough and then $\re^{\dagger}_{1,0,0,q}\in(0,1)$ close enough to $1$ so that
  $$\nu_{1,0,0,q}:= -\inf\left\{\widetilde{\bm{C}}_{0,q}(1-\re)^{p \over 2} + \sigma_0(\Lambda_{q}(\delta)-1) + q\kappa_{\re}\,;\,\re \in (\re_{1,0,0,q}^{\dagger},1),\, \delta \in (0,\delta_{1,0,0,q}^{\dagger})  \right\} >0 $$
and get the result. }
\smallskip  

\noindent
{Let us now consider the case $\ell=1$ and introduce the norm 
$$\vertiii{h}= \|h\|_{L^{1}_{v}L^2_{x}(\m_{q})}+   \|\nabla_{x} h\|_{L^{1}_{v}L^2_{x}(\m_{q})} +\eta\,\|\nabla_{v} h\|_{L^{1}_{v}L^2_{x}(\m_{q})}  , $$
for some $\eta>0$, the value of which shall be fixed later on. This norm is equivalent to the classical $\W^{1,1}_v\W^{1,2}_{x} (\m_{q})$-norm. We shall prove that for some $\nu_{1,1,1,q}>0$, $ \mathcal{B}_{\re,\e}^{(\delta)}+\e^{-2}\nu_{1,1,1,q}$ is dissipative in $\W^{1,1}_v\W^{1,2}_{x} (\m_{q})$ for the norm $\vertiii{\cdot}$.  
Notice first that the $x$-derivative commutes with all the above terms $C_{i}(h)$, $i=0,\ldots,3$, i.e.
$$\nabla_{x} \mathcal{B}_{\re,\e}^{(\delta)}h(x,v)= \mathcal{B}_{\re,\e}^{(\delta)}\nabla_{x}h(x,v)$$
so that, according to the previous step
\begin{align}\label{eq:Baldx}
\begin{split}
\mathcal{J}_{x}:&=\int_{\R^d} \|\nabla_x h(\cdot,v)\|_{L^2_x}^{-1} \left(\int_{\T^d}  \nabla_x \,\mathcal{B}_{\re,\e}^{(\delta)}( h)(x,v) \cdot \nabla_x h(x,v) \, \d x\right) \, \m_{q}(v)\, \d v\\
&\leq -\e^{-2} \nu_{1,0,0,q}\|\nabla_{x}h\|_{L^{1}_{v}L^{2}_{x}(\m_{q})}.
\end{split}
\end{align}
Consider now the quantity
$$\mathcal{J}_{v}:=\int_{\R^d} \|\nabla_v h(\cdot,v)\|_{L^2_x}^{-1} \left(\int_{\T^d}  \nabla_v \, \mathcal{B}_{\re,\e}^{(\delta)}( h)(x,v) \cdot \nabla_v h(x,v) \, \d x\right) \, \m_{q}(v)\, \d v.$$
Using the notations above, one notices that $\nabla_{v} C_{1}(h)=-\e^{-1} \nabla_{x} h + C_{1}(\nabla_{v}h)$, so that
\begin{align}\label{eq:nabB}
\begin{split}
\nabla_{v} ( \mathcal{B}_{\re,\e}^{(\delta)} h(x,v))&=\e^{-2} \nabla_{v} (\mathcal{B}_{1}^{(\delta)}h) -\e^{-1} \nabla_{x}h + C_{1}(\nabla_{v}h) \\
&\hspace{4cm}+ \e^{-2} \nabla_{v}(\P_{\re}h) + \e^{-2} \nabla_{v} (T_{\re}h)\\
&=\e^{-2} \nabla_{v} [\mathscr{L}_{1}^{R,\delta}h -\Sigma_\M\,h ]  -\e^{-1} \nabla_{x}h + C_{1}(\nabla_{v}h) \\
&\hspace{4.5cm}+\e^{-2} \nabla_{v} (\P_{\re}h) + \e^{-2} \nabla_{v} (T_{\re}h).
\end{split}
\end{align} 
Then, it follows from Corollary~\ref{cor:Palpha} that 
\begin{equation}\label{eq:nab}
\|\nabla_{v} (\P_{\re} h) \|_{L^{1}_{v}L^2_x(\m_{q})} \leq \widetilde{\bm{C}}_{1,q}(1-\re)^{p \over 2}\left(\|h\|_{L^{1}_{v}L^2_x(\m_{q+1})}+\|\nabla_{v} h\|_{L^{1}_{v}L^2_x(\m_{q+1})}\right)\,.\end{equation}
Now, 
$$\nabla_{v} [\mathscr{L}_{1}^{{R,\delta}}h -\Sigma_\M  h ]
= \mathscr{L}_{1}^{R,\delta}(\nabla_{v} h) -\Sigma_\M  \nabla_{v} h 
+{\mathcal R} (h) , $$
where 
$${\mathcal R}( h )=\Q_1(h,\nabla_{v} \M)+ \Q_1(\nabla_{v} \M,h) 
- (\nabla_{v}{\mathcal A}^{(\delta)})(h) - {\mathcal A}^{(\delta)}(\nabla_{v} h). $$
From the proof of~\cite[Lemma~4.14]{GMM}, we have that
$$\| {\mathcal R}( h )\|_{L^1_{v}L^2_x(\m_{q})}  
\leq  C_\delta \|h\|_{L^1_{v}L^2_x(\m_{q+1})}. $$
From~\cite[Lemma~4.12]{GMM}, we also have that 
$$\|\mathscr{L}_{1}^{{R,\delta}}(\nabla_{v} h)\|_{L^1_{v}L^2_x(\m_{q})}
\leq \Lambda_q^{(1)}(\delta) \|\nabla_{v} h\|_{L^1_{v}L^2_x( \m_{q}\Sigma_\M)},$$
where $\Lambda_q^{(1)}(\delta)$ was introduce in Lemma~\ref{prop:hypo1}.  Then, one has that
\begin{equation*}
\begin{split}
- \int_{\R^d} \|\nabla_v h(\cdot,v)\|_{L^2_x}^{-1} \left(\int_{\T^d}  \Sigma_\M(v) \, |\nabla_v h(x,v)|^2 \, \d x\right) \, \m_{q}(v)\, \d v
& \leq  -\|\nabla_{v} h\|_{L^1_{v}L^2_x(\m_{q}\Sigma_\M)}.
\end{split}
\end{equation*}
Therefore, recalling that $\delta$ is such that $\Lambda_q^{(1)}(\delta)<1$ and that $\Sigma_\M(v) \geq \sigma_0 \langle v \rangle$, we get 
\begin{equation}\label{eq:nabL0}
\|\nabla_{v} [\mathscr{L}_{1}^{{R,\delta}}(h) -\Sigma_\M\, h ]\|_{L^{1}_{v}L^2_x(\m_{q})} \leq C_\delta	\|h\|_{L^{1}_{v}L^2_x(\m_{q+1})}+\sigma_0\left(\Lambda_q^{(1)}(\delta)-1\right) \|\nabla_{v} h\|_{L^1_{v}L^2_x(\m_{q+1})}.\end{equation}
Finally, using the short-hand notation 
$$
\nabla_v \cdot (v \, \nabla_v h) =\big(\nabla_v \cdot (v \, \partial_{v_1} h),\cdots,\nabla_v \cdot (v \,\partial_{v_d} h)\big),
$$
we have
$$
\nabla_v h \cdot \nabla_v \left(\nabla_v \cdot (vh)\right) = |\nabla_v h|^2 + \left[\nabla_v \cdot (v\, \nabla_vh)\right] \cdot \nabla_v h. 
$$
Doing similar computations as the ones leading to~\eqref{eq:I3}, we obtain:
\begin{equation}\label{eq:div}
\begin{aligned} 
&\int_{\R^d} \|\nabla_v h(\cdot,v)\|_{L^2_x}^{-1} \left(\int_{\T^d}  \nabla_v \, T_\re( h)(x,v) \cdot \nabla_v h(x,v) \, \d x\right) \, \m_{q}(v)\, \d v
\\
&\hspace{8cm} \leq (q-1) \kappa_{\re} \|\nabla_v h\|_{L^1_vL^2_x(\m_{q+1})}.
\end{aligned}
\end{equation}
Coming back to~\eqref{eq:nabB}, Cauchy-Schwarz inequality and estimates \eqref{eq:nab},~\eqref{eq:nabL0} and~\eqref{eq:div} give that  
\begin{multline*}
\mathcal{J}_{v} \leq \e^{-2} (C_\delta +\widetilde{\bm{C}}_{1,q}(1-\re)^{\frac{p}{2}} )\| h\|_{L^{1}_{v}L^2_{x}(\m_{q+1})} + \e^{-1} \|\nabla_x h\|_{L^1_vL^2_x(\m_q)}
\\ +\e^{-2}(\widetilde{\bm{C}}_{1,q}(1-\re)^{p \over 2}+\,C\kappa_{\re}+\sigma_0(\Lambda_q^{(1)}(\delta)-1)) \|\nabla_{v} h\|_{L^{1}_{v}L^2_{x}(\m_{q+1})}\,,
\end{multline*}
where we used that the contribution to $\mathcal{J}_{v}$ of the term $C_{1}(\nabla_{v}h)$ vanishes. 
Hence, combining this estimate with \eqref{eq:Bal} and \eqref{eq:Baldx} and using that $\e \leq 1$, it follows that
\begin{multline*}
\mathcal{I} + \mathcal{J}_{x}+ \eta\,\mathcal{J}_{v} 
\leq \e^{-2} \bigg(\left[-\nu_{1,0,0,q} + \eta\left(C_{\delta}+\widetilde{\bm{C}}_{1,q}(1-\re)^{p \over 2}\right)\right]\|h\|_{L^{1}_{v}L^{2}_{x}(\m_{q+1})} \\ 
-(\nu_{1,0,0,q}- \eta)\|\nabla_{x}h\|_{L^{1}_{v}L^{2}_{x}(\m_{q+1})} + \eta\left[\widetilde{\bm{C}}_{1,q}(1-\re)^{p \over 2}+C\kappa_{\re}+\sigma_0(\Lambda_q^{(1)}(\delta)-1)\right]\|\nabla_{v} h\|_{L^1_{v}L^{2}_{x}(\m_{q+1})}\bigg).
\end{multline*}
Consequently, there exists $\re^{\dagger}_{1,1,q} >0$ and $\delta_{1,1,q}^{\dagger} >0$ so that
$$\bm{C}_{1,q}(1-\re)+C\kappa_{\re}+\sigma_0(\Lambda_q^{(1)}(\delta)-1) <0 \qquad \forall \re \in (\re^{\dagger}_{1,1,q},1), \:\:\delta \in (0,\delta^{\dagger}_{1,1,q}).$$ 
Choosing $\eta>0$ small enough such that 
$$
\nu_{1,0,0,q}-\eta\, \max\bigg(1,\sup_{\delta \in (0,\delta^\dagger_{1,1,q})}C_\delta +\widetilde{\bm{C}}_{1,q}(1-\re^\dagger_{1,1,q})^{p\over 2}\bigg) >0,
$$ 
we finally obtain that there exists $\nu_{1,1,1,q}>0$ such that for $\re \in (\re^{\dagger}_{1,1,q},1)$ and $\delta \in (0,\delta^{\dagger}_{1,1,1,q})$,
\begin{equation*}\begin{split}
\mathcal{I} + \mathcal{J}_{x}+ \eta\,\mathcal{J}_{v}
&\leq  -\e^{-2} \nu_{1,1,1,q}\left [\|h\|_{L^1_{v}L^{2}_{x}(\m_{q+1})} + \|\nabla_{x}h\|_{L^{1}_{v}L^{2}_{x}(\m_{q+1})}
+ \eta\,  \|\nabla_{v} h\|_{L^1_{v}L^{2}_{x}(\m_{q+1})}\right]\\
&\leq -\e^{-2}\nu_{1,1,1,q} \vertiii{h}.
\end{split}\end{equation*}
This proves that ${\mathcal B}_{\re,\e}^{(\delta)}+\e^{-2} \nu_{1,1,1,q}$ is hypo-dissipative in $\W^{1,1}_{v}\W^{1,2}_x(\m_{q})$. We prove the result for higher order derivatives in the same way considering now the norm 
$$\vertiii{h}=\sum_{|\bm{\beta}_{1}|+|\bm{\beta}_{2}| \leq k}\eta^{|\bm{\beta}_{1}|}\left\|\nabla^{|\bm{\beta}_{1}|}_{v}\nabla_{x}^{|\bm{\beta}_{2}|}h\right\|_{L^{1}_{v}L^{2}_{x}(\m_{q})}$$
for some $\eta>0$ to be chosen sufficiently small.}
\end{proof}

\section{About the original problem in physical variables}\label{Appendix-PP}
Let $F_{\e}(t,x,v)$ be the solution of the Boltzmann equation \eqref{Bol-e} with associated Knudsen number $\e$.  Recall that the time-scale functions $\tau_{\e}(t),V_{\e}(t)$ that relate the problem in original (physical) variables to its self-similar counterpart 
$$F_{\e}(t,x,v) = V_{\e}(t)^{d}f_{\e}\big(\tau_{\e}(t),x,V_{\e}(t)v\big)$$
are given by
$$ \tau_{\e}(t):=\frac{1}{c_{\e}}\ln(1+ c_{\e}\,t)\,, \quad V_{\e}(t)=1+c_{\e}\,t\,,\qquad t \geq0,$$
where $c_{\e}=\frac{1-\re({\e})}{\e^2}$.   It follows that the explicit equation for $f_{\e}$ is given by
\begin{equation*}
\partial_{\tau}f_{\e}+ \e^{-1}w\cdot\nabla_{x}f_{\e} = \e^{-2} \Q(f_{\e},f_{\e}) - \e^{-2}(1-\re)\nabla_{w}(w\,f_{\e})\,,\quad w = V_{\e}(t)\,v\,
\end{equation*}
as observed in \eqref{BE}. Set $f_{\varepsilon}(\tau,x,w)=G_{\re(\e)}(w)+\varepsilon \,h_{\varepsilon}(\tau,x,w)$ and denote $\bm{h}(\tau,x,w)$ the weak$-\star$ limit in the space $L^{\infty}\big((0,\infty);\E\big)$ of the (sub-)sequence $\{h_{\e}\}$. 
Define
$$e_{\e}(t,x,v)={V}_{\e}(t)^{d}\,\big( h_{\e}(\tau_{\e}(t),x,V_{\e}(t)v) - \bm{h}(\tau_{\e}(t),x,V_{\e}(t)v) \big)\,.$$ 
The following error estimate holds.
\begin{lem}\label{lemma-app-pp}
Under Assumption \ref{hyp:re} and in the regime $\e \ll 1\,$, $\lambda_{0}>0$, the following estimation holds for any $a\in(0,1/2)$, up to possibly extracting a subsequence,
\begin{equation*}
\big| \langle e_{\e}(t),|v|^{\kappa}\varphi\rangle \big| \leq C_{\varphi}\sqrt{\mathcal{K}_0}\,V_{\e}(t)^{-\kappa-a}\,,\qquad \varphi\in  \mathcal{C}^{1}_{v,b}L^{\infty}_{x}\,,\;\; 0\leq\kappa \leq q - 1\,,
\end{equation*}
where we denoted by $\mathcal{C}^{1}_{v,b}$ the set of $\mathcal{C}^{1}$ functions in $v$ that are bounded as well as their derivatives and where for any $t\geq0$,
$$V_{\e}(t) \approx V_{0}(t)=\big( 1+\lambda_0\,t \big)\,, \qquad \text{ as } \;\e \approx 0\,.$$
\end{lem}
\begin{proof}
After a change of variables it follows that, for any test-function $\varphi$,
\begin{align*}
&\langle e_{\e}(t),|v|^{\kappa}\varphi\rangle \\
&\quad=V_{\e}(t)^{-\kappa}\int_{\T^d \times \R^d} \Big( h_{\e}(\tau_{\e}(t),x,v) - \bm{h}(\tau_{\e}(t),x,v) \Big)|v|^{\kappa}\big(\varphi(x,V_{\e}(t)^{-1}v)-\varphi(x,0)\big)\,\d v\,\d x\\
&\qquad+V_{\e}(t)^{-\kappa}\int_{\T^d \times \R^d} \Big( h_{\e}(\tau_{\e}(t),x,v) - \bm{h}(\tau_{\e}(t),x,v) \Big)|v|^{\kappa}\varphi(x,0)\,\d v\,\d x\\
&\quad=\mathcal{I}_{1}(t) + \mathcal{I}_{2}(t)\,.
\end{align*}
Note that, up to a subsequence, $\bm{h}$ is the weak$-\star$ limit of $\{h_{\e}\}_{\e}$ in $L^{\infty}\big((0,\infty);\E\big)$.  Thus, {for any $t >0$,
$\|\bm{h}\|_{L^{\infty}((t,\infty)\,;\,\E)} \leq \liminf_{\e\searrow0}\|h_{\e}\|_{L^{\infty}((t,\infty)\,;\,\E)}.$
Consequently, thanks to Theorem~\ref{theo:h-relaxation}, it holds that
\begin{equation}\label{tiempito}
\|\bm{h}\|_{L^{\infty}((t,\infty)\,;\,\E)} \leq C\,\sqrt{\mathcal{K}_0}\,e^{-\frac{\lambda_{\e}}{2}t}\,,\qquad t>0\,.
\end{equation}}
As a consequence, recalling that $\lambda_{\e} \simeq c_{\e}$, it follows that for any $a\in(0,1/2)$
\begin{align*}
&\Big|\int_{\mathbb{T}^{d}\times \mathbb{R}^{d} }\Big( h_{\e}(\tau,x,v) - \bm{h}(\tau,x,v) \Big)|v|^{\kappa}\d v\,\d x \Big| \\
&\qquad\leq \| h_{\e}(\tau_{\e}(t)) - \bm{h}(\tau_{\e}(t)) \|_{ L^{1}_{v,x}(\m_{q}) } \leq C\,\sqrt{\mathcal{K}_0}\,e^{-\frac{\lambda_{\e}}{2}\tau_{\e}(t)}\leq C \, V_{\e}(t)^{- a}\, \sqrt{\mathcal{K}}_0\,,
\end{align*}
{where we used once again that $\|\cdot\|_{L^{1}_{v}(\m_{q})} \lesssim \|\cdot\|_{L^{2}_{v}(\m_{q+\kappa})}$ for $\kappa >\frac{d}{2}$.}
Now, in regard of $\mathcal{I}_{1}(t)$, note that
$$\big| \varphi(x,V_{\e}(t)^{-1}v) - \varphi(x,0)\big | \leq V_{\e}(t)^{-1}|v| \sup_{x}\sup_{v} \big|\partial_{v}\varphi(x,v) \big|=C_{\varphi}V_{\e}(t)^{-1}|v|\,,$$
so that the following holds:
\begin{align*}
\big|\mathcal{I}_{1}(t)\big| &\leq {C}_{\varphi}V_{\e}(t)^{-\kappa-1}\| h_{\e}(\tau_{\e}(t)) - \bm{h}(\tau_{\e}(t)) \|_{ L^{1}_{v,x}(\m_{q}) } \leq C_{\varphi} \, V_{\e}(t)^{-\kappa - 1 - a}\, \sqrt{\mathcal{K}}_0\,.
\end{align*}
Similarly,
\begin{align*}
\Big|\mathcal{I}_{2}(t) \Big| \leq  C_{\varphi} \, V_{\e}(t)^{-\kappa - a}\, \sqrt{\mathcal{K}}_0\,,
\end{align*}
which proves the desired estimate.
\end{proof}

The above computations also allow to provide a \textbf{\textit{local version of Haff's Law}}. Namely, note that
\begin{align*}
\int_{\mathbb{R}^{d}}f_{\e}&(\tau_{\e}(t),x,w)|w|^{\kappa}\d w \\
&= \int_{\mathbb{R}^{d}}G_{\re}(w)|w|^{\kappa}\d w + \e\,\int_{\mathbb{R}^{d}}h_{\e}(\tau_{\e}(t),x,w)|w|^{\kappa}\d w\,,\qquad 0 \leq \kappa \leq q\,.
\end{align*}
Thanks to Sobolev embedding it holds that
\begin{equation*}
\bigg| \sup_{x\in\mathbb{T}^{d}}\int_{\mathbb{R}^{d}}h_{\e}(\tau_{\e}(t),x,w)|w|^{\kappa}\d v \bigg| \leq C_{\kappa}\| h_{\e}( \tau_{\e}(t) ) \|_{\E} \leq  C_{\kappa}\sqrt{K_0}\,.
\end{equation*}
Therefore, for sufficiently small $\e>0$ there exists two positive constants $C_{\kappa}$ and $c_{\kappa}$ such that
\begin{equation*}
c_{\kappa}\leq\int_{\mathbb{R}^{d}}f_{\e}(\tau_{\e}(t),x,w)|w|^{\kappa}\d w \leq C_{\kappa}\,,\qquad 0 \leq \kappa \leq q\,,\;\, t\geq0\,,
\end{equation*}
which leads, for the physical problem, to
\begin{equation*}
V_{\e}(t)^{-\kappa}c_{\kappa}\leq\int_{\mathbb{R}^{d}}F_{\e}(t,x,v)|v|^{\kappa}\d v \leq V_{\e}(t)^{-\kappa}C_{\kappa}\,,\qquad  0 \leq \kappa \leq q\,,\;\, t\geq0\,.
\end{equation*}
In particular, this estimate renders a local version of Haff's law
\begin{equation*}
\int_{\mathbb{R}^{d}}F_{\e}(t,x,v)|v|^{2}\d v \sim \big( 1 + c_{\e}\, t \big)^{-2} \,, \qquad \forall t \geq0, \quad x \in \T^{d}.
\end{equation*}

\section{Tools for the Hydrodynamic limit}\label{sec:hydro1}
We collect several tools that are used in Section \ref{sec:hydro} to derive the modified incompressible Navier-Stokes system.   Various known computations regarding the elastic Boltzmann operator are needed.   As in the classical case, we introduce the traceless tensor
$$\bm{A}(v)=v\otimes v - \frac{1}{d}|v|^{2}\mathbf{Id}\,.$$
Notice that that \eqref{eq:vnah}  can be rewritten thanks to \eqref{eq:incomp} as
$$v \cdot \nabla_{x}\bm{h}\,=\,\bm{A}(v)\M(v):\nabla_{x}u+\bm{b}(v)\M(v)\,\cdot \nabla_{x}\,\vE\,,$$
with 
$$\bm{b}(v)=\frac{1}{2}\left(|v|^{2}-(d+2)\en_{1}\right)v \in\R^{d}.$$
\begin{lem}\label{lem:phipsi} One has that $\bm{A},\bm{b} \in \mathrm{Range}(\mathbf{I}-\bm{\pi}_{0})$ and there exists two radial functions $\chi_{i}=\chi_{i}(|v|)$, $i=1,2$, such that
$$\phi(v)=\chi_{1}(|v|)\bm{A}(v) \in \mathscr{M}_{d}(\R)\,, \quad \text{and} \quad \psi(v)=\chi_{2}(|v|)\bm{b}(v) \in \R^{d}\,,$$
satisfy
\begin{equation}\label{eq:L1chi}
\mathbf{L}_{1}(\phi\,\M)=-\bm{A}\,\M\,,\qquad \mathbf{L}_{1}(\psi\,\M)=-\bm{b}\,\M\,.
\end{equation}
Moreover,
\begin{multline}\label{eq:visco}
\la \phi^{i,j}\mathbf{L}_{1}(\phi^{k,\ell}\M)\ra=-{\bm \nu}\left(\delta_{ik}\delta_{j\ell}+\delta_{i\ell}\delta_{jk}-\frac{2}{d}\delta_{ij}\delta_{kl}\right)\\
\la \psi_{i}\mathbf{L}_{1}(\psi_{j}\M)\ra=-\frac{d+2}{2}\gamma\,\delta_{ij}, \qquad i,j,k,\ell \in \{1,\ldots,d\}\,,\end{multline}
with
$${\bm \nu}:=-\frac{1}{(d-1)(d+2)}\la \phi\,:\,\mathbf{L}_{1}(\phi\M)\ra \geq 0, \qquad \gamma:=-\frac{2}{d(d+2)}\la \psi \cdot \mathbf{L}_{1}(\psi\M)\ra \geq 0.$$
{Finally, 
$$\phi^{i,j}(v) \lesssim \m_{3}(v), \qquad \psi_{i}(v) \lesssim \m_{4}(v), \qquad i,j \in \{1,\ldots,d\}.$$}
\end{lem}
\begin{proof} The tensor $\bm{A}$ and the vector $\bm{b}$ satisfy
\begin{equation}\label{eq:ortho}
\la \bm{A}^{k,\ell}\Psi_{i}\,\M \ra=0, \quad \la \bm{b}\Psi_{i}\,\M\ra=0\,, \quad \forall i=1,\ldots,d+2, \qquad k,\ell \in \{1,\ldots,d\}\,,
\end{equation}
from which we get that $\bm{A},\bm{b} \in \mathrm{Range}(\mathbf{I}-\bm{\pi}_{0})$. We refer to {\cite{DesGol}} and {\cite{BaGoLe2}} for the proof of the second part of the Lemma, just mind that the linearized Boltzmann operator considered in such references is defined as $Lg=-\M^{-1}\mathbf{L}_{1}(\M\,g)$.
We refer to \cite[Lemma 4.4]{BaGoLe2} for the proof of \eqref{eq:visco}. {We refer to \cite[Proposition 6.5]{golseSR2}
for the last estimates on $\phi^{i,j}$ and $\psi$.}\end{proof}
\begin{nb} Notice that if $\zeta=\zeta(|v|)$ is radially symmetric, then
$$\la \zeta\,\bm{A}^{i,j}\,\M\ra=\la \zeta\,\mathbf{L}_{1}(\phi\,\M)\ra=0, \qquad \forall\, i,j=1,\ldots,d.$$
\end{nb}
\begin{lem}\label{lem:Q1hh} For $\bm{h}$ given by \eqref{eq:hlim}, it  holds that
$$\la \phi\,\Q_{1}(\bm{h},\bm{h})\ra=\en_{1}^{2}\left(u \otimes u-\frac{2}{d}|u|^{2}\mathbf{Id}\right)\,,$$
for any $i,j=1,\ldots,d.$
\end{lem}
\begin{proof} As observed in \cite[Eq. (60)]{Cerci}, if $g\M \in \mathrm{Ker}(\mathbf{L}_{1})$ then $\Q_{1}(g\M,g\M)=-\frac{1}{2}\mathbf{L}_{1}(g^{2}\M)$. Therefore, with $g=\varrho+u\cdot v+\frac{1}{2}(|v|^{2}-2\en_{1})$,
\begin{equation}\label{eq:Q1sim}
\Q_{1}(\bm{h},\bm{h})=-\frac{1}{2}\mathbf{L}_{1}((u\cdot v)^{2}\M)-\frac{1}{8}\vE^{2}\mathbf{L}_{1}(|v|^{4}\M)+\vE\,u\cdot \mathbf{L}_{1}(\tfrac{1}{2}|v|^{2}v\M).
\end{equation}
One checks that 
$$\la \phi^{i,j}\mathbf{L}_{1}(|v|^{4}\M)\ra=0\,,$$
whereas $\mathbf{L}_{1}(\tfrac{1}{2}|v|^{2}v\M)=\mathbf{L}_{1}(\bm{b}\M)$, from which
$$\la \phi^{i,j}\mathbf{L}_{1}(\tfrac{1}{2}|v|^{2}v\M)\ra=\la \bm{b}\mathbf{L}_{1}(\phi^{i,j}\M)\ra=-\la \bm{b}\,\bm{A}^{i,j}\M\ra=0\,,$$
since $\bm{b}\,\bm{A}^{i,j}$ is an even function.  Therefore, we obtain that
\begin{equation}\label{eq:Q1hh}
\la \phi^{i,j}\Q_{1}(\bm{h},\bm{h})\ra=-\frac{1}{2}\sum_{k,\ell}u_{k}u_{\ell}\la \phi^{i,j}\mathbf{L}_{1}(v_{k}v_{\ell}\M)\ra\\
=\frac{1}{2}\sum_{k,\ell}u_{k}u_{\ell}\la v_{k}v_{\ell}\bm{A}^{i,j}\M\ra\,.\end{equation}
As for \eqref{eq:visco}, one checks that if $i \neq j$
$$\sum_{k,\ell}u_{k}u_{\ell}\la v_{k}v_{\ell}\bm{A}^{i,j}\M\ra=\sum_{\{k,\ell\}=\{i,j\}}u_{k}u_{\ell}\la v_{i}^{2}v_{j}^{2}\M\ra=2u_{i}u_{j}\la v_{i}^{2}v_{j}^{2}\M\ra\,,$$
whereas, for $i=j$,
$$\sum_{k,\ell}u_{k}u_{\ell}\la v_{k}v_{\ell}\bm{A}^{i,i}\M\ra=\sum_{k=1}^{d}u_{k}^{2}\left(\la v_{i}^{2}v_{k}^{2}\M\ra-\frac{1}{d}\la v_{k}^{2}|v|^{2}\M\ra\right)\,.$$
Notice that $a:=\la v_{i}^{2}v_{j}^{2}\M\ra$ is independent of $i,j$, thus, it is not difficult to check that 
$$(d-1)a=\frac{1}{d}\int_{\R^{d}}|v|^{4}\M\d v-\int_{\R^{d}}v_{1}^{4}\M(v)\d v=(d-1)\en_{1}^{2}\,,$$
that is, $a=\en_{1}^{2}$. In the same way, for any $k \in \{1,\ldots,d\}$
$$\la v_{k}^{2}|v|^{2}\M\ra=\frac{1}{d}\la |v|^{4}\M\ra=(d+2)\en_{1}^{2}\,,$$
whereas
$$\la v_{k}^{2}v_{i}^{2}\M\ra=
\begin{cases}
\qquad a=\en_{1}^{2} \quad &\text{ if }\; k\neq i\,,\\
\la v_{i}^{4}\M\ra=3\en_{1}^{2} \quad &\text{ if }\; k=i\,,
\end{cases}$$
so that,
$$\sum_{k,\ell}u_{k}u_{\ell}\la v_{k}v_{\ell}\bm{A}^{i,i}\M\ra=\en_{1}^{2}\sum_{k\neq i}u_{k}^{2}+3\en_{1}^{2}u_{i}^{2}-\frac{d+2}{d}|u|^{2}\en_{1}^{2}=2\en_{1}^{2}u_{i}^{2}-\frac{2}{d}\en_{1}^{2}|u|^{2}\,.$$
Gathering these last computations, we get
$$\la \phi^{i,j}\Q_{1}((u\cdot v)\M,(u\cdot v)\M)\ra =\en_{1}^{2}\left(u_{i}u_{j}-\frac{2}{d}|u|^{2}\delta_{i,j}\right)\,,$$
which, combined with \eqref{eq:Q1hh} gives the result.
\end{proof}
\begin{lem}\label{lem:phinu}
Let $\bm{h}$ be given by \eqref{eq:hlim}.   For any $i,j=1,\ldots,d$ it holds that
$$\la v_{\ell}\,\phi^{i,j}\,\bm{h}\ra =\begin{cases} \qquad\quad {\bm \nu}\,u_{j}  \quad &\text{ if }\; i \neq j, \;\, \ell=i,\\
\qquad\quad {\bm \nu}\,u_{i}  \quad &\text{ if }\; i \neq j, \;\, \ell=j,\\
- \frac{2}{d}{\bm \nu}\,u_{\ell}+2{\bm \nu}\,u_{i}\delta_{i\ell} \quad &\text{ if }\; i=j,\\
\qquad\quad 0 \quad &\text{ else}.
\end{cases}$$
\end{lem}
\begin{proof} 
Using the fact that $\chi_{1}$ is radial, similar computations to that of Lemma \ref{lem:Q1hh} imply that for $\ell \in \{1,\ldots,d\}$, 
\begin{equation*}
\begin{split}
\la v_{\ell}\,\phi^{i,j}\,\bm{h}\ra&=\sum_{k=1}^{d}u_{k}\la v_{\ell}v_{k}\,\phi^{i,j}\,\M\ra=\sum_{k=1}^{d}u_{k}\la v_{\ell}v_{k}\,\phi^{i,j}\,\M\ra\\
&=\sum_{k=1}^{d}u_{k}\left(\la \bm{A}^{k,\ell}\,\phi^{i,j}\,\M\ra+\frac{1}{d}\la |v|^{2}\,\phi^{i,j}\,\M\ra\delta_{k\ell}\right)\\
&=-\sum_{k=1}^{d}u_{k}\la \phi^{i,j}\mathbf{L}_{1}(\phi^{k,\ell}\M)\ra\,,
\end{split}
\end{equation*}
where we used that $\mathbf{L}_{1}(\phi\M)=-\bm{A}\M$ and $\la |v|^{2}\phi^{i,j}\M\ra=0$. This gives the result thanks to \eqref{eq:visco}.\end{proof}
\begin{lem}\label{lem:psiQ} Let $\bm{h}$ be given by \eqref{eq:hlim}.   For any $i=1,\ldots,d$ it holds that
$$\la \psi_{i}\,\Q_{1}(\bm{h},\bm{h})\ra=\frac{d+2}{2}\en_{1}^{3}\,(\vE\,u_{i})\,,$$
and, if $\varrho$ and $\vE$ satisfies Boussinesq relation \eqref{eq:boussi}, then
$$\mathrm{div}_{x}\la \psi_{i}\,\bm{h}\,v\ra=\gamma\frac{d+2}{2}\partial_{x_{i}}\vE.$$
\end{lem}
\begin{proof} On the one hand, using \eqref{eq:Q1sim}  it holds that
$$\la \psi_{i} \Q_{1}(\bm{h},\bm{h})\ra=\vE\,u\cdot \la \psi_{i}\mathbf{L}_{1}(\tfrac{1}{2}|v|^{2}v\M)\ra=\vE\,u\cdot \la \psi_{i}\,\mathbf{L}_{1}(\bm{b}\M)\ra\,,$$
since, $\psi_{i}$ being odd, one has $\la \psi_{i}\mathbf{L}_{1}((u\cdot v)^{2}\M)\ra=\la \psi_{i}\mathbf{L}_{1}(|v|^{4}\M)\ra=0.$ Now, 
$$\la \psi_{i}\,\mathbf{L}_{1}(\bm{b}\M)\ra=\la \bm{b}\,\mathbf{L}_{1}(\psi_{i}\M)\ra=-\la \bm{b}\M\,\bm{b}_{i}\ra\,,$$
and a direct computations show that
$$
\la \bm{b}_{j}\bm{b}_{i}\,\M\ra=-\frac{1}{4d}\la \left(|v|^{2}-(d+2)\en_{1}\right)^{2}|v|^{2}\M\ra\,\delta_{ij}
=-\frac{d+2}{2}\en_{1}^{3}\,\delta_{ij}\,,$$
which gives the expression for $\la \psi_{i}\,\Q_{1}(\bm{h},\bm{h})\ra$. 
On the other hand, using symmetry properties, one checks that 
$$\la \psi_{i}\bm{h}\,v_{\ell}\ra=\varrho\la \psi_{i}v_{i}\M\ra\delta_{i\ell}+\frac{1}{2}\vE\la \psi_{i}(|v|^{2}-d\en_{1})v_{i}\M\ra\,\delta_{i\ell}\,,$$
from which
$$\mathrm{div}_{x}\la \psi_{i}\,\bm{h}\,v\ra=\la \psi_{i}v_{i}\M\ra\,\partial_{x_{i}}\varrho + \frac{1}{2}\la \psi_{i}\big( |v|^{2}-d\en_{1} \big)v_{i}\M\ra\,\partial_{x_{i}}\vE.$$
Writing $\frac{1}{2}\la \psi_{i}(|v|^{2}-d\en_{1})v_{i}\M\ra=\la \psi_{i}\,\bm{b}_{i}\M\ra + \en_{1}\la \psi_{i}\,v_{i}\M\ra$ and using Boussinesq relation~\eqref{eq:boussi}, one gets that
$$\mathrm{div}_{x}\la \psi_{i}\,\bm{h}\,v\ra=\la \psi_{i}\,\bm{b}_{i}\M\ra\,\partial_{x_{i}}\vE=\gamma\frac{d+2}{2}\partial_{x_{i}}\vE\,,$$
where the identity $\la \psi_{i}\,\bm{b}_{i}\M\ra=-\la \psi_{i}\mathbf{L}_{1}(\psi_{i}\M)\ra$ was used together with \eqref{eq:visco}.
\end{proof}

In Lemma~\ref{prop:limitA}, we study the convergence of some term involving the source term $\bm {S}_\e$ defined in~\eqref{eq:source}. To do that, we use the next Lemma which provides a strong convergence to $0$ of this source term.

{\begin{lem} \label{lem:source} Let $\bm{S}_\e$ defined in~\eqref{eq:source}. We have that 
$$
\|\bm{S}_\e\|_{L^1((0,T); {L^1_vL^2_x(\m_{q-1}))}} \lesssim \e.
$$
\end{lem}}
{\begin{proof}
We decompose $\bm{S}_\e$ into three parts using the splitting $h_\e=h^0_\e+h^1_\e$; namely, $\bm{S}_\e=\bm{S}^0_\e + \bm{S}^1_\e+ \bm{S}^2_\e$ with 
$$
\bm{S}^j_{\e}:=\e^{-1}\left(\mathbf{L}_{\re}h^j_{\e}-\mathbf{L}_{1}h^j_{\e}\right)+\Q_{\re}(h^j_{\e},h^j_{\e})-\Q_{1}(h^j_{\e},h^j_{\e})-\e^{-1}\kappa_{\re}\nabla_{v}\cdot (vh^j_{\e}), \quad {j}=0,1,
$$
and
$$
\bm{S}^2_{\e}:=2\widetilde \Q_{\re}(h^0_{\e},h^1_{\e})-2 \widetilde\Q_{1}(h^0_{\e},h^1_{\e}).
$$
The terms $\bm{S}^0_\e$ and $\bm{S}^2_\e$ are treated using the estimate on $h^0_\e$ and $h^1_\e$ stated in ~\eqref{eq:sequ}-\eqref{eq:seqo}. Indeed, using standard estimates on $\Q_\re$ and $\Q_1$ and  {Proposition~\ref{prop:converLLL0}}, we have:
$$
\begin{aligned}
&\|\bm{S}^0_\e+\bm{S}^2_\e\|_{L^1((0,T); {L^1_vL^2_x(\m_{q-1}))}}  \lesssim  {\e \|h^0_\e\|_{L^1((0,T);\E_1)}} \\
&\qquad \quad  + \|h^0_\e\|_{L^1((0,T);\E_1)} \left(\|h_\e^0\|_{L^\infty((0,T);\E)} + \|h_\e^1\|_{L^\infty((0,T);\H)} \right)+ \e \, \|h_\e^0\|_{L^1((0,T);\E_1)} \lesssim \e^2. 
\end{aligned}
$$
Using now~\eqref{eq:sequ}, Proposition~\ref{prop:converLLL0}  {and Remark~\ref{nb:diffQXY}}, we have:
\begin{align*}
&\|\bm{S^1}_\e\|_{L^1((0,T); {L^1_vL^2_x(\m_{q-1}))}}  \lesssim \e \|h^1_\e\|_{L^1((0,T);\H)} \\
&\qquad \quad + \e^2 \|h^1_\e\|_{L^1((0,T);\H)} \|h^1_\e\|_{L^\infty((0,T);\H)} + \e \|h^1_\e\|_{L^1((0,T);\H)} \lesssim \e, 
\end{align*}
which yields the result.
\end{proof}}
To handle the convergence of nonlinear terms, we will need to resort to the following compensated compactness result extracted from \cite{lions-masm} (see also {\cite[Lemma 13.1, Appendix D]{golseSR}}. The original result in {\cite{lions-masm}} is proven in the whole space but is easily adapted to the case of the torus.
\begin{prop}\label{prop:LM}
Let $c \neq 0$ and $T >0$. Consider two families $\{\phi_{\e}\}_{\e}$ and $\{\psi_{\e}\}$ bounded in $L^{\infty}((0,T)\,;L^{2}_{x}(\T^{d}))$ and in $L^{\infty}((0,T)\,;\,\W^{1,2}_{x}(\T^{d}))$ respectively, such that
\begin{equation*}\begin{cases}
\partial_{t}\nabla_{x}\psi_{\e} + \dfrac{c^{2}}{\e}\nabla_{x}\phi_{\e}=\dfrac{1}{\e}F_{\e}\\
\\
\partial_{t}\phi_{\e}+ \dfrac{1}{\e}\Delta_{x}\psi_{\e}=\dfrac{1}{\e}G_{\e}\end{cases}\end{equation*}
where $F_{\e}$ and $G_{\e}$ converge strongly to $0$ in $L^{1}((0,T)\,;\,L^{2}_{x}(\T^{d}))$. Then, 
$$\mathcal{P}\mathrm{Div}_{x}\left(\nabla_{x}\psi_{\e} \otimes \nabla_{x}\psi_{\e}\right) \longrightarrow 0, \qquad \mathrm{div}_{x}\left(\phi_{\e}\nabla_{x}\psi_{\e}\right) \longrightarrow 0$$
in the sense of distributions on $(0,T) \times \T^{d}$.
\end{prop}
 \section{Proof of Theorem \ref{theo:G1e} and Proposition \ref{prop:Bree}}\label{appen:G1e}
\begin{theo}[See Theorem 2.1, {\cite{bmam}}] \label{bria}  {Let $j=1,2$.} There exists $\e_{0} \in (0,1)$ such that, for all $\ell,s \in \N$ with $\ell \geq s$  and  {$q >q_j$ (where $q_j$ is defined in~\eqref{eq:qstar})} and any $\e \in (0,\e_{0})$, the full transport operator $\G_{1,\e}$ generates a $C_{0}$-semigroup $\{\mathcal{V}_{1,\e}(t)\;;\;t \geq 0\}$ on $ {\W^{s,j}_{v}\W^{\ell,2}_{x}(\m_{q_j})}$  such that, for all $t_{\star} >0$ there exist $C_{0}(t_{\star})$, $\mu_{\star} >0$ satisfying
\begin{multline}\label{eq:decay_{0}}
\left\|\mathcal{V}_{1,\e}(t)h-\mathbf{P}_{0}h\right\|_{ {\W^{s,j}_{v}\W^{\ell,2}_{x}}(\m_{q})} \\
\leq C_{0}(t_{\star})\exp(-\mu_{\star}t)\,\|h-\mathbf{P}_{0}h\|_{ {\W^{s,j}_{v}\W^{\ell,2}_{x}}(\m_{q})}\,, \qquad \forall \, t > t_{\star}\,,j=1,2
\end{multline}
holds true for any $h_{0} \in {\W^{s,j}_{v}\W^{\ell,2}_{x}}(\m_{q})$, $j=1,2$ where $\mathbf{P}_{0}$ is the spectral projection onto $\mathrm{Ker}(\G_{1,\e})=\mathrm{Ker}(\mathscr{L}_{1})$ which is \emph{independent of} $\e$ and given by \eqref{eq:P0}.  
\end{theo}
{Note that a similar estimate also holds on the space $\W^{s,1}_{v}\W^{\ell,1}_{x}(\m_{q})$.} The difference between Theorems \ref{theo:G1e} and \ref{bria} lies in the fact that, in Theorem \ref{theo:G1e}, we allow $t_{\star}=0$ in the decay estimate \eqref{eq:decay_{0}}.   The ``initial layer'' dependence on $t_{\star} >0$ in \eqref{eq:decay_{0}} is inherent to the method of the enlargement semigroup theory of {\cite{GMM}} {(see Remarks~\ref{nb:G1e} and~\ref{nb:G1e2})}.

\smallskip
\noindent
Theorem \ref{bria} ensures that $\G_{1,\e}$ is the generator of a $C_{0}$-semigroup $\{\mathcal{V}_{1,\e}(t)\;;\;t\geq0\}$ on $\E$ as soon as $q >\frac{9}{2}$ and $\e \in (0,\e_{0})$. We focus on extending \eqref{eq:decay_{0}} to $t_{\star}:=0$.
\begin{proof}[Proof of Theorem \ref{theo:G1e}] We adopt the decomposition of the nonlinear part of {\cite{bmam}} that we used in Section \ref{sec:non}. Namely, for some \emph{fixed}
$h \in  {\W^{s,2}_{v}\W^{\ell,2}_{x}}(\m_{q})$ we set
$$f_{\mathrm{in}}:=h-\mathbf{P}_{0}h\,,$$
and write $f(t)=\mathcal{V}_{1,\e}(t)f_{\mathrm{in}}$ as $f(t)=f^{0}(t)+f^{1}(t)$
with $f^{0} \in {\W^{s,2}_{v}\W^{\ell,2}_{x}}(\m_{q})$ solution to
\begin{equation}\label{eq:2-1}
\partial_{t}f^{0}(t)=\mathcal{B}_{1,\e}f^{0}, \qquad f^{0}(0)=f_{\mathrm{in}}\,,
\end{equation}
whereas $f^{1} \in \H:=\W^{\ell,2}_{v,x}(\M^{-\frac{1}{2}})$,  is solution to
\begin{equation}\label{eq:2-2}
\partial_{t}f^{1}(t)=\G_{1,\e}f^{1}(t)+\mathcal{A}_{\e}f^{0}(t), \qquad f^{1}(0)=0.
\end{equation}
As before, the same notations for the operators $\G_{1,\e},$ $\mathcal{V}_{1,\e}(t)$ acting on various different spaces is used. The definition should be clear from the context. Of course,
$$f^{0}(t)=\mathcal{S}_{1,\e}(t)f_{ \mathrm{in}}\,,$$
and
\begin{equation}\label{eq:2-3}
\|f^{0}(t)\|_{ {\W^{s,2}_{v}\W^{\ell,2}_{x}}(\m_{q})} \leq C_{0}\,\exp(-\e^{-2}\nu_{0}t)\|f_{\mathrm{in}}\|_{ {\W^{s,2}_{v}\W^{\ell,2}_{x}}(\m_{q})}\,,
\end{equation}
since $\mathcal{B}_{1,\e}$ is $\e^{-2}\nu_{0}$ hypo-dissipative ($\nu_{0}$ depends on $\ell,m$). The constant $C_{0}$ is independent of~$\e$.  Let us investigate $\|f^{1}(t)\|_{\H}$. Notice that, since $\mathbf{P}_{0}f=0$, $\mathbf{P}_{0}f^{1}=-\mathbf{P}_{0}f^{0}$ (recall that the projection is the same in $\H$ and $ {\W^{s,2}_{v}\W^{\ell,2}_{x}}(\m_{q})$ and independent of $\e$), the estimate for $\mathbf{P}_{0}f^{1}$ is thus straightforward
\begin{equation} \label{eq:P0f1}
\|\mathbf{P}_{0}f^{1}(t)\|_{\H} \leq C_{1}\,\exp(-\e^{-2}\nu_{0}t)\|f_{\mathrm{in}}\|_{ {\W^{s,2}_{v}\W^{\ell,2}_{x}}(\m_{q})}\,,
\end{equation}
where the constant $C_{1}$ differs from $C_{0}$ just because the norm of the eigenfunctions are different in $\H$ and $ {\W^{s,2}_{v}\W^{\ell,2}_{x}}(\m_{q})$.  We now focus on 
$$\psi(t)=\mathbf{P}_{0}^{\perp}f^{1}(t)=(\mathbf{Id-P}_{0})f^{1}(t).$$ One has
$$\partial_{t}\psi(t)=\G_{1,\e}\psi(t)+\mathbf{P}_{0}^{\perp}\mathcal{A}_{\e}f^{0}(t)\,,$$
{and, arguing as in \cite[Section 7.2]{briant} (see also \cite[Theorem 4.7 and Remark~4.8]{bmam}), one has
$$\frac{1}{2}\|\psi(t)\|_{\H}^{2} \leq \frac{1}{2}\|\psi(0)\|_{\H}^{2}e^{- \mu_{\star} \,t}+\int_{0}^{t}e^{- \mu_{\star} (t-s)}\|\psi(s)\|_{\H}\,\|\mathbf{P}_{0}^{\perp}\A_{\e}\,f^{0}(s)\|_{\H}\d s$$ with $\mu_{\star} >0$ independent of $\e$ which is the size of the spectral gap of $\G_{1,\e}$. Recalling that $\psi(0)=0$ and  $\|\mathcal{A}_{\e}\|_{\mathscr{B}( {\W^{s,2}_{v}\W^{\ell,2}_{x}}(\m_{q}),\H)} \leq C_{A}\e^{-2}$, we get that
$$\|\psi(t)\|_{\H}^{2} \leq \frac{2C_{A}}{\e^{2}}\int_{0}^{t}e^{-\mu_{\star}(t-s)}\|\psi(s)\|_{\H}\|f^{0}(s)\|_{ {\W^{s,2}_{v}\W^{\ell,2}_{x}}(\m_{q})}\d s.$$}
We use \eqref{eq:2-3} to deduce that
$$\|\psi(t)\|_{\H}^{2} \leq \frac{2C_{0}C_{A}}{\e^{2}}\int_{0}^{t}e^{-\mu_{\star}(t-s)}e^{-\frac{\nu_{0}}{\e^{2}}s}\|\psi(s)\|_{\H}\|f_{\mathrm{in}}\|_{ {\W^{s,2}_{v}\W^{\ell,2}_{x}}(\m_{q})}\d s.$$
Then, Young's inequality leads to
\begin{multline*}
\|\psi(t)\|_{\H}^{2} \leq \frac{C_{0}C_{A}e^{-\mu_{\star}\,t}}{\e^{2}}\int_{0}^{t}e^{-(\frac{\nu_{0}}{\e^{2}}-\mu_{\star})s}\|\psi(s)\|_{\H}^{2}\d s \\
+ \frac{C_{0}C_{A}e^{-\mu_{\star}\,t}}{\e^{2}}\|f_{\mathrm{in}}\|_{ {\W^{s,2}_{v}\W^{\ell,2}_{x}}(\m_{q})}^{2}
\int_{0}^{t}e^{-(\frac{\nu_{0}}{\e^{2}}-\mu_{\star})s}\d s.
\end{multline*}
If $\e^{-2}\nu_{0} > 2\mu_{\star}$  we get after integration that
$$\|\psi(t)\|_{\H}^{2} \leq \frac{2C_{0}C_{A}e^{-\mu_{\star}\,t}}{\nu_{0}}\|f_{\mathrm{in}}\|_{ {\W^{s,2}_{v}\W^{\ell,2}_{x}}(\m_{q})}^{2} + \frac{C_{0}C_{A}e^{-\mu_{\star}\,t}}{\e^{2}}\int_{0}^{t}e^{-(\frac{\nu_{0}}{\e^{2}}-\mu_{\star})s}\|\psi(s)\|_{\H}^{2}\d s.$$
With $x(t)=e^{\mu_{\star}\,t}\|\psi(t)\|_{\H}^{2}$ it follows that
$$x(t) \leq \frac{2C_{0}C_{A}}{\nu_{0}}\|f_{\mathrm{in}}\|_{ {\W^{s,2}_{v}\W^{\ell,2}_{x}}(\m_{q})}^{2} + \frac{C_{0}C_{A}}{\e^{2}}\int_{0}^{t}e^{-\frac{\nu_{0}}{\e^{2}}s}x(s)\d s\,,$$
and Gronwall's lemma gives 
$$x(t) \leq \frac{2C_{0}C_{A}}{\nu_{0}}\|f_{\mathrm{in}}\|_{ {\W^{s,2}_{v}\W^{\ell,2}_{x}}(\m_{q})}^{2}\exp\left(\frac{C_{0}C_{A}}{\nu_{0}}\right)=:C_{2}\|f_{\mathrm{in}}\|_{ {\W^{s,2}_{v}\W^{\ell,2}_{x}}(\m_{q})}^{2}\,,$$
with $C_{2} >0$ independent of $\e$. Therefore
$$\|\psi(t)\|_{\H}^{2} \leq C_{2}\|f_{\mathrm{in}}\|_{ {\W^{s,2}_{v}\W^{\ell,2}_{x}}(\m_{q})}^{2}\exp(-\mu_{\star}\,t).$$
This combined with \eqref{eq:P0f1} gives that
$$\|f^{1}(t)\|_{\H}^{2} \leq (C_{2}+C_{1})\|f_{\mathrm{in}}\|_{ {\W^{s,2}_{v}\W^{\ell,2}_{x}}(\m_{q})}^{2}\exp(-\mu_{\star}\,t).$$
Overall, the estimates for $f^{0}$ and $f^{1}$ lead to
$$\|f(t)\|_{\E} \leq C_{3}\|f_{\mathrm{in}}\|_{ {\W^{s,2}_{v}\W^{\ell,2}_{x}}(\m_{q})}\exp\left(-\frac{\mu_{\star}}{2}t\right), \qquad \forall\, t \geq0\,,$$
with $C_{3}$ independent of $\e$ and given by $C_{0}+\sqrt{C_{1}+C_{2}}$ as long as $\nu_{0}\e^{-2} > \mu_{\star}$.\end{proof}
\begin{proof}[Proof of Proposition \ref{prop:Bree}] We aim to prove here that, on the Banach space ${\W^{s,j}_{v}\W^{\ell,2}_{x}}(\m_{q})$, $j=1,2$, the operator
$$\mathcal{B}_{\re,\e}^{(\delta)}=\mathcal{B}_{1,\e}^{(\delta)}+\e^{-2}\mathcal{P}_{\re}+\e^{-2}T_{\re}$$
with domain $\D(\mathcal{B}_{\re,\e}^{(\delta)})=\W^{s+1,j}_{v}\W^{\ell+1,2}_{x}(\m_{q+1})$ generates a $C_{0}$-semigroup. Since we will resort to an approach introduced in \cite{ABL} and some computations made earlier in  \cite{AC}, it will be more convenient to prove that $\mathcal{B}_{\re,\e}^{(\delta)}$ is the generator of a $C_{0}$-semigroup on the space
$$\W^{s,1}_{v}\W^{\ell,2}_{x}(\overline{m}_{a,\beta})$$
where $\overline{m}_{a,\beta}$ is the exponential weight
$$\overline{m}_{a,\beta}(v)=\exp\left(a\langle v\rangle^{\beta}\right), \qquad v \in \R^{d}, \quad a >0, \beta \in (0,1)$$ and then use some enlargement result \cite[Theorem 2.13]{GMM}. It is easy to adapt the proof of Proposition \ref{prop:hypo} and find an equivalent norm on ${\W^{s,1}_{v}\W^{\ell,2}_{x}}(\overline{m}_{a,\beta})$ for which $\mathcal{B}_{\re,\e}^{(\delta)}+\e^{-2}\nu_{\ell,s}$ is dissipative (see \cite{Tr} for useful computations in spaces with exponential weights whenever $\e=1$). According to Lumer-Phillips Theorem, see \cite[Proposition 3.14 \& Theorem 3.15]{engel}, in order to show that $\mathcal{B}_{\re,\e}^{(\delta)}$ generates a $C_{0}$-semigroup it suffices to prove that there exists $\lambda >0$ large enough such that 
\begin{equation}\label{eq:range}
\mathrm{Range}(\lambda-\mathcal{B}_{\re,\e}^{(\delta)})={\W^{s,1}_{v}\W^{\ell,2}_{x}}(\overline{m}_{a,\beta}).\end{equation}
Clearly, one can replace without loss of generality $\mathcal{B}_{\re,\e}^{(\delta)}$ with $\e^{2}\mathcal{B}_{\re,\e}^{(\delta)}$.  Denote for simplicity
$$X={\W^{s,1}_{v}\W^{\ell,2}_{x}}(\bm{\mu}_{a}), \qquad B_{\re}:=\e^{2}\mathcal{B}_{\re,\e}^{(\delta)}\,,$$
omitting the dependence with respect to $\e$ and $\delta$.  It follows that
$$B_{\re}=\mathscr{L}_{1}^{{R,\delta}}- \Sigma_\M-\e v\cdot \nabla_{x}\, + \mathcal{P}_{\re}+ T_{\re}\,.$$
Introduce the following operator
$$\mathcal{T}_{\re}h:=-\e v\cdot \nabla_{x}h + T_{\re}h-\Sigma_\M\,h=-\e v \cdot \nabla_{x}h- \kappa_{\alpha}\mathrm{div}_{v}(vh)-\Sigma_{\M}h$$
with domain $\D(\mathcal{T}_{\re})=\W_{v}^{s+1,1}\W^{\ell+1,2}_{x}(\langle \cdot \rangle\overline{m}_{a,\beta}).$
It is not difficult to check that $\mathcal{T}_{\re}$ generates a $C_{0}$-semigroup in $X$ given by
$$e^{t\mathcal{T}_{\re}}g(x,v)=\exp\left(-\int_{0}^{t}d\kappa_{\re}+\Sigma_{\M}(ve^{\kappa_{\re}(s-t)})\d s\right)g\left(x-\frac{\e}{\kappa_{\re}}\left(1-e^{-\kappa_{\re}t}\right)v, ve^{-\kappa_{\re}t}\right).$$
In particular,
\begin{equation}\label{eq:resoT}
\lim_{\lambda\to\infty}\|\Rs(\lambda,\mathcal{T}_{\re})\|_{\mathscr{B}(X)}=0.
\end{equation}
Moreover, one has the following gain of integrability for the resolvent of $\mathcal{T}_{\re}$: there is $\re_{1}\in (0,1)$ such that, for $\re \in (\re_{1},1)$ there is $c >0$ and $\lambda(\re) >0$ 
\begin{equation}\label{eq:gainT}
\left\|\Rs(\lambda,\mathcal{T}_{\re})\right\|_{\mathscr{B}({\W^{s,1}_{v}\W^{\ell,2}_{x}}(\overline{m}_{a,\beta}),{\W^{s,1}_{v}\W^{\ell,2}_{x}}(\langle \cdot\rangle \overline{m}_{a,\beta}))} \leq \frac{1}{\sigma-c\kappa_{\re}}\,, \qquad \forall\, \lambda > \lambda(\re)\,,\end{equation}
where $\sigma$ is an explicit positive constant depending only on $\Sigma_{\M}$. The proof of such a property is an easy adaptation of Lemma C.14 in \cite{ABL}
whenever $k=s=0$ and extends to $k\geq s\geq0$ following techniques from {\cite{neuman}}, we leave the details to the reader. One also has the following result, see the proof of \cite[Lemma B.1 \& Proposition B.2]{AC}: there exists $\tau(\delta) >0$ such that $\lim_{\delta\to0}\tau(\delta)=0$ and
\begin{equation}\label{eq:L1+}\left\|\mathscr{L}_{1,+}^{R,\delta}\right\|_{\mathscr{B}(\W^{s,1}_{v}\W^{\ell,2}_{x}(\langle \cdot\rangle \overline{m}_{a,\beta}),\W^{s,1}_{v}\W^{\ell,2}_{x}(\overline{m}_{a,\beta}))} \leq \tau(\delta)\,,\end{equation}
while $\mathscr{L}_{1,-}^{R,\delta} \in \mathscr{B}(X)$.  
With these two properties, introduce the sum $\mathcal{C}_{\re}:=\mathscr{L}_{1,+}^{R,\delta}+\mathcal{T}_{\re}$ with domain $\D(\mathcal{C}_{\re})=\D(\mathcal{T}_{\re})$. We have directly from the previous two properties \eqref{eq:gainT} and \eqref{eq:L1+}
$$\left\|\mathscr{L}_{1,+}^{R,\delta}\Rs(\lambda,\mathcal{T}_{\re})\right\|_{\mathscr{B}(X)} \leq \frac{\tau(\delta)}{\sigma-c\kappa_{\re}}, \qquad \forall \lambda >\lambda(\re)\,,$$
from which, choosing $\delta >0$ sufficiently small such that $\frac{\tau(\delta)}{\sigma-c\kappa_{\re}}< 1$, we obtain that $(\mathbf{Id}-\mathscr{L}_{1,+}^{R,\delta}\Rs(\lambda,\mathcal{T}_{\re}))$ is invertible.  We deduce that
\begin{equation*}\begin{split}\Rs(\lambda,\mathcal{C}_{\re})&=\Rs(\lambda,\mathcal{T}_{\re})\left(\mathbf{Id}-\mathscr{L}_{1,+}^{R,\delta}\Rs(\lambda,\mathcal{C}_{\re})\right)^{-1}\\
&=\Rs(\lambda,\mathcal{T}_{\re})\sum_{n=0}^{\infty}\left[\mathscr{L}_{1,+}^{R,\delta}\Rs(\lambda,\mathcal{T}_{\re})\right]^{n}\,, \qquad \forall\, \lambda > \lambda(\re)\,,
\end{split}
\end{equation*}
simply observing that $(\lambda-\mathcal{C}_{\re})=(\mathbf{Id}-\mathscr{L}_{1,+}^{R,\delta}\Rs(\lambda,\mathcal{C}_{\re}))(\lambda-\mathcal{T}_{\re})$. In particular,
$$\left\|\Rs(\lambda,\mathcal{C}_{\re})\right\|_{\mathscr{B}(X)} \leq \frac{1}{\sigma-c\kappa_{\re}-\tau(\delta)}, \qquad \forall \lambda >\lambda(\re)\,,$$
with
$$\lim_{\lambda\to \infty}\left\|\Rs(\lambda,\mathcal{C}_{\re})\right\|_{\mathscr{B}(X)}=0$$
by virtue of \eqref{eq:resoT}. Set then 
$$\mathcal{C}_{\re}^{1}:=\mathcal{C}_{\re}+\mathcal{P}_{\re}.$$ 
With the estimate of $\mathcal{P}_{\re}$
$$\left\|\mathcal{P}_{\re}\Rs(\lambda,\mathcal{C}_{\re})\right\|_{\mathscr{B}(X)} \leq C\frac{1-\re}{\sigma-c\kappa_{\re}-\tau(\delta)}$$
and, choosing $\alpha$ sufficiently close to $1$, the operator $\mathbf{Id}+\mathcal{P}_{\re}\Rs(\lambda,\mathcal{C}_{\re})$ is invertible and so is $\lambda-\mathcal{C}_{\re}^{1}$.  Finally, since
$$B_{\re}=\mathcal{C}_{\re}^{1}-\mathscr{L}_{1,-}^{R,\delta}\,,$$ 
one can chose $\lambda >0$ sufficiently large so that 
$$\|\mathscr{L}_{1,-}^{R,\delta}\Rs(\lambda,\mathcal{C}_{\re}^{1})\|_{\mathscr{B}(X)} \leq \|\mathscr{L}_{1,-}^{R,\delta}\|_{\mathscr{B}(X)}\,\|\Rs(\lambda,\mathcal{C}_{\re}^{1})\|_{\mathscr{B}(X)} <1$$ and obtain that $\lambda-B_{\re}$ is invertible. In particular, \eqref{eq:range} holds true and this proves the result on the space $X$. As said before, we deduce that $\mathcal{B}_{\re,\e}$ still generates a $C_{0}$-semigroup on the larger spaces $\W^{s,1}_{v}\W^{\ell,2}_{x}(\m_{q})$ and $\W^{s,2}_{v}\W^{\ell,2}_{x}(\m_{q})$ under the assumptions thanks to Theorem 2.13 and Remark 2.14 (i) of \cite{GMM}. We leave the technical details to the reader.
\end{proof}

\bibliographystyle{plainnat-linked}

\end{document}